\documentclass{article}
\usepackage[left=4cm, right=4cm, top=4cm, bottom=2.25cm]{geometry}
\usepackage{amsmath,color}
\usepackage{amsfonts,amsthm}
\usepackage{amssymb,enumerate,enumitem,verbatim}
\usepackage[pdftex]{graphicx}

\usepackage{amsmath,setspace,scalefnt}
\usepackage[usenames,dvipsnames,svgnames,table]{xcolor}
\usepackage{amsfonts,amsthm}
\usepackage{amssymb,enumitem,verbatim}
\usepackage{graphicx}
\usepackage{tikz,caption,subcaption,thm-restate}
\usetikzlibrary{calc}
\usepackage[titles]{tocloft}
\usepackage[symbol]{footmisc}
\usepackage{hyperref}

\newcounter{capitalcounter}

\newcounter{claimcounter}

\newcounter{myfootnote}[page]

\newcommand{\Footnote}[1]{\stepcounter{myfootnote}\footnote{#1}}

\newtheorem{lemma}{Lemma}[section]
\newtheorem{corollary}[lemma]{Corollary}
\newtheorem{theorem}[lemma]{Theorem}

\newtheorem{prop}[lemma]{Proposition}
\newtheorem{conjecture}[lemma]{Conjecture}

\theoremstyle{definition}
\newtheorem{defn}[lemma]{Definition}

\newtheorem{claim}{Claim}

\theoremstyle{remark}
\newtheorem*{rmk}{Remark}

\newtheorem*{example}{Example}
\newtheorem*{nte}{Note}

\global\long\def\eps{\varepsilon}

\global\long\def\N{\mathbb{N}}
\global\long\def\R{\mathbb{R}}

\global\long\def\P{\mathbb{P}}

\global\long\def\E{\mathbb{E}}

\global\long\def\re{\begin{rmk}}
\global\long\def\mark{\end{rmk}}
\global\long\def\ex{\begin{example}}
\global\long\def\ple{\end{example}}
\global\long\def\no{\begin{nte}}
\global\long\def\ted{\end{nte}}
\global\long\def\en{\begin{compactenum}}
\global\long\def\um{\end{compactenum}}
\global\long\def\li{\begin{compactitem}}
\global\long\def\st{\end{compactitem}}
\global\long\def\de{\begin{defn}}
\global\long\def\fn{\end{defn}}
\global\long\def\cor{\begin{corollary}}
\global\long\def\ary{\end{corollary}}
\global\long\def\lem{\begin{lemma}}
\global\long\def\ma{\end{lemma}}
\global\long\def\arr{\begin{array}}
\global\long\def\ay{\end{array}}
\global\long\def\pr{\begin{proof}}
\global\long\def\oof{\end{proof}}

\global\long\def\gapp{\vspace{1cm}\indent}

\newcounter{propcounter}

\newcommand{\ggpoly}{\stackrel{\scriptscriptstyle{\text{\sc poly}}}{\gg}}
\newcommand{\llpoly}{\stackrel{\scriptscriptstyle{\text{\sc poly}}}{\ll}}
\newcommand{\itref}[1]{\emph{\ref{#1}}}
\newcommand{\eref}[1]{\emph{\ref{#1}}}

\newif\ifpicfirst
\picfirsttrue
\newif\ifpicsecond
\picsecondtrue
\newif\ifpicthird
\picthirdtrue

\newif\ifboxpicfirst
\newif\ifboxpicsecond
\newif\ifboxpicthird
\newif\ifboxpicfourth
\newif\ifboxpicfifth
\newif\ifcolboxpicfirst
\newif\ifcolboxpicsecond
\newif\ifcolboxpicthird
\newif\ifcolboxpicfourth
\newif\ifcolboxpicfifth

\title{A proof of the Ryser-Brualdi-Stein conjecture for large even $n$}
\author{Richard\ Montgomery\thanks{Mathematics Institute, University of Warwick, Coventry, CV4 7AL, UK.
richard.montgomery@warwick.ac.uk. Supported by the European Research Council (ERC) under the European Union Horizon 2020 research and innovation programme (grant agreement No.\ 947978) and the Leverhulme trust.}}

\begin{document}

\maketitle

\begin{abstract}
A Latin square of order $n$ is an $n$ by $n$ grid filled using $n$ symbols so that each symbol appears exactly once in each row and column. A transversal in a Latin square is a collection of cells which share no symbol, row or column.
The Ryser-Brualdi-Stein conjecture, with origins from 1967, states that every Latin square of order $n$ contains a transversal with $n-1$ cells, and a transversal with $n$ cells if $n$ is odd.
Keevash, Pokrovskiy, Sudakov and Yepremyan recently improved the long-standing best known bounds towards this conjecture by showing that every Latin square of order $n$ has a transversal with  $n-O(\log n/\log\log n)$ cells.
Here, we show, for sufficiently large $n$, that every Latin square of order $n$ has a transversal with $n-1$ cells.

We also apply our methods to show that, for sufficiently large $n$, every Steiner triple system of order $n$ has a matching containing at least $(n-4)/3$ edges. This improves a recent result of Keevash, Pokrovskiy, Sudakov and Yepremyan, who found such matchings with $n/3-O(\log n/\log\log n)$ edges, and proves a conjecture of Brouwer from 1981 for large $n$.
\end{abstract}

\newpage

{\tableofcontents}

\newpage


\section{Introduction}\label{sec:intro}
 The study of transversals in Latin squares dates back at least to the 18th century when Euler considered Latin squares which can be decomposed into full transversals~\cite{OGeuler}. A \emph{Latin square of order $n$} is an $n$ by $n$ grid filled with $n$ symbols, so that every symbol appears exactly once in each row and column. A \emph{transversal} of a Latin square of order $n$ is a collection of cells in the grid which share no row, column or symbol, while a \emph{full transversal} is a transversal with $n$ cells. For more background on Latin squares, see the surveys by Andersen~\cite{andersen2007history}, Wanless~\cite{WanlessSurvey}, and the current author~\cite{Mysurvey}.

Key examples of Latin squares include the multiplication tables of finite groups, which easily provide examples that, if $n$ is even, then there are Latin squares of order $n$ with no full transversal (e.g., the multiplication table for $\mathbb{Z}_2$). In 1967, Ryser~\cite{Ryser} conjectured that no such Latin square of order $n$ exists when $n$ is odd (see also~\cite{Rysertranslation}). Brualdi (see~\cite{Brualdi}) later conjectured that every Latin square of order $n$ has a transversal with $n-1$ cells, while Stein~\cite{Stein} made some related, stronger, conjectures in the 1970's. The following combined  conjecture has become known as the Ryser-Brualdi-Stein conjecture and is the most significant open problem on transversals in Latin squares.

\begin{conjecture} [The Ryser-Brualdi-Stein conjecture]\label{conj:RBS}
Every Latin square of order $n$ has a transversal with $n-1$ cells, and a full transversal if $n$ is odd.
\end{conjecture}

Towards Conjecture~\ref{conj:RBS}, increasingly large transversals were shown to exist in any Latin square by Koksma~\cite{koksma}, and Drake~\cite{drake}, before Brouwer, De Vries and Wieringa~\cite{brouwer1978lower} and Woolbright~\cite{woolbright} independently showed that every Latin square of order $n$ has a transversal with at least $n-\sqrt{n}$ cells. In 1982, Shor~\cite{shor} showed that a transversal with $n-O(\log^2 n)$ cells exists in any Latin square of order $n$, though the proof had an error that was only noticed and corrected by Hatami and Shor in 2008~\cite{hatamishor}. This bound stood until the breakthrough work of Keevash, Pokrovskiy, Sudakov and Yepremyan~\cite{KPSY} in 2020, which showed that every Latin square of order $n$ has a transversal with $n-O(\log n/\log\log n)$ elements. Here the bound $O(\log n/\log\log n)$ on the missing elements is a natural barrier, and it seems likely this is the best bound achievable with methods that approach each Latin square in the same manner.

In this paper, we introduce the first techniques to identify and exploit the possible algebraic properties behind the entries in a Latin square. This will allow us to find transversals missing at most one symbol in large Latin squares, as follows.

\begin{theorem}\label{thm:RBSeven}
There is some $n_0\in \N$ such that every Latin square of order $n\geq n_0$ contains a transversal with $n-1$ cells.
\end{theorem}


As noted above, the multiplication tables of finite groups give important examples of Latin squares, and even in this particular case the Ryser-Brualdi-Stein conjecture is very difficult. For these examples, Hall and Paige~\cite{hallpaige} considered when a full transversal should exist in such a Latin square corresponding to the multiplication table of a finite group, conjecturing in 1955 that it is exactly when the 2-Sylow subgroups of the corresponding group are trivial or non-cyclic (and, thus, in particular, such a transversal should always exist in a group of odd order). This conjecture was eventually confirmed through a combination of work by Wilcox, Evans, and Bray~\cite{wilcox,evans}, completed in 2009, using computer algebra and the classification of finite simple groups. For large groups, an alternative proof of the conjecture (along with sharp asymptotics for the number of transversals) was given by Eberhard, Manners and Mrazovi\'c~\cite{greenalites} using tools from analytic number theory. Very recently, a combinatorial proof of the Hall-Paige conjecture was given for large groups by M\"uyesser and Pokrovskiy~\cite{muyesser2022random}, as part of a more general result.

\medskip

A \emph{generalised Latin square}, or \emph{Latin array}, of order $n$ is an $n$ by $n$ grid filled with symbols so that every symbol appears at most once in each row and each column (thus allowing more than $n$ symbols). In~\cite{montgomery2018decompositions}, the current author showed with Pokrovskiy and Sudakov that any large Latin array which is sufficiently far from a Latin square has a full transversal (see Theorem~\ref{thm-farnoworry}). For large~$n$, the Latin arrays not covered by the results in~\cite{montgomery2018decompositions} are close enough to Latin squares to apply the techniques introduced here for Theorem~\ref{thm:RBSeven}, allowing the following generalisation of Theorem~\ref{thm:RBSeven} without much additional work.

\begin{theorem}\label{thm:generalLS}
There is some $n_0\in \N$ such that every Latin array of order $n\geq n_0$ contains a transversal with $n-1$ cells.
\end{theorem}

Theorem~\ref{thm-farnoworry}, as quoted from~\cite{montgomery2018decompositions}, shows that Latin arrays of order $n$ in which at most $(1-o(1))n$ symbols appear more than $(1-o(1))n$ times have a full transversal. Another possible condition forcing Latin arrays to have a full transversal was suggested by Akbari and Alipour~\cite{akbari2004transversals}, who conjectured that any Latin array of order $n$ containing at least $n^2/2$ different symbols should contain a full transversal. This was confirmed for large $n$ in a strong sense by the results of~\cite{montgomery2018decompositions}, as well as in independent work by Keevash and Yepremyan~\cite{keevash2020number} who showed the stronger bound that, for large $n$, Latin arrays of order $n$ with at least $n^{399/200}$ different symbols contain a full transversal. In their recent work, Keevash, Pokrovskiy, Sudakov and Yepremyan~\cite{KPSY} improved this much further, showing that $O(n\log n/\log\log n)$ different symbols can suffice to force a full transversal. Using our techniques for Theorem~\ref{thm:RBSeven}, we will show that in fact $O(n)$ different symbols can suffice.

\begin{theorem}\label{thm:symbolnum}
There is some $n_0\in \N$ such that every Latin array of order $n\geq n_0$ with at least $250n$ different symbols contains a full transversal.
\end{theorem}

The constant 250 in Theorem~\ref{thm:symbolnum} could be lowered using the techniques in this paper, but doing so would not give a likely optimal constant. Considering the known extremal examples for Theorem~\ref{thm:RBSeven}, it is feasible that Theorem~\ref{thm:symbolnum} could hold with $250n$ replaced even by $n+1$.

The Ryser-Brualdi-Stein conjecture has a natural expression as a hypergraph matching problem, and, similarly, we can apply our methods for Theorem~\ref{thm:RBSeven} to find large matchings in Steiner triple systems.
 A \emph{Steiner triple system of order $n$} is an $n$-vertex 3-uniform hypergraph in which each pair of vertices is in exactly one edge, while a \emph{matching} is a set of edges which share no vertices. Steiner triple systems are a type of \emph{design}, a general combinatorial object whose study dates back to the 19th century. While the existence of designs in general was not proved until the famous result of Keevash from 2014~\cite{keevash2014existence}, Steiner triple systems were shown to exist for each $n\equiv 1,3\mod 6$ by Kirkman in 1847~\cite{kirkman}. Simple divisibility conditions show that no Steiner triple system of order $n$ exists if $n\not\equiv 1,3\mod 6$. In 1981, Brouwer~\cite{brouwer1981size} conjectured that Steiner triple systems should contain matchings missing at most 4 vertices, as follows.

\begin{conjecture}[Brouwer]\label{conj:brouwer}
Every Steiner triple system of order $n$ contains a matching with at least $(n-4)/3$ edges.
\end{conjecture}

Due to constructions of Wilson (see~\cite{colbourn1992directed}) and Bryant and Horsley~\cite{bryant2013second,bryant2015steiner}, it is known that Brouwer's conjecture would be tight for infinitely many values of $n$. Towards the conjecture, increasingly large matchings  were shown to exist in any Steiner triple system by Wang~\cite{wang}, Lindner and Phelps~\cite{lindner1978note} and then Brouwer~\cite{brouwer1981size}, before Alon, Kim and Spencer~\cite{AKS} showed that any Steiner triple system of order $n$ has a matching containing all but $O(n^{1/2}\log^{3/2}n)$ vertices. Keevash, Pokrovskiy, Sudakov and Yepremyan applied their methods in~\cite{KPSY} to drastically reduce this to find matchings containing all but $O(\log n/\log\log n)$ vertices, by translating the problem to one on rainbow matchings in pseudorandom bipartite graphs. Following this translation, our methods require a little modification, but allow us to prove Brouwer's conjecture for large Steiner triple systems, as follows.

\begin{theorem}\label{thm:brouwer} There is some $n_0\in \N$ such that every Steiner triple system of order $n\geq n_0$ contains a matching with at least $(n-4)/3$ edges.
\end{theorem}

To prove our results, we use a combination of the semi-random method and the absorption method. Since its codification in 2006 as a general approach by R\"odl, Ruci\'nski and Szemer\'edi~\cite{rodl2006dirac}, absorption has been a critical tool in turning approximate results into exact results. That is, we aim to find a transversal with $(1-o(1))n$ cells in a Latin square of order $n$ using known methods (in particular, here, the semi-random method), before using the absorption method to turn this into a transversal with $n$ cells.
However, the extremal examples showing a full transversal may not always exist (when $n$ is even) demonstrate the challenge of using the absorption method in this setting. In these examples, the same algebraic properties behind the entries in the Latin square prevent the existence of the typical absorbers used for an application of the absorbing method. Latin squares arising as the multiplication tables of groups are good examples of such algebraic properties, but other extremal examples show these properties can be more complicated still (see Section~\ref{sec:alg}).

To prove Theorem~\ref{thm:RBSeven}, we introduce the first methods to identify and exploit algebraic properties of the entries of Latin squares.
We use this to construct an `absorption structure' with a natural but limited absorption property. We then introduce an `addition structure' and use this to strengthen the properties of the absorption structure. In combination, the absorption structure and addition structure will allow us to adjust a transversal with $(1-o(1))n$ cells to one with $n-1$ cells. Each of the main three parts of our proof (identifying the algebraic properties, creating and using the absorption structure, and creating and using the addition structure) carries significant novelty and we sketch each of these in detail in Section~\ref{subsec:discuss}.

\smallskip

In this paper, we will work throughout with the now-standard equivalent formulation of Latin squares as properly coloured complete bipartite graphs. In Section~\ref{sec:expo}, we recall this formulation and state some basic notation, before giving a brief snapshot of the proof for readers very familiar with techniques in the area and then discussing our proofs in more detail. The rest of the paper is outlined at the start of Section~\ref{sec:prelim}, but we finish in Section~\ref{sec:final} by discussing the potential of the methods introduced here, the difficulty of finding full transversals, and further related problems.


\section{Exposition}\label{sec:expo}
As is often standard (see, for example,~\cite{KPSY,montgomery2018decompositions}) for the results in Section~\ref{sec:intro}, we will work in the equivalent setup of rainbow matchings in properly coloured bipartite graphs. An edge colouring is \emph{proper} if no pair of edges of the same colour share a vertex, an \emph{optimal colouring} is a proper colouring which uses the minimum number of colours, and a graph is \emph{rainbow} if each edge in the graph has a different colour. Given a Latin square $L$ of order $n$, we form a complete bipartite edge-coloured graph $G(L)$ by creating a vertex for each row and each column, and for each row/column pair putting an edge between the corresponding vertices whose `colour' is the symbol in the cell of the Latin square in that row and column. The matchings in $G(L)$ correspond exactly to the sets of cells in $L$ sharing no row or column, and the rainbow matchings in $G(L)$ correspond exactly to the transversals in $L$. Similarly, given an optimal colouring of the complete bipartite graph with $n$ vertices in each class, $K_{n,n}$, we can create a corresponding Latin square.

To show Theorem~\ref{thm:RBSeven}, we thus need to demonstrate that, for sufficiently large $n$, any optimally coloured $K_{n,n}$ has a rainbow matching with at least $n-1$ edges. For Theorems~\ref{thm:generalLS} and~\ref{thm:symbolnum} we use a similar translation into coloured graphs (potentially using more than $n$ colours). Transforming the problem of finding matchings in Steiner triple systems into a rainbow matching problem for the proof of Theorem~\ref{thm:brouwer} is a little more involved, and here we follow Keevash, Pokrovskiy, Sudakov and Yepremyan~\cite{KPSY} (See Section~\ref{sec:brouwer}). This results not in a complete bipartite graph, but in a dense bipartite graph with certain `pseudorandomness' conditions. In order to prove our main theorems in a unified manner, we carry out our proofs for a class of dense pseudorandom bipartite graphs (see Section~\ref{sec:pseud}).


\subsection{Notation}\label{sec:notation}
Our notation is generally standard, but we note here the most important.
A graph $G$ has vertex set $V(G)$ and edge set $E(G)$, and $|G|=|V(G)|$ and $e(G)=|E(G)|$. If $G$ has an edge colouring, then $C(G)$ is the set of colours appearing on the edges of $G$, and, for each $c\in C(G)$, $E_c(G)$ is the set of edges of $G$ with colour $c$. The colour of an edge $e\in E(G)$ is $C(e)$ and its vertex set is $V(e)$. Given a vertex set $A$ of a graph $G$, $G[A]$ is the induced subgraph of $G$ with vertex set $A$ and edge set $\{uv\in E(G):u,v\in A\}$. Given a further vertex subset $B\subset V(G)$, $G[A,B]$ is the graph with vertex set $A\cup B$ and edge set $\{uv\in E(G):u\in A,v\in B\}$ and $e_G(A,B)=|\{(u,v):u\in A,v\in B,uv\in E(G)\}|$. Given an edge set $E\subset E(G)$, $V(E)=\cup_{e\in E}V(e)$. Given graphs $G$ and $H$, $G-H$ has vertex set $V(G)$ and edge set $E(G)\setminus E(H)$. Given $V\subset V(G)$, $G-V$ is the graph $G[V(G)\setminus V]$. Given $U\subset V(G)$, for each $i\geq 0$, $B^i_G(U)$ is the ball of radius $i$ around $U$ in $G$, that is, the set of vertices within distance $i$ of $U$ in $G$.

Given a ground set $V$ and $p\in [0,1]$, a set $A\subset V$ is \emph{$p$-random} if each element of $V$ is included in $A$ independently at random with probability $p$.
A \emph{matching} in a graph (or hypergraph) is a set of edges which share no vertices. In a coloured graph, a subgraph is \emph{$C$-rainbow} if it is rainbow with edge colours in $C$, and \emph{exactly-$C$-rainbow} if it has exactly one edge of each colour in $C$, with no other colours appearing.  A \emph{balanced} bipartite graph is one where the two vertex classes have the same size.

We use hierarchies of constants to record dependencies between the constants in our proofs. We write $\alpha \ll \beta$ to mean there exists some positive increasing function $f:(0,1]\to \mathbb{R}$ so that the remainder of the proof follows if $\alpha\leq f(\beta)$. We use $\alpha \llpoly \beta$ to mean there exists some fixed $C>0$ such that the remainder of the proof follows if $\alpha \leq \beta^C/C$. Where there are several constants in the hierarchy, the constants/functions are chosen from right to left. For more details on this notation, see~\cite[Section~3.2]{montgomery2018decompositions}. We also use `big O' notation, where the functions involved are always functions of $n$. Given $a,b\in \R$ and $c\geq 0$, we say $x=(a\pm c)b$ if $(a-c)b\leq x\leq (a+c)b$. For each integer $\ell\geq 1$, we use $[\ell]=\{1,\ldots,\ell\}$ and $[\ell]_0=\{0,1,\ldots,\ell\}$.



\subsection{Proof snapshot}\label{subsec:snapshot}
Before discussing the proof in detail, we make some remarks for readers very familiar with absorption techniques and the Ryser-Brualdi-Stein conjecture. Other readers may find it more helpful to move on to Section~\ref{subsec:discuss} and return here for later reference. A discussion of the proof in similar terms to that which we give now, but with a little more detail, also appears in a recent survey paper by the author~\cite{Mysurvey}.

Suppose we have a graph $G$ which is an optimally coloured copy of $K_{n,n}$, and wish to find a large rainbow matching using absorption and the semi-random method.
In order to use absorption, for some $\ell\in \N$, we would like to find sets $V^{\mathrm{abs}}\subset V(G)$ and $C^{\mathrm{abs}}\subset C(G)$ so that $(V^{\mathrm{abs}},C^{\mathrm{abs}})$ can absorb any balanced set $W\subset V(G)\setminus V^{\mathrm{abs}}$ of size $2\ell$ (by finding an exactly-$C^{\mathrm{abs}}$-rainbow matching with vertex set $V^{\mathrm{abs}}\cup W$). However, if, for example, $G$ corresponds to the addition\Footnote{For examples and as an analogy we consider from now on only abelian groups, and use addition rather than multiplication for its group rule.} table of some abelian group $H$, then simple calculations (like those later at \eqref{eqn:contra}) would show that, if such a matching exists for $W\subset V(G)$, then
\begin{equation}\label{eqn:lastone}
\sum_{v\in W}v=\sum_{c\in C^{\mathrm{abs}}}c-\sum_{v\in V^{\mathrm{abs}}}v,
\end{equation}
so, for example, we might only hope to be able to absorb vertex sets $W$ that sum to 0. This is related to the absorption of `zero-sum' sets in recent work of M\"uyesser and Pokrovskiy~\cite{muyesser2022random} and Bowtell and Keevash~\cite{bowtell2021n}, but in contrast for our general colourings we are not given an algebraic structure which gives an clear condition for which sets we should be able to absorb.

Instead, we pick (fairly arbitrarily) an `identity colour' $c_0$, and the condition for absorption is that the vertex set $W$ above must additionally be the vertex set of a matching of colour-$c_0$ edges. Where the colouring of $G$ arises from an abelian group $H$ as above, note that if $c_0=0$ then any such vertex set sums to 0, so this is more restrictive than the `zero-sum' condition. Where there is a corresponding abelian group $H$, the construction of our absorber could be a fairly straightforward implementation of distributive absorption (by constructing small absorbers which can absorb the vertex set of one of 100 specified colour-$c_0$ edges). In general, the construction of the absorber is the same at heart, but in practice considerably more complex as we need to construct this in conjunction with our work finding some approximate algebraic structure in the colouring. This latter work is hard to introduce more briefly than in Section~\ref{sec:alg}, so we do not elaborate on this aspect here, but instead discuss our \emph{addition structure}.

The condition we have for absorption is quite restrictive, and so we introduce an addition structure $(V^{\mathrm{add}},C^{\mathrm{add}})$, which takes a more general vertex set $W\subset V(G)\setminus V^{\mathrm{add}}$ and outputs two `remainder vertices' and a vertex set $\hat{W}$ which does satisfy the absorption condition. More specifically, given any small enough balanced set $W\subset V(G)\setminus V^{\mathrm{add}}$, there are vertex-disjoint matchings $M^{\mathrm{rb}}$ and $M^{\mathrm{id}}$ and vertices $w$ and $z$ such that together they have vertex set $V^{\mathrm{add}}\cup W$, $M^{\mathrm{rb}}$ is an exactly-$C^{\mathrm{add}}$-rainbow matching, and $M^{\mathrm{id}}$ is a matching of colour-$c_0$ edges (whose vertex set $\hat{W}=V(M^{\mathrm{id}})$ thus satisfies the absorption condition). The set $\hat{W}$ will then be absorbed, while the vertices $w$ and $z$ are the two vertices not used in the final $(n-1)$-edge rainbow matching.

Thus, leaving aside for now the analysis of the colouring, we do the following.

\vspace{-0.2cm}

\begin{itemize}  \setlength\itemsep{-0.3em}
\item Find an absorption structure and an addition structure.
\item Find a large rainbow matching using most of the remaining vertices and all but one of the remaining colours (using the semi-random method, and that we set aside fewer vertices than colours for the absorption and addition structures).
\item Transform the small set of unused vertices using the addition structure, giving us the two vertices to omit and a vertex set which is then absorbed by the absorption structure.
\end{itemize}

\vspace{-0.2cm}

The addition structure works iteratively. Essentially we start with the two matchings $M^{\mathrm{rb}}$ and $M^{\mathrm{id}}$ and two remainder vertices $w$ and $z$ and set $V^{\mathrm{add}}=V(M^{\mathrm{rb}}\cup M^{\mathrm{id}})\cup \{w,z\}$ and $C^{\mathrm{add}}=V(M^{\mathrm{rb}})$.
Then, we iteratively update these with small adjustments to cover together more and more vertices. In this, $M^{\mathrm{id}}$ increases in size but is always a matching of colour-$c_0$ edges, and $M^{\mathrm{rb}}$ is always an exactly-$C^{\mathrm{add}}$-rainbow matching.
Where $G$ has a corresponding abelian group $H$, the sum of $V(M^{\mathrm{rb}})\cup V(M^{\mathrm{id}})$ is fixed ($=\sum_{c\in C^{\mathrm{add}}}c$) so the remainder vertices $w,z$ may have to sum to any element of $H$ depending on the sum of $W$. Here, it is very important that we have two spare vertices so that this is possible. The construction of the addition structure is discussed in detail in Section~\ref{sec:add}.


\subsection{Proof sketch}\label{subsec:discuss}
Our sketch here has a lot of detail, but the actual implementation is more complicated in various places and thus varies a little from the proof sketch. Where these differences arise they are highlighted at the start of the relevant section. Once we have introduced the main components of our proof along with some notation, we give an overview of the proof using this notation, as highlighted by lines in the margin.

Instead of proving the theorems given in Section~\ref{sec:intro} directly, we will prove a more general and technical result finding rainbow matchings with $n-1$ edges in \emph{properly-pseudorandom} bipartite graphs with $2n$ vertices (see Theorem~\ref{thm-technical}), along with a variant which, under slightly stronger conditions, can find $n$-edge rainbow matchings (see Theorem~\ref{thm-technical-variant}). These theorems follow without any additional conceptual difficulty from our methods applied more directly for Theorem~\ref{thm:RBSeven}, and so here we will discuss only our methods in this setting. That is, we will take an optimally coloured complete bipartite graph $G$ with vertex classes $A$ and $B$ of size $n$, where $n$ is large, and look for a rainbow matching with $n-1$ edges.
To find a rainbow matching in $G$ with $n-1$ edges, we use the \emph{semi-random method} and the \emph{absorption method}.

\smallskip

\noindent \textbf{The semi-random method.} The semi-random method (also known as the R\"odl nibble) was introduced by R\"odl~\cite{rodlandhisnibble} in 1985 to find (equivalently) large matchings in complete hypergraphs (and, hence, approximate designs). Frankl and R\"odl~\cite{FRANKLRODL}, and, in unpublished work, Pippenger showed that this method could, more generally, find large matchings in almost-regular hypergraphs (for details on subsequent developments, see the recent survey by Kang, Kelly, K\"uhn, Osthus and Methuku~\cite{kang2021graph}). Our graph $G$ has a natural expression as a regular (uncoloured) 3-partite 3-uniform hypergraph, which allows the semi-random method to be applied as standard to find a large matching in this hypergraph, which then corresponds to a large rainbow matching in $G$.
The application of the semi-random method we use (via quoted results) is standard, requiring only the record of good bounds on the error terms involved.

More specifically, we will use a standard result from the semi-random method (see Corollary~\ref{cor:nibble}) to show that we can find large matchings in $G$ using only some chosen vertices and colours, as long as a significant proportion of these vertices and colours are chosen randomly. In this, we follow the work of the current author, Pokrovskiy and Sudakov proving Ringel's conjecture on tree packings in large complete graphs~\cite{montgomery2021proof}. More specifically, we will set aside random sets $V^{\text{s-r}}\subset V(G)$ and $C^{\text{s-r}}\subset C(G)$ for the application of the semi-random method, choosing them by including each element independently at random with probability $3/4$ (i.e., they are independent $(3/4)$-random subsets). Then, with high probability (i.e., with probability $1-o(1)$), the following property will hold (as we show in Section~\ref{sec:semirandom}).
\stepcounter{propcounter}
\begin{enumerate}[label = {{\textbf{\Alph{propcounter}\arabic{enumi}}}}]
\item Given any balanced $V\subset V(G)$ and $C\subset C(G)$, with $|V|=2|C|$, $C^{\text{s-r}}\subset C$ and $V^{\text{s-r}}\subset V$, there is a $C$-rainbow matching in $G[V]$ with $(1-o(1))|C|$ edges.\label{prop:sketch1}
\end{enumerate}
This allows us to use colours in $C(G)\setminus C^{\text{s-r}}$ and vertices in $V(G)\setminus V^{\text{s-r}}$ to construct substructures within $G$, before covering most of the unused vertices with a rainbow matching using most of the unused colours. Our goal is to construct such a substructure that will allow us to extend the rainbow matching from the semi-random method to a rainbow matching with $n-1$ edges. That is, we use the \emph{absorption method}.

\medskip

\noindent \textbf{The absorption method.} Following its origins in the work of  Erd\H{o}s, Gy\'arf\'as and Pyber~\cite{erdHos1991vertex} and Krivelevich~\cite{krivelevich1997triangle}, absorption was introduced as a general method by R\"odl, Ruci\'nski and Szemer\'edi~\cite{rodl2006dirac} in 2006, since when it has been used extensively for embedding and packing problems in many different contexts. When using it to find a rainbow matching, we aim to find a rainbow matching in $G$ (as our \emph{absorption structure}) to which we could make adjustments (adding and removing edges) in order to extend the matching to precisely use some extra vertices and colours. These `extra vertices and colours' would ultimately be those left unused after a large rainbow matching is found disjointly from the absorption structure using the semi-random method.

For simplicity, let us think about only absorbing the `extra vertices'. An ideal implementation of absorption for this would, for the random sets $V\subset V(G)$ and $C\subset C(G)$ described above and $\ell=\eps n$ for some small $\eps>0$, find a balanced set $V^{\mathrm{abs}}\subset V(G)\setminus V$ (that is, with equally many vertices in $A$ and $B$) and a set $C^{\mathrm{abs}}\subset C(G)\setminus C$ with $|V^{\mathrm{abs}}|=2|C^{\mathrm{abs}}|-2\ell$ such that the following holds
\begin{enumerate}[label = {{\textbf{\Alph{propcounter}\arabic{enumi}}}}]\addtocounter{enumi}{1}
\item Given any balanced set $W\subset V(G)\setminus V^{\mathrm{abs}}$ with $|W|=2\ell$, there is an exactly-$C^{\mathrm{abs}}$-rainbow matching in $G[V^{\mathrm{abs}}\cup W]$. (That is, a matching using each colour in $C^{\mathrm{abs}}$ exactly once, and so which has the vertex set $V^{\mathrm{abs}}\cup W$.)\label{prop:sketch2}
\end{enumerate}

If such sets $V^{\mathrm{abs}}\subset V(G)\setminus V$ and $C^{\mathrm{abs}}\subset C(G)\setminus C$ exist when $n$ is large, then, it is not hard to find a rainbow perfect matching in $G$. Indeed, by \ref{prop:sketch1}, there is a rainbow matching $M_1$ in $G[V(G)\setminus V^{\mathrm{abs}}]$ with colours in $C(G)\setminus C^{\mathrm{abs}}$
 and at least $n-|C^{\mathrm{abs}}|-\eps n$ edges. As $|V^{\mathrm{abs}}|=2|C^{\mathrm{abs}}|-2\ell$, there are more vertices in $V(G)\setminus (V^{\mathrm{abs}}\cup V(M_1))$ than colours in $C(G)\setminus (C^{\mathrm{abs}}\cup C(M_1))$, and this makes it relatively easy to find a rainbow matching $M_2$ within these vertices using exactly these colours\Footnote{If we have ensured a random-like set of vertices appears among them. As this is not actually how the proof proceeds, we skip over this to keep our illustration as simple as possible.}.
 Then, by \ref{prop:sketch2}, there is an exactly-$C^{\mathrm{abs}}$-rainbow matching $M_3$ in $G[V(G)\setminus V(M_1\cup M_2)]$, and $M_1\cup M_2\cup M_3$ is then an $n$-edge rainbow matching in $G$.

Of course, we know such a matching need not exist, and thus it is not always possible to find sets $V^{\mathrm{abs}}$ and $C^{\mathrm{abs}}$ satisfying \ref{prop:sketch2}.
Indeed, in certain colourings with some algebraic properties, the possible colour sets $C(M)$ of a matching $M$ may be restricted by the vertex set $V(M)$ (i.e., once ${V}^{\mathrm{abs}}$ and ${C}^{\mathrm{abs}}$ are chosen, there are only certain sets $W$ for which \ref{prop:sketch2} can hold).

\medskip

\noindent \textbf{Our approach.} We will take a two stage approach to create our full absorption property. First, with $\ell=\eps n$, we will create an absorption structure with sets $V^{\mathrm{abs}}$ and ${C}^{\mathrm{abs}}$ such that $|V^{\mathrm{abs}}|=2|{C}^{\mathrm{abs}}|-2\ell$ so that \ref{prop:sketch2} holds for sets $W$ with size $2\ell$ \emph{satisfying a certain condition}.
Having picked an `identity colour', $c_0$, this condition is (essentially) that $W$ is the vertex set of a matching with $\ell$ colour-$c_0$ edges. Crucially, we are using this as a condition for which sets can be absorbed and the colour-$c_0$ edges do not end up in the rainbow matching (which may have no colour-$c_0$ edges) resulting from the absorption. This condition is quite limiting, so we create an `addition structure' separately from the absorption structure which allows us to transform more general vertex subsets into a set suitable for the absorption structure, while also identifying two vertices which we leave out of the final matching.

The proof is described in more detail in what remains of this sketch. We start in Section~\ref{sec:alg} by describing some more extremal colourings (those without perfect rainbow matchings) and then how we analyse any given colouring to find colour classes with approximate algebraic properties. Then, in Section~\ref{sec:abs}, we discuss the absorption structure further, before discussing the addition structure in Section~\ref{sec:add}. Finally, in Section~\ref{sec:thetroubles}, we note some complications that arise for proving Theorem~\ref{thm:brouwer}, for which we need to make some adjustments to our methods as outlined in this sketch.


\subsubsection{Colourings with algebraic properties}\label{sec:alg}
Here, we first discuss extremal colourings arising from group addition tables, before highlighting a key algebraic property of these colourings (Property~\textbf{P} below). Not every extremal colouring will have this property, but if the property does not hold then we will see that this allows us to `switch' between two colours and we will discuss how we develop a partition of colour classes so that \emph{a)} we can `switch' between using any two colours of the same class and \emph{b)} these classes have instead some (at least approximate) algebraic property. We then illustrate this with an example from a more general collection of extremal colourings. Finally, we give a very brief summary of some of the tools we use to find our colour partition in general and then comment on some of the variables we use. The use of the `switching' properties of the colour classes we find is highlighted in our subsequent description of the absorption structure.

\smallskip

\noindent\textbf{Extremal colourings from group addition tables.} Given an abelian group $H$ of order $n$, let $G(H)$ be the coloured bipartite graph corresponding to the addition table of $H$. That is, $G(H)$ has two disjoint copies of $H$, $A$ and $B$ say, where the edge between $a\in A$ and $b\in B$ has colour $C(ab)=a+b$. If $G(H)$ contains a perfect rainbow matching, $M=\{x_iy_i:i\in [n]\}$ say, then
\begin{equation}\label{eqn:contra}
\sum_{v\in H}v=\sum_{i\in [n]}C(x_iy_i)=\sum_{i\in [n]}(x_i+y_i)=2\sum_{v\in H}v,
\end{equation}
and hence $\sum_{v\in H}v=0$. Thus, if $\sum_{v\in H}v\neq 0$, then $G(H)$ has no perfect rainbow matching. In particular, when $n$ is even, $G(\mathbb{Z}_n)$ contains no perfect rainbow matching, and this is the canonical extremal example for Conjecture~\ref{conj:RBS} known to Euler. More generally  (and including non-abelian groups), $G(H)$ has no perfect rainbow matching if $H$ has a non-trivial cyclic Sylow 2-subgroup and this is the topic of the Hall-Paige conjecture discussed in the introduction.

\medskip

\noindent\textbf{Our key algebraic property.} For each abelian\Footnote{For non-abelian groups $H$, $G(H)$ does not quite have this property, as the colour of the 4th edge can depend on which vertex class the path begins in.} group $H$, $G(H)$ has the following property.
\begin{enumerate}[label = \textbf{P}]
\item Given any two paths with length 3 with the same colours on their edges in the same order, the edges completing each path into a cycle have the same colour.
\end{enumerate}
Indeed, for any path $P=v_1v_2v_3v_4$ in $G(H)$, whose edges have colour $c_1,c_2,c_3\in G$ in that order, we have
\[
v_1+v_4=v_1+v_2-(v_2+v_3)+v_3+v_4=c_1-c_2+c_3,
\]
and therefore the colour of $v_1v_4$ is always $c_1-c_2+c_3\in H$.

On the other hand, if the colouring of $G$ has property \textbf{P}, it is possible to see that it must arise from the addition table of some abelian group. Indeed, after choosing an identity colour $c_0$, we can take the sum of any two non-identity colours $c$ and $d$ to be the colour of any edge completing a length 3 path with colours $c$, $c_0$, $d$ (in that order) into a cycle. It it not hard to show that this will give an abelian group structure on $C(G)$, nor to then label vertices with elements of $C(G)$ and show that the colouring of $G$ can thus arise as the addition table of this group.

\smallskip

\noindent\textbf{Colour classes with the key algebraic property.} Now, suppose that, in the colouring of $G$, the Property \textbf{P} fails for many pairs of length 3 paths, as follows. Suppose we have two vertex-disjoint paths $P_1=w_1x_1y_1z_1$ and $P_2=w_2x_2y_2z_2$ with edge colours $c_1,c_2,c_3$ in that order, but that $w_1z_1$ has colour $c$ and $w_2z_2$ has colour $d$ with $c\neq d$ (see Figure~\ref{fig:switcher}).
Let $M_1=\{x_1y_1,w_1z_1,w_2x_2,y_2z_2\}$ and $M_2=\{w_1x_1,y_1z_1,x_2y_2,w_2z_2\}$. Then, $M_1$ and $M_2$ are rainbow matchings with the same vertex set $V(M_1)=V(M_2)$ but $M_1$ uses colours $c_1,c_2,c_3$ and $c$ and $M_2$ uses colours $c_1,c_2,c_3$ and $d$ (see again Figure~\ref{fig:switcher}). We will say that $M_1$ and $M_2$ allow us to \emph{switch between using $c$ and using $d$}: if we include $M_1$ in any rainbow matching, then by switching $M_1$ for $M_2$ we can switch $c$ out of the colour set and $d$ into the colour set without changing the other colours or the vertex set.
We will say that $(M_1,M_2)$ forms a \emph{$c,d$-colour-switcher} of order $|M_1|=|M_2|=4$ (see Definition~\ref{defn:colourswitcher}). For brevity we will usually call this a \emph{$c,d$-switcher} if it is clear we are switching between colours; we also later use edge switchers (see Figure~\ref{fig:edgeswitcher}).

\begin{figure}[b]
\hspace{-0.7cm}
\begin{minipage}{1.1\textwidth}
\picfirsttrue\picsecondfalse\picthirdfalse
\begin{tikzpicture}[scale=1]
\def\vxrad{0.07cm}
\def\horunit{1.2}
\def\edgelength{0.4}
\def\betweenrows{0.5}


\def\gapratio{1}

\foreach \num/\parity/\parityy in {1/-1/-1,2/-1/1}
{
\coordinate (Y\num) at ($(0,0)+\parityy*\gapratio*(\horunit,0)+(0.5*\parity*\horunit,0.5*\horunit)$);
\coordinate (W\num) at ($(0,0)+\parityy*\gapratio*(\horunit,0)+(-0.5*\parity*\horunit,0.5*\horunit)$);
\coordinate (X\num) at ($(0,0)+\parityy*\gapratio*(\horunit,0)+(0.5*\parity*\horunit,-0.5*\horunit)$);
\coordinate (Z\num) at ($(0,0)+\parityy*\gapratio*(\horunit,0)+(-0.5*\parity*\horunit,-0.5*\horunit)$);
}

\ifpicfirst
\foreach \x/\y/\col/\n in {Z/W/magenta/1,W/Z/teal/2}
{
\draw [thick,\col,densely dashed] (\x\n) -- (\y\n);
}
\fi

\ifpicsecond
\foreach \x/\y/\col/\n in {Z/W/magenta/1}
{
\draw [thick,\col] (\x\n) -- (\y\n);
}
\else
\foreach \x/\y/\col in {X/W/red,Y/Z/blue}
\foreach \n in {1}
{
\draw [thick,\col] (\x\n) -- (\y\n);
}
\foreach \x/\y/\col in {X/Y/orange}
\foreach \n in {2}
{
\draw [thick,\col] (\x\n) -- (\y\n);
}
\fi
\ifpicthird
\foreach \x/\y/\col/\n in {W/Z/teal/2}
{
\draw [thick,\col] (\x\n) -- (\y\n);
}
\else
\foreach \x/\y/\col in {X/Y/orange}
\foreach \n in {1}
{
\draw [thick,\col] (\x\n) -- (\y\n);
}

\foreach \x/\y/\col in {X/W/red,Y/Z/blue}
\foreach \n in {2}
{
\draw [thick,\col] (\x\n) -- (\y\n);
}
\fi

\foreach \x/\lab/\offs in {X/x/-1.2,Y/y/1,Z/z/-1.2,W/w/1}
\foreach \num in {1,2}
{
\draw  ($(\x\num)+\offs*(0,0.25)$) node {$\lab_\num$};
}

\ifpicthird
\else
\draw [red] ($0.5*(X2)+0.5*(W2)+(0.025,-0.25)+0.15*(-\horunit,-\horunit)$) node {$c_1$};
\draw [blue] ($0.5*(Y2)+0.5*(Z2)+(0.025,0.22)+0.15*(-\horunit,\horunit)$) node {$c_3$};
\draw [orange] ($0.5*(X1)+0.5*(Y1)-(0.2,0)$) node {$c_2$};
\draw [magenta] ($0.5*(W1)+0.5*(Z1)+(0.2,0)$) node {$c$};
\fi
\ifpicsecond
\else
\draw [red] ($0.5*(X1)+0.5*(W1)+(0.025,-0.25)+0.15*(-\horunit,-\horunit)$) node {$c_1$};
\draw [blue] ($0.5*(Y1)+0.5*(Z1)+(0.025,0.22)+0.15*(-\horunit,\horunit)$) node {$c_3$};
\draw [orange] ($0.5*(X2)+0.5*(Y2)-(0.2,0)$) node {$c_2$};
\draw [teal] ($0.5*(W2)+0.5*(Z2)+(0.2,0)$) node {$d$};
\fi

\ifpicsecond
\draw ($0.5*(X1)+0.5*(Y1)-(0.8,0)$) node {$M_1$:};
\fi

\ifpicthird
\draw ($0.5*(X1)+0.5*(Y1)-(0.6,0)$) node {$M_2$:};
\fi

\foreach \x in {X,W,Y,Z}
\foreach \n in {1,2}
{
\draw [fill] (\x\n) circle [radius=\vxrad];
}

\end{tikzpicture}
\begin{tikzpicture}
\draw [white] (-0.5,0) -- (0.5,0);
\draw [dashed] (0,-1.2) -- (0,1);.2
\end{tikzpicture}
\picfirstfalse\picsecondtrue\picthirdfalse
\begin{tikzpicture}[scale=1]
\def\vxrad{0.07cm}
\def\horunit{1.2}
\def\edgelength{0.4}
\def\betweenrows{0.5}


\def\gapratio{1}

\foreach \num/\parity/\parityy in {1/-1/-1,2/-1/1}
{
\coordinate (Y\num) at ($(0,0)+\parityy*\gapratio*(\horunit,0)+(0.5*\parity*\horunit,0.5*\horunit)$);
\coordinate (W\num) at ($(0,0)+\parityy*\gapratio*(\horunit,0)+(-0.5*\parity*\horunit,0.5*\horunit)$);
\coordinate (X\num) at ($(0,0)+\parityy*\gapratio*(\horunit,0)+(0.5*\parity*\horunit,-0.5*\horunit)$);
\coordinate (Z\num) at ($(0,0)+\parityy*\gapratio*(\horunit,0)+(-0.5*\parity*\horunit,-0.5*\horunit)$);
}

\ifpicfirst
\foreach \x/\y/\col/\n in {Z/W/magenta/1,W/Z/teal/2}
{
\draw [thick,\col,densely dashed] (\x\n) -- (\y\n);
}
\fi

\ifpicsecond
\foreach \x/\y/\col/\n in {Z/W/magenta/1}
{
\draw [thick,\col] (\x\n) -- (\y\n);
}
\else
\foreach \x/\y/\col in {X/W/red,Y/Z/blue}
\foreach \n in {1}
{
\draw [thick,\col] (\x\n) -- (\y\n);
}
\foreach \x/\y/\col in {X/Y/orange}
\foreach \n in {2}
{
\draw [thick,\col] (\x\n) -- (\y\n);
}
\fi
\ifpicthird
\foreach \x/\y/\col/\n in {W/Z/teal/2}
{
\draw [thick,\col] (\x\n) -- (\y\n);
}
\else
\foreach \x/\y/\col in {X/Y/orange}
\foreach \n in {1}
{
\draw [thick,\col] (\x\n) -- (\y\n);
}

\foreach \x/\y/\col in {X/W/red,Y/Z/blue}
\foreach \n in {2}
{
\draw [thick,\col] (\x\n) -- (\y\n);
}
\fi

\foreach \x/\lab/\offs in {X/x/-1.2,Y/y/1,Z/z/-1.2,W/w/1}
\foreach \num in {1,2}
{
\draw  ($(\x\num)+\offs*(0,0.25)$) node {$\lab_\num$};
}

\ifpicthird
\else
\draw [red] ($0.5*(X2)+0.5*(W2)+(0.025,-0.25)+0.15*(-\horunit,-\horunit)$) node {$c_1$};
\draw [blue] ($0.5*(Y2)+0.5*(Z2)+(0.025,0.22)+0.15*(-\horunit,\horunit)$) node {$c_3$};
\draw [orange] ($0.5*(X1)+0.5*(Y1)-(0.2,0)$) node {$c_2$};
\draw [magenta] ($0.5*(W1)+0.5*(Z1)+(0.2,0)$) node {$c$};
\fi
\ifpicsecond
\else
\draw [red] ($0.5*(X1)+0.5*(W1)+(0.025,-0.25)+0.15*(-\horunit,-\horunit)$) node {$c_1$};
\draw [blue] ($0.5*(Y1)+0.5*(Z1)+(0.025,0.22)+0.15*(-\horunit,\horunit)$) node {$c_3$};
\draw [orange] ($0.5*(X2)+0.5*(Y2)-(0.2,0)$) node {$c_2$};
\draw [teal] ($0.5*(W2)+0.5*(Z2)+(0.2,0)$) node {$d$};
\fi

\ifpicsecond
\draw ($0.5*(X1)+0.5*(Y1)-(0.8,0)$) node {$M_1$:};
\fi

\ifpicthird
\draw ($0.5*(X1)+0.5*(Y1)-(0.6,0)$) node {$M_2$:};
\fi

\foreach \x in {X,W,Y,Z}
\foreach \n in {1,2}
{
\draw [fill] (\x\n) circle [radius=\vxrad];
}

\end{tikzpicture}
\begin{tikzpicture}
\draw [white] (-0.5,0) -- (0.5,0);
\draw [dashed] (0,-1.2) -- (0,1);.2
\end{tikzpicture}
\picfirstfalse\picsecondfalse\picthirdtrue
\begin{tikzpicture}[scale=1]
\def\vxrad{0.07cm}
\def\horunit{1.2}
\def\edgelength{0.4}
\def\betweenrows{0.5}


\def\gapratio{1}

\foreach \num/\parity/\parityy in {1/-1/-1,2/-1/1}
{
\coordinate (Y\num) at ($(0,0)+\parityy*\gapratio*(\horunit,0)+(0.5*\parity*\horunit,0.5*\horunit)$);
\coordinate (W\num) at ($(0,0)+\parityy*\gapratio*(\horunit,0)+(-0.5*\parity*\horunit,0.5*\horunit)$);
\coordinate (X\num) at ($(0,0)+\parityy*\gapratio*(\horunit,0)+(0.5*\parity*\horunit,-0.5*\horunit)$);
\coordinate (Z\num) at ($(0,0)+\parityy*\gapratio*(\horunit,0)+(-0.5*\parity*\horunit,-0.5*\horunit)$);
}

\ifpicfirst
\foreach \x/\y/\col/\n in {Z/W/magenta/1,W/Z/teal/2}
{
\draw [thick,\col,densely dashed] (\x\n) -- (\y\n);
}
\fi

\ifpicsecond
\foreach \x/\y/\col/\n in {Z/W/magenta/1}
{
\draw [thick,\col] (\x\n) -- (\y\n);
}
\else
\foreach \x/\y/\col in {X/W/red,Y/Z/blue}
\foreach \n in {1}
{
\draw [thick,\col] (\x\n) -- (\y\n);
}
\foreach \x/\y/\col in {X/Y/orange}
\foreach \n in {2}
{
\draw [thick,\col] (\x\n) -- (\y\n);
}
\fi
\ifpicthird
\foreach \x/\y/\col/\n in {W/Z/teal/2}
{
\draw [thick,\col] (\x\n) -- (\y\n);
}
\else
\foreach \x/\y/\col in {X/Y/orange}
\foreach \n in {1}
{
\draw [thick,\col] (\x\n) -- (\y\n);
}

\foreach \x/\y/\col in {X/W/red,Y/Z/blue}
\foreach \n in {2}
{
\draw [thick,\col] (\x\n) -- (\y\n);
}
\fi

\foreach \x/\lab/\offs in {X/x/-1.2,Y/y/1,Z/z/-1.2,W/w/1}
\foreach \num in {1,2}
{
\draw  ($(\x\num)+\offs*(0,0.25)$) node {$\lab_\num$};
}

\ifpicthird
\else
\draw [red] ($0.5*(X2)+0.5*(W2)+(0.025,-0.25)+0.15*(-\horunit,-\horunit)$) node {$c_1$};
\draw [blue] ($0.5*(Y2)+0.5*(Z2)+(0.025,0.22)+0.15*(-\horunit,\horunit)$) node {$c_3$};
\draw [orange] ($0.5*(X1)+0.5*(Y1)-(0.2,0)$) node {$c_2$};
\draw [magenta] ($0.5*(W1)+0.5*(Z1)+(0.2,0)$) node {$c$};
\fi
\ifpicsecond
\else
\draw [red] ($0.5*(X1)+0.5*(W1)+(0.025,-0.25)+0.15*(-\horunit,-\horunit)$) node {$c_1$};
\draw [blue] ($0.5*(Y1)+0.5*(Z1)+(0.025,0.22)+0.15*(-\horunit,\horunit)$) node {$c_3$};
\draw [orange] ($0.5*(X2)+0.5*(Y2)-(0.2,0)$) node {$c_2$};
\draw [teal] ($0.5*(W2)+0.5*(Z2)+(0.2,0)$) node {$d$};
\fi

\ifpicsecond
\draw ($0.5*(X1)+0.5*(Y1)-(0.8,0)$) node {$M_1$:};
\fi

\ifpicthird
\draw ($0.5*(X1)+0.5*(Y1)-(0.6,0)$) node {$M_2$:};
\fi

\foreach \x in {X,W,Y,Z}
\foreach \n in {1,2}
{
\draw [fill] (\x\n) circle [radius=\vxrad];
}

\end{tikzpicture}
\end{minipage}

\caption{A $c,d$-colour-switcher of order 4 on the left with vertex set $\{w_1,x_1,y_1,z_1,w_2,x_2,y_2,z_2\}$ and colour set $\{c_1,c_2,c_3\}$ (see Definition~\ref{defn:colourswitcher}), along with matchings $M_1$ and $M_2$ demonstrating its properties.}\label{fig:switcher}
\end{figure}
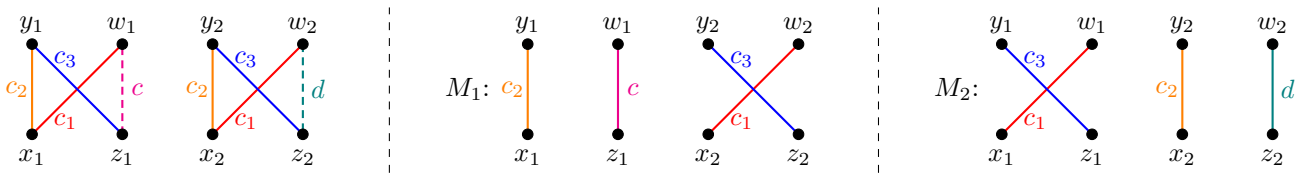

Effectively, we will find a set of colour classes\Footnote{Our colour classes are sets of colours. We use \emph{classes} here to distinguish between the sets in $\mathcal{C}$ and the other sets of colours we use like $C^{\mathrm{abs}}$ as they play a very different role.} $\mathcal{C}$ such that, for each $C\in \mathcal{C}$, given any distinct $c,d\in C$ we can find many $c,d$-switchers in $G$ with order $O(\log^4n)$, while, moreover, without using some other vertices or colours we wish to avoid (so long as, roughly speaking, we do not try to avoid too many colours in $C$). The `exchangeable' classes of colours we use are defined in Definition~\ref{defn:exchange}, and found by Theorem~\ref{thm-manyinter}. Our aim is to find such classes so that overall the colour classes have an approximate version of the following analogue of property \textbf{P} for colour classes (see \itref{prop-C-3} in Definition~\ref{defn:exchange}).
\begin{enumerate}[label = \textbf{P'}]
\item Given any two paths with length 3 with the same colours on their edges in the same order, the edges completing each path into a cycle have \emph{colours within in the same colour class}.
\end{enumerate}

\smallskip

\noindent\textbf{An example: further extremal colourings.} We can use the extremal colourings from group addition tables described above to generate other extremal colourings with no perfect rainbow matching and a less obvious algebraic structure, using the following `blow-up' construction (see Figure~\ref{fig:extremal} for an example with $H=\mathbb{Z}_2$). Suppose $H$ is an abelian group of order $n$ and let $m\in \N$. Take disjoint sets $A_v$, $B_v$, $C_v$, $v\in H$, all with size $m$.
Form a complete bipartite graph with vertex classes $\cup_{v\in H}A_v$ and $\cup_{v\in H}B_v$. For each $v,w\in H$, colour the edges between $A_v$ and $B_w$ using the colours from $C_{v+w}$ in \emph{any} proper manner. This can be easily seen to give an optimally coloured copy of $K_{mn,mn}$,  where, if there is a perfect rainbow matching $x_iy_i$, $i\in [mn]$, then, letting $v_i$ and $w_i$ be such that $x_i\in A_{v_i}$ and $y_i\in B_{w_i}$ for each $i\in [mn]$, the corresponding calculation to \eqref{eqn:contra} is that
\[
m\cdot \sum_{v\in H}v=\sum_{v\in H}|C_v|\cdot v=\sum_{v\in H}\;\sum_{i\in [mn]:C(x_iy_i)\in C_u}u=\sum_{i\in [mn]}(v_i+w_i)
=2m\cdot \sum_{v\in H}v,
\]
and therefore if $m\cdot \sum_{v\in H}v\neq 0$ then there is no perfect rainbow matching in $G(H)$. More specifically, taking and odd $m$ and $H=\mathbb{Z}_n$ for any even $n$ gives an optimally bipartite coloured graph with no rainbow matching, recovering a construction for the corresponding Latin squares given by Maillet~\cite{maillet1894carres} in 1894.

In these examples, for Property~\textbf{P'}, the colour classes $C_v$, $v\in H$, would be a good starting point for our partition of colours. However, depending on the colourings we chose between the sets $A_x$ and $B_y$ (and whether these colourings have some algebraic structure) we may need to refine this partition of colours to get classes that satisfy \textbf{P'} approximately.

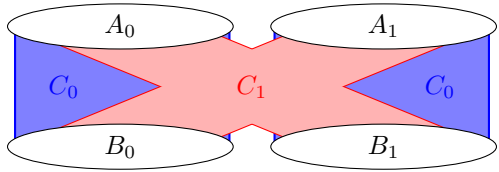
\begin{figure}[b]
\begin{center}
\begin{tikzpicture}[scale=1]
\def\vxrad{0.07cm}
\def\horunit{0.8}
\def\verunit{1.75}
\def\edgelength{0.4}
\def\betweenrows{0.5}


\def\gapratio{1}

\foreach \lab/\yy in {A/1,B/-1}
{
\foreach \num/\xx in {1/-1,2/1}
{
\coordinate (\lab\num) at ($(\xx*\verunit,\yy*\horunit)$);
}
}

\coordinate (L11) at ($(A1)-(1,0)-(0.2,0)$);
\coordinate (L12) at ($(B1)-(1,0)-(0.2,0)$);
\coordinate (L21) at ($(A2)-(1,0)$);
\coordinate (L22) at ($(B2)-(1,0)$);
\coordinate (R11) at ($(A1)+(1,0)$);
\coordinate (R12) at ($(B1)+(1,0)$);
\coordinate (R21) at ($(A2)+(1,0)+(0.2,0)$);
\coordinate (R22) at ($(B2)+(1,0)+(0.2,0)$);

\foreach \cooo in {L11,L12}
\foreach \w in {w}
{
\coordinate (\cooo\w) at ($(\cooo)-(0.2,0)$);
}
\foreach \cooo in {L21,L22}
\foreach \w in {w}
{
\coordinate (\cooo\w) at ($(\cooo)-(0.45,0)$);
}
\foreach \cooo in {R21,R22}
\foreach \w in {w}
{
\coordinate (\cooo\w) at ($(\cooo)+(0.2,0)$);
}
\foreach \cooo in {R11,R12}
\foreach \w in {w}
{
\coordinate (\cooo\w) at ($(\cooo)+(0.45,0)$);
}

\fill [fill=blue!50] ($(L11w)$) -- ($(L12w)$) -- (R12w) -- (R11w) -- (L11w);
\fill [fill=blue!50] ($(L21w)$) -- ($(L22w)$) -- (R22w) -- (R21w) -- (L21w);
\draw [thick,blue] ($(L11w)$) -- ($(L12w)$);
\draw [thick,blue] ($(L21w)$) -- ($(L22w)$);
\draw [thick,blue] ($(R21w)$) -- ($(R22w)$);
\draw [thick,blue] ($(R11w)$) -- ($(R12w)$);

\draw [red,thick] (L11w) -- (L22) -- (R22w) -- (R11) -- (L11w);
\draw [red,thick]  (L21) -- (L12w) -- (R12) -- (R21w) -- (L21);
\fill [black,fill=red!30] (L11w) -- (L22) -- (R22w) -- (R11) -- (L11w);
\fill [black,fill=red!30] (L21) -- (L12w) -- (R12) -- (R21w) -- (L21w);

\foreach \lab/\num/\labb in {A/1/A_0,A/2/A_1,B/1/B_0,B/2/B_1}
{
\draw [fill=white] (\lab\num) circle [x radius=1.5cm,y radius=0.3cm];
\draw (\lab\num) node {$\labb$};
}

\draw [blue] ($0.5*(A1)+0.5*(B1)-(0.75,0)$) node {$C_0$};
\draw [blue] ($0.5*(A2)+0.5*(B2)+(0.75,0)$) node {$C_0$};
\draw [red] ($0.25*(A1)+0.25*(B1)+0.25*(A2)+0.25*(B2)$) node {$C_1$};

\end{tikzpicture}
\end{center}

\caption{An extremal colouring formed by taking disjoint equal-sized sets $A_0,A_1,B_0,B_1,C_0,C_1$ with equal odd order, and optimally colouring the complete bipartite graph with classes $A_0\cup A_1$ and $B_0\cup B_1$ so that, for each $i,j\in \mathbb{Z}_2$, edges between $A_i$ and $B_j$ have colour in $C_{i+j}$.}\label{fig:extremal}
\end{figure}

\smallskip

\noindent\textbf{Tools for developing colour classes.} As we introduced above, if \textbf{P} fails for many pairs of paths, then we will be able to find switchers of order 4 for many pairs of colours $c,d$. To develop our colour classes, we will consider the auxiliary graph whose vertex set is the set of colours and whose edges between pairs of colours are weighted by the number of switchers of order 4 for that colour pair in $G$. This auxiliary graph can be very different for different colourings: it can be dense with generally low edge weights or sparse with high edge weights, or indeed some combination of these cases and anything in between. Very roughly speaking, we partition the edges of the auxiliary graph into $O(\log n)$ parts based on these edge-weights before finding well-connected subgraphs in each part. For most of these parts, this is done using Koml\'os-Szemer\'edi sublinear graph expansion (as discussed in Section~\ref{subsec:expand}). Within these subgraphs we can then build switchers for any pair of colours by chaining together colour switchers of order 4 (where it is important that these chains are not too long). This is discussed further in Section~\ref{sec:exchangingcolours}.

\smallskip

\noindent\textbf{A comment on our variables.} We now comment briefly on our main variables, as, though loosely worded at this stage, it may give some intuition over their use. The typical overarching hierarchy\Footnote{The use of `$\llpoly$' is described in Section~\ref{sec:notation}, but let us note here that \eqref{eqn:sampleh} means that there are increasingly large constants $C_\beta,C_\gamma,C_\eta,C_\eps, C_0$ for which our proof holds after setting $\beta=\frac{1}{C_\beta}\log^{-C_\beta}n$, $\gamma=\frac{1}{C_\gamma}\log^{-C_\gamma}n$, $\eta=\frac{1}{C_\eta}\log^{-C_\eta}n$ and $\eps=\frac{1}{C_\eps}\log^{-C_\eps}n$,
and ensuring that $\frac{1}{n}\leq \frac{1}{C_0}\log^{-C_0}n$ by taking $n$ to be sufficiently large.} we use for finding an $(n-1)$-edge rainbow matching in $G$ is
\begin{equation}\label{eqn:sampleh}
\frac{1}{n}\llpoly \eps \llpoly \eta \llpoly\gamma\llpoly \beta \llpoly \log^{-1}n.
\end{equation}
Due to our use of Koml\'os-Szemer\'edi sublinear graph expansion, the switchers we construct in the auxiliary graph have size $O(\log^4n)$ and we split the auxiliary graph into $O(\log n)$ graphs. This explains the right-most variable of \eqref{eqn:sampleh}, and we will then be able to construct substructures in $G$ using around $\beta n$ colours and vertices while using colour switchers if $\beta \llpoly \log^{-1}n$.
Creating an absorption structure of such a size will allow us to absorb the vertex set of around $\gamma n$ `identity colour' edges with $\gamma\llpoly \beta$. The addition structure will then be able to take a set of around $\eta n$ vertices with $\eta \llpoly \gamma$ and put out a rainbow matching (with a specified colour set) and a set which is the vertex set of $\gamma n$ `identity colour' edges (which is then `absorbed' by the absorption structure). Therefore, we need to find an almost-perfect rainbow matching covering all but around $\eta n$ vertices not in the addition or absorption structure, using colours not set aside for these structures. To do this we show the subgraph $G'$ of $G$ formed by the removal of these vertices and the removal of edges with these colours has vertex degrees $(1\pm \eps)D$ for some $D=(1-o(1))n$, where we want $1/n\llpoly \eps\llpoly \eta$ so that we can apply the semi-random method.
However,
by setting aside random sets of vertices and colours with probability $3/4$ we can ensure that this regularity condition holds for $\eps=n^{-\alpha}$ for some small universal constant $\alpha>0$ (see the proof of Lemma~\ref{lem:matchinhyper}), and thus (for large $n$) we can choose the variables we require as in \eqref{eqn:sampleh}.
Beyond this sketch, for our more general theorems, $G$ will be only approximately regular with errors proportional to $\eps$ and there we show the graph $G'$ has vertex degrees $(1\pm \eps')D$ for some $D$, where  $\eps\llpoly \eps'\llpoly \eta$.


\subsubsection{Absorption structure}\label{sec:abs}
To discuss the absorption structure, for simplification we will now assume that the colouring has property \textbf{P}. As we have noted, this is equivalent to the colouring arising from the addition table for an abelian group $H$ of order $n$, so that Theorem~\ref{thm:RBSeven} was determined long ago in a stronger form by Hall and Paige~\cite{hallpaige}. This simplification allows us to discuss the main mechanism of our construction, without the complications due to the colour classes (though we highlight at the key moment where colour switchers could be used if we do not have this property).

As discussed above, we will construct an absorber which can absorb any vertex set $W$ as long as it is the vertex set of a matching of the right number of colour-$c_0$ edges, where $c_0$ is some (fairly arbitrarily) chosen `identity colour'. (Moreover, we will have that $W$ is a subset of some larger vertex set so that the size of the absorption structure constructed is manageable.) To do this we will use \emph{distributive absorption}, a form of absorption introduced by the current author in~\cite{montgomery2018spanning} which has proven very useful in the efficient creation of absorption properties and in using absorption for rather rigid substructures. We start, then, by recalling the use of distributive absorption, and how it can be used to create a global absorption property from a robust local absorption property. Next, we describe how to `switch' in a matching between using the vertex sets of two edges with the same colour. Finally, we describe how we use this to create a local absorption property.

\smallskip

\noindent\textbf{Distributive absorption.}
The following is a typical aim when finding our absorption structure. For some colour $c_0$ and large $m$ (for example $m\approx\gamma n$ in our discussion at the end of Section~\ref{sec:alg}),
 we wish to take a set $E$ of $2m$ edges in $G$ with colour $c_0$ and find sets $V^{\mathrm{abs}}\subset V(G)\setminus V(M)$ and $C^{\mathrm{abs}}\subset C(G)$ with $|V^{\mathrm{abs}}|=2|C^{\mathrm{abs}}|-2\ell_1$ such that, given any set $E'\subset E$ of $\ell_1$ edges, there is an exactly-$C^{\mathrm{abs}}$-rainbow matching in $G[V^{\mathrm{abs}}\cup V(E')]$.
  We can think of this as finding an absorption structure $(V^{\mathrm{abs}},C^{\mathrm{abs}})$ that can absorb any set $W\subset V(E)$ which is the vertex set
 of $\ell_1$ edges with colour $c_0$. Using distributive absorption we can build this efficiently if we can do it robustly for small $\ell_1$ using small absorbers, using a `template graph' such as that provided by the following lemma.

\begin{lemma}[\cite{montgomery2018spanning}]\label{Lemma_H_graph}
There is a constant $h_0 \in \mathbb N$ such that, for every $h \geq h_0$ with $3|h$, there exists
a bipartite graph $K$ with maximum degree at most $100$ and vertex classes $X$ and $Y \cup Z$, with
$|X| = h$, and $|Y| = |Z| = 2h/3$, so that the following is true. If $Z_0 \subseteq Z$ and $|Z_0 | = h/3$, then
there is a matching between $X$ and $Y \cup Z_0$.
\end{lemma}

Using the template from Lemma~\ref{Lemma_H_graph} (sometimes known as a robustly matchable bipartite graph) with $h=3m$, we can use this to combine small absorbers into an absorber with a global absorption property as described above for vertex sets from $V(E)$, as follows.
Let $E'\subset E_{c_0}(G)$ be a set of $2m$ edges disjoint from $E$, and identify $E'$ with $Y$ and $E$ with $Z$ in $K$. Suppose we can find disjoint vertex sets $V_x\subset V(G)\setminus V(E\cup E')$, $x\in X$, and disjoint colour sets $C_x\subset C(G)$, $x\in X$, such that, for each $x\in X$, $|V_x|=2|C_x|-2$, and, for each $e\in N_K(x)$, there is an exactly-$C_x$-rainbow matching in $G[V_x\cup V(e)]$. We say that $(V_x,C_x)$ is an \emph{absorber} for the edges in $N_K(x)$ (noting that the vertex set of the edge is used but the edge itself is not usually used in the final matching).
Then, we can combine these absorbers to get an absorber for any vertex set from $V(E)$ which is the vertex set of $m$ edges. Indeed, consider $(V^{\mathrm{abs}},C^{\mathrm{abs}})$, where $V^{\mathrm{abs}}=(\cup_{x\in X}V_x)\cup V(E')$ and $C^{\mathrm{abs}}=\cup_{c\in C}C_x$. Letting $\bar{E}\subset E$ be an arbitrary set of $m$ edges, we show that $(V^{\mathrm{abs}},C^{\mathrm{abs}})$ can absorb $V(\bar{E})$.
By the property of $K$, there is a matching between $X$ and $\bar{E}\cup E'$ in $K$, say with edges $x\phi(x)$, $x\in X$. Combining, for each $x\in X$, an exactly-$C_x$-rainbow matching in $G[V_x\cup V(\phi(x))]$, this gives an exactly-$C^{\mathrm{abs}}$-rainbow matching with vertex set
\[
\cup_{x\in X}(V_x\cup V(\phi(x)))=\left(\cup_{x\in X}V_x\right)\cup V(\phi(X))=\left(\cup_{x\in X}V_x\right)\cup V(E\cup \bar{E})=V^{\mathrm{abs}}\cup V(\bar{E}),
\]
showing that $(V^{\mathrm{abs}},C^{\mathrm{abs}})$ can absorb $V(\bar{E})$, as required.
 As $\Delta(K)\leq 100$, we need only construct absorbers for small sets of edges, and, as the graph $K$ is sparse, we can build the global absorption property while using a modest number of absorbers.

 \smallskip

 \noindent\textbf{Switching between edges of the same colour.} To build these small absorbers, we start by building switchers between edges with the same colour. Suppose $e_1=u_1v_1$ and $e_2=u_2v_2$ are two edges in $G$ with the same colour, $c$ say (see Figure~\ref{fig:edgeswitcher}). Pick two colours $d,d'\in C(G)$ such that, if $w_i$ and $x_i$ are the $d$- and $d'$-neighbour of $u_i$ and $v_i$ respectively for each $i\in [2]$, then $u_1,v_1,w_1,x_1,u_2,v_2,w_2,x_2$ are distinct vertices and $w_1x_1$ does not have colour $c$. By property \textbf{P}, $w_1x_1$ and $w_2x_2$ have the same colour, $c'$ say. Let $\hat{C}=\{d,d',c'\}$ and $\hat{V}=\{w_1,x_1,w_2,x_2\}$, and note that each of $G[\hat{V}\cup V(e_1)]$ and $G[\hat{V}\cup V(e_2)]$ has an exactly-$\hat{C}$-rainbow matching. We call here $(\hat{V},\hat{C})$ an \emph{$e_1,e_2$-edge-switcher} (of order 4) which we can turn to cover either $V(e_1)$ or $V(e_2)$ (see also Definition~\ref{defn:edgeexchanger}). As with colour switchers, we will usually call this an $e_1,e_2$-switcher when it is clear from context that we are switching between edges.

 \begin{figure}[b]
 \hspace{-0.7cm}
 \begin{minipage}{1.1\textwidth}
 \picfirsttrue\picsecondfalse\picthirdfalse
 \begin{tikzpicture}[scale=1]
\def\vxrad{0.07cm}
\def\horunit{1.2}
\def\edgelength{0.4}
\def\betweenrows{0.5}


\def\gapratio{1}

\foreach \num/\parity/\parityy in {1/-1/-1,2/-1/1}
{
\coordinate (Y\num) at ($(0,0)+\parityy*\gapratio*(\horunit,0)+(0.5*\parity*\horunit,0.5*\horunit)$);
\coordinate (W\num) at ($(0,0)+\parityy*\gapratio*(\horunit,0)+(-0.5*\parity*\horunit,0.5*\horunit)$);
\coordinate (X\num) at ($(0,0)+\parityy*\gapratio*(\horunit,0)+(0.5*\parity*\horunit,-0.5*\horunit)$);
\coordinate (Z\num) at ($(0,0)+\parityy*\gapratio*(\horunit,0)+(-0.5*\parity*\horunit,-0.5*\horunit)$);
}

\ifpicfirst
\foreach \x/\y/\col/\n in {Z/W/magenta/1,W/Z/magenta/2}
{
\draw [thick,\col] (\x\n) -- (\y\n);
}
\foreach \x/\y/\col in {X/Y/orange}
\foreach \n in {1}
{
\draw [thick,\col, densely dashed] (\x\n) -- (\y\n);
}
\foreach \x/\y/\col in {X/Y/orange}
\foreach \n in {2}
{
\draw [thick,\col, densely dashed] (\x\n) -- (\y\n);
}
\fi

\ifpicsecond
\foreach \x/\y/\col/\n in {Z/W/magenta/1}
{
\draw [thick,\col] (\x\n) -- (\y\n);
}
\else
\foreach \x/\y/\col in {X/W/red,Y/Z/blue}
\foreach \n in {1}
{
\draw [thick,\col] (\x\n) -- (\y\n);
}

\fi
\ifpicthird
\foreach \x/\y/\col/\n in {W/Z/magenta/2}
{
\draw [thick,\col] (\x\n) -- (\y\n);
}
\else

\foreach \x/\y/\col in {X/W/red,Y/Z/blue}
\foreach \n in {2}
{
\draw [thick,\col] (\x\n) -- (\y\n);
}
\fi

\foreach \x/\lab/\offs in {X/v/-1.2,Y/u/1,Z/w/-1.2,W/x/1}
\foreach \num in {1,2}
{
\draw  ($(\x\num)+\offs*(0,0.25)$) node {$\lab_\num$};
}

\ifpicfirst
\draw ($0.5*(X1)+0.5*(Y1)-(0.3,0)$) node {$e_1$};
\draw ($0.5*(X2)+0.5*(Y2)-(0.3,0)$) node {$e_2$};
\fi

\ifpicfirst
\draw [orange] ($0.5*(X2)+0.5*(Y2)+(0.2,0)$) node {$c$};
\draw [orange] ($0.5*(X1)+0.5*(Y1)+(0.2,0)$) node {$c$};
\fi
\ifpicthird
\else
\draw [red] ($0.5*(X2)+0.5*(W2)+(0.025,-0.25)+0.15*(-\horunit,-\horunit)$) node {$d'$};
\draw [blue] ($0.5*(Y2)+0.5*(Z2)+(0.025,0.22)+0.15*(-\horunit,\horunit)$) node {$d$};

\draw [magenta] ($0.5*(W1)+0.5*(Z1)+(0.2,0)$) node {$c'$};
\fi
\ifpicsecond
\else
\draw [red] ($0.5*(X1)+0.5*(W1)+(0.025,-0.25)+0.15*(-\horunit,-\horunit)$) node {$d'$};
\draw [blue] ($0.5*(Y1)+0.5*(Z1)+(0.025,0.22)+0.15*(-\horunit,\horunit)$) node {$d$};

\draw [magenta] ($0.5*(W2)+0.5*(Z2)+(0.2,0)$) node {$c'$};
\fi

\ifpicsecond
\fi

\ifpicthird
\fi

\foreach \x in {X,W,Y,Z}
\foreach \n in {1,2}
{
\draw [fill] (\x\n) circle [radius=\vxrad];
}

\end{tikzpicture}
 \begin{tikzpicture}
 \draw [white] (-0.5,0) -- (0.5,0);
 \draw [dashed] (0,-1.2) -- (0,1);.2
 \end{tikzpicture}
 \picfirstfalse\picsecondtrue\picthirdfalse
 \begin{tikzpicture}[scale=1]
\def\vxrad{0.07cm}
\def\horunit{1.2}
\def\edgelength{0.4}
\def\betweenrows{0.5}


\def\gapratio{1}

\foreach \num/\parity/\parityy in {1/-1/-1,2/-1/1}
{
\coordinate (Y\num) at ($(0,0)+\parityy*\gapratio*(\horunit,0)+(0.5*\parity*\horunit,0.5*\horunit)$);
\coordinate (W\num) at ($(0,0)+\parityy*\gapratio*(\horunit,0)+(-0.5*\parity*\horunit,0.5*\horunit)$);
\coordinate (X\num) at ($(0,0)+\parityy*\gapratio*(\horunit,0)+(0.5*\parity*\horunit,-0.5*\horunit)$);
\coordinate (Z\num) at ($(0,0)+\parityy*\gapratio*(\horunit,0)+(-0.5*\parity*\horunit,-0.5*\horunit)$);
}

\ifpicfirst
\foreach \x/\y/\col/\n in {Z/W/magenta/1,W/Z/magenta/2}
{
\draw [thick,\col] (\x\n) -- (\y\n);
}
\foreach \x/\y/\col in {X/Y/orange}
\foreach \n in {1}
{
\draw [thick,\col, densely dashed] (\x\n) -- (\y\n);
}
\foreach \x/\y/\col in {X/Y/orange}
\foreach \n in {2}
{
\draw [thick,\col, densely dashed] (\x\n) -- (\y\n);
}
\fi

\ifpicsecond
\foreach \x/\y/\col/\n in {Z/W/magenta/1}
{
\draw [thick,\col] (\x\n) -- (\y\n);
}
\else
\foreach \x/\y/\col in {X/W/red,Y/Z/blue}
\foreach \n in {1}
{
\draw [thick,\col] (\x\n) -- (\y\n);
}

\fi
\ifpicthird
\foreach \x/\y/\col/\n in {W/Z/magenta/2}
{
\draw [thick,\col] (\x\n) -- (\y\n);
}
\else

\foreach \x/\y/\col in {X/W/red,Y/Z/blue}
\foreach \n in {2}
{
\draw [thick,\col] (\x\n) -- (\y\n);
}
\fi

\foreach \x/\lab/\offs in {X/v/-1.2,Y/u/1,Z/w/-1.2,W/x/1}
\foreach \num in {1,2}
{
\draw  ($(\x\num)+\offs*(0,0.25)$) node {$\lab_\num$};
}

\ifpicfirst
\draw ($0.5*(X1)+0.5*(Y1)-(0.3,0)$) node {$e_1$};
\draw ($0.5*(X2)+0.5*(Y2)-(0.3,0)$) node {$e_2$};
\fi

\ifpicfirst
\draw [orange] ($0.5*(X2)+0.5*(Y2)+(0.2,0)$) node {$c$};
\draw [orange] ($0.5*(X1)+0.5*(Y1)+(0.2,0)$) node {$c$};
\fi
\ifpicthird
\else
\draw [red] ($0.5*(X2)+0.5*(W2)+(0.025,-0.25)+0.15*(-\horunit,-\horunit)$) node {$d'$};
\draw [blue] ($0.5*(Y2)+0.5*(Z2)+(0.025,0.22)+0.15*(-\horunit,\horunit)$) node {$d$};

\draw [magenta] ($0.5*(W1)+0.5*(Z1)+(0.2,0)$) node {$c'$};
\fi
\ifpicsecond
\else
\draw [red] ($0.5*(X1)+0.5*(W1)+(0.025,-0.25)+0.15*(-\horunit,-\horunit)$) node {$d'$};
\draw [blue] ($0.5*(Y1)+0.5*(Z1)+(0.025,0.22)+0.15*(-\horunit,\horunit)$) node {$d$};

\draw [magenta] ($0.5*(W2)+0.5*(Z2)+(0.2,0)$) node {$c'$};
\fi

\ifpicsecond
\fi

\ifpicthird
\fi

\foreach \x in {X,W,Y,Z}
\foreach \n in {1,2}
{
\draw [fill] (\x\n) circle [radius=\vxrad];
}

\end{tikzpicture}
 \begin{tikzpicture}
 \draw [white] (-0.5,0) -- (0.5,0);
 \draw [dashed] (0,-1.2) -- (0,1);.2
 \end{tikzpicture}
 \picfirstfalse\picsecondfalse\picthirdtrue
 \begin{tikzpicture}[scale=1]
\def\vxrad{0.07cm}
\def\horunit{1.2}
\def\edgelength{0.4}
\def\betweenrows{0.5}


\def\gapratio{1}

\foreach \num/\parity/\parityy in {1/-1/-1,2/-1/1}
{
\coordinate (Y\num) at ($(0,0)+\parityy*\gapratio*(\horunit,0)+(0.5*\parity*\horunit,0.5*\horunit)$);
\coordinate (W\num) at ($(0,0)+\parityy*\gapratio*(\horunit,0)+(-0.5*\parity*\horunit,0.5*\horunit)$);
\coordinate (X\num) at ($(0,0)+\parityy*\gapratio*(\horunit,0)+(0.5*\parity*\horunit,-0.5*\horunit)$);
\coordinate (Z\num) at ($(0,0)+\parityy*\gapratio*(\horunit,0)+(-0.5*\parity*\horunit,-0.5*\horunit)$);
}

\ifpicfirst
\foreach \x/\y/\col/\n in {Z/W/magenta/1,W/Z/magenta/2}
{
\draw [thick,\col] (\x\n) -- (\y\n);
}
\foreach \x/\y/\col in {X/Y/orange}
\foreach \n in {1}
{
\draw [thick,\col, densely dashed] (\x\n) -- (\y\n);
}
\foreach \x/\y/\col in {X/Y/orange}
\foreach \n in {2}
{
\draw [thick,\col, densely dashed] (\x\n) -- (\y\n);
}
\fi

\ifpicsecond
\foreach \x/\y/\col/\n in {Z/W/magenta/1}
{
\draw [thick,\col] (\x\n) -- (\y\n);
}
\else
\foreach \x/\y/\col in {X/W/red,Y/Z/blue}
\foreach \n in {1}
{
\draw [thick,\col] (\x\n) -- (\y\n);
}

\fi
\ifpicthird
\foreach \x/\y/\col/\n in {W/Z/magenta/2}
{
\draw [thick,\col] (\x\n) -- (\y\n);
}
\else

\foreach \x/\y/\col in {X/W/red,Y/Z/blue}
\foreach \n in {2}
{
\draw [thick,\col] (\x\n) -- (\y\n);
}
\fi

\foreach \x/\lab/\offs in {X/v/-1.2,Y/u/1,Z/w/-1.2,W/x/1}
\foreach \num in {1,2}
{
\draw  ($(\x\num)+\offs*(0,0.25)$) node {$\lab_\num$};
}

\ifpicfirst
\draw ($0.5*(X1)+0.5*(Y1)-(0.3,0)$) node {$e_1$};
\draw ($0.5*(X2)+0.5*(Y2)-(0.3,0)$) node {$e_2$};
\fi

\ifpicfirst
\draw [orange] ($0.5*(X2)+0.5*(Y2)+(0.2,0)$) node {$c$};
\draw [orange] ($0.5*(X1)+0.5*(Y1)+(0.2,0)$) node {$c$};
\fi
\ifpicthird
\else
\draw [red] ($0.5*(X2)+0.5*(W2)+(0.025,-0.25)+0.15*(-\horunit,-\horunit)$) node {$d'$};
\draw [blue] ($0.5*(Y2)+0.5*(Z2)+(0.025,0.22)+0.15*(-\horunit,\horunit)$) node {$d$};

\draw [magenta] ($0.5*(W1)+0.5*(Z1)+(0.2,0)$) node {$c'$};
\fi
\ifpicsecond
\else
\draw [red] ($0.5*(X1)+0.5*(W1)+(0.025,-0.25)+0.15*(-\horunit,-\horunit)$) node {$d'$};
\draw [blue] ($0.5*(Y1)+0.5*(Z1)+(0.025,0.22)+0.15*(-\horunit,\horunit)$) node {$d$};

\draw [magenta] ($0.5*(W2)+0.5*(Z2)+(0.2,0)$) node {$c'$};
\fi

\ifpicsecond
\fi

\ifpicthird
\fi

\foreach \x in {X,W,Y,Z}
\foreach \n in {1,2}
{
\draw [fill] (\x\n) circle [radius=\vxrad];
}

\end{tikzpicture}
 \end{minipage}

 \caption{An $e_1,e_2$-edge-switcher of order 4 depicted on the left with vertex set $\hat{V}=\{w_1,x_1,w_2,x_2\}$ and colour set $\hat{C}=\{d,d',c'\}$ (see Definition~\ref{defn:edgeexchanger}),
 along with two exactly-$\{d,d',c'\}$-rainbow matchings with vertex set $\hat{V}\cup \{u_2,v_2\}$ and $\hat{V}\cup \{u_1,v_1\}$, respectively.}\label{fig:edgeswitcher}
 \end{figure}
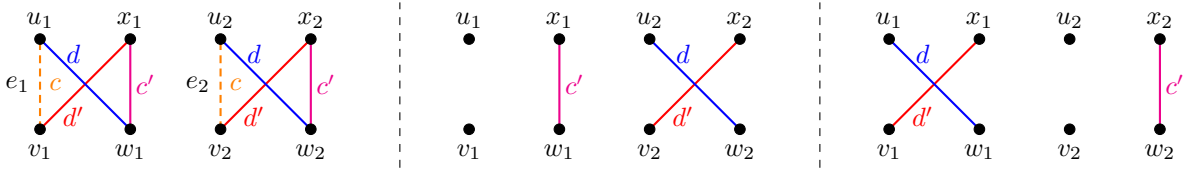

   \smallskip

   \noindent\textbf{A note on more difficult colourings.} If we do not have property \textbf{P}, then $w_1x_1$ and $w_2x_2$ may not have the same colour. However, we will have started by finding a set of colour classes $\mathcal{C}$ such that, for typical choices of the edges and vertices, the colour of $w_1x_1$ and $w_2x_2$, $c'$ and $c''$ say, are in some colour class $C\in \mathcal{C}$ together (i.e., such that an approximate version of property \textbf{P'} holds). This class will have the property that we can find a set $\hat{V}'$ of unused vertices and a set $\hat{C}'$ of unused colours so that $|\hat{V}'|=2|\hat{C}|+2$ such that $G[\hat{V}']$ has an exactly-$(\hat{C}\cup \{c'\})$-rainbow matching and  an exactly-$(\hat{C}\cup \{c''\})$-rainbow matching, i.e.\ we will find a colour switcher which can switch between using $c'$ and $c''$. We can then observe that $(\hat{V}\cup \hat{V}',\hat{C}\cup \hat{C'}\cup \{c''\})$ is an $e_1,e_2$-switcher.

  \smallskip

  \noindent\textbf{Absorbing one edge from a set of 100 monochromatic edges.} Suppose then we have colour $c_0$-edges $e_i=u_iv_i$, $i\in [100]$, and, as above, find colours $d,d',c'$ and label vertices such that, for each $i\in [100]$, $u_iv_ix_iw_iu_i$ is a 4-cycle in $G$ with edge colours $c_0,d,c',d'$ in that order, where these 4-cycles are all vertex disjoint (see Figure~\ref{fig:abs}). Using new vertices and colours, and a new edge $wx$ with colour $c'$, suppose that we can find, disjointly, for each $i\in [100]$, a $w_ix_i,wx$-switcher $(\hat{V}_i,\hat{C}_i)$. Then,
  \[
  (\{w,x\}\cup (\cup_{i\in [100]}\hat{V}_i\cup\{w_i,x_i\}),\{d,d'\}\cup(\cup_{i\in [100]}\hat{C}_i))
  \]
  can `absorb' $V(e_i)$ for any one edge $e_i$, $i\in [100]$. Indeed, for each $i\in [100]$, if we wish to absorb $V(e_i)$, then we can use the edges $u_iw_i,v_ix_i$ to cover $V(e_i)\cup\{w_i,x_i\}$, use the switch $(\hat{V}_i,\hat{C}_i)$ to cover $\{w,x\}$, and use each switch $(\hat{V}_j,\hat{C}_j)$, $j\in [100]\setminus\{i\}$ to cover $\{w_j,x_j\}$ (see Figure~\ref{fig:abs} for a depiction with $i=2$).

When we do not have property \textbf{P}, carrying out this construction while additionally switching between colours in the same class is certainly more difficult, but this approach using distribution absorption (via Lemma~\ref{Lemma_H_graph}) and the constructions depicted in Figures~\ref{fig:edgeswitcher} and~\ref{fig:abs}, is the heart of our construction.

  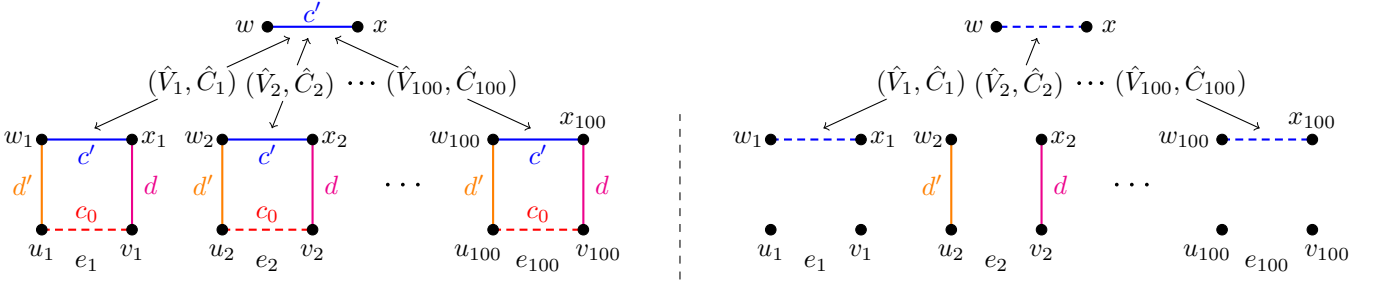
\begin{figure}
  \hspace{-01cm}
  \begin{minipage}{1.1\textwidth}
  \picfirsttrue\picsecondfalse\picthirdfalse
  \begin{tikzpicture}[scale=1]
\def\vxrad{0.07cm}
\def\horunit{1.2}
\def\edgelength{0.4}
\def\betweenrows{0.5}


\def\gapratio{1}

\foreach \num/\parity/\parityy in {1/-1/-1,2/-1/1,3/-1/4}
{
\coordinate (X\num) at ($(0,0)+\parityy*\gapratio*(\horunit,0)+(0.5*\parity*\horunit,0.5*\horunit)$);
\coordinate (W\num) at ($(0,0)+\parityy*\gapratio*(\horunit,0)+(-0.5*\parity*\horunit,0.5*\horunit)$);
\coordinate (U\num) at ($(0,0)+\parityy*\gapratio*(\horunit,0)+(0.5*\parity*\horunit,-0.5*\horunit)$);
\coordinate (V\num) at ($(0,0)+\parityy*\gapratio*(\horunit,0)+(-0.5*\parity*\horunit,-0.5*\horunit)$);
}

\coordinate (X) at ($0.5*(X1)+0.5*(W3)+(0,1.5)-0.5*(\horunit,0)$);
\coordinate (W) at ($0.5*(X1)+0.5*(W3)+(0,1.5)+0.5*(\horunit,0)$);

\foreach \x/\lab/\offs in {U/u/-1.2,V/v/-1.2}
\foreach \num/\labb in {1/1,2/2}
{
\draw  ($(\x\num)+\offs*(0,0.25)$) node {$\lab_{\labb}$};
}
\foreach \x/\lab/\offs/\ooffs in {U/u/-1.2/-0.2,V/v/-1.2/0.2}
\foreach \num/\labb in {3/100}
{
\draw  ($(\x\num)+\offs*(0,0.25)+\ooffs*(1,0)$) node {$\lab_{\labb}$};
}
\foreach \x/\lab/\offs in {W/x/1,X/w/-1}
\foreach \num/\labb in {1/1,2/2}
{
\draw  ($(\x\num)+\offs*(0.3,0)$) node {$\lab_{\labb}$};
}
\foreach \x/\lab/\offs/\ooffs in {W/x/1/0.2}
\foreach \num/\labb in {3/100}
{
\draw  ($(\x\num)+\offs*(0,0.25)+\ooffs*(0,0)$) node {$\lab_{\labb}$};
}
\foreach \x/\lab/\offs/\ooffs in {X/w/-1/-0.5}
\foreach \num/\labb in {3/100}
{
\draw  ($(\x\num)+\offs*(0,0)+\ooffs*(1,0)$) node {$\lab_{\labb}$};
}
\foreach \num/\labb in {1/1,2/2,3/100}
{
\draw  ($0.5*(U\num)+0.5*(V\num)-(0,0.45)$) node {$e_{\labb}$};
}

\draw ($(X)-(0.3,0)$) node {$w$};
\draw ($(W)+(0.3,0)$) node {$x$};

\ifpicfirst
\foreach \x/\y/\col in {U/V/red}
\foreach \n in {1,2,3}
{
\draw [thick,\col,densely dashed] (\x\n) -- (\y\n);
}
\foreach \x/\y/\col in {U/X/orange}
\foreach \n in {1,2,3}
{
\draw [thick,\col] (\x\n) -- (\y\n);
}
\foreach \x/\y/\col in {W/V/magenta}
\foreach \n in {1,2,3}
{
\draw [thick,\col] (\x\n) -- (\y\n);
}
\foreach \x/\y/\col in {X/W/blue}
\foreach \n in {1,2,3}
{
\draw [thick,\col] (\x\n) -- (\y\n);
}

\draw [thick,blue] (W) -- (X);
\fi
\ifpicsecond
\draw [thick,blue, densely dashed] (W) -- (X);
\foreach \x/\y/\col in {X/W/blue}
\foreach \n in {1,3}
{
\draw [thick,\col,densely dashed] (\x\n) -- (\y\n);
}
\foreach \x/\y/\col in {W/V/magenta}
\foreach \n in {2}
{
\draw [thick,\col] (\x\n) -- (\y\n);
}
\foreach \x/\y/\col in {U/X/orange}
\foreach \n in {2}
{
\draw [thick,\col] (\x\n) -- (\y\n);
}
\fi


\foreach\num in {2}
{
\draw [orange] ($0.5*(U\num)+0.5*(X\num)+(-0.25,0)$) node {$d'$};
\draw [magenta] ($0.5*(V\num)+0.5*(W\num)+(0.25,0)$) node {$d$};
}

\ifpicfirst
\foreach\num in {2}
{
\draw [red] ($0.5*(U\num)+0.5*(V\num)+(0,0.2)$) node {$c_0$};
\draw [blue] ($0.5*(W\num)+0.5*(X\num)+(0,-0.2)$) node {$c'$};
}
\draw [blue] ($0.5*(W)+0.5*(X)+(0,0.2)$) node {$c'$};
\foreach\num in {1,3}
{
\draw [red] ($0.5*(U\num)+0.5*(V\num)+(0,0.2)$) node {$c_0$};
\draw [orange] ($0.5*(U\num)+0.5*(X\num)+(-0.25,0)$) node {$d'$};
\draw [magenta] ($0.5*(V\num)+0.5*(W\num)+(0.25,0)$) node {$d$};
\draw [blue] ($0.5*(W\num)+0.5*(X\num)+(0,-0.2)$) node {$c'$};
}
\fi

\coordinate (R1) at ($0.5*(X1)+0.5*(W1)$);
\coordinate (R2) at ($0.5*(X2)+0.5*(W2)$);
\coordinate (R3) at ($0.5*(X3)+0.5*(W3)$);

\coordinate (T) at ($0.5*(X)+0.5*(W)$);

\foreach \num in {2}
{
\coordinate (S\num1) at ($0.1*(R\num)+0.9*(T)$);
\coordinate (S\num2) at ($0.9*(R\num)+0.1*(T)$);
\coordinate (M\num) at ($0.5*(S\num1)+0.5*(S\num2)$);
}
\foreach \num in {1,3}
{
\coordinate (S\num1) at ($0.1*(R\num)+0.9*(T)$);
\coordinate (S\num2) at ($0.9*(R\num)+0.1*(T)+\num*(0.2,0)-2*(0.2,0)$);
\coordinate (M\num) at ($0.5*(S\num1)+0.5*(S\num2)$);
}

\ifpicfirst
\foreach \num in {1,2,3}
{
\draw [<->] (S\num1) -- (S\num2);
}
\fi
\ifpicsecond
\draw [<-] (S12) -- (M1);
\draw [<-] (S21) -- (M2);
\draw [<-] (S32) -- (M3);
\fi

\foreach \num in {2}
{
\draw [fill=white,white] (M\num) circle [radius=0.25];
\draw (M\num) node {$(\hat{V}_\num,\hat{C}_\num)$};
}
\foreach \num/\lab in {1/1}
{
\draw [fill=white,white] (M\num) circle [radius=0.5];
\draw ($(M\num)+(0,0.025)$) node {$(\hat{V}_{\lab},\hat{C}_{\lab})$};
}
\foreach \num/\lab in {3/100}
{
\draw [fill=white,white] (M\num) circle [radius=0.5];
\draw ($(M\num)+(0.25,0.025)$) node {$(\hat{V}_{\lab},\hat{C}_{\lab})$};
}

\foreach \x in {X,W,U,V}
\foreach \n in {1,2,3}
{
\draw [fill] (\x\n) circle [radius=\vxrad];
}

\draw [fill] (W) circle [radius=\vxrad];
\draw [fill] (X) circle [radius=\vxrad];

\foreach \y in {-1,0,1}
{
\draw  [fill] ($0.25*(X2)+0.25*(V2)+0.25*(W3)+0.25*(U3)+\y*(0.2,0)$) circle [radius=0.25*\vxrad];
\draw  [fill] ($0.5*(M2)+0.5*(M3)+\y*(0.15,0)$) circle [radius=0.25*\vxrad];
}

\end{tikzpicture}
  \begin{tikzpicture}
  \draw [white] (-0.5,0) -- (0.5,0);
  \draw [dashed] (0,-1.2) -- (0,1);.2
  \end{tikzpicture}
  \picfirstfalse\picsecondtrue\picthirdfalse
  \begin{tikzpicture}[scale=1]
\def\vxrad{0.07cm}
\def\horunit{1.2}
\def\edgelength{0.4}
\def\betweenrows{0.5}


\def\gapratio{1}

\foreach \num/\parity/\parityy in {1/-1/-1,2/-1/1,3/-1/4}
{
\coordinate (X\num) at ($(0,0)+\parityy*\gapratio*(\horunit,0)+(0.5*\parity*\horunit,0.5*\horunit)$);
\coordinate (W\num) at ($(0,0)+\parityy*\gapratio*(\horunit,0)+(-0.5*\parity*\horunit,0.5*\horunit)$);
\coordinate (U\num) at ($(0,0)+\parityy*\gapratio*(\horunit,0)+(0.5*\parity*\horunit,-0.5*\horunit)$);
\coordinate (V\num) at ($(0,0)+\parityy*\gapratio*(\horunit,0)+(-0.5*\parity*\horunit,-0.5*\horunit)$);
}

\coordinate (X) at ($0.5*(X1)+0.5*(W3)+(0,1.5)-0.5*(\horunit,0)$);
\coordinate (W) at ($0.5*(X1)+0.5*(W3)+(0,1.5)+0.5*(\horunit,0)$);

\foreach \x/\lab/\offs in {U/u/-1.2,V/v/-1.2}
\foreach \num/\labb in {1/1,2/2}
{
\draw  ($(\x\num)+\offs*(0,0.25)$) node {$\lab_{\labb}$};
}
\foreach \x/\lab/\offs/\ooffs in {U/u/-1.2/-0.2,V/v/-1.2/0.2}
\foreach \num/\labb in {3/100}
{
\draw  ($(\x\num)+\offs*(0,0.25)+\ooffs*(1,0)$) node {$\lab_{\labb}$};
}
\foreach \x/\lab/\offs in {W/x/1,X/w/-1}
\foreach \num/\labb in {1/1,2/2}
{
\draw  ($(\x\num)+\offs*(0.3,0)$) node {$\lab_{\labb}$};
}
\foreach \x/\lab/\offs/\ooffs in {W/x/1/0.2}
\foreach \num/\labb in {3/100}
{
\draw  ($(\x\num)+\offs*(0,0.25)+\ooffs*(0,0)$) node {$\lab_{\labb}$};
}
\foreach \x/\lab/\offs/\ooffs in {X/w/-1/-0.5}
\foreach \num/\labb in {3/100}
{
\draw  ($(\x\num)+\offs*(0,0)+\ooffs*(1,0)$) node {$\lab_{\labb}$};
}
\foreach \num/\labb in {1/1,2/2,3/100}
{
\draw  ($0.5*(U\num)+0.5*(V\num)-(0,0.45)$) node {$e_{\labb}$};
}

\draw ($(X)-(0.3,0)$) node {$w$};
\draw ($(W)+(0.3,0)$) node {$x$};

\ifpicfirst
\foreach \x/\y/\col in {U/V/red}
\foreach \n in {1,2,3}
{
\draw [thick,\col,densely dashed] (\x\n) -- (\y\n);
}
\foreach \x/\y/\col in {U/X/orange}
\foreach \n in {1,2,3}
{
\draw [thick,\col] (\x\n) -- (\y\n);
}
\foreach \x/\y/\col in {W/V/magenta}
\foreach \n in {1,2,3}
{
\draw [thick,\col] (\x\n) -- (\y\n);
}
\foreach \x/\y/\col in {X/W/blue}
\foreach \n in {1,2,3}
{
\draw [thick,\col] (\x\n) -- (\y\n);
}

\draw [thick,blue] (W) -- (X);
\fi
\ifpicsecond
\draw [thick,blue, densely dashed] (W) -- (X);
\foreach \x/\y/\col in {X/W/blue}
\foreach \n in {1,3}
{
\draw [thick,\col,densely dashed] (\x\n) -- (\y\n);
}
\foreach \x/\y/\col in {W/V/magenta}
\foreach \n in {2}
{
\draw [thick,\col] (\x\n) -- (\y\n);
}
\foreach \x/\y/\col in {U/X/orange}
\foreach \n in {2}
{
\draw [thick,\col] (\x\n) -- (\y\n);
}
\fi


\foreach\num in {2}
{
\draw [orange] ($0.5*(U\num)+0.5*(X\num)+(-0.25,0)$) node {$d'$};
\draw [magenta] ($0.5*(V\num)+0.5*(W\num)+(0.25,0)$) node {$d$};
}

\ifpicfirst
\foreach\num in {2}
{
\draw [red] ($0.5*(U\num)+0.5*(V\num)+(0,0.2)$) node {$c_0$};
\draw [blue] ($0.5*(W\num)+0.5*(X\num)+(0,-0.2)$) node {$c'$};
}
\draw [blue] ($0.5*(W)+0.5*(X)+(0,0.2)$) node {$c'$};
\foreach\num in {1,3}
{
\draw [red] ($0.5*(U\num)+0.5*(V\num)+(0,0.2)$) node {$c_0$};
\draw [orange] ($0.5*(U\num)+0.5*(X\num)+(-0.25,0)$) node {$d'$};
\draw [magenta] ($0.5*(V\num)+0.5*(W\num)+(0.25,0)$) node {$d$};
\draw [blue] ($0.5*(W\num)+0.5*(X\num)+(0,-0.2)$) node {$c'$};
}
\fi

\coordinate (R1) at ($0.5*(X1)+0.5*(W1)$);
\coordinate (R2) at ($0.5*(X2)+0.5*(W2)$);
\coordinate (R3) at ($0.5*(X3)+0.5*(W3)$);

\coordinate (T) at ($0.5*(X)+0.5*(W)$);

\foreach \num in {2}
{
\coordinate (S\num1) at ($0.1*(R\num)+0.9*(T)$);
\coordinate (S\num2) at ($0.9*(R\num)+0.1*(T)$);
\coordinate (M\num) at ($0.5*(S\num1)+0.5*(S\num2)$);
}
\foreach \num in {1,3}
{
\coordinate (S\num1) at ($0.1*(R\num)+0.9*(T)$);
\coordinate (S\num2) at ($0.9*(R\num)+0.1*(T)+\num*(0.2,0)-2*(0.2,0)$);
\coordinate (M\num) at ($0.5*(S\num1)+0.5*(S\num2)$);
}

\ifpicfirst
\foreach \num in {1,2,3}
{
\draw [<->] (S\num1) -- (S\num2);
}
\fi
\ifpicsecond
\draw [<-] (S12) -- (M1);
\draw [<-] (S21) -- (M2);
\draw [<-] (S32) -- (M3);
\fi

\foreach \num in {2}
{
\draw [fill=white,white] (M\num) circle [radius=0.25];
\draw (M\num) node {$(\hat{V}_\num,\hat{C}_\num)$};
}
\foreach \num/\lab in {1/1}
{
\draw [fill=white,white] (M\num) circle [radius=0.5];
\draw ($(M\num)+(0,0.025)$) node {$(\hat{V}_{\lab},\hat{C}_{\lab})$};
}
\foreach \num/\lab in {3/100}
{
\draw [fill=white,white] (M\num) circle [radius=0.5];
\draw ($(M\num)+(0.25,0.025)$) node {$(\hat{V}_{\lab},\hat{C}_{\lab})$};
}

\foreach \x in {X,W,U,V}
\foreach \n in {1,2,3}
{
\draw [fill] (\x\n) circle [radius=\vxrad];
}

\draw [fill] (W) circle [radius=\vxrad];
\draw [fill] (X) circle [radius=\vxrad];

\foreach \y in {-1,0,1}
{
\draw  [fill] ($0.25*(X2)+0.25*(V2)+0.25*(W3)+0.25*(U3)+\y*(0.2,0)$) circle [radius=0.25*\vxrad];
\draw  [fill] ($0.5*(M2)+0.5*(M3)+\y*(0.15,0)$) circle [radius=0.25*\vxrad];
}

\end{tikzpicture}
  \end{minipage}

\vspace{-0.4cm}

  \caption{An absorber on the left with vertex set $\{w,x\}\cup (\cup_{i\in [100]}\hat{V}_i\cup\{w_i,x_i\})$ and colour set $\{d,d'\}\cup(\cup_{i\in [100]}\hat{C}_i)$ (see Definition~\ref{defn:Eabs}), which can absorb the two vertices in $\{u_i,v_i\}$ for any one $i\in [100]$, and incorporates the $wx,w_ix_i$-switchers $(\hat{V}_i,\hat{C}_i)$, $i\in [100]$. On the right, $u_2$ and $v_2$ are absorbed by turning $(\hat{V}_2,\hat{C}_2)$ to cover $x$ and $w$, turning each other switch $(\hat{V}_i,\hat{C}_i)$ to cover $x_i$ and $w_i$, and using the exactly-$\{d,d'\}$-rainbow matching $\{u_2w_2,v_2x_2\}$ to cover $u_2,v_2,w_2$ and $x_2$.}\label{fig:abs}
  \end{figure}


\subsubsection{Addition structure}\label{sec:add}
As indicated above, we want to construct an addition structure, which will take as its input a general balanced vertex subset (disjoint from the addition structure) of a certain size and output a rainbow matching with predictable colour set (added to the final matching) and a matching of identity colour edges, whose vertex set can then be absorbed, as well as two vertices (which we think of as `remainder vertices'). Having selected some (relatively arbitrary) identity colour $c_0$, the addition structure will, for some $\ell_0\leq \ell_1$, consist of sets ${V}^{\mathrm{add}}\subset V(G)$ and ${C}^{\mathrm{add}}\subset C(G)\setminus\{c_0\}$
with $|{V}^{\mathrm{add}}|=2|C^{\mathrm{add}}|-2\ell_0-2\ell_1$
such that,
for any balanced set $W\subset V(G)\setminus {V}^{\mathrm{add}}$ of $2\ell_0+2$ vertices, $G[V^{\mathrm{add}}\cup W]$ contains
an exactly-$C^{\mathrm{add}}$-rainbow matching $M^{\mathrm{rb}}$ and a matching of $\ell_1$ colour-$c_0$ edges $M^{\mathrm{id}}$ such that these matchings are vertex-disjoint and thus together cover all but 2 vertices in ${V}^{\mathrm{add}}\cup W$ (which we call the two `remainder vertices').
Furthermore, we ensure that all the edges of the colour-$c_0$ matching are in $G[{V}^{\mathrm{add}}]$, so that we may create our absorption structure only for the vertex sets of colour-$c_0$ matchings in $G[{V}^{\mathrm{add}}]$.

\smallskip

\hspace{-1cm}\rule{0.3pt}{9.75cm}\hspace{\textwidth}\hspace{1cm}\rule{0.3pt}{9.75cm}

\vspace{-9.95cm}

\noindent \textbf{Overview.}
We can now give an overview of the structure of the proof (using illustrative variables matching those in~\eqref{eqn:sampleh}), where the vertex and colour partitions used are depicted in Figure~\ref{fig:abstract}. We set aside random sets $V^{\text{s-r}},V^{\text{ex}}\subset V(G)$ and $C^{\text{s-r}}\subset C(G)$ such that $V^{\text{s-r}}$ and $C^{\text{s-r}}$ are $(3/4)$-random and $V^{\text{ex}}$ is an extra $\eta$-random set which we use to incorporate unused colours before applying the addition and absorption structures to the remaining vertices. For $\ell_0=\eta n$ and $\ell_1=\gamma n$, we find
disjointly $V^{\mathrm{abs}},V^{\mathrm{add}}\subset V(G)\setminus (V^{\text{s-r}}\cup V^{\text{ex}})$ and $C^{\mathrm{abs}},C^{\mathrm{add}}\subset C(G)\setminus C^{\text{s-r}}$ with the properties we have discussed (with more detail as we use them below). As $|V^{\text{ex}}|\approx 2\gamma n=2\ell_0$ (with high probability), and thus $|V^{\mathrm{add}}|+|V^{\mathrm{ex}}|\approx 2|C^{\mathrm{add}}|$,
\begin{equation}\label{eqn:justabove}
|V(G)\setminus (V^{\mathrm{abs}}\cup V^{\mathrm{add}}\cup V^{\mathrm{ex}})|\approx 2|C(G)\setminus (C^{\mathrm{abs}}\cup C^{\mathrm{add}})|,
\end{equation}
so using \ref{prop:sketch1} we find a rainbow matching $M_1$ covering all but $2\eps n$ vertices and $\eps n$ colours in the sets whose sizes are compared in~\eqref{eqn:justabove} (where $\eps\llpoly \eta$). Then, using that $V^{\mathrm{ex}}$ has many more vertices than the remaining $\eps n$ colours and is chosen randomly, we expect to be able to greedily find a rainbow matching $M_2$ in $G[V^{\mathrm{ex}}]$ using all but one of these colours\Footnote{Or, if we wished, all of the colours, but we need to leave out one colour at some point.}, $c$ say.
That is, together we have $M_1\cup M_2$, a rainbow matching disjoint from $V^{\mathrm{abs}}\cup V^{\mathrm{add}}$ which uses all the colours not in $C^{\mathrm{abs}}\cup C^{\mathrm{add}}$, except for $c$. Setting $W=V(G)\setminus (V^{\mathrm{abs}}\cup V^{\mathrm{add}}\cup V(M_1\cup M_2))$, we observe that $|W|=2\ell_0$.
Then, using the addition structure property we find in $G[V^{\mathrm{add}}\cup W]$ an exactly-$C^{\mathrm{add}}$-rainbow matching $M^{\mathrm{rb}}$ and a matching of $\ell_1$ colour-$c_0$ edges $M^{\mathrm{id}}$ such that these matchings are vertex-disjoint and thus together cover all but 2 vertices in ${V}^{\mathrm{add}}\cup W$ -- say these two vertices are $w$ and $x$.
Then, we set $M_3=M^{\mathrm{rb}}$ and $\hat{W}=V(M^{\mathrm{id}})$, where we will additionally have that $\hat{W}\subset V^{\mathrm{add}}$. As $\hat{W}$ is the vertex set of $\ell_1$ colour-$c_0$ edges, we then will have a version of \ref{prop:sketch2} for such sets within $V^{\mathrm{add}}$ and so can find an exactly-$C^{\mathrm{abs}}$-rainbow matching $M_4$ with vertex set $V^{\mathrm{abs}}\cup W$. Putting this all together, we will have that $M_1\cup M_2\cup M_3\cup M_4$ is a rainbow matching covering all the vertices of $V(G)$ except for $w$ and $z$ (and using all the colours except for $c$), which, therefore, has $n-1$ edges.



\begin{figure}
\centering
\def\spacedown{0.3cm}

\hspace{-0.6cm}\boxpicfirsttrue\boxpicsecondfalse\boxpicthirdfalse\boxpicfourthfalse\boxpicfifthfalse\begin{tikzpicture}[scale=1]

\def\hgt{0.8};
\def\labh{0.2};

\def\gapp{0.1}
\def\wVsr{2.5};
\def\wVadd{1.2};
\def\wVabs{1.2};
\def\wVex{0.8};
\def\wVrest{1.5};
\def\wVrestone{0.3};
\def\wVresttwo{0.5};
\def\wVrestthree{0.6};
\def\wVrestfour{1.5};
\def\wMone{\wVsr+\wVrest-\wVrestone};
\def\wMtwo{\wVrestone+\wVex-\wVresttwo};
\def\wMthree{\wVresttwo+\wVadd-\wVrestthree};
\def\wMfour{\wVabs+\wVrestthree+\gapp};

\def\Vsr{\draw [thick] (0,0) rectangle ++(\wVsr,\hgt);
\draw ($0.5*(\wVsr,\hgt)$) node {$V^{\text{s-r}}$};}

\def\Vadd{\draw [thick] (0,0) rectangle ++(\wVadd,\hgt);
\draw ($0.5*(\wVadd,\hgt)$) node {$V^{\text{add}}$};}

\def\Vabs{
\draw [thick] (0,0) rectangle ++(\wVabs,\hgt);
\draw ($0.5*(\wVabs,\hgt)$) node {$V^{\text{abs}}$};
}

\def\Vex{\draw [thick] (0,0) rectangle ++(\wVex,\hgt);
\draw ($0.5*(\wVex,\hgt)$) node {$V^{\text{ex}}$};}

\def\Mone{\draw [thick,densely dashed] (0,0) rectangle ++(\wMone,\hgt);
\draw ($0.5*(\wMone,\hgt)$) node {$M_1$};}

\def\Mtwo{\draw [thick,densely dashed] (0,0) rectangle ++(\wMtwo,\hgt);
\draw ($0.5*(\wMtwo,\hgt)$) node {$M_2$};}

\def\Mthree{\draw [thick,densely dashed] (0,0) rectangle ++(\wMthree,\hgt);
\draw ($0.5*(\wMthree,\hgt)$) node {$M_3$};}

\def\boxout{0.15};

\def\Mfour{\draw [thick,densely dashed] (2*\boxout,0) -- (\wMfour,0) -- ++(0,\hgt) -- (0,\hgt) -- (0,\boxout) -- ++(2*\boxout,0) --++(0,-\boxout);
\draw ($0.5*(\wMfour,\hgt)$) node {$M_4$};
\draw [fill] (0.4*\boxout,0.4*\boxout) circle [radius = 0.02];
\draw [fill] (1.4*\boxout,0.4*\boxout) circle [radius = 0.02];}

\def\Vrest{\draw [thick] (0,0) rectangle ++(\wVrest,\hgt);}

\def\Vrestone{\draw [thick] (0,0) rectangle ++(\wVrestone,\hgt);}

\def\Vresttwo{\draw [thick] (0,0) rectangle ++(\wVresttwo,\hgt);
\draw ($0.5*(\wVresttwo,\hgt)$) node {$W$};
}

\def\Vrestthree{\draw [thick] (2*\boxout,0) -- (\wVrestthree,0) -- ++(0,\hgt) -- ++(-\wVrestthree,0) -- (0,\boxout) -- ++(2*\boxout,0) --++(0,-\boxout);
\draw [fill] (0.4*\boxout,0.4*\boxout) circle [radius = 0.02];
\draw [fill] (1.4*\boxout,0.4*\boxout) circle [radius = 0.02];
\draw ($0.5*(\wVrestthree,\hgt)+(0,0.05)$) node {$\hat{W}$};}

\def\Vrestfour{\draw [thick] (0,0) rectangle ++(\wVrestfour,\hgt);}

\draw [white] (-0.4,0.5*\hgt) node {iii)};

\ifboxpicfirst
\draw (-0.4,0.5*\hgt) node {i)};
\Vsr
\begin{scope}[shift={(\wVsr+\gapp,0)}]
\Vrest
\end{scope}
\begin{scope}[shift={(\wVrest+\wVsr+2*\gapp,0)}]
\Vex
\end{scope}
\begin{scope}[shift={(\wVrest+\wVsr+\wVex+3*\gapp,0)}]
\Vadd
\end{scope}
\begin{scope}[shift={(\wVrest+\wVsr+\wVex+\wVadd+4*\gapp,0)}]
\Vabs
\end{scope}

\fi

\ifboxpicsecond
\draw (-0.4,0.5*\hgt) node {ii)};
\Mone
\begin{scope}[shift={(\wMone+\gapp,0)}]
\Vrestone
\end{scope}
\begin{scope}[shift={(\wMone+\wVrestone+2*\gapp,0)}]
\Vex
\end{scope}
\begin{scope}[shift={(\wMone+\wVrestone+\wVex+3*\gapp,0)}]
\Vadd
\end{scope}
\begin{scope}[shift={(\wMone+\wVrestone+\wVex+\wVadd+4*\gapp,0)}]
\Vabs
\end{scope}
\fi

\ifboxpicthird
\draw (-0.4,0.5*\hgt) node {iii)};
\Mone
\begin{scope}[shift={(\wMone+\gapp,0)}]
\Mtwo
\end{scope}
\begin{scope}[shift={(\wMone+\wMtwo+2*\gapp,0)}]
\Vresttwo
\end{scope}
\begin{scope}[shift={(\wMone+\wMtwo+\wVresttwo+3*\gapp,0)}]
\Vadd
\end{scope}
\begin{scope}[shift={(\wMone+\wMtwo+\wVresttwo+\wVadd+4*\gapp,0)}]
\Vabs
\end{scope}
\fi

\ifboxpicfourth
\draw (-0.4,0.5*\hgt) node {iv)};
\Mone
\begin{scope}[shift={(\wMone+\gapp,0)}]
\Mtwo
\end{scope}
\begin{scope}[shift={(\wMone+\wMtwo+2*\gapp,0)}]
\Mthree
\end{scope}
\begin{scope}[shift={(\wMone+\wMtwo+\wMthree+3*\gapp,0)}]
\Vrestthree
\end{scope}
\begin{scope}[shift={(\wMone+\wMtwo+\wMthree+\wVrestthree+4*\gapp,0)}]
\Vabs
\end{scope}
\fi

\ifboxpicfifth
\draw (-0.4,0.5*\hgt) node {v)};
\Mone
\begin{scope}[shift={(\wMone+\gapp,0)}]
\Mtwo
\end{scope}
\begin{scope}[shift={(\wMone+\wMtwo+2*\gapp,0)}]
\Mthree
\end{scope}
\begin{scope}[shift={(\wMone+\wMtwo+\wMthree+3*\gapp,0)}]
\Mfour
\end{scope}
\fi

\end{tikzpicture}\hspace{0.5cm}\begin{tikzpicture}[scale=1]

\def\hgt{0.8};
\def\labh{0.2};

\def\gapp{0.1}
\def\wVsr{2.5};
\def\wVaddd{1.1};
\def\wVabss{1.2};
\def\wVex{0.8};
\def\wVrestt{1.5};
\def\wVrestonee{0.3};
\def\wVresttwoo{0.5};
\def\wVrestthreee{0.6};
\def\wVrestfourr{1.5};
\def\wMone{\wVsr+\wVrestt-\wVrestonee};
\def\wMtwo{\wVrestonee+\wVex-\wVresttwoo};
\def\wMthree{\wVresttwoo+\wVaddd-\wVrestthreee};
\def\wMfour{\wVabss+\wVrestthreee+\gapp};

\def\wVadd{\wVresttwoo+\wVaddd-\wVrestthreee};
\def\wVabs{\wVabss+\wVrestthreee+\gapp};
\def\wVrest{\wMone+\wMtwo-\wVsr};
\def\wVrestone{\wMtwo};
\def\wVresttwo{0.5};
\def\wVrestthree{0.6};
\def\wVrestfour{1.5};

\def\Vsr{\draw [thick] (0,0) rectangle ++(\wVsr,\hgt);
\draw ($0.5*(\wVsr,\hgt)$) node {$C^{\text{s-r}}$};}

\def\Vadd{\draw [thick] (0,0) rectangle ++(\wVadd,\hgt);
\draw ($0.5*(\wVadd,\hgt)$) node {$C^{\text{add}}$};}

\def\Vabs{
\draw [thick] (0,0) rectangle ++(\wVabs,\hgt);
\draw ($0.5*(\wVabs,\hgt)$) node {$C^{\text{abs}}$};
}

\def\Vex{\draw [thick] (0,0) rectangle ++(\wVex,\hgt);
\draw ($0.5*(\wVex,\hgt)$) node {$C^{\text{ex}}$};}

\def\Mone{\draw [thick,densely dashed] (0,0) rectangle ++(\wMone,\hgt);
\draw ($0.5*(\wMone,\hgt)$) node {$M_1$};}

\def\boxout{0.15};
\def\Mtwo{\draw [thick,densely dashed] (\boxout,0) -- (\wMtwo,0) -- ++(0,\hgt) -- (0,\hgt) -- (0,\boxout) -- ++(\boxout,0) --++(0,-\boxout);
\draw ($0.5*(\wMtwo,\hgt)$) node {$M_2$};
\draw [fill] (0.4*\boxout,0.4*\boxout) circle [radius = 0.02];}

\def\Mthree{\draw [thick,densely dashed] (0,0) rectangle ++(\wMthree,\hgt);
\draw ($0.5*(\wMthree,\hgt)$) node {$M_3$};}

\def\Mfour{\draw [thick,densely dashed] (0,0) rectangle ++(\wMfour,\hgt);
\draw ($0.5*(\wMfour,\hgt)$) node {$M_4$};}

\def\Vrest{\draw [thick] (0,0) rectangle ++(\wVrest,\hgt);}

\def\Vrestone{\draw [thick] (0,0) rectangle ++(\wVrestone,\hgt);}

\def\Vresttwo{\draw [thick] (0,0) rectangle ++(\wVresttwo,\hgt);
\draw ($0.5*(\wVresttwo,\hgt)$) node {$W$};
}

\def\Vrestthree{\draw [thick] (0,0) rectangle ++(\wVrestthree,\hgt);
\draw ($0.5*(\wVrestthree,\hgt)+(0,0.05)$) node {$\hat{W}$};}

\def\Vrestfour{\draw [thick] (0,0) rectangle ++(\wVrestfour,\hgt);}

\ifboxpicfirst
\Vsr
\begin{scope}[shift={(\wVsr+\gapp,0)}]
\Vrest
\end{scope}
\begin{scope}[shift={(\wVrest+\wVsr+2*\gapp,0)}]
\Vadd
\end{scope}
\begin{scope}[shift={(\wVrest+\wVsr+\wVadd+3*\gapp,0)}]
\Vabs
\end{scope}

\fi

\ifboxpicsecond
\Mone
\begin{scope}[shift={(\wMone+\gapp,0)}]
\Vrestone
\end{scope}
\begin{scope}[shift={(\wMone+\wVrestone+2*\gapp,0)}]
\Vadd
\end{scope}
\begin{scope}[shift={(\wMone+\wVrestone+\wVadd+3*\gapp,0)}]
\Vabs
\end{scope}
\fi

\ifboxpicthird
\Mone
\begin{scope}[shift={(\wMone+\gapp,0)}]
\Mtwo
\end{scope}
\begin{scope}[shift={(\wMone+\wMtwo+2*\gapp,0)}]
\Vadd
\end{scope}
\begin{scope}[shift={(\wMone+\wMtwo+\wVadd+3*\gapp,0)}]
\Vabs
\end{scope}
\fi

\ifboxpicfourth
\Mone
\begin{scope}[shift={(\wMone+\gapp,0)}]
\Mtwo
\end{scope}
\begin{scope}[shift={(\wMone+\wMtwo+2*\gapp,0)}]
\Mthree
\end{scope}
\begin{scope}[shift={(\wMone+\wMtwo+\wMthree+3*\gapp,0)}]
\Vabs
\end{scope}
\fi

\ifboxpicfifth
\Mone
\begin{scope}[shift={(\wMone+\gapp,0)}]
\Mtwo
\end{scope}
\begin{scope}[shift={(\wMone+\wMtwo+2*\gapp,0)}]
\Mthree
\end{scope}
\begin{scope}[shift={(\wMone+\wMtwo+\wMthree+3*\gapp,0)}]
\Mfour
\end{scope}
\fi

\end{tikzpicture}

\vspace{\spacedown}

\hspace{-0.6cm}\boxpicfirstfalse\boxpicsecondtrue\begin{tikzpicture}[scale=1]

\def\hgt{0.8};
\def\labh{0.2};

\def\gapp{0.1}
\def\wVsr{2.5};
\def\wVadd{1.2};
\def\wVabs{1.2};
\def\wVex{0.8};
\def\wVrest{1.5};
\def\wVrestone{0.3};
\def\wVresttwo{0.5};
\def\wVrestthree{0.6};
\def\wVrestfour{1.5};
\def\wMone{\wVsr+\wVrest-\wVrestone};
\def\wMtwo{\wVrestone+\wVex-\wVresttwo};
\def\wMthree{\wVresttwo+\wVadd-\wVrestthree};
\def\wMfour{\wVabs+\wVrestthree+\gapp};

\def\Vsr{\draw [thick] (0,0) rectangle ++(\wVsr,\hgt);
\draw ($0.5*(\wVsr,\hgt)$) node {$V^{\text{s-r}}$};}

\def\Vadd{\draw [thick] (0,0) rectangle ++(\wVadd,\hgt);
\draw ($0.5*(\wVadd,\hgt)$) node {$V^{\text{add}}$};}

\def\Vabs{
\draw [thick] (0,0) rectangle ++(\wVabs,\hgt);
\draw ($0.5*(\wVabs,\hgt)$) node {$V^{\text{abs}}$};
}

\def\Vex{\draw [thick] (0,0) rectangle ++(\wVex,\hgt);
\draw ($0.5*(\wVex,\hgt)$) node {$V^{\text{ex}}$};}

\def\Mone{\draw [thick,densely dashed] (0,0) rectangle ++(\wMone,\hgt);
\draw ($0.5*(\wMone,\hgt)$) node {$M_1$};}

\def\Mtwo{\draw [thick,densely dashed] (0,0) rectangle ++(\wMtwo,\hgt);
\draw ($0.5*(\wMtwo,\hgt)$) node {$M_2$};}

\def\Mthree{\draw [thick,densely dashed] (0,0) rectangle ++(\wMthree,\hgt);
\draw ($0.5*(\wMthree,\hgt)$) node {$M_3$};}

\def\boxout{0.15};

\def\Mfour{\draw [thick,densely dashed] (2*\boxout,0) -- (\wMfour,0) -- ++(0,\hgt) -- (0,\hgt) -- (0,\boxout) -- ++(2*\boxout,0) --++(0,-\boxout);
\draw ($0.5*(\wMfour,\hgt)$) node {$M_4$};
\draw [fill] (0.4*\boxout,0.4*\boxout) circle [radius = 0.02];
\draw [fill] (1.4*\boxout,0.4*\boxout) circle [radius = 0.02];}

\def\Vrest{\draw [thick] (0,0) rectangle ++(\wVrest,\hgt);}

\def\Vrestone{\draw [thick] (0,0) rectangle ++(\wVrestone,\hgt);}

\def\Vresttwo{\draw [thick] (0,0) rectangle ++(\wVresttwo,\hgt);
\draw ($0.5*(\wVresttwo,\hgt)$) node {$W$};
}

\def\Vrestthree{\draw [thick] (2*\boxout,0) -- (\wVrestthree,0) -- ++(0,\hgt) -- ++(-\wVrestthree,0) -- (0,\boxout) -- ++(2*\boxout,0) --++(0,-\boxout);
\draw [fill] (0.4*\boxout,0.4*\boxout) circle [radius = 0.02];
\draw [fill] (1.4*\boxout,0.4*\boxout) circle [radius = 0.02];
\draw ($0.5*(\wVrestthree,\hgt)+(0,0.05)$) node {$\hat{W}$};}

\def\Vrestfour{\draw [thick] (0,0) rectangle ++(\wVrestfour,\hgt);}

\draw [white] (-0.4,0.5*\hgt) node {iii)};

\ifboxpicfirst
\draw (-0.4,0.5*\hgt) node {i)};
\Vsr
\begin{scope}[shift={(\wVsr+\gapp,0)}]
\Vrest
\end{scope}
\begin{scope}[shift={(\wVrest+\wVsr+2*\gapp,0)}]
\Vex
\end{scope}
\begin{scope}[shift={(\wVrest+\wVsr+\wVex+3*\gapp,0)}]
\Vadd
\end{scope}
\begin{scope}[shift={(\wVrest+\wVsr+\wVex+\wVadd+4*\gapp,0)}]
\Vabs
\end{scope}

\fi

\ifboxpicsecond
\draw (-0.4,0.5*\hgt) node {ii)};
\Mone
\begin{scope}[shift={(\wMone+\gapp,0)}]
\Vrestone
\end{scope}
\begin{scope}[shift={(\wMone+\wVrestone+2*\gapp,0)}]
\Vex
\end{scope}
\begin{scope}[shift={(\wMone+\wVrestone+\wVex+3*\gapp,0)}]
\Vadd
\end{scope}
\begin{scope}[shift={(\wMone+\wVrestone+\wVex+\wVadd+4*\gapp,0)}]
\Vabs
\end{scope}
\fi

\ifboxpicthird
\draw (-0.4,0.5*\hgt) node {iii)};
\Mone
\begin{scope}[shift={(\wMone+\gapp,0)}]
\Mtwo
\end{scope}
\begin{scope}[shift={(\wMone+\wMtwo+2*\gapp,0)}]
\Vresttwo
\end{scope}
\begin{scope}[shift={(\wMone+\wMtwo+\wVresttwo+3*\gapp,0)}]
\Vadd
\end{scope}
\begin{scope}[shift={(\wMone+\wMtwo+\wVresttwo+\wVadd+4*\gapp,0)}]
\Vabs
\end{scope}
\fi

\ifboxpicfourth
\draw (-0.4,0.5*\hgt) node {iv)};
\Mone
\begin{scope}[shift={(\wMone+\gapp,0)}]
\Mtwo
\end{scope}
\begin{scope}[shift={(\wMone+\wMtwo+2*\gapp,0)}]
\Mthree
\end{scope}
\begin{scope}[shift={(\wMone+\wMtwo+\wMthree+3*\gapp,0)}]
\Vrestthree
\end{scope}
\begin{scope}[shift={(\wMone+\wMtwo+\wMthree+\wVrestthree+4*\gapp,0)}]
\Vabs
\end{scope}
\fi

\ifboxpicfifth
\draw (-0.4,0.5*\hgt) node {v)};
\Mone
\begin{scope}[shift={(\wMone+\gapp,0)}]
\Mtwo
\end{scope}
\begin{scope}[shift={(\wMone+\wMtwo+2*\gapp,0)}]
\Mthree
\end{scope}
\begin{scope}[shift={(\wMone+\wMtwo+\wMthree+3*\gapp,0)}]
\Mfour
\end{scope}
\fi

\end{tikzpicture}\hspace{0.5cm}\begin{tikzpicture}[scale=1]

\def\hgt{0.8};
\def\labh{0.2};

\def\gapp{0.1}
\def\wVsr{2.5};
\def\wVaddd{1.1};
\def\wVabss{1.2};
\def\wVex{0.8};
\def\wVrestt{1.5};
\def\wVrestonee{0.3};
\def\wVresttwoo{0.5};
\def\wVrestthreee{0.6};
\def\wVrestfourr{1.5};
\def\wMone{\wVsr+\wVrestt-\wVrestonee};
\def\wMtwo{\wVrestonee+\wVex-\wVresttwoo};
\def\wMthree{\wVresttwoo+\wVaddd-\wVrestthreee};
\def\wMfour{\wVabss+\wVrestthreee+\gapp};

\def\wVadd{\wVresttwoo+\wVaddd-\wVrestthreee};
\def\wVabs{\wVabss+\wVrestthreee+\gapp};
\def\wVrest{\wMone+\wMtwo-\wVsr};
\def\wVrestone{\wMtwo};
\def\wVresttwo{0.5};
\def\wVrestthree{0.6};
\def\wVrestfour{1.5};

\def\Vsr{\draw [thick] (0,0) rectangle ++(\wVsr,\hgt);
\draw ($0.5*(\wVsr,\hgt)$) node {$C^{\text{s-r}}$};}

\def\Vadd{\draw [thick] (0,0) rectangle ++(\wVadd,\hgt);
\draw ($0.5*(\wVadd,\hgt)$) node {$C^{\text{add}}$};}

\def\Vabs{
\draw [thick] (0,0) rectangle ++(\wVabs,\hgt);
\draw ($0.5*(\wVabs,\hgt)$) node {$C^{\text{abs}}$};
}

\def\Vex{\draw [thick] (0,0) rectangle ++(\wVex,\hgt);
\draw ($0.5*(\wVex,\hgt)$) node {$C^{\text{ex}}$};}

\def\Mone{\draw [thick,densely dashed] (0,0) rectangle ++(\wMone,\hgt);
\draw ($0.5*(\wMone,\hgt)$) node {$M_1$};}

\def\boxout{0.15};
\def\Mtwo{\draw [thick,densely dashed] (\boxout,0) -- (\wMtwo,0) -- ++(0,\hgt) -- (0,\hgt) -- (0,\boxout) -- ++(\boxout,0) --++(0,-\boxout);
\draw ($0.5*(\wMtwo,\hgt)$) node {$M_2$};
\draw [fill] (0.4*\boxout,0.4*\boxout) circle [radius = 0.02];}

\def\Mthree{\draw [thick,densely dashed] (0,0) rectangle ++(\wMthree,\hgt);
\draw ($0.5*(\wMthree,\hgt)$) node {$M_3$};}

\def\Mfour{\draw [thick,densely dashed] (0,0) rectangle ++(\wMfour,\hgt);
\draw ($0.5*(\wMfour,\hgt)$) node {$M_4$};}

\def\Vrest{\draw [thick] (0,0) rectangle ++(\wVrest,\hgt);}

\def\Vrestone{\draw [thick] (0,0) rectangle ++(\wVrestone,\hgt);}

\def\Vresttwo{\draw [thick] (0,0) rectangle ++(\wVresttwo,\hgt);
\draw ($0.5*(\wVresttwo,\hgt)$) node {$W$};
}

\def\Vrestthree{\draw [thick] (0,0) rectangle ++(\wVrestthree,\hgt);
\draw ($0.5*(\wVrestthree,\hgt)+(0,0.05)$) node {$\hat{W}$};}

\def\Vrestfour{\draw [thick] (0,0) rectangle ++(\wVrestfour,\hgt);}

\ifboxpicfirst
\Vsr
\begin{scope}[shift={(\wVsr+\gapp,0)}]
\Vrest
\end{scope}
\begin{scope}[shift={(\wVrest+\wVsr+2*\gapp,0)}]
\Vadd
\end{scope}
\begin{scope}[shift={(\wVrest+\wVsr+\wVadd+3*\gapp,0)}]
\Vabs
\end{scope}

\fi

\ifboxpicsecond
\Mone
\begin{scope}[shift={(\wMone+\gapp,0)}]
\Vrestone
\end{scope}
\begin{scope}[shift={(\wMone+\wVrestone+2*\gapp,0)}]
\Vadd
\end{scope}
\begin{scope}[shift={(\wMone+\wVrestone+\wVadd+3*\gapp,0)}]
\Vabs
\end{scope}
\fi

\ifboxpicthird
\Mone
\begin{scope}[shift={(\wMone+\gapp,0)}]
\Mtwo
\end{scope}
\begin{scope}[shift={(\wMone+\wMtwo+2*\gapp,0)}]
\Vadd
\end{scope}
\begin{scope}[shift={(\wMone+\wMtwo+\wVadd+3*\gapp,0)}]
\Vabs
\end{scope}
\fi

\ifboxpicfourth
\Mone
\begin{scope}[shift={(\wMone+\gapp,0)}]
\Mtwo
\end{scope}
\begin{scope}[shift={(\wMone+\wMtwo+2*\gapp,0)}]
\Mthree
\end{scope}
\begin{scope}[shift={(\wMone+\wMtwo+\wMthree+3*\gapp,0)}]
\Vabs
\end{scope}
\fi

\ifboxpicfifth
\Mone
\begin{scope}[shift={(\wMone+\gapp,0)}]
\Mtwo
\end{scope}
\begin{scope}[shift={(\wMone+\wMtwo+2*\gapp,0)}]
\Mthree
\end{scope}
\begin{scope}[shift={(\wMone+\wMtwo+\wMthree+3*\gapp,0)}]
\Mfour
\end{scope}
\fi

\end{tikzpicture}

\vspace{\spacedown}

\hspace{-0.6cm}\boxpicsecondfalse\boxpicthirdtrue\begin{tikzpicture}[scale=1]

\def\hgt{0.8};
\def\labh{0.2};

\def\gapp{0.1}
\def\wVsr{2.5};
\def\wVadd{1.2};
\def\wVabs{1.2};
\def\wVex{0.8};
\def\wVrest{1.5};
\def\wVrestone{0.3};
\def\wVresttwo{0.5};
\def\wVrestthree{0.6};
\def\wVrestfour{1.5};
\def\wMone{\wVsr+\wVrest-\wVrestone};
\def\wMtwo{\wVrestone+\wVex-\wVresttwo};
\def\wMthree{\wVresttwo+\wVadd-\wVrestthree};
\def\wMfour{\wVabs+\wVrestthree+\gapp};

\def\Vsr{\draw [thick] (0,0) rectangle ++(\wVsr,\hgt);
\draw ($0.5*(\wVsr,\hgt)$) node {$V^{\text{s-r}}$};}

\def\Vadd{\draw [thick] (0,0) rectangle ++(\wVadd,\hgt);
\draw ($0.5*(\wVadd,\hgt)$) node {$V^{\text{add}}$};}

\def\Vabs{
\draw [thick] (0,0) rectangle ++(\wVabs,\hgt);
\draw ($0.5*(\wVabs,\hgt)$) node {$V^{\text{abs}}$};
}

\def\Vex{\draw [thick] (0,0) rectangle ++(\wVex,\hgt);
\draw ($0.5*(\wVex,\hgt)$) node {$V^{\text{ex}}$};}

\def\Mone{\draw [thick,densely dashed] (0,0) rectangle ++(\wMone,\hgt);
\draw ($0.5*(\wMone,\hgt)$) node {$M_1$};}

\def\Mtwo{\draw [thick,densely dashed] (0,0) rectangle ++(\wMtwo,\hgt);
\draw ($0.5*(\wMtwo,\hgt)$) node {$M_2$};}

\def\Mthree{\draw [thick,densely dashed] (0,0) rectangle ++(\wMthree,\hgt);
\draw ($0.5*(\wMthree,\hgt)$) node {$M_3$};}

\def\boxout{0.15};

\def\Mfour{\draw [thick,densely dashed] (2*\boxout,0) -- (\wMfour,0) -- ++(0,\hgt) -- (0,\hgt) -- (0,\boxout) -- ++(2*\boxout,0) --++(0,-\boxout);
\draw ($0.5*(\wMfour,\hgt)$) node {$M_4$};
\draw [fill] (0.4*\boxout,0.4*\boxout) circle [radius = 0.02];
\draw [fill] (1.4*\boxout,0.4*\boxout) circle [radius = 0.02];}

\def\Vrest{\draw [thick] (0,0) rectangle ++(\wVrest,\hgt);}

\def\Vrestone{\draw [thick] (0,0) rectangle ++(\wVrestone,\hgt);}

\def\Vresttwo{\draw [thick] (0,0) rectangle ++(\wVresttwo,\hgt);
\draw ($0.5*(\wVresttwo,\hgt)$) node {$W$};
}

\def\Vrestthree{\draw [thick] (2*\boxout,0) -- (\wVrestthree,0) -- ++(0,\hgt) -- ++(-\wVrestthree,0) -- (0,\boxout) -- ++(2*\boxout,0) --++(0,-\boxout);
\draw [fill] (0.4*\boxout,0.4*\boxout) circle [radius = 0.02];
\draw [fill] (1.4*\boxout,0.4*\boxout) circle [radius = 0.02];
\draw ($0.5*(\wVrestthree,\hgt)+(0,0.05)$) node {$\hat{W}$};}

\def\Vrestfour{\draw [thick] (0,0) rectangle ++(\wVrestfour,\hgt);}

\draw [white] (-0.4,0.5*\hgt) node {iii)};

\ifboxpicfirst
\draw (-0.4,0.5*\hgt) node {i)};
\Vsr
\begin{scope}[shift={(\wVsr+\gapp,0)}]
\Vrest
\end{scope}
\begin{scope}[shift={(\wVrest+\wVsr+2*\gapp,0)}]
\Vex
\end{scope}
\begin{scope}[shift={(\wVrest+\wVsr+\wVex+3*\gapp,0)}]
\Vadd
\end{scope}
\begin{scope}[shift={(\wVrest+\wVsr+\wVex+\wVadd+4*\gapp,0)}]
\Vabs
\end{scope}

\fi

\ifboxpicsecond
\draw (-0.4,0.5*\hgt) node {ii)};
\Mone
\begin{scope}[shift={(\wMone+\gapp,0)}]
\Vrestone
\end{scope}
\begin{scope}[shift={(\wMone+\wVrestone+2*\gapp,0)}]
\Vex
\end{scope}
\begin{scope}[shift={(\wMone+\wVrestone+\wVex+3*\gapp,0)}]
\Vadd
\end{scope}
\begin{scope}[shift={(\wMone+\wVrestone+\wVex+\wVadd+4*\gapp,0)}]
\Vabs
\end{scope}
\fi

\ifboxpicthird
\draw (-0.4,0.5*\hgt) node {iii)};
\Mone
\begin{scope}[shift={(\wMone+\gapp,0)}]
\Mtwo
\end{scope}
\begin{scope}[shift={(\wMone+\wMtwo+2*\gapp,0)}]
\Vresttwo
\end{scope}
\begin{scope}[shift={(\wMone+\wMtwo+\wVresttwo+3*\gapp,0)}]
\Vadd
\end{scope}
\begin{scope}[shift={(\wMone+\wMtwo+\wVresttwo+\wVadd+4*\gapp,0)}]
\Vabs
\end{scope}
\fi

\ifboxpicfourth
\draw (-0.4,0.5*\hgt) node {iv)};
\Mone
\begin{scope}[shift={(\wMone+\gapp,0)}]
\Mtwo
\end{scope}
\begin{scope}[shift={(\wMone+\wMtwo+2*\gapp,0)}]
\Mthree
\end{scope}
\begin{scope}[shift={(\wMone+\wMtwo+\wMthree+3*\gapp,0)}]
\Vrestthree
\end{scope}
\begin{scope}[shift={(\wMone+\wMtwo+\wMthree+\wVrestthree+4*\gapp,0)}]
\Vabs
\end{scope}
\fi

\ifboxpicfifth
\draw (-0.4,0.5*\hgt) node {v)};
\Mone
\begin{scope}[shift={(\wMone+\gapp,0)}]
\Mtwo
\end{scope}
\begin{scope}[shift={(\wMone+\wMtwo+2*\gapp,0)}]
\Mthree
\end{scope}
\begin{scope}[shift={(\wMone+\wMtwo+\wMthree+3*\gapp,0)}]
\Mfour
\end{scope}
\fi

\end{tikzpicture}\hspace{0.5cm}\begin{tikzpicture}[scale=1]

\def\hgt{0.8};
\def\labh{0.2};

\def\gapp{0.1}
\def\wVsr{2.5};
\def\wVaddd{1.1};
\def\wVabss{1.2};
\def\wVex{0.8};
\def\wVrestt{1.5};
\def\wVrestonee{0.3};
\def\wVresttwoo{0.5};
\def\wVrestthreee{0.6};
\def\wVrestfourr{1.5};
\def\wMone{\wVsr+\wVrestt-\wVrestonee};
\def\wMtwo{\wVrestonee+\wVex-\wVresttwoo};
\def\wMthree{\wVresttwoo+\wVaddd-\wVrestthreee};
\def\wMfour{\wVabss+\wVrestthreee+\gapp};

\def\wVadd{\wVresttwoo+\wVaddd-\wVrestthreee};
\def\wVabs{\wVabss+\wVrestthreee+\gapp};
\def\wVrest{\wMone+\wMtwo-\wVsr};
\def\wVrestone{\wMtwo};
\def\wVresttwo{0.5};
\def\wVrestthree{0.6};
\def\wVrestfour{1.5};

\def\Vsr{\draw [thick] (0,0) rectangle ++(\wVsr,\hgt);
\draw ($0.5*(\wVsr,\hgt)$) node {$C^{\text{s-r}}$};}

\def\Vadd{\draw [thick] (0,0) rectangle ++(\wVadd,\hgt);
\draw ($0.5*(\wVadd,\hgt)$) node {$C^{\text{add}}$};}

\def\Vabs{
\draw [thick] (0,0) rectangle ++(\wVabs,\hgt);
\draw ($0.5*(\wVabs,\hgt)$) node {$C^{\text{abs}}$};
}

\def\Vex{\draw [thick] (0,0) rectangle ++(\wVex,\hgt);
\draw ($0.5*(\wVex,\hgt)$) node {$C^{\text{ex}}$};}

\def\Mone{\draw [thick,densely dashed] (0,0) rectangle ++(\wMone,\hgt);
\draw ($0.5*(\wMone,\hgt)$) node {$M_1$};}

\def\boxout{0.15};
\def\Mtwo{\draw [thick,densely dashed] (\boxout,0) -- (\wMtwo,0) -- ++(0,\hgt) -- (0,\hgt) -- (0,\boxout) -- ++(\boxout,0) --++(0,-\boxout);
\draw ($0.5*(\wMtwo,\hgt)$) node {$M_2$};
\draw [fill] (0.4*\boxout,0.4*\boxout) circle [radius = 0.02];}

\def\Mthree{\draw [thick,densely dashed] (0,0) rectangle ++(\wMthree,\hgt);
\draw ($0.5*(\wMthree,\hgt)$) node {$M_3$};}

\def\Mfour{\draw [thick,densely dashed] (0,0) rectangle ++(\wMfour,\hgt);
\draw ($0.5*(\wMfour,\hgt)$) node {$M_4$};}

\def\Vrest{\draw [thick] (0,0) rectangle ++(\wVrest,\hgt);}

\def\Vrestone{\draw [thick] (0,0) rectangle ++(\wVrestone,\hgt);}

\def\Vresttwo{\draw [thick] (0,0) rectangle ++(\wVresttwo,\hgt);
\draw ($0.5*(\wVresttwo,\hgt)$) node {$W$};
}

\def\Vrestthree{\draw [thick] (0,0) rectangle ++(\wVrestthree,\hgt);
\draw ($0.5*(\wVrestthree,\hgt)+(0,0.05)$) node {$\hat{W}$};}

\def\Vrestfour{\draw [thick] (0,0) rectangle ++(\wVrestfour,\hgt);}

\ifboxpicfirst
\Vsr
\begin{scope}[shift={(\wVsr+\gapp,0)}]
\Vrest
\end{scope}
\begin{scope}[shift={(\wVrest+\wVsr+2*\gapp,0)}]
\Vadd
\end{scope}
\begin{scope}[shift={(\wVrest+\wVsr+\wVadd+3*\gapp,0)}]
\Vabs
\end{scope}

\fi

\ifboxpicsecond
\Mone
\begin{scope}[shift={(\wMone+\gapp,0)}]
\Vrestone
\end{scope}
\begin{scope}[shift={(\wMone+\wVrestone+2*\gapp,0)}]
\Vadd
\end{scope}
\begin{scope}[shift={(\wMone+\wVrestone+\wVadd+3*\gapp,0)}]
\Vabs
\end{scope}
\fi

\ifboxpicthird
\Mone
\begin{scope}[shift={(\wMone+\gapp,0)}]
\Mtwo
\end{scope}
\begin{scope}[shift={(\wMone+\wMtwo+2*\gapp,0)}]
\Vadd
\end{scope}
\begin{scope}[shift={(\wMone+\wMtwo+\wVadd+3*\gapp,0)}]
\Vabs
\end{scope}
\fi

\ifboxpicfourth
\Mone
\begin{scope}[shift={(\wMone+\gapp,0)}]
\Mtwo
\end{scope}
\begin{scope}[shift={(\wMone+\wMtwo+2*\gapp,0)}]
\Mthree
\end{scope}
\begin{scope}[shift={(\wMone+\wMtwo+\wMthree+3*\gapp,0)}]
\Vabs
\end{scope}
\fi

\ifboxpicfifth
\Mone
\begin{scope}[shift={(\wMone+\gapp,0)}]
\Mtwo
\end{scope}
\begin{scope}[shift={(\wMone+\wMtwo+2*\gapp,0)}]
\Mthree
\end{scope}
\begin{scope}[shift={(\wMone+\wMtwo+\wMthree+3*\gapp,0)}]
\Mfour
\end{scope}
\fi

\end{tikzpicture}

\vspace{\spacedown}

\hspace{-0.6cm}\boxpicthirdfalse\boxpicfourthtrue\begin{tikzpicture}[scale=1]

\def\hgt{0.8};
\def\labh{0.2};

\def\gapp{0.1}
\def\wVsr{2.5};
\def\wVadd{1.2};
\def\wVabs{1.2};
\def\wVex{0.8};
\def\wVrest{1.5};
\def\wVrestone{0.3};
\def\wVresttwo{0.5};
\def\wVrestthree{0.6};
\def\wVrestfour{1.5};
\def\wMone{\wVsr+\wVrest-\wVrestone};
\def\wMtwo{\wVrestone+\wVex-\wVresttwo};
\def\wMthree{\wVresttwo+\wVadd-\wVrestthree};
\def\wMfour{\wVabs+\wVrestthree+\gapp};

\def\Vsr{\draw [thick] (0,0) rectangle ++(\wVsr,\hgt);
\draw ($0.5*(\wVsr,\hgt)$) node {$V^{\text{s-r}}$};}

\def\Vadd{\draw [thick] (0,0) rectangle ++(\wVadd,\hgt);
\draw ($0.5*(\wVadd,\hgt)$) node {$V^{\text{add}}$};}

\def\Vabs{
\draw [thick] (0,0) rectangle ++(\wVabs,\hgt);
\draw ($0.5*(\wVabs,\hgt)$) node {$V^{\text{abs}}$};
}

\def\Vex{\draw [thick] (0,0) rectangle ++(\wVex,\hgt);
\draw ($0.5*(\wVex,\hgt)$) node {$V^{\text{ex}}$};}

\def\Mone{\draw [thick,densely dashed] (0,0) rectangle ++(\wMone,\hgt);
\draw ($0.5*(\wMone,\hgt)$) node {$M_1$};}

\def\Mtwo{\draw [thick,densely dashed] (0,0) rectangle ++(\wMtwo,\hgt);
\draw ($0.5*(\wMtwo,\hgt)$) node {$M_2$};}

\def\Mthree{\draw [thick,densely dashed] (0,0) rectangle ++(\wMthree,\hgt);
\draw ($0.5*(\wMthree,\hgt)$) node {$M_3$};}

\def\boxout{0.15};

\def\Mfour{\draw [thick,densely dashed] (2*\boxout,0) -- (\wMfour,0) -- ++(0,\hgt) -- (0,\hgt) -- (0,\boxout) -- ++(2*\boxout,0) --++(0,-\boxout);
\draw ($0.5*(\wMfour,\hgt)$) node {$M_4$};
\draw [fill] (0.4*\boxout,0.4*\boxout) circle [radius = 0.02];
\draw [fill] (1.4*\boxout,0.4*\boxout) circle [radius = 0.02];}

\def\Vrest{\draw [thick] (0,0) rectangle ++(\wVrest,\hgt);}

\def\Vrestone{\draw [thick] (0,0) rectangle ++(\wVrestone,\hgt);}

\def\Vresttwo{\draw [thick] (0,0) rectangle ++(\wVresttwo,\hgt);
\draw ($0.5*(\wVresttwo,\hgt)$) node {$W$};
}

\def\Vrestthree{\draw [thick] (2*\boxout,0) -- (\wVrestthree,0) -- ++(0,\hgt) -- ++(-\wVrestthree,0) -- (0,\boxout) -- ++(2*\boxout,0) --++(0,-\boxout);
\draw [fill] (0.4*\boxout,0.4*\boxout) circle [radius = 0.02];
\draw [fill] (1.4*\boxout,0.4*\boxout) circle [radius = 0.02];
\draw ($0.5*(\wVrestthree,\hgt)+(0,0.05)$) node {$\hat{W}$};}

\def\Vrestfour{\draw [thick] (0,0) rectangle ++(\wVrestfour,\hgt);}

\draw [white] (-0.4,0.5*\hgt) node {iii)};

\ifboxpicfirst
\draw (-0.4,0.5*\hgt) node {i)};
\Vsr
\begin{scope}[shift={(\wVsr+\gapp,0)}]
\Vrest
\end{scope}
\begin{scope}[shift={(\wVrest+\wVsr+2*\gapp,0)}]
\Vex
\end{scope}
\begin{scope}[shift={(\wVrest+\wVsr+\wVex+3*\gapp,0)}]
\Vadd
\end{scope}
\begin{scope}[shift={(\wVrest+\wVsr+\wVex+\wVadd+4*\gapp,0)}]
\Vabs
\end{scope}

\fi

\ifboxpicsecond
\draw (-0.4,0.5*\hgt) node {ii)};
\Mone
\begin{scope}[shift={(\wMone+\gapp,0)}]
\Vrestone
\end{scope}
\begin{scope}[shift={(\wMone+\wVrestone+2*\gapp,0)}]
\Vex
\end{scope}
\begin{scope}[shift={(\wMone+\wVrestone+\wVex+3*\gapp,0)}]
\Vadd
\end{scope}
\begin{scope}[shift={(\wMone+\wVrestone+\wVex+\wVadd+4*\gapp,0)}]
\Vabs
\end{scope}
\fi

\ifboxpicthird
\draw (-0.4,0.5*\hgt) node {iii)};
\Mone
\begin{scope}[shift={(\wMone+\gapp,0)}]
\Mtwo
\end{scope}
\begin{scope}[shift={(\wMone+\wMtwo+2*\gapp,0)}]
\Vresttwo
\end{scope}
\begin{scope}[shift={(\wMone+\wMtwo+\wVresttwo+3*\gapp,0)}]
\Vadd
\end{scope}
\begin{scope}[shift={(\wMone+\wMtwo+\wVresttwo+\wVadd+4*\gapp,0)}]
\Vabs
\end{scope}
\fi

\ifboxpicfourth
\draw (-0.4,0.5*\hgt) node {iv)};
\Mone
\begin{scope}[shift={(\wMone+\gapp,0)}]
\Mtwo
\end{scope}
\begin{scope}[shift={(\wMone+\wMtwo+2*\gapp,0)}]
\Mthree
\end{scope}
\begin{scope}[shift={(\wMone+\wMtwo+\wMthree+3*\gapp,0)}]
\Vrestthree
\end{scope}
\begin{scope}[shift={(\wMone+\wMtwo+\wMthree+\wVrestthree+4*\gapp,0)}]
\Vabs
\end{scope}
\fi

\ifboxpicfifth
\draw (-0.4,0.5*\hgt) node {v)};
\Mone
\begin{scope}[shift={(\wMone+\gapp,0)}]
\Mtwo
\end{scope}
\begin{scope}[shift={(\wMone+\wMtwo+2*\gapp,0)}]
\Mthree
\end{scope}
\begin{scope}[shift={(\wMone+\wMtwo+\wMthree+3*\gapp,0)}]
\Mfour
\end{scope}
\fi

\end{tikzpicture}\hspace{0.5cm}\begin{tikzpicture}[scale=1]

\def\hgt{0.8};
\def\labh{0.2};

\def\gapp{0.1}
\def\wVsr{2.5};
\def\wVaddd{1.1};
\def\wVabss{1.2};
\def\wVex{0.8};
\def\wVrestt{1.5};
\def\wVrestonee{0.3};
\def\wVresttwoo{0.5};
\def\wVrestthreee{0.6};
\def\wVrestfourr{1.5};
\def\wMone{\wVsr+\wVrestt-\wVrestonee};
\def\wMtwo{\wVrestonee+\wVex-\wVresttwoo};
\def\wMthree{\wVresttwoo+\wVaddd-\wVrestthreee};
\def\wMfour{\wVabss+\wVrestthreee+\gapp};

\def\wVadd{\wVresttwoo+\wVaddd-\wVrestthreee};
\def\wVabs{\wVabss+\wVrestthreee+\gapp};
\def\wVrest{\wMone+\wMtwo-\wVsr};
\def\wVrestone{\wMtwo};
\def\wVresttwo{0.5};
\def\wVrestthree{0.6};
\def\wVrestfour{1.5};

\def\Vsr{\draw [thick] (0,0) rectangle ++(\wVsr,\hgt);
\draw ($0.5*(\wVsr,\hgt)$) node {$C^{\text{s-r}}$};}

\def\Vadd{\draw [thick] (0,0) rectangle ++(\wVadd,\hgt);
\draw ($0.5*(\wVadd,\hgt)$) node {$C^{\text{add}}$};}

\def\Vabs{
\draw [thick] (0,0) rectangle ++(\wVabs,\hgt);
\draw ($0.5*(\wVabs,\hgt)$) node {$C^{\text{abs}}$};
}

\def\Vex{\draw [thick] (0,0) rectangle ++(\wVex,\hgt);
\draw ($0.5*(\wVex,\hgt)$) node {$C^{\text{ex}}$};}

\def\Mone{\draw [thick,densely dashed] (0,0) rectangle ++(\wMone,\hgt);
\draw ($0.5*(\wMone,\hgt)$) node {$M_1$};}

\def\boxout{0.15};
\def\Mtwo{\draw [thick,densely dashed] (\boxout,0) -- (\wMtwo,0) -- ++(0,\hgt) -- (0,\hgt) -- (0,\boxout) -- ++(\boxout,0) --++(0,-\boxout);
\draw ($0.5*(\wMtwo,\hgt)$) node {$M_2$};
\draw [fill] (0.4*\boxout,0.4*\boxout) circle [radius = 0.02];}

\def\Mthree{\draw [thick,densely dashed] (0,0) rectangle ++(\wMthree,\hgt);
\draw ($0.5*(\wMthree,\hgt)$) node {$M_3$};}

\def\Mfour{\draw [thick,densely dashed] (0,0) rectangle ++(\wMfour,\hgt);
\draw ($0.5*(\wMfour,\hgt)$) node {$M_4$};}

\def\Vrest{\draw [thick] (0,0) rectangle ++(\wVrest,\hgt);}

\def\Vrestone{\draw [thick] (0,0) rectangle ++(\wVrestone,\hgt);}

\def\Vresttwo{\draw [thick] (0,0) rectangle ++(\wVresttwo,\hgt);
\draw ($0.5*(\wVresttwo,\hgt)$) node {$W$};
}

\def\Vrestthree{\draw [thick] (0,0) rectangle ++(\wVrestthree,\hgt);
\draw ($0.5*(\wVrestthree,\hgt)+(0,0.05)$) node {$\hat{W}$};}

\def\Vrestfour{\draw [thick] (0,0) rectangle ++(\wVrestfour,\hgt);}

\ifboxpicfirst
\Vsr
\begin{scope}[shift={(\wVsr+\gapp,0)}]
\Vrest
\end{scope}
\begin{scope}[shift={(\wVrest+\wVsr+2*\gapp,0)}]
\Vadd
\end{scope}
\begin{scope}[shift={(\wVrest+\wVsr+\wVadd+3*\gapp,0)}]
\Vabs
\end{scope}

\fi

\ifboxpicsecond
\Mone
\begin{scope}[shift={(\wMone+\gapp,0)}]
\Vrestone
\end{scope}
\begin{scope}[shift={(\wMone+\wVrestone+2*\gapp,0)}]
\Vadd
\end{scope}
\begin{scope}[shift={(\wMone+\wVrestone+\wVadd+3*\gapp,0)}]
\Vabs
\end{scope}
\fi

\ifboxpicthird
\Mone
\begin{scope}[shift={(\wMone+\gapp,0)}]
\Mtwo
\end{scope}
\begin{scope}[shift={(\wMone+\wMtwo+2*\gapp,0)}]
\Vadd
\end{scope}
\begin{scope}[shift={(\wMone+\wMtwo+\wVadd+3*\gapp,0)}]
\Vabs
\end{scope}
\fi

\ifboxpicfourth
\Mone
\begin{scope}[shift={(\wMone+\gapp,0)}]
\Mtwo
\end{scope}
\begin{scope}[shift={(\wMone+\wMtwo+2*\gapp,0)}]
\Mthree
\end{scope}
\begin{scope}[shift={(\wMone+\wMtwo+\wMthree+3*\gapp,0)}]
\Vabs
\end{scope}
\fi

\ifboxpicfifth
\Mone
\begin{scope}[shift={(\wMone+\gapp,0)}]
\Mtwo
\end{scope}
\begin{scope}[shift={(\wMone+\wMtwo+2*\gapp,0)}]
\Mthree
\end{scope}
\begin{scope}[shift={(\wMone+\wMtwo+\wMthree+3*\gapp,0)}]
\Mfour
\end{scope}
\fi

\end{tikzpicture}

\vspace{\spacedown}

\hspace{-0.6cm}\boxpicfourthfalse\boxpicfifthtrue\begin{tikzpicture}[scale=1]

\def\hgt{0.8};
\def\labh{0.2};

\def\gapp{0.1}
\def\wVsr{2.5};
\def\wVadd{1.2};
\def\wVabs{1.2};
\def\wVex{0.8};
\def\wVrest{1.5};
\def\wVrestone{0.3};
\def\wVresttwo{0.5};
\def\wVrestthree{0.6};
\def\wVrestfour{1.5};
\def\wMone{\wVsr+\wVrest-\wVrestone};
\def\wMtwo{\wVrestone+\wVex-\wVresttwo};
\def\wMthree{\wVresttwo+\wVadd-\wVrestthree};
\def\wMfour{\wVabs+\wVrestthree+\gapp};

\def\Vsr{\draw [thick] (0,0) rectangle ++(\wVsr,\hgt);
\draw ($0.5*(\wVsr,\hgt)$) node {$V^{\text{s-r}}$};}

\def\Vadd{\draw [thick] (0,0) rectangle ++(\wVadd,\hgt);
\draw ($0.5*(\wVadd,\hgt)$) node {$V^{\text{add}}$};}

\def\Vabs{
\draw [thick] (0,0) rectangle ++(\wVabs,\hgt);
\draw ($0.5*(\wVabs,\hgt)$) node {$V^{\text{abs}}$};
}

\def\Vex{\draw [thick] (0,0) rectangle ++(\wVex,\hgt);
\draw ($0.5*(\wVex,\hgt)$) node {$V^{\text{ex}}$};}

\def\Mone{\draw [thick,densely dashed] (0,0) rectangle ++(\wMone,\hgt);
\draw ($0.5*(\wMone,\hgt)$) node {$M_1$};}

\def\Mtwo{\draw [thick,densely dashed] (0,0) rectangle ++(\wMtwo,\hgt);
\draw ($0.5*(\wMtwo,\hgt)$) node {$M_2$};}

\def\Mthree{\draw [thick,densely dashed] (0,0) rectangle ++(\wMthree,\hgt);
\draw ($0.5*(\wMthree,\hgt)$) node {$M_3$};}

\def\boxout{0.15};

\def\Mfour{\draw [thick,densely dashed] (2*\boxout,0) -- (\wMfour,0) -- ++(0,\hgt) -- (0,\hgt) -- (0,\boxout) -- ++(2*\boxout,0) --++(0,-\boxout);
\draw ($0.5*(\wMfour,\hgt)$) node {$M_4$};
\draw [fill] (0.4*\boxout,0.4*\boxout) circle [radius = 0.02];
\draw [fill] (1.4*\boxout,0.4*\boxout) circle [radius = 0.02];}

\def\Vrest{\draw [thick] (0,0) rectangle ++(\wVrest,\hgt);}

\def\Vrestone{\draw [thick] (0,0) rectangle ++(\wVrestone,\hgt);}

\def\Vresttwo{\draw [thick] (0,0) rectangle ++(\wVresttwo,\hgt);
\draw ($0.5*(\wVresttwo,\hgt)$) node {$W$};
}

\def\Vrestthree{\draw [thick] (2*\boxout,0) -- (\wVrestthree,0) -- ++(0,\hgt) -- ++(-\wVrestthree,0) -- (0,\boxout) -- ++(2*\boxout,0) --++(0,-\boxout);
\draw [fill] (0.4*\boxout,0.4*\boxout) circle [radius = 0.02];
\draw [fill] (1.4*\boxout,0.4*\boxout) circle [radius = 0.02];
\draw ($0.5*(\wVrestthree,\hgt)+(0,0.05)$) node {$\hat{W}$};}

\def\Vrestfour{\draw [thick] (0,0) rectangle ++(\wVrestfour,\hgt);}

\draw [white] (-0.4,0.5*\hgt) node {iii)};

\ifboxpicfirst
\draw (-0.4,0.5*\hgt) node {i)};
\Vsr
\begin{scope}[shift={(\wVsr+\gapp,0)}]
\Vrest
\end{scope}
\begin{scope}[shift={(\wVrest+\wVsr+2*\gapp,0)}]
\Vex
\end{scope}
\begin{scope}[shift={(\wVrest+\wVsr+\wVex+3*\gapp,0)}]
\Vadd
\end{scope}
\begin{scope}[shift={(\wVrest+\wVsr+\wVex+\wVadd+4*\gapp,0)}]
\Vabs
\end{scope}

\fi

\ifboxpicsecond
\draw (-0.4,0.5*\hgt) node {ii)};
\Mone
\begin{scope}[shift={(\wMone+\gapp,0)}]
\Vrestone
\end{scope}
\begin{scope}[shift={(\wMone+\wVrestone+2*\gapp,0)}]
\Vex
\end{scope}
\begin{scope}[shift={(\wMone+\wVrestone+\wVex+3*\gapp,0)}]
\Vadd
\end{scope}
\begin{scope}[shift={(\wMone+\wVrestone+\wVex+\wVadd+4*\gapp,0)}]
\Vabs
\end{scope}
\fi

\ifboxpicthird
\draw (-0.4,0.5*\hgt) node {iii)};
\Mone
\begin{scope}[shift={(\wMone+\gapp,0)}]
\Mtwo
\end{scope}
\begin{scope}[shift={(\wMone+\wMtwo+2*\gapp,0)}]
\Vresttwo
\end{scope}
\begin{scope}[shift={(\wMone+\wMtwo+\wVresttwo+3*\gapp,0)}]
\Vadd
\end{scope}
\begin{scope}[shift={(\wMone+\wMtwo+\wVresttwo+\wVadd+4*\gapp,0)}]
\Vabs
\end{scope}
\fi

\ifboxpicfourth
\draw (-0.4,0.5*\hgt) node {iv)};
\Mone
\begin{scope}[shift={(\wMone+\gapp,0)}]
\Mtwo
\end{scope}
\begin{scope}[shift={(\wMone+\wMtwo+2*\gapp,0)}]
\Mthree
\end{scope}
\begin{scope}[shift={(\wMone+\wMtwo+\wMthree+3*\gapp,0)}]
\Vrestthree
\end{scope}
\begin{scope}[shift={(\wMone+\wMtwo+\wMthree+\wVrestthree+4*\gapp,0)}]
\Vabs
\end{scope}
\fi

\ifboxpicfifth
\draw (-0.4,0.5*\hgt) node {v)};
\Mone
\begin{scope}[shift={(\wMone+\gapp,0)}]
\Mtwo
\end{scope}
\begin{scope}[shift={(\wMone+\wMtwo+2*\gapp,0)}]
\Mthree
\end{scope}
\begin{scope}[shift={(\wMone+\wMtwo+\wMthree+3*\gapp,0)}]
\Mfour
\end{scope}
\fi

\end{tikzpicture}\hspace{0.5cm}\begin{tikzpicture}[scale=1]

\def\hgt{0.8};
\def\labh{0.2};

\def\gapp{0.1}
\def\wVsr{2.5};
\def\wVaddd{1.1};
\def\wVabss{1.2};
\def\wVex{0.8};
\def\wVrestt{1.5};
\def\wVrestonee{0.3};
\def\wVresttwoo{0.5};
\def\wVrestthreee{0.6};
\def\wVrestfourr{1.5};
\def\wMone{\wVsr+\wVrestt-\wVrestonee};
\def\wMtwo{\wVrestonee+\wVex-\wVresttwoo};
\def\wMthree{\wVresttwoo+\wVaddd-\wVrestthreee};
\def\wMfour{\wVabss+\wVrestthreee+\gapp};

\def\wVadd{\wVresttwoo+\wVaddd-\wVrestthreee};
\def\wVabs{\wVabss+\wVrestthreee+\gapp};
\def\wVrest{\wMone+\wMtwo-\wVsr};
\def\wVrestone{\wMtwo};
\def\wVresttwo{0.5};
\def\wVrestthree{0.6};
\def\wVrestfour{1.5};

\def\Vsr{\draw [thick] (0,0) rectangle ++(\wVsr,\hgt);
\draw ($0.5*(\wVsr,\hgt)$) node {$C^{\text{s-r}}$};}

\def\Vadd{\draw [thick] (0,0) rectangle ++(\wVadd,\hgt);
\draw ($0.5*(\wVadd,\hgt)$) node {$C^{\text{add}}$};}

\def\Vabs{
\draw [thick] (0,0) rectangle ++(\wVabs,\hgt);
\draw ($0.5*(\wVabs,\hgt)$) node {$C^{\text{abs}}$};
}

\def\Vex{\draw [thick] (0,0) rectangle ++(\wVex,\hgt);
\draw ($0.5*(\wVex,\hgt)$) node {$C^{\text{ex}}$};}

\def\Mone{\draw [thick,densely dashed] (0,0) rectangle ++(\wMone,\hgt);
\draw ($0.5*(\wMone,\hgt)$) node {$M_1$};}

\def\boxout{0.15};
\def\Mtwo{\draw [thick,densely dashed] (\boxout,0) -- (\wMtwo,0) -- ++(0,\hgt) -- (0,\hgt) -- (0,\boxout) -- ++(\boxout,0) --++(0,-\boxout);
\draw ($0.5*(\wMtwo,\hgt)$) node {$M_2$};
\draw [fill] (0.4*\boxout,0.4*\boxout) circle [radius = 0.02];}

\def\Mthree{\draw [thick,densely dashed] (0,0) rectangle ++(\wMthree,\hgt);
\draw ($0.5*(\wMthree,\hgt)$) node {$M_3$};}

\def\Mfour{\draw [thick,densely dashed] (0,0) rectangle ++(\wMfour,\hgt);
\draw ($0.5*(\wMfour,\hgt)$) node {$M_4$};}

\def\Vrest{\draw [thick] (0,0) rectangle ++(\wVrest,\hgt);}

\def\Vrestone{\draw [thick] (0,0) rectangle ++(\wVrestone,\hgt);}

\def\Vresttwo{\draw [thick] (0,0) rectangle ++(\wVresttwo,\hgt);
\draw ($0.5*(\wVresttwo,\hgt)$) node {$W$};
}

\def\Vrestthree{\draw [thick] (0,0) rectangle ++(\wVrestthree,\hgt);
\draw ($0.5*(\wVrestthree,\hgt)+(0,0.05)$) node {$\hat{W}$};}

\def\Vrestfour{\draw [thick] (0,0) rectangle ++(\wVrestfour,\hgt);}

\ifboxpicfirst
\Vsr
\begin{scope}[shift={(\wVsr+\gapp,0)}]
\Vrest
\end{scope}
\begin{scope}[shift={(\wVrest+\wVsr+2*\gapp,0)}]
\Vadd
\end{scope}
\begin{scope}[shift={(\wVrest+\wVsr+\wVadd+3*\gapp,0)}]
\Vabs
\end{scope}

\fi

\ifboxpicsecond
\Mone
\begin{scope}[shift={(\wMone+\gapp,0)}]
\Vrestone
\end{scope}
\begin{scope}[shift={(\wMone+\wVrestone+2*\gapp,0)}]
\Vadd
\end{scope}
\begin{scope}[shift={(\wMone+\wVrestone+\wVadd+3*\gapp,0)}]
\Vabs
\end{scope}
\fi

\ifboxpicthird
\Mone
\begin{scope}[shift={(\wMone+\gapp,0)}]
\Mtwo
\end{scope}
\begin{scope}[shift={(\wMone+\wMtwo+2*\gapp,0)}]
\Vadd
\end{scope}
\begin{scope}[shift={(\wMone+\wMtwo+\wVadd+3*\gapp,0)}]
\Vabs
\end{scope}
\fi

\ifboxpicfourth
\Mone
\begin{scope}[shift={(\wMone+\gapp,0)}]
\Mtwo
\end{scope}
\begin{scope}[shift={(\wMone+\wMtwo+2*\gapp,0)}]
\Mthree
\end{scope}
\begin{scope}[shift={(\wMone+\wMtwo+\wMthree+3*\gapp,0)}]
\Vabs
\end{scope}
\fi

\ifboxpicfifth
\Mone
\begin{scope}[shift={(\wMone+\gapp,0)}]
\Mtwo
\end{scope}
\begin{scope}[shift={(\wMone+\wMtwo+2*\gapp,0)}]
\Mthree
\end{scope}
\begin{scope}[shift={(\wMone+\wMtwo+\wMthree+3*\gapp,0)}]
\Mfour
\end{scope}
\fi

\end{tikzpicture}

\caption{The initial partition of $V(G)$ and $C(G)$ at i), and the subsequent development of these partitions as the matchings $M_1$ to $M_4$ are found, where one colour and two vertices in $G$ are not used.}\label{fig:abstract}
\end{figure}
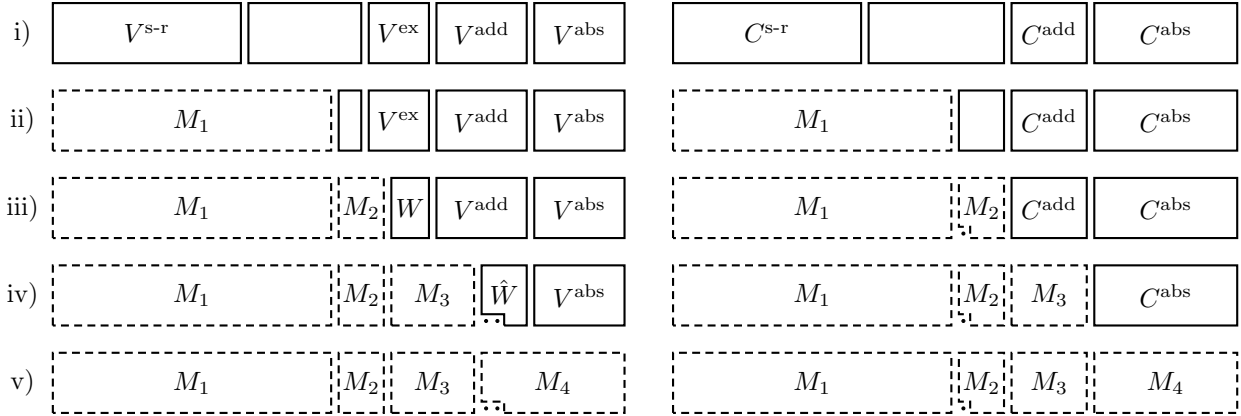


\smallskip

\noindent \textbf{The mechanism of the addition structure.} 
Our addition structure (as depicted at the top of Figure~\ref{fig:add}) will consist of two large vertex-disjoint matchings ($\hat{M}^{\mathrm{id}}$ and $\hat{M}^{\mathrm{rb}}$) and two `remainder vertices' ($\hat{w}$ and $\hat{z}$), which we iteratively update to incorporate two more vertices (from different sides of the bipartition) each time.
We set $V^{\mathrm{add}}=V(\hat{M}^{\mathrm{id}}\cup \hat{M}^{\mathrm{rb}})\cup \{\hat{w},\hat{z}\}$ and $C^{\mathrm{add}}=C(\hat{M}^{\mathrm{rb}})$, where $(V^{\mathrm{add}},C^{\mathrm{add}})$ is then formally the addition structure.
One of the two matchings ($\hat{M}^{\mathrm{id}}$) is, essentially, a randomly chosen monochromatic matching of identity colour edges, while the other matching ($\hat{M}^{\mathrm{rb}}$) is also chosen randomly, but in a much more careful manner. We will outline this choice later, but we will first describe one iterative step to update $\hat{M}^{\mathrm{id}}$, $\hat{M}^{\mathrm{rb}}$, $\hat{w}$ and $\hat{z}$ to cover 2 more vertices. This iterative step will be robust enough that, given a balanced set of $2\ell_0$ vertices $W$ as above, we can pair the vertices and then apply the iterative step $\ell_0$ times. Thus, to keep things simpler, let us assume that $\ell_0=1$ and $W=\{x,y\}$, so that we need only apply our iterative step once. 



When we apply the addition structure to $W=\{x,y\}$, we want to make small alterations to the matchings to find new remainder vertices and get $({M}^{\mathrm{id}},{M}^{\mathrm{rb}},w,z)$ so that they together cover the vertices appearing in $(\hat{M}^{\mathrm{id}},\hat{M}^{\mathrm{rb}},\hat{w},\hat{z})$ as well as the vertices $x$ and $y$. Like $\hat{M}^{\mathrm{rb}}$, ${M}^{\mathrm{rb}}$ is an exactly-$C^{\mathrm{add}}$-rainbow matching and, like $\hat{M}^{\mathrm{id}}$, ${M}^{\mathrm{id}}$ is a matching of colour-$c_0$ edges, but has one more edge as the addition structure after this step also covers $x$ and $y$.


To find $({M}^{\mathrm{id}},{M}^{\mathrm{rb}},w,z)$, we will use the following three stages (depicted in Figure~\ref{fig:add}).
\begin{enumerate}[label = \roman{enumi})]
\item We find a set $E_1\subset \hat{M}^{\mathrm{id}}$ of 4 edges such that $G[V(E_1)\cup \{\hat{w},x,y,\hat{z}\}]$ contains a ${C}^{\mathrm{add}}$-rainbow matching $F_1$ with 6 edges.
\item We find a set $E_2$ of 24 colour-$c_0$ edges such that there is a matching $F_2\subset \hat{M}^{\mathrm{rb}}$ with 24 edges and vertex set $V(E_2)$ such that $C(F_1)\subset C(F_2)$.
\item We find a set of 19 colour-$c_0$ edges $E_3\subset \hat{M}^{\mathrm{id}}\setminus E_1$ such that there is a matching $F_3$ of 18 edges and vertices $w,z$ with $V(F_3)\cup \{w,z\}=V(E_3)$, such that $F_3$ is exactly-$(C(F_2)\setminus C(F_1))$-rainbow.
\end{enumerate}

Now, removing $E_1$, adding $E_2$ and removing $E_3$ from $\hat{M}^{\mathrm{id}}$ to get ${M}^{\mathrm{id}}$ adds 1 colour-$c_0$ edge in total, while adding the vertices in $V(E_2)$ and removing the vertices in $V(E_1\cup E_3)$. On the other hand, adding $F_1$, removing $F_2$ and adding $F_3$ to $\hat{M}^{\mathrm{rb}}$ to get ${M}^{\mathrm{rb}}$ adds and removes one edge with each colour in $C(F_2)$ to make no overall change in $C(M_2)$, while adding the vertices in $V(F_1)\cup V(F_3)=(V(E_1)\cup V(E_3)\cup \{\hat{w},\hat{z}\})\setminus \{w,z\}$ and removing
the vertices in $V(F_2)=V(E_2)$. Thus, $({M}^{\mathrm{id}},{M}^{\mathrm{rb}},{w},{z})$ will have the properties we required above.

To understand this structure, it may be helpful to consider the case where the colouring arises from an $n$-element abelian group $H$ where $c_0$ is the identity, and note that the sum of colours in ${M}^{\mathrm{id}}\cup {M}^{\mathrm{rb}}$ is the same as $\hat{M}^{\mathrm{id}}\cup \hat{M}^{\mathrm{rb}}$. In stage i), the colours of $F_1$ have sum $\hat{w}+x+y+\hat{z}+4c_0=\hat{w}+x+y+\hat{z}\in H$,
and we want to add the edges of $F_1$ to $\hat{M}^{\mathrm{rb}}$. To do this we drop the edges with colour in $C(F_1)$ which are already in $\hat{M}^{\mathrm{rb}}$, but we want the vertex set of any edges we drop from $\hat{M}^{\mathrm{rb}}$ to be the vertex set of a colour-$c_0$ matching that we can then add to $\hat{M}^{\mathrm{id}}$. As we may have $w+x+y+z\neq 0$, we will have to drop additional edges whose sum is the inverse of $w+x+y+z$, dropping in total the edges in $F_2$ in stage ii). We then need to add an edge of each colour in $C(F_2)\setminus C(F_1)$ back into $M_2$, which we do in stage iii) with the matching $F_3$. To get ${M}^{\mathrm{id}}$ disjoint from $\hat{M}^{\mathrm{rb}}$
we will need to drop any colour-$c_0$ edge in $\hat{M}^{\mathrm{id}}$ with a vertex in $F_3$ (here dropped vertices would become remainder vertices), so we want $V(F_3)$ to touch as few colour-$c_0$ edges as possible. As the sum of the colours in $C(F_3)=C(F_2)\setminus C(F_1)$ is $-(w+x+y+z)$ which may not be 0, we cannot always do this so that $V(F_3)$ is the vertex set of a colour-$c_0$ matching in $M_1$, but we do manage this so that we only need to drop $|F_3|+1$ colour-$c_0$ edges from $M_1$, where the additional two vertices not in $V(F_3)$ are the `remainder vertices' $w$ and $z$, which will have sum $\hat{w}+x+y+\hat{z}$, so that the sum of the `remainder vertices' will have increased by $x+y$.

\begin{figure}
\hspace{-2cm}
\begin{minipage}{1.2\textwidth}
\begin{center}
\picfirsttrue\picsecondfalse\picthirdfalse
\begin{tikzpicture}[scale=1]
\def\vxrad{0.035cm}
\def\horunit{0.4}
\def\edgelength{0.4}
\def\betweenrows{0.5}

\draw [white] ($-5*(\horunit,0)$) -- ($45*(\horunit,0)$);

\foreach\n in {1,...,40}
{
\coordinate (A\n) at ($(\horunit*\n,0)$);
\coordinate (B\n) at ($(\horunit*\n,-\edgelength)$);
\coordinate (C\n) at ($(\horunit*\n,-\betweenrows-\edgelength)$);
\coordinate (D\n) at ($(\horunit*\n,-\betweenrows-2*\edgelength)$);
}

\ifpicfirst
\foreach\n in {1,...,40}
{
\draw [thick,red!=50] (A\n) -- (B\n);
}
\fi
\ifpicsecond
\foreach\n in {1,...,40}
{
\draw [thick,red!=50] (A\n) -- (B\n);
}
\fi
\ifpicthird
\foreach\n in {1,2,5,8,9,10,11,12,13,14}
{
\draw [thick,red!=50] (A\n) -- (B\n);
}
\foreach\n in {34,...,40}
{
\draw [thick,red!=50] (A\n) -- (B\n);
}
\fi

\def\wxyzup{0.3}
\def\wxyzuplabel{0.2}
\coordinate (w) at ($0.5*(A3)+0.5*(A4)+(0,\wxyzup)$);
\coordinate (x) at ($0.5*(A4)+0.5*(A5)+(0,\wxyzup)$);
\coordinate (y) at ($0.5*(A6)+0.5*(A7)+(0,\wxyzup)$);
\coordinate (z) at ($0.5*(A7)+0.5*(A8)+(0,\wxyzup)$);

\def\Mlabel{0.6}
\ifpicfirst
\draw [red] ($0.5*(A1)+0.5*(B1)-(\Mlabel,0)$) node {$\hat{M}^{\mathrm{id}}$};
\draw ($0.5*(C1)+0.5*(D1)-(\Mlabel,0)$) node {$\hat{M}^{\mathrm{rb}}$};
\fi
\ifpicthird
\fi

\ifpicfirst
\else
\draw [thick,green] (w) -- (A3);
\draw [thick,orange] (B3) -- (A4);
\draw [thick,purple] (B4) -- (x);

\draw [thick,blue] (y) -- (A6);
\draw [thick,cyan] (B6) -- (A7);
\draw [thick,teal] (B7) -- (z);
\fi

\foreach \x/\y/\z/\zz/\col in {1/2/3/4/green,5/6/7/8/blue,9/10/11/12/orange,17/18/19/20/purple,29/30/31/32/cyan,37/38/39/40/teal}
{
\ifpicthird
\else
\draw [thick,\col] (C\x) -- (D\x);
\draw [thick,densely dash dot,\col] (C\y) -- (D\y);
\draw [thick,densely dotted,\col] (C\z) -- (D\z);
\draw [thick,densely dashed,\col] (C\zz) -- (D\zz);
\fi
\ifpicfirst
\else
\draw [thick,red] (C\x) -- (D\y);
\draw [thick,red] (C\y) -- (D\z);
\draw [thick,red] (C\z) -- (D\zz);
\draw [thick,red] (C\zz) -- (D\x);
\fi
}
\foreach \x/\y/\z/\zz/\col in {13/14/15/16/violet,21/22/23/24/black,25/26/27/28/pink,33/34/35/36/magenta}
{
\draw [thick,\col] (C\x) -- (D\x);
\draw [thick,densely dash dot,\col] (C\y) -- (D\y);
\draw [thick,densely dotted,\col] (C\z) -- (D\z);
\draw [thick,densely dashed,\col] (C\zz) -- (D\zz);
}

\def\vertlab{0.125}
\ifpicfirst
\else
\draw ($(B15)-(0,\vertlab+0.075)$) node {$w$};
\ifpicsecond
\draw ($(A33)+(\vertlab+0.05,0.1)$) node {$z$};
\else
\draw ($(A33)+(0,\vertlab+0.1)$) node {$z$};
\fi
\fi

\ifpicfirst
\else
\foreach \x/\y/\dashstyle/\col in {11/12/densely dash dot/green,12/13/densely dotted/green,13/14/densely dashed/green,
14/15/densely dash dot/orange,15/16/densely dotted/orange,16/17/densely dashed/orange,
17/18/densely dash dot/purple,18/19/densely dotted/purple,19/20/densely dashed/purple,
20/21/densely dash dot/blue,21/22/densely dotted/blue,22/23/densely dashed/blue,
23/24/densely dash dot/cyan,24/25/densely dotted/cyan,25/26/densely dashed/cyan,
26/27/densely dash dot/teal,27/28/densely dotted/teal,28/29/densely dashed/teal}
{
\draw [thick,\dashstyle,\col] ($(A\x)+4*(\horunit,0)$) -- ($(B\y)+4*(\horunit,0)$);
}
\fi

\foreach\n in {1,...,40}
{
\foreach \x in {A\n,B\n,C\n,D\n}
{
\draw [fill] (\x) circle [radius=\vxrad];
}
}
\ifpicsecond
\foreach \x in {y}
{
\draw [fill] (\x) circle [radius=\vxrad];
\draw ($(\x)+(-\wxyzuplabel,-0.05)$) node {${\x}$};
}
\foreach \x in {w}
{
\draw [fill] (\x) circle [radius=\vxrad];
\draw ($(\x)+(-\wxyzuplabel,0)$) node {$\hat{\x}$};
}
\foreach \x in {z}
{
\draw [fill] (\x) circle [radius=\vxrad];
\draw ($(\x)+(\wxyzuplabel,0)$) node {$\hat{\x}$};
}
\foreach \x in {x}
{
\draw [fill] (\x) circle [radius=\vxrad];
\draw ($(\x)+(\wxyzuplabel,-0.05)$) node {$\x$};
}

\else
\foreach \x in {x,y}
{
\draw [fill] (\x) circle [radius=\vxrad];
\draw ($(\x)+(0,\wxyzuplabel)$) node {$\x$};
}
\foreach \x in {w,z}
{
\draw [fill] (\x) circle [radius=\vxrad];
\draw ($(\x)+(0,0.05+\wxyzuplabel)$) node {$\hat{\x}$};
}
\fi

\ifpicsecond
\draw ($(A3)+(-0.1,0.4)$) --  ($(A3)+(-0.1,0.6)$);
\draw ($(A6)+(-0.1,0.4)$) --  ($(A6)+(-0.1,0.6)$);
\draw ($0.5*(A7)+0.5*(A8)+(0,0.4)$) --  ($0.5*(A7)+0.5*(A8)+(0,0.6)$);
\draw ($0.5*(A4)+0.5*(A5)+(0,0.4)$) --  ($0.5*(A4)+0.5*(A5)+(0,0.6)$);
\draw ($(A3)+(-0.1,0.5)$) -- ($0.5*(A4)+0.5*(A5)+(0,0.5)$);
\draw ($(A6)+(-0.1,0.5)$) -- ($0.5*(A7)+0.5*(A8)+(0,0.5)$);
\draw ($(A3)+(0,0.5)-(0.85,0)$) node {$E_1\cup F_1:$};
\fi

\ifpicsecond
\foreach \n in {1,12,17,20,29,32,37,40}
\draw ($(D\n)+(0,-0.2)$) -- ($(D\n)+(0,-0.4)$);
\foreach \n/\nn in {1/12,17/20,29/32,37/40}
\draw ($(D\n)+(0,-0.3)$) -- ($(D\nn)+(0,-0.3)$);
\draw ($(D1)+(0,-0.3)-(0.85,0)$) node {$E_2\cup F_2:$};

\fi

\ifpicsecond
\draw ($(A15)+(0,0.4)$) --  ($(A15)+(0,0.6)$);
\draw ($(A33)+(0,0.4)$) --  ($(A33)+(0,0.6)$);
\draw ($(A33)+(0,0.5)$) -- ($(A15)+(0,0.5)$);
\draw ($(A15)+(0,0.5)-(0.85,0)$) node {$E_3\cup F_3:$};
\fi

\end{tikzpicture}

\vspace{-0.1cm}

$\downarrow$

\vspace{0.05cm}

\picsecondtrue\picfirstfalse\picthirdfalse
\begin{tikzpicture}[scale=1]
\def\vxrad{0.035cm}
\def\horunit{0.4}
\def\edgelength{0.4}
\def\betweenrows{0.5}

\draw [white] ($-5*(\horunit,0)$) -- ($45*(\horunit,0)$);

\foreach\n in {1,...,40}
{
\coordinate (A\n) at ($(\horunit*\n,0)$);
\coordinate (B\n) at ($(\horunit*\n,-\edgelength)$);
\coordinate (C\n) at ($(\horunit*\n,-\betweenrows-\edgelength)$);
\coordinate (D\n) at ($(\horunit*\n,-\betweenrows-2*\edgelength)$);
}

\ifpicfirst
\foreach\n in {1,...,40}
{
\draw [thick,red!=50] (A\n) -- (B\n);
}
\fi
\ifpicsecond
\foreach\n in {1,...,40}
{
\draw [thick,red!=50] (A\n) -- (B\n);
}
\fi
\ifpicthird
\foreach\n in {1,2,5,8,9,10,11,12,13,14}
{
\draw [thick,red!=50] (A\n) -- (B\n);
}
\foreach\n in {34,...,40}
{
\draw [thick,red!=50] (A\n) -- (B\n);
}
\fi

\def\wxyzup{0.3}
\def\wxyzuplabel{0.2}
\coordinate (w) at ($0.5*(A3)+0.5*(A4)+(0,\wxyzup)$);
\coordinate (x) at ($0.5*(A4)+0.5*(A5)+(0,\wxyzup)$);
\coordinate (y) at ($0.5*(A6)+0.5*(A7)+(0,\wxyzup)$);
\coordinate (z) at ($0.5*(A7)+0.5*(A8)+(0,\wxyzup)$);

\def\Mlabel{0.6}
\ifpicfirst
\draw [red] ($0.5*(A1)+0.5*(B1)-(\Mlabel,0)$) node {$\hat{M}^{\mathrm{id}}$};
\draw ($0.5*(C1)+0.5*(D1)-(\Mlabel,0)$) node {$\hat{M}^{\mathrm{rb}}$};
\fi
\ifpicthird
\fi

\ifpicfirst
\else
\draw [thick,green] (w) -- (A3);
\draw [thick,orange] (B3) -- (A4);
\draw [thick,purple] (B4) -- (x);

\draw [thick,blue] (y) -- (A6);
\draw [thick,cyan] (B6) -- (A7);
\draw [thick,teal] (B7) -- (z);
\fi

\foreach \x/\y/\z/\zz/\col in {1/2/3/4/green,5/6/7/8/blue,9/10/11/12/orange,17/18/19/20/purple,29/30/31/32/cyan,37/38/39/40/teal}
{
\ifpicthird
\else
\draw [thick,\col] (C\x) -- (D\x);
\draw [thick,densely dash dot,\col] (C\y) -- (D\y);
\draw [thick,densely dotted,\col] (C\z) -- (D\z);
\draw [thick,densely dashed,\col] (C\zz) -- (D\zz);
\fi
\ifpicfirst
\else
\draw [thick,red] (C\x) -- (D\y);
\draw [thick,red] (C\y) -- (D\z);
\draw [thick,red] (C\z) -- (D\zz);
\draw [thick,red] (C\zz) -- (D\x);
\fi
}
\foreach \x/\y/\z/\zz/\col in {13/14/15/16/violet,21/22/23/24/black,25/26/27/28/pink,33/34/35/36/magenta}
{
\draw [thick,\col] (C\x) -- (D\x);
\draw [thick,densely dash dot,\col] (C\y) -- (D\y);
\draw [thick,densely dotted,\col] (C\z) -- (D\z);
\draw [thick,densely dashed,\col] (C\zz) -- (D\zz);
}

\def\vertlab{0.125}
\ifpicfirst
\else
\draw ($(B15)-(0,\vertlab+0.075)$) node {$w$};
\ifpicsecond
\draw ($(A33)+(\vertlab+0.05,0.1)$) node {$z$};
\else
\draw ($(A33)+(0,\vertlab+0.1)$) node {$z$};
\fi
\fi

\ifpicfirst
\else
\foreach \x/\y/\dashstyle/\col in {11/12/densely dash dot/green,12/13/densely dotted/green,13/14/densely dashed/green,
14/15/densely dash dot/orange,15/16/densely dotted/orange,16/17/densely dashed/orange,
17/18/densely dash dot/purple,18/19/densely dotted/purple,19/20/densely dashed/purple,
20/21/densely dash dot/blue,21/22/densely dotted/blue,22/23/densely dashed/blue,
23/24/densely dash dot/cyan,24/25/densely dotted/cyan,25/26/densely dashed/cyan,
26/27/densely dash dot/teal,27/28/densely dotted/teal,28/29/densely dashed/teal}
{
\draw [thick,\dashstyle,\col] ($(A\x)+4*(\horunit,0)$) -- ($(B\y)+4*(\horunit,0)$);
}
\fi

\foreach\n in {1,...,40}
{
\foreach \x in {A\n,B\n,C\n,D\n}
{
\draw [fill] (\x) circle [radius=\vxrad];
}
}
\ifpicsecond
\foreach \x in {y}
{
\draw [fill] (\x) circle [radius=\vxrad];
\draw ($(\x)+(-\wxyzuplabel,-0.05)$) node {${\x}$};
}
\foreach \x in {w}
{
\draw [fill] (\x) circle [radius=\vxrad];
\draw ($(\x)+(-\wxyzuplabel,0)$) node {$\hat{\x}$};
}
\foreach \x in {z}
{
\draw [fill] (\x) circle [radius=\vxrad];
\draw ($(\x)+(\wxyzuplabel,0)$) node {$\hat{\x}$};
}
\foreach \x in {x}
{
\draw [fill] (\x) circle [radius=\vxrad];
\draw ($(\x)+(\wxyzuplabel,-0.05)$) node {$\x$};
}

\else
\foreach \x in {x,y}
{
\draw [fill] (\x) circle [radius=\vxrad];
\draw ($(\x)+(0,\wxyzuplabel)$) node {$\x$};
}
\foreach \x in {w,z}
{
\draw [fill] (\x) circle [radius=\vxrad];
\draw ($(\x)+(0,0.05+\wxyzuplabel)$) node {$\hat{\x}$};
}
\fi

\ifpicsecond
\draw ($(A3)+(-0.1,0.4)$) --  ($(A3)+(-0.1,0.6)$);
\draw ($(A6)+(-0.1,0.4)$) --  ($(A6)+(-0.1,0.6)$);
\draw ($0.5*(A7)+0.5*(A8)+(0,0.4)$) --  ($0.5*(A7)+0.5*(A8)+(0,0.6)$);
\draw ($0.5*(A4)+0.5*(A5)+(0,0.4)$) --  ($0.5*(A4)+0.5*(A5)+(0,0.6)$);
\draw ($(A3)+(-0.1,0.5)$) -- ($0.5*(A4)+0.5*(A5)+(0,0.5)$);
\draw ($(A6)+(-0.1,0.5)$) -- ($0.5*(A7)+0.5*(A8)+(0,0.5)$);
\draw ($(A3)+(0,0.5)-(0.85,0)$) node {$E_1\cup F_1:$};
\fi

\ifpicsecond
\foreach \n in {1,12,17,20,29,32,37,40}
\draw ($(D\n)+(0,-0.2)$) -- ($(D\n)+(0,-0.4)$);
\foreach \n/\nn in {1/12,17/20,29/32,37/40}
\draw ($(D\n)+(0,-0.3)$) -- ($(D\nn)+(0,-0.3)$);
\draw ($(D1)+(0,-0.3)-(0.85,0)$) node {$E_2\cup F_2:$};

\fi

\ifpicsecond
\draw ($(A15)+(0,0.4)$) --  ($(A15)+(0,0.6)$);
\draw ($(A33)+(0,0.4)$) --  ($(A33)+(0,0.6)$);
\draw ($(A33)+(0,0.5)$) -- ($(A15)+(0,0.5)$);
\draw ($(A15)+(0,0.5)-(0.85,0)$) node {$E_3\cup F_3:$};
\fi

\end{tikzpicture}

\vspace{-0.1cm}

$\downarrow$

\vspace{-0.1cm}

\picfirstfalse\picsecondfalse\picthirdtrue
\begin{tikzpicture}[scale=1]
\def\vxrad{0.035cm}
\def\horunit{0.4}
\def\edgelength{0.4}
\def\betweenrows{0.5}

\draw [white] ($-5*(\horunit,0)$) -- ($45*(\horunit,0)$);

\foreach\n in {1,...,40}
{
\coordinate (A\n) at ($(\horunit*\n,0)$);
\coordinate (B\n) at ($(\horunit*\n,-\edgelength)$);
\coordinate (C\n) at ($(\horunit*\n,-\betweenrows-\edgelength)$);
\coordinate (D\n) at ($(\horunit*\n,-\betweenrows-2*\edgelength)$);
}

\ifpicfirst
\foreach\n in {1,...,40}
{
\draw [thick,red!=50] (A\n) -- (B\n);
}
\fi
\ifpicsecond
\foreach\n in {1,...,40}
{
\draw [thick,red!=50] (A\n) -- (B\n);
}
\fi
\ifpicthird
\foreach\n in {1,2,5,8,9,10,11,12,13,14}
{
\draw [thick,red!=50] (A\n) -- (B\n);
}
\foreach\n in {34,...,40}
{
\draw [thick,red!=50] (A\n) -- (B\n);
}
\fi

\def\wxyzup{0.3}
\def\wxyzuplabel{0.2}
\coordinate (w) at ($0.5*(A3)+0.5*(A4)+(0,\wxyzup)$);
\coordinate (x) at ($0.5*(A4)+0.5*(A5)+(0,\wxyzup)$);
\coordinate (y) at ($0.5*(A6)+0.5*(A7)+(0,\wxyzup)$);
\coordinate (z) at ($0.5*(A7)+0.5*(A8)+(0,\wxyzup)$);

\def\Mlabel{0.6}
\ifpicfirst
\draw [red] ($0.5*(A1)+0.5*(B1)-(\Mlabel,0)$) node {$\hat{M}^{\mathrm{id}}$};
\draw ($0.5*(C1)+0.5*(D1)-(\Mlabel,0)$) node {$\hat{M}^{\mathrm{rb}}$};
\fi
\ifpicthird
\fi

\ifpicfirst
\else
\draw [thick,green] (w) -- (A3);
\draw [thick,orange] (B3) -- (A4);
\draw [thick,purple] (B4) -- (x);

\draw [thick,blue] (y) -- (A6);
\draw [thick,cyan] (B6) -- (A7);
\draw [thick,teal] (B7) -- (z);
\fi

\foreach \x/\y/\z/\zz/\col in {1/2/3/4/green,5/6/7/8/blue,9/10/11/12/orange,17/18/19/20/purple,29/30/31/32/cyan,37/38/39/40/teal}
{
\ifpicthird
\else
\draw [thick,\col] (C\x) -- (D\x);
\draw [thick,densely dash dot,\col] (C\y) -- (D\y);
\draw [thick,densely dotted,\col] (C\z) -- (D\z);
\draw [thick,densely dashed,\col] (C\zz) -- (D\zz);
\fi
\ifpicfirst
\else
\draw [thick,red] (C\x) -- (D\y);
\draw [thick,red] (C\y) -- (D\z);
\draw [thick,red] (C\z) -- (D\zz);
\draw [thick,red] (C\zz) -- (D\x);
\fi
}
\foreach \x/\y/\z/\zz/\col in {13/14/15/16/violet,21/22/23/24/black,25/26/27/28/pink,33/34/35/36/magenta}
{
\draw [thick,\col] (C\x) -- (D\x);
\draw [thick,densely dash dot,\col] (C\y) -- (D\y);
\draw [thick,densely dotted,\col] (C\z) -- (D\z);
\draw [thick,densely dashed,\col] (C\zz) -- (D\zz);
}

\def\vertlab{0.125}
\ifpicfirst
\else
\draw ($(B15)-(0,\vertlab+0.075)$) node {$w$};
\ifpicsecond
\draw ($(A33)+(\vertlab+0.05,0.1)$) node {$z$};
\else
\draw ($(A33)+(0,\vertlab+0.1)$) node {$z$};
\fi
\fi

\ifpicfirst
\else
\foreach \x/\y/\dashstyle/\col in {11/12/densely dash dot/green,12/13/densely dotted/green,13/14/densely dashed/green,
14/15/densely dash dot/orange,15/16/densely dotted/orange,16/17/densely dashed/orange,
17/18/densely dash dot/purple,18/19/densely dotted/purple,19/20/densely dashed/purple,
20/21/densely dash dot/blue,21/22/densely dotted/blue,22/23/densely dashed/blue,
23/24/densely dash dot/cyan,24/25/densely dotted/cyan,25/26/densely dashed/cyan,
26/27/densely dash dot/teal,27/28/densely dotted/teal,28/29/densely dashed/teal}
{
\draw [thick,\dashstyle,\col] ($(A\x)+4*(\horunit,0)$) -- ($(B\y)+4*(\horunit,0)$);
}
\fi

\foreach\n in {1,...,40}
{
\foreach \x in {A\n,B\n,C\n,D\n}
{
\draw [fill] (\x) circle [radius=\vxrad];
}
}
\ifpicsecond
\foreach \x in {y}
{
\draw [fill] (\x) circle [radius=\vxrad];
\draw ($(\x)+(-\wxyzuplabel,-0.05)$) node {${\x}$};
}
\foreach \x in {w}
{
\draw [fill] (\x) circle [radius=\vxrad];
\draw ($(\x)+(-\wxyzuplabel,0)$) node {$\hat{\x}$};
}
\foreach \x in {z}
{
\draw [fill] (\x) circle [radius=\vxrad];
\draw ($(\x)+(\wxyzuplabel,0)$) node {$\hat{\x}$};
}
\foreach \x in {x}
{
\draw [fill] (\x) circle [radius=\vxrad];
\draw ($(\x)+(\wxyzuplabel,-0.05)$) node {$\x$};
}

\else
\foreach \x in {x,y}
{
\draw [fill] (\x) circle [radius=\vxrad];
\draw ($(\x)+(0,\wxyzuplabel)$) node {$\x$};
}
\foreach \x in {w,z}
{
\draw [fill] (\x) circle [radius=\vxrad];
\draw ($(\x)+(0,0.05+\wxyzuplabel)$) node {$\hat{\x}$};
}
\fi

\ifpicsecond
\draw ($(A3)+(-0.1,0.4)$) --  ($(A3)+(-0.1,0.6)$);
\draw ($(A6)+(-0.1,0.4)$) --  ($(A6)+(-0.1,0.6)$);
\draw ($0.5*(A7)+0.5*(A8)+(0,0.4)$) --  ($0.5*(A7)+0.5*(A8)+(0,0.6)$);
\draw ($0.5*(A4)+0.5*(A5)+(0,0.4)$) --  ($0.5*(A4)+0.5*(A5)+(0,0.6)$);
\draw ($(A3)+(-0.1,0.5)$) -- ($0.5*(A4)+0.5*(A5)+(0,0.5)$);
\draw ($(A6)+(-0.1,0.5)$) -- ($0.5*(A7)+0.5*(A8)+(0,0.5)$);
\draw ($(A3)+(0,0.5)-(0.85,0)$) node {$E_1\cup F_1:$};
\fi

\ifpicsecond
\foreach \n in {1,12,17,20,29,32,37,40}
\draw ($(D\n)+(0,-0.2)$) -- ($(D\n)+(0,-0.4)$);
\foreach \n/\nn in {1/12,17/20,29/32,37/40}
\draw ($(D\n)+(0,-0.3)$) -- ($(D\nn)+(0,-0.3)$);
\draw ($(D1)+(0,-0.3)-(0.85,0)$) node {$E_2\cup F_2:$};

\fi

\ifpicsecond
\draw ($(A15)+(0,0.4)$) --  ($(A15)+(0,0.6)$);
\draw ($(A33)+(0,0.4)$) --  ($(A33)+(0,0.6)$);
\draw ($(A33)+(0,0.5)$) -- ($(A15)+(0,0.5)$);
\draw ($(A15)+(0,0.5)-(0.85,0)$) node {$E_3\cup F_3:$};
\fi

\end{tikzpicture}
\end{center}
\end{minipage}

\vspace{0.1cm}

\caption{The initial addition structure with vertices $x$ and $y$ to cover, and one iterative step finding matchings $E_1,F_1,E_2,F_2,E_3,F_3$ in the middle, to then produce below two new matchings and two new remainder vertices $w$ and $z$.}\label{fig:add}
\end{figure}
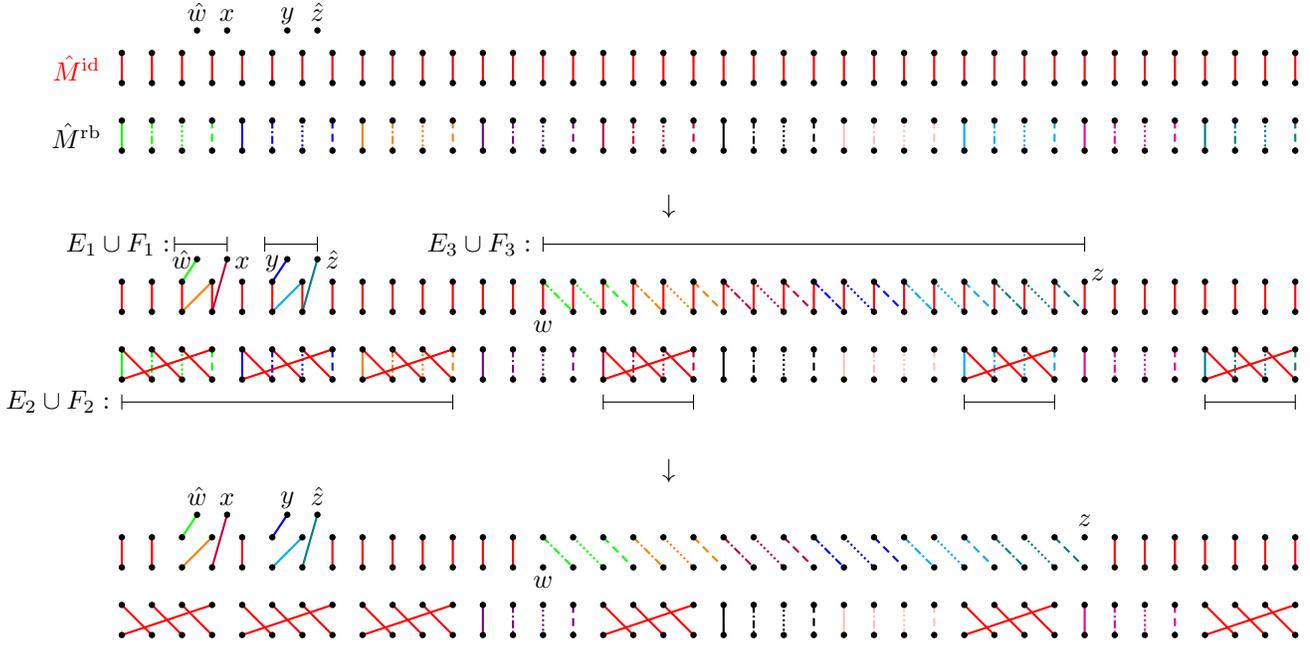


\smallskip

\noindent \textbf{Notes on finding the addition structure.} We now discuss in more detail how we find the addition structure so that stages i) to iii) above of an iterative step can be performed. While we only discuss one alteration to cover $x,y$, enough of the small paths and cycles discussed will exist in order to do the alterations iteratively while avoiding previously used colours, vertices and edges.

\smallskip

\emph{Stage i).} The large colour-$c_0$ matching $\hat{M}^{\mathrm{id}}$ will be chosen randomly along with a random colour set $C_0$. In the colouring of $G$, it is easy to observe that, between any two vertices $u,v$ there are many paths of length 5 whose 2nd and 4th edge have colour-$c_0$ and which are otherwise rainbow without colour $c_0$. Indeed, there are enough of these paths that, with high probability, for each $u,v$, for many of these paths the two colour-$c_0$ edges will lie in $\hat{M}^{\mathrm{id}}$ and the other colours will lie in $C_0$. Then, for any vertices $x$ and $y$ in stage i) above, taking such paths $P_{wx}$ and $P_{yz}$ between $w$ and $x$ and $y$ and $z$ respectively (see Figure~\ref{fig:add}), so that their vertices and non-$c_0$-colours are disjoint, we can collect together the colour-$c_0$ edges of $P_{wx}$ and $P_{yz}$ to give $E_1$ and let the other edges be $F_1$. (The property we use to find these paths is later stated as \ref{prop-pseud-add-2} in Section~\ref{sec:pseud}.)

\smallskip

 \emph{Stage ii).} Instead of finding the matchings $E_2$ and $F_2$ at once, we find a pair of matchings $E_{2,c}$ and $F_{2,c}$ for each $c\in C(F_1)$, each with size 4, so that $c\in C(F_{2,c})$, $V(E_{2,c})=V(F_{2,c})$, $E_{2,c}$ is a colour-$c_0$ matching, and $F_{2,c}$ is a rainbow matching, finding them so that we can take $E_2=\cup_{c\in C(F_1)}E_{2,c}$ and $F_2=\cup_{c\in C(F_1)}F_{2,c}$.
  For any $c\in C_0$, we can see that there are many options for such a pair of matchings. Indeed, selecting a colour-$c$ edge, and picking its two neighbouring colour-$c_0$ edges, if two more colour-$c_0$ edges are selected, then it is likely the three edges completing these edges into an 8-cycle have new, different, colours.
  To find $M^{\mathrm{rb}}$, we take our random set of colours $C_0$ and, vertex-disjointly from $\hat{M}^{\mathrm{id}}$, find, for each $c\in C_0$, the matchings $F_{2,c}$ and $E_{2,c}$, so that the matchings $F_{2,c}$ are colour- and vertex-disjoint, and set $\hat{M}^{\mathrm{rb}}=\cup_{c\in C_0}F_{2,c}$.
  (The property we use to find these cycles is later stated as \ref{prop-pseud-add-3} in Section~\ref{sec:pseud}.)

  Thinking of the case where the colouring is from an abelian group again, this is setting up $\hat{M}^{\mathrm{id}}$ so that for any $c\in C_0$ which we might want to use in the matching $E_1$, an edge with colour $c$ is sitting in $\hat{M}^{\mathrm{id}}$ with three other edges whose sum is $-c$, so that we can drop the colour $c$ out of  $\hat{M}^{\mathrm{id}}$ at the expense of also dropping out the other three colours representing the inverse of $c$ so that the vertices on their edges can be covered by 4 identity coloured edges.

  \smallskip

  \emph{Stage iii).} For stage iii), we observe that in our colouring of $G$ (as every colour appears at every vertex), for a typical set $D$ of 18 colours, we can typically start at an arbitrary vertex $w'$ and alternate in a path between an edge with colour $c_0$ and an edge with colour in $D$ (so that each such colour is used once). Furthermore, if $\hat{M}^{\mathrm{id}}$ is a relatively large matching then (though a small proportion) many such paths will have all their colour-$c_0$ edges in $\hat{M}^{\mathrm{id}}$. When $D=C(F_2)\setminus C(F_1)$ in stage ii), letting $E_3$ be the colour-$c_0$ edges of such a path, and letting $F_3$ be the edges with colour in $D$, we have the required matchings (where $w$ and $z$ are the endvertices of the path).
   There will be enough of these paths that we will be able to do this while avoiding vertices in $V(E_1)\subset V(\hat{M}^{\mathrm{id}})$. For some sets $D$, alternating edges between colour
  $c_0$ and different colours in $D$ may not result in a path, so we may need to take a collection of cycles in addition to a path (whose union is the same as $E_3\cup F_3$), but this does not add any additional difficulty as our graph is bipartite.
  The property we need in $G$ for this to be possible as sketched is below, which we record as we now discuss how we have to alter this sketch for Theorem~\ref{thm:brouwer}, where the coloured graph we consider may not have this property.

\begin{enumerate}[label = ($\dagger$)]
\item  For some small, fixed $\alpha$, the following holds. For any $c_0\in C(G)$ and any $D\subset C(G)\setminus \{c_0\}$ with $|D|=18$, there are vertex-disjoint sets $\bar{V}_1,\ldots,\bar{V}_{\alpha n}\subset V(G)$ such that, for each $i\in [\alpha n]$, $|\bar{V}_i|=38$ and $G[\bar{V}_i]$ contains both a matching of $19$ colour-$c_0$ edges and a $D$-rainbow matching with $18$ edges.\label{isthisadagger}
\end{enumerate}

\subsubsection{Changes due to Theorem~\ref{thm:brouwer}}\label{sec:thetroubles}
The approach sketched above works well for Theorem~\ref{thm:RBSeven}, but an issue arises when applied for Theorem~\ref{thm:brouwer}. Before discussing this further, we recall the approach of Keevash, Pokrovskiy, Sudakov and Yepremyan~\cite{KPSY} for studying matchings in Steiner triple systems (STSs) via rainbow matchings. Let $S$ be a Steiner triple system (STS) with vertex set $[n]$, and assume that $n\equiv 3\mod 6$ (the case  $n\equiv 1\mod 6$ follows similarly after the deletion of an arbitrary vertex). Let $m=n/3$, and let $[3m]=A\cup B\cup C$
be a partition created by, for each $i\in [3m]$,
choosing the set for $i$ independently at random such
that $\P(v\in A)=\P(v\in B)=\P(v\in C)=1/3$. Let $G$ be the bipartite graph with vertex classes $A$ and $B$ where $ab$ with $a\in A$ and $b\in B$ is an edge with colour $c$ exactly when $abc\in S$.

It is not hard to see that $G$ is a properly coloured bipartite graph. Furthermore, if $M$ is a rainbow matching in $G$, then  $\{abc\in S:ab\in E(M)\}$ is a matching in $S$, so to prove Theorem~\ref{thm:brouwer} it would suffice to find an $(m-1)$-edge rainbow matching in $G$ (for some partition $A\cup B\cup C$). However, $G$ is not a complete bipartite graph; instead we expect its edges to have density around $1/3$. Keevash, Pokrovskiy, Sudakov and Yepremyan~\cite{KPSY} showed that, roughly speaking, with positive probability $|A|=|B|=|C|$ and $G$ satisfies some natural pseudorandomness conditions. However, for our constructions we require some more esoteric conditions.
In particular, in the auxiliary coloured bipartite graph used, we cannot show that \ref{isthisadagger} is likely to hold, and so must modify our sketch.

In the sketch above, using the property from \ref{isthisadagger}, in stage iii) we could incorporate any set of 18 colours while dropping out two new remainder vertices. Now, instead of \ref{isthisadagger}, we show that if we have in addition 72 more colours which have not yet been incorporated, within these $18+72=90$ colours there is a set of 18 that can be incorporated into the matching at the expense of two remainder vertices (see~\ref{prop-pseud-add-new-1}). Roughly speaking (and by leaving out stage iii) in the final iteration), we will be able to use this to incorporate all of the vertices from $W$ and all but at most 100 of the missing colours.

Our final task then is, given a set $D$ of at most 100 colours to incorporate all of them, where we can leave out two remainder vertices and one colour. As we cannot use a corresponding version of \ref{isthisadagger}, we do this separately from the main addition structure using a supplementary addition structure with a similar structure, which also has a rainbow matching and a colour-$c_0$ matching. We first use the rainbow matching to drop out another set of 100 colours $D'$ which looks quite random subject to their edges in the rainbow matching forming a cycle with colour-$c_0$ edges (see~\ref{prop-pseud-add-new-2} for the actual condition we use). As $D'$ looks quite random, we will have a version of \ref{isthisadagger} that holds for $D\cup D'$ with one colour removed, which we can then use with the matching of colour-$c_0$ edges in the supplementary addition structure to incorporate all but 1 of the missing colours in $D\cup D'$ while finding the final two remainder vertices.
 This is elaborated further when it is carried out, in Section~\ref{sec:addition}.



\section{Preliminaries}\label{sec:prelim}
As noted in Section~\ref{sec:expo}, we will carry out our main techniques in a wider class of coloured graphs than bipartite complete graphs, a class that we call \emph{properly-pseudorandom}. We postpone the technical definition of an  $(n,p,\eps)$-properly-pseudorandom bipartite graph $G$ to Section~\ref{sec:pseud}, and, before then, note only that a good example is a graph $G$ formed from an optimally-coloured copy of $K_{n,n}$ by selecting each edge independently at random with probability $p$, where $\eps$ represents the proportionate deviation of some random variables (like vertex degrees) from their mean, and $p$ (and, where it appears, $q$), are small constants, which are fixed outside of the hierarchy of variables discussed in Section~\ref{sec:pseud}. We have two main results on rainbow matchings in properly-pseudorandom graphs: one finding a rainbow matching missing two vertices (Theorem~\ref{thm-technical}) and one finding a perfect rainbow matching under the condition there are slightly more colours (Theorem~\ref{thm-technical-variant}). These are stated in Section~\ref{sec:tech}.

As sketched in Section~\ref{subsec:discuss}, our approach uses the semi-random method in combination with an absorption structure and an addition structure. In Section~\ref{sec:comp}, we state our main result covering each of these three components, Theorems~\ref{thm:RSBcoveringstep}--\ref{thm:RSBaddition} respectively, along with a variant result, Theorem~\ref{thm:RSBaddition-variant}, for the addition structure used for the second technical theorem, Theorem~\ref{thm-technical-variant}. Developing an approximate algebraic structure for colours is a critical part of our proof, but this is carried out within the proof of the absorption result (as discussed further below when describing the relevant sections). The statement of the component results then allows us to prove the technical theorems from them in Section~\ref{sec:genproof}. This allows us to cover the main structure of the proof, as sketched in the outline in Section~\ref{sec:add}, but with the appropriate detail.

We then discuss typicality of hypergraphs for the semi-random results in Section~\ref{sec:typical} and give our definition of pseudorandomness in Section~\ref{sec:pseud}. We deduce Theorem~\ref{thm:RBSeven} and Theorem~\ref{thm:generalLS} from Theorem~\ref{thm-technical} in Section~\ref{sec:what} and Section~\ref{sec:otherthmsfromgen1} respectively. We deduce Theorem~\ref{thm:symbolnum} from Theorem~\ref{thm-technical-variant} in Section~\ref{sec:otherthmsfromgen2}. In Section~\ref{subsec:expand}, we detail the graph expansion that we use and prove two key results, Theorem~\ref{thm-expander} and Lemma~\ref{lem-connectinexp}. Finally, to complete our preliminaries, we state the concentration results we will use in Section~\ref{sec:conc} and cover a simple counting result for 4-cycles in Section~\ref{sec:count4}.

This leads us to the proofs of the main components in the remaining sections, which are outlined in more detail at the start of each section. In Section~\ref{sec:semirandom}, we prove the component theorem that gives us the large rainbow matching from the semi-random method (Theorem~\ref{thm:RSBcoveringstep}). In Section~\ref{sec:exchangingcolours}, we find colour classes with, approximately, the algebraic property discussed in Section~\ref{sec:alg}. In Section~\ref{sec:exchangingedges}, we use this to build `edge switchers' which can switch between covering the vertices of one edge with another, when the edges have the same colour (for most of the edges in the graph). In Section~\ref{sec:exchangingedgesinclass}, we do this similarly with edges which have colour in the same class (again for most of the edges).
In Section~\ref{sec:absorbingedges}, we use this to construct our absorption structure, while in Section~\ref{sec:addition} we construct our addition structure, thus completing the proof of our technical theorems. In Section~\ref{sec:brouwer}, we prove Theorem~\ref{thm:brouwer} from these technical theorems. We then finish with some concluding remarks in Section~\ref{sec:final}.


\subsection{Technical theorems and component results}\label{sec:tech}
For Theorems~\ref{thm:RBSeven},~\ref{thm:generalLS} and~\ref{thm:brouwer}, we will prove the following unified technical theorem, where, as noted above, the full definition of an $(n,p,\eps)$-properly-pseudorandom bipartite graph is postponed to Section~\ref{sec:pseud}.

\begin{theorem}\label{thm-technical} Let $1/n\ll p\leq 1$ and $1/n\llpoly \eps\llpoly\log^{-1}n$. Then, any $(n,p,\eps)$-properly-pseudorandom bipartite graph contains a rainbow matching with $n-1$ edges.
\end{theorem}

For Theorem~\ref{thm:symbolnum}, we prove the following variant where a perfect rainbow matching is found if the properly-pseudorandom graph has slightly more than $n$ colours.

\begin{theorem}\label{thm-technical-variant} Let $1/n\ll p\leq 1$ and $1/n\llpoly \eps\llpoly\log^{-1}n$. Then, any $(n,p,\eps)$-properly-pseudorandom bipartite graph $G$ with $|C(G)|\geq n+100$ contains a rainbow matching with $n$ edges.
\end{theorem}


\subsection{Main component results}\label{sec:comp}
We can now state the main theorem that covers each of the three components for the technical theorems, along with a variant. The variables used in this section fit into the hierarchy
\[\frac{1}{n}\llpoly \eps\llpoly \eta\llpoly \gamma \llpoly \beta\llpoly\alpha\llpoly\log^{-1}n\]
which corresponds to the discussion of the hierarchy at \eqref{eqn:sampleh} in Section~\ref{sec:expo}. The extra variable $\alpha$ here, though, governs the `approximate' nature of the algebraic properties of the colours, where (for example) there are at most $\alpha n$ colours which we avoid when picking the `identity colour'.

We start with our result finding an almost-perfect rainbow matching in a subgraph whose vertices and colours contain a large random set of vertices and colours, respectively, which represents the `semi-random' portion of the proof (proved in Section~\ref{sec:semirandom}).

\begin{theorem}[Almost-perfect rainbow matchings]\label{thm:RSBcoveringstep} Let $1/n\ll p,q\leq 1$ and $1/n\llpoly \eps \llpoly \eta \llpoly \log^{-1}n$. Let $2q/3\leq q_V,q_C\leq q$. Let $G$ be an $(n,p,\eps)$-properly-pseudorandom bipartite graph with vertex classes $A$ and $B$.
Independently, let $V^{\mathrm{s-r}}$ be a $q_V$-random subset of $V(G)$ and let $C^{\mathrm{s-r}}$ be a $q_C$-random subset of $C(G)$.
Then, with high probability, the following holds.

Given any sets $\bar{A}\subset A$, $\bar{B}\subset B$, $\bar{C}\subset C(G)$ with size $qn$ such that $V^{\mathrm{s-r}}\subset \bar{A}\cup \bar{B}$ and $C^{\mathrm{s-r}}\subset \bar{C}$, there is a $\bar{C}$-rainbow matching in $G[\bar{A},\bar{B}]$ with at least $qn-\eta n$ edges.
\end{theorem}

Our next theorem allows us to find our absorption structure (and is proved across Sections~\ref{sec:exchangingcolours}--\ref{sec:absorbingedges}). To tie this in to the proof sketch, we note that, at \ref{prop:rando} in \itref{prop:thmabs} below, $(V^{\mathrm{abs}},C^{\mathrm{abs}})$ can `absorb' $W=V(\bar{E})$, which is the vertex set of a colour-$c$ matching (to be used where $c=c_0$ is the `identity colour'). The purpose of the sets $V'$ and $C'$ in \itref{prop:thmabs} is so that the absorption structure can be found disjointly from the addition structure.

\begin{theorem}[Absorption structure]\label{thm:RSBabsorption} Let $1/n\ll p,q_V,q_C\leq 1$. Let $1/n\llpoly\eps\llpoly\gamma \llpoly \beta\llpoly \alpha \llpoly\log^{-1}n$. Let $G$ be an $(n,p,\eps)$-properly-pseudorandom
bipartite graph with vertex classes $A$ and $B$.
Independently, let $V$ be a $q_V$-random subset of $V(G)$ and let $C$ be a $q_C$-random subset of $C(G)$.
Then, with high probability, for all but at most $\alpha n$ colours $c\in C$, there is a set $E_c\subset E_{c}(G)$ with $|E_{c}(G[V])\setminus E_c|\leq \alpha n$ such that the following hold.

\stepcounter{propcounter}
\begin{enumerate}[label = {\textbf{\emph{\Alph{propcounter}}}}]
\item For each $0\leq \ell_0\leq \ell_1\leq \gamma n$, and sets $E\subset E_c$ with $|E|=\ell_1$, and $V'\subset V(G)$ and $C'\subset C(G)$ with $|C'|,|V'|\leq 10\gamma n$, there are sets $V^{\mathrm{abs}}\subset V\setminus (V(E')\cup V')$ and $C^{\mathrm{abs}}\subset C\setminus C'$
such that $|A\cap V^{\mathrm{abs}}|=|B\cap V^{\mathrm{abs}}|=\beta n-\ell_0$ and $|C^{\mathrm{abs}}|=\beta n$, and the following property holds.\label{prop:thmabs}

\vspace{-0.2cm}

\begin{enumerate}[label = $(\bullet)$]
\item  For any set $\bar{E}\subset E$ with $|\bar{E}|=\ell_0$, $G[V^{\mathrm{abs}}\cup V(\bar{E})]$ has a $C^{\mathrm{abs}}$-rainbow matching with size $\beta n$.\label{prop:rando}
\end{enumerate}
\end{enumerate}
\end{theorem}

Next, we  give our result which can find our addition structure. Beyond the description in Section~\ref{sec:expo}, we augment the power of this structure slightly so that (as at \itref{prop:addthm2} below)
the structure found, $(V^{\mathrm{add}},C^{\mathrm{add}})$
takes as its input a small set of colours $\hat{C}$ in addition to a small set of vertices $\hat{A}\cup \hat{B}$, and finds two vertex-disjoint matchings which use exactly all but two vertices in $V^{\mathrm{add}}\cup \hat{A}\cup \hat{B}$ such that one is a monochromatic matching with a specific colour and the other is a rainbow matching using exactly all but one of the colours in $C^{\mathrm{add}}\cup \hat{C}$. This augmentation is easy to do, and convenient for the application of the structure.

\begin{theorem}[Addition structure]\label{thm:RSBaddition}
Let $1/n\ll p,q_V,q_C\leq 1$.  Let $1/n\llpoly \eps \llpoly \eta  \llpoly \gamma \llpoly \alpha \llpoly\log^{-1}n$. Let $G$ be an $(n,p,\eps)$-properly-pseudorandom bipartite graph with vertex classes $A$ and $B$. Independently, let $V$ be a $q_V$-random subset of $V(G)$ and let $C$ be a $q_C$-random subset of $C(G)$. Then, with high probability, the following holds.

Given any $c\in C(G)$ and any $E_c\subset E_c(G)$ with $|E_c(G)\setminus E_c|\leq \alpha n$, there are sets $V^{\mathrm{add}}\subset V$ and $C^{\mathrm{add}}\subset C$ such that
\stepcounter{propcounter}
\begin{enumerate}[label = {\textbf{\emph{\Alph{propcounter}\arabic{enumi}}}}]
\item $|A\cap V^{\mathrm{add}}|=|B\cap V^{\mathrm{add}}|=2\gamma n+1$ and $|C^{\mathrm{add}}|=\gamma n+1$, and,\label{prop:addthm1}
\item for any $\hat{A}\subset V(G)\setminus V^{\mathrm{add}}$ and $\hat{B}\subset V(G)\setminus V^{\mathrm{add}}$ with $|\hat{A}|=|\hat{B}|\leq \eta n$, and any $\hat{C}\subset C(G)\setminus C^{\mathrm{add}}$ with $|\hat{C}|\leq \eta n$, $G[V^{\mathrm{add}}\cup \hat{A}\cup \hat{B}]$ contains vertex-disjoint matchings $M^{\textrm{id}}$ and $M^{\textrm{rb}}$ such that $M^{\textrm{id}}$ has $\gamma n+|\hat{A}|-|\hat{C}|$ edges in $E_c$ with both vertices in $V^{\mathrm{add}}$
and $M^{\textrm{rb}}$ is a $(C^{\mathrm{add}}\cup \hat{C})$-rainbow matching with size $|C^{\mathrm{add}}\cup \hat{C}|-1$.\label{prop:addthm2}
\end{enumerate}
\end{theorem}

For Theorem~\ref{thm:symbolnum}, we use a slight variant of Theorem~\ref{thm:RSBaddition}, as follows, where the matchings  $M^{\textrm{id}}$ and $M^{\textrm{rb}}$ that are ultimately found in the addition structure use exactly the vertices $V^{\mathrm{add}}\cup \hat{A}\cup \hat{B}$, but there is now more than 1 colour missing from $M^{\textrm{rb}}$ as $C^{\mathrm{add}}$ is slightly larger.

\begin{theorem}[Addition structure variant]\label{thm:RSBaddition-variant}
Let $1/n\ll p,q_V,q_C\leq 1$.  Let $1/n\llpoly \eps \llpoly \eta  \llpoly \gamma \llpoly \alpha \llpoly\log^{-1}n$. Let $G$ be an $(n,p,\eps)$-properly-pseudorandom bipartite graph with vertex classes $A$ and $B$. Let $V$ be a $q_V$-random subset of $V(G)$ and let $C$ be a $q_C$-random subset of $C(G)$. Then, with high probability, the following holds.

Given any $c\in C(G)$ and any $E_c\subset E_c(G)$ with $|E_c(G)\setminus E_c|\leq \alpha n$, there are sets $V^{\mathrm{add}}\subset V$ and $C^{\mathrm{add}}\subset C$ such that
\stepcounter{propcounter}
\begin{enumerate}[label = {\textbf{\emph{\Alph{propcounter}\arabic{enumi}}}}]
\item $|A\cap V^{\mathrm{add}}|=|B\cap V^{\mathrm{add}}|=2\gamma n$ and $|C^{\mathrm{add}}|=\gamma n+100$, and,\label{prop:addthmvar1}
\item for any $\hat{A}\subset V(G)\setminus V^{\mathrm{add}}$ and $\hat{B}\subset V(G)\setminus V^{\mathrm{add}}$ with $|\hat{A}|=|\hat{B}|\leq \eta n$, and any $\hat{C}\subset C(G)\setminus C^{\mathrm{add}}$ with $|\hat{C}|\leq \eta n$, $G[V^{\mathrm{add}}\cup \hat{A}\cup \hat{B}]$ contains vertex-disjoint matchings $M^{\textrm{id}}$ and $M^{\textrm{rb}}$ such that $M^{\textrm{id}}$ has $\gamma n+|\hat{A}|-|\hat{C}|$ edges in $E_c$ with both vertices in $V^{\mathrm{add}}$
and $M^{\textrm{rb}}$ is a $(C^{\mathrm{add}}\cup \hat{C})$-rainbow matching with size $|C^{\mathrm{add}}\cup \hat{C}|-100$.\label{prop:addthmvar2}
\end{enumerate}
\end{theorem}


\subsection{Proof of the technical theorems from the component results}\label{sec:genproof}
Given Theorems~\ref{thm:RSBcoveringstep},~\ref{thm:RSBabsorption}, and~\ref{thm:RSBaddition}, we can now combine them to prove Theorem~\ref{thm-technical}.
\begin{proof}[Proof of Theorem~\ref{thm-technical} from Theorems~\ref{thm:RSBcoveringstep},~\ref{thm:RSBabsorption}, and~\ref{thm:RSBaddition}]
Let $\eta,\gamma,\beta$ and $\alpha$ satisfy
\[
\frac{1}{n}\llpoly \eps\llpoly \eta\llpoly \gamma \llpoly \beta\llpoly\alpha\llpoly\log^{-1}n.
\]
Let $G$ be an $(n,p,\eps)$-properly-pseudorandom bipartite graph on vertex classes $A$ and $B$, where this will imply that (see \ref{prop-pseud-basic-1} later) $|A|=|B|=n$ and $n\leq |C(G)|\leq (1+\eps)n$.
Let $V_0\cup V^{\mathrm{s-r}}$ be a random partition of $V(G)$ so that $V_0$ is $(1/4)$-random and $V^{\mathrm{s-r}}$ is $(3/4)$-random.
Independently, let $C_0\cup C^{\mathrm{s-r}}$ be a random partition of $V(G)$ so that $C_0$ is $(1/4)$-random and $C^{\mathrm{s-r}}$ is $(3/4)$-random. As $\E|C^{\mathrm{s-r}}|\leq 3(1+\eps)n/4$, we have, by a simple application of Chernoff's bound (see Lemma~\ref{Lemma_Chernoff} later), that $|C^{\mathrm{s-r}}|\leq 7n/8$ with high probability.
Thus, by Theorems~\ref{thm:RSBcoveringstep},~\ref{thm:RSBabsorption} (applied with $\gamma'=2\gamma$), and~\ref{thm:RSBaddition}, we can assume the following properties hold concurrently.
\stepcounter{propcounter}
\begin{enumerate}[label = {\textbf{\Alph{propcounter}\arabic{enumi}}}]
\item $|C^{\mathrm{s-r}}|\leq 7n/8$.\label{newD1}
\item For any sets $\bar{A}\subset A$, $\bar{B}\subset B$, $\bar{C}\subset C(G)$ with $|\bar{A}|=|\bar{B}|=|\bar{C}|$ and such that $V^{\mathrm{s-r}}\subset \bar{A}\cup \bar{B}$ and $C^{\mathrm{s-r}}\subset \bar{C}$, there is a $\bar{C}$-rainbow matching in $G[\bar{A},\bar{B}]$ with at least $|\bar{C}|-\eta n$ edges.\label{prop:main1}
\item For all but at most $\alpha n$ colours $c\in C(G)$, there is a set $E_c\subset E_{c}(G[V_0])$ with $|E_{c}(G[V_0])\setminus E_c|\leq \alpha n$ such that the following holds.\label{prop:main2}
\begin{enumerate}[label = \roman{enumii})]
\item For each set $E\subset E_c$ with $|E|\leq 2\gamma n$ and sets $\bar{V}\subset V_0$ and $\bar{C}\subset C_0$ with $|\bar{V}|,|\bar{C}|\leq 10\gamma n$, and $0\leq \ell\leq |E|$, there are sets $V^{\mathrm{abs}}\subset V_0\setminus(\bar{V}\cup V(E))$
 and $C^{\mathrm{abs}}\subset C_0\setminus \bar{C}$ such that $|A\cap V^{\mathrm{abs}}|=|B\cap V^{\mathrm{abs}}|=\beta n-\ell$ and $|C^{\mathrm{abs}}|=\beta n$, and the following property holds.\label{prop:main22}
\begin{itemize}
\item For any $\bar{E}\subset E$ with $|\bar{E}|=\ell$, $G[V^{\mathrm{abs}}\cup V(\bar{E})]$ has a $C^{\mathrm{abs}}$-rainbow matching with size $\beta n$.
\end{itemize}
\end{enumerate}
\item Given any $c\in C(G)$ and any $E_0\subset E_c(G[V])$ with $|E_c(G[V])\setminus E_0|\leq \alpha n$, there are sets $V^{\mathrm{add}}\subset V_0$ and $C^{\mathrm{add}}\subset C_0$ such that\label{prop:main3}
\begin{enumerate}[label = \roman{enumii})]\addtocounter{enumii}{1}
\item $|A\cap V^{\mathrm{add}}|=|B\cap V^{\mathrm{add}}|=2\gamma n+1$ and $|C^{\mathrm{add}}|=\gamma n+1$, and,\label{prop:main31}
\item for any $\hat{A}\subset A\setminus V^{\mathrm{add}}$ and $\hat{B}\subset B\setminus V^{\mathrm{add}}$ with $|\hat{A}|=|\hat{B}|\leq \eta n$, and any $\hat{C}\subset C(G)\setminus C^{\mathrm{add}}$
with $|\hat{C}|\leq \eta n$, $G[V^{\mathrm{add}}\cup \hat{A}\cup \hat{B}]$ contains two vertex-disjoint matchings: one of $\gamma n+|\hat{A}|-|\hat{C}|$ edges in $E_0$ with vertices in $V^{\mathrm{add}}$, and one a $(C^{\mathrm{add}}\cup \hat{C})$-rainbow matching of size $|C^{\mathrm{add}}\cup \hat{C}|-1$.\label{prop:main32}
\end{enumerate}
\end{enumerate}

Using \ref{prop:main2} (and that $|C(G)|\geq n$), pick $c_0\in C(G)$ and a set $E_0\subset E_c(G[V])$ with $|E_{c}(G[V])\setminus E_0|\leq \alpha n$ so that \ref{prop:main22} holds with $c=c_0$ and $E_c=E_0$.
Using \ref{prop:main3}, find sets $V^{\mathrm{add}}\subset V_0$ and $C^{\mathrm{add}}\subset C_0$ such that \ref{prop:main31} and \ref{prop:main32} hold.
Let $E_1$ be the set of edges in $E_0$ with both vertices in $V^{\mathrm{add}}$, noting that $|E_1|\leq |V^{\mathrm{add}}|/2\leq 2\gamma n$.
Using \ref{prop:main22}, find sets $V^{\mathrm{abs}}\subset V_0\setminus (V^{\mathrm{add}}\cup V(E_1))$ and $C^{\mathrm{abs}}\subset C_0\setminus C^{\mathrm{add}}$ such that $|A\cap V^{\mathrm{abs}}|=|B\cap V^{\mathrm{abs}}|=\beta n-\gamma n$ and $|C^{\mathrm{abs}}|=\beta n$ and the following holds.
\begin{enumerate}[label = \roman{enumi})]\addtocounter{enumi}{3}
\item  For any set $\bar{E}\subset E_1$ with $|\bar{E}|=\gamma n$, $G[V^{\mathrm{abs}}\cup V(\bar{E})]$ contains a $C^{\mathrm{abs}}$-rainbow matching with size $\beta n$.\label{prop:main5}
\end{enumerate}

Now, $V^{\mathrm{add}}$ and $V^{\mathrm{abs}}$ both have an equal number of vertices in $A$ and $B$, and $V^{\mathrm{add}}\cup V^{\mathrm{abs}}\subset V_0$. Let $\bar{A}=A\setminus (V^{\mathrm{add}}\cup V^{\mathrm{abs}})$ and $\bar{B}=B\setminus (V^{\mathrm{add}}\cup V^{\mathrm{abs}})$, so that $|\bar{A}|=|\bar{B}|=n-(\beta n-\gamma n)-(2\gamma n+1)=n-\beta n-\gamma n-1$ and $V^{\mathrm{s-r}}=V(G)\setminus V_0\subset \bar{A}\cup \bar{B}$.
We have, as $|C^{\mathrm{add}}|=\gamma n+1$ and $|C^{\mathrm{abs}}|=\beta n$,
\begin{equation}\label{eqn:Clower}
|C(G)\setminus (C^{\mathrm{add}}\cup C^{\mathrm{abs}})|\geq n-(\gamma n+1)-\beta n=|\bar{A}|.
\end{equation}
Thus, as, using \ref{newD1}, we have $n-(\gamma n+1)-\beta n\geq |C^{\mathrm{s-r}}|$, we can choose $\bar{C}\subset C(G)\setminus (C^{\mathrm{add}}\cup C^{\mathrm{abs}})$ with $C^{\mathrm{s-r}}\subset \bar{C}$ and $|\bar{C}|=|\bar{A}|$ (so that, then, $|\bar{C}|=n-|C^{\mathrm{add}}|-|C^{\mathrm{abs}}|$). Therefore, by \ref{prop:main1}, $G[\bar{A}\cup \bar{B}]$ contains a $\bar{C}$-rainbow matching, $M_1$ say, with at least $|\bar{C}|-\eta n$ edges.

Let $\hat{A}=\bar{A}\setminus V(M_1)$, $\hat{B}=\bar{B}\setminus V(M_1)$ and $\hat{C}=\bar{C}\setminus C(M_1)$, so that $|\hat{A}|=|\hat{B}|=|\hat{C}|=|\bar{C}|-|M_1|\leq \eta n$. Then, by \ref{prop:main32}, $G[V^{\mathrm{add}}\cup \hat{A}\cup \hat{B}]$ contains vertex disjoint matchings $M_2$ and $M_2'$ so that $M_2$ is a $(C^{\mathrm{add}}\cup \hat{C})$-rainbow matching with $|C^{\mathrm{add}}\cup \hat{C}|-1$ edges
and $M_2'\subset E_0\cap E_c(G[V^{\mathrm{add}}])$ is a matching of $\gamma n$ edges in $E_1$.

Then, using \ref{prop:main5}, $G[V^{\mathrm{abs}}\cap V(M_2')]$ contains a $C^{\mathrm{abs}}$-rainbow matching $M_3$ with size $|C^{\mathrm{abs}}|=\beta n$.
Let $M=M_1\cup M_2\cup M_3$. Noting that $V^{\mathrm{add}}\cup V^{\mathrm{abs}}\cup (\bar{A}\cup \bar{B})=V^{\mathrm{add}}\cup V^{\mathrm{abs}}\cup V(M_1)\cup (\hat{A}\cup \hat{B})$ are partitions of $V(G)$, and recalling that $M_2'$ and $M_2$ are vertex-disjoint matchings in $G[V^{\mathrm{add}}\cup \hat{A}\cup \hat{B}]$ and $V(M_3)\subset V^{\mathrm{abs}}\cup V(M_2')$, we have that $M$ is the union of vertex-disjoint matchings.
Furthermore, $C(M_1)=\bar{C}\setminus \hat{C}$, $C(M_2)\subset (
C^{\mathrm{add}}\cup \hat{C})$ and $C(M_3)=C^{\mathrm{abs}}$, and $C^{\mathrm{add}}$ and $C^{\mathrm{add}}$ are disjoint sets, while $\hat{C}\subset \hat{C}$. Thus, $M$ is a union of 3 rainbow matchings with disjoint colour sets, and is thus rainbow. Finally,
\begin{align}
|M|&=|M_1|+|M_2|+|M_3|=|\bar{C}\setminus \hat{C}|+|C^{\mathrm{add}}\cup \hat{C}|-1+|C^{\mathrm{abs}}|=|\bar{C}|+|C^{\mathrm{add}}|+|C^{\mathrm{abs}}|-1=n-1,\label{eqn:tochange}
\end{align}
so that $M$ is a rainbow matching with $n-1$ edges, as required.
\end{proof}

Very similarly, we can also now prove Theorem~\ref{thm-technical-variant}, except using Theorem~\ref{thm:RSBaddition-variant} in place of Theorem~\ref{thm:RSBaddition}, as follows.

\begin{proof}[Proof of Theorem~\ref{thm-technical-variant} from Theorems~\ref{thm:RSBcoveringstep},~\ref{thm:RSBabsorption}, and~\ref{thm:RSBaddition-variant}]
This proceeds identically to the proof of Theorem~\ref{thm-technical}, except we apply Theorem~\ref{thm:RSBaddition-variant} instead of Theorem~\ref{thm:RSBaddition}, so that \ref{prop:main31} and \ref{prop:main32} becomes
\begin{enumerate}[label = \roman{enumi})']\addtocounter{enumi}{1}
\item $|A\cap V^{\mathrm{add}}|=|B\cap V^{\mathrm{add}}|=2\gamma n$ and $|C^{\mathrm{add}}|=\gamma n+100$, and,\label{prop:main31pr}
\item for any $\hat{A}\subset A\setminus V^{\mathrm{add}}$ and $\hat{B}\subset B\setminus V^{\mathrm{add}}$ with $|\hat{A}|=|\hat{B}|\leq \eta n$, and any $\hat{C}\subset C(G)\setminus C^{\mathrm{add}}$
with $|\hat{C}|\leq \eta n$, $G[V^{\mathrm{add}}\cup \hat{A}\cup \hat{B}]$ contains two vertex-disjoint matchings: one of $\gamma n+|\hat{A}|-|\hat{C}|$ edges in $E_0$ with vertices in $V^{\mathrm{add}}$, and one a $(C^{\mathrm{add}}\cup \hat{C})$-rainbow matching of size $|C^{\mathrm{add}}\cup \hat{C}|-100$.\label{prop:main32pr}
\end{enumerate}
 Then, proceeding as above, we have instead that $|\bar{A}|=|\bar{B}|=n-(\beta n-\gamma n)-(2\gamma n)=n-\beta n-\gamma n$ and \eqref{eqn:Clower} becomes
\begin{equation*}
|C(G)\setminus (C^{\mathrm{add}}\cup C^{\mathrm{abs}})|= |C(G)|-\gamma n-\beta n\geq n+100-\gamma n-\beta n\geq |\bar{A}|+100,
\end{equation*}
so we choose $\bar{C}$ as above, but so that $|\bar{C}|=|\bar{A}|+100$, and hence $|\bar{C}|+|C^{\mathrm{add}}|+|C^{\mathrm{abs}}|=n+100$.
Finally, continuing as above, we can instead have when applying \ref{prop:main32pr} instead of \ref{prop:main32} that $M_2$ is a $(C^\mathrm{add}\cup \hat{C})$-rainbow matching with $|C^\text{add}\cup \hat{C}|-100$ edges, so that \eqref{eqn:tochange} becomes
\begin{align*}
|M|&=|M_1|+|M_2|+|M_3|=|\bar{C}\setminus \hat{C}|+|C^{\mathrm{add}}\cup \hat{C}|-100+|C^{\mathrm{abs}}|=|\bar{C}|+|C^{\mathrm{add}}|+|C^{\mathrm{abs}}|-100=n,
\end{align*}
so that the matching found has $n$ edges, as required.
\end{proof}


\subsection{Typical hypergraphs}\label{sec:typical}
To apply results shown using the semi-random method, we will work with 3-partite 3-uniform hypergraphs formed from our coloured graphs as follows.

\begin{defn}\label{defn:HG}
Given a coloured bipartite graph $G$ with vertex classes $A$ and $B$, we denote by $\mathcal{H}(G)$ the 3-partite 3-uniform hypergraph with vertex classes $A$, $B$ and $C(G)$, and edge set $\{abc:a\in A,b\in B,c\in C(G),ab\in E_c(G)\}$.
\end{defn}

Note that if $G$ is properly coloured, then $\mathcal{H}(G)$ is simple (i.e., each pair of vertices is in at most one edge). A common application of the semi-random method is to show that an almost-regular simple hypergraph will contain an almost-perfect matching (see, for example, Corollary~\ref{cor:nibble}). To find almost-perfect matchings in subgraphs (with a large portion of random vertices), we will use the stronger condition of \emph{typicality}. Here, our definition of typicality for a 3-partite balanced 3-uniform hypergraph follows that of Keevash, Pokrovskiy, Sudakov and Yepremyan~\cite{KPSY} (stated equivalently for coloured bipartite graphs), but we record the error term in the definition differently to match our polynomial bounds. This is only relevant later when we use one result from~\cite{KPSY} (see Lemma~\ref{lem:typicalpseudo}).
We  start by defining a typical bipartite graph as follows.
\begin{defn}\label{defn:typical}
A bipartite graph $H$ with vertex classes $A$ and $B$ is \emph{$(n,p,\eps)$-typical} if the following hold.
\begin{itemize}
\item $|A|=(1\pm \eps)n$ and $|B|=(1\pm \eps)n$.
\item For each $v\in V(H)$, $d_{H}(v)=(1\pm \eps)pn$.
\item For each distinct $u,v\in V(H)$ with $u,v\in A$ or $u,v\in B$, we have $|N_H(u)\cap N_H(v)|=(1\pm \eps)p^2n$.
\end{itemize}
\end{defn}

Essentially, a 3-partite simple 3-uniform hypergraph is typical if removing the vertices in any vertex class from $V(\mathcal{H})$ and from each edge in $E(\mathcal{H})$ gives a typical bipartite graph. To formalise this we use the following two definitions.

\begin{defn} Given a 3-uniform hypergraph $\mathcal{H}$ and any disjoint sets $X,Y\subset V(\mathcal{H})$, let $\mathcal{H}_{X,Y}$ be the bipartite graph with vertex classes $X$ and $Y$, with edges $xy$ exactly those $x\in X$ and $y\in Y$ for which there exists some $z\in V(H)\setminus \{x,y\}$ with $\{x,y,z\}\in E(\mathcal{H})$.
\end{defn}

\begin{defn}
A 3-partite simple 3-uniform hypergraph $\mathcal{H}$ with vertex classes $A$, $B$ and $C$ is \emph{$(n,p,\eps)$-typical} if each of $\mathcal{H}_{AB}$, $\mathcal{H}_{BC}$ and $\mathcal{H}_{AC}$ is $(n,p,\eps)$-typical.
\end{defn}


\subsection{Proper-pseudorandomness}\label{sec:pseud}

We now state precisely our definition of proper pseudorandomness (Definition~\ref{defn:pseud}), where the canonical example of an $(n,p,\eps)$-properly-pseudorandom graph is an optimally coloured copy of $K_{n,n}$ with each edge deleted independently at random with probability $p$. Definition~\ref{defn:pseud} contains conditions \ref{prop-pseud-basic-1}--\ref{prop-pseud-add-new-2}. Of these conditions, \ref{prop-pseud-basic-1} and \ref{prop-pseud-basic-2new} are natural conditions for pseudorandomness here, while \ref{prop-pseud-abs-prime} is used in creating the absorption structure (more specifically, to ensure 4-cycles depicted in Figure~\ref{fig:abs} exist), and \ref{prop-pseud-add-2}--\ref{prop-pseud-add-new-2} are used for the addition structure. Of these last four conditions, in Figure~\ref{fig:add}, \ref{prop-pseud-add-2} is used to find the $w,x$- and $y,z$-paths in $E_1\cup F_1$ while \ref{prop-pseud-add-3} is used to find the 8-cycles in $E_2\cup F_2$. As discussed in Section~\ref{sec:thetroubles}, the path $E_3\cup F_3$ cannot be found in Figure~\ref{fig:add} using exactly the colours we wish, but \ref{prop-pseud-add-new-1} will allow us to do this if we have 80 extra colours, while \ref{prop-pseud-add-new-2} will allow us to complete the final step where we incorporate all but 1 of the final colours we are incorporating. The structures discussed in \ref{prop-pseud-abs-prime}--\ref{prop-pseud-add-3} are pictured in Figure~\ref{fig:pseudpics} (with labelling for the proof of a later result, Proposition~\ref{prop:nearcompletepseudorandom}).

\begin{defn}\label{defn:pseud}
A bipartite graph $G$, with vertex classes $A$ and $B$, is \emph{$(n,p,\eps)$-properly-pseudorandom} if it is properly coloured and the following hold with $\alpha=p^{12}/10^{100}$.
\stepcounter{propcounter}
\begin{enumerate}[label = {\textbf{\Alph{propcounter}\arabic{enumi}}}]
\item $|A|=|B|=n$ and $n\leq |C(G)|\leq (1+\eps)n$. \label{prop-pseud-basic-1}
\item $\mathcal{H}(G)$ is $(n,p,\eps)$-typical.\label{prop-pseud-basic-2new}
\item For each $c\in C(G)$ and $e\in E_c(G)$, for all but at most $\sqrt{n}$ edges $f\in E_c(G)\setminus \{e\}$, there are at least $\alpha n^2$ pairs $(S_1,S_2)$ such that $S_1$ and $S_2$ are vertex-disjoint rainbow 4-cycles with $e\in E(S_1)$ and $f\in E(S_2)$, and
the colour sets of the neighbouring edges of $e$ in $S_1$ and the neighbouring edges of $f$ in $S_2$ are the same.\label{prop-pseud-abs-prime}
\item For each $u\in A$, $v\in B$ and $c_0\in C(G)$, there are disjoint sets ${V}_1,\ldots,{V}_{\alpha n}\subset V(G)\setminus \{u,v\}$ and disjoint sets $C_1,\ldots,C_{\alpha n}$ in $C(G)\setminus\{c_0\}$ such that,
for each $i\in [\alpha n]$, $|{V}_i|=4$, $|C_i|=3$, $G[V_i]$ contains 2 colour-$c_0$ edges and $G[\{u,v\}\cup V_i]$ contains an exactly-$C_i$-rainbow matching in $E(G)\setminus \{uv\}$.\label{prop-pseud-add-2}
\item For each distinct $c_0,d\in C(G)$, there are disjoint sets $V_1,\ldots,V_{\alpha n/12}$ in $V(G)$ and disjoint sets $C_1,\ldots,C_{\alpha n/12}$ in $C(G)\setminus\{c_0,d\}$,
so that, for each $i\in [\alpha n/12]$, $|V_i|=8$ and $|C_i|=3$, and $G[V_i]$ contains a matching of $4$ colour-$c_0$ edges
and an exactly-$(C_i\cup \{d\})$-rainbow matching.\label{prop-pseud-add-3}
\item  For any $c_0\in C(G)$, $0\leq k\leq 20$, and any $\bar{C}\subset C(G)\setminus \{c_0\}$ with $|\bar{C}|\geq 5k$, there are vertex-disjoint sets $\bar{V}_1,\ldots,\bar{V}_{\alpha n}\subset V(G)$
such that,
for each $i\in [\alpha n]$, $|\bar{V}_i|=2k+2$ and $G[\bar{V}_i]$ contains both a matching of $k+1$ colour-$c_0$ edges and a $\bar{C}$-rainbow matching with $k$ edges. \label{prop-pseud-add-new-1}
\item Setting $k=100$, for each $c_0\in C(G)$, there is some $r\in \N$ and disjoint sets ${V}_1,\ldots,{V}_{r}$ in $V(G)$ and disjoint sets $C_1,\ldots,C_{r}$ in $C(G)\setminus\{c_0\}$
 with $|V_i|=2k$ and $|C_i|=k$ for each $i\in [r]$ such that $G[V_i]$ contains an exactly-$C_i$-rainbow matching and a perfect matching of colour-$c_0$ edges, and the following holds. For every $\bar{C}\subset C(G)\setminus \{c_0\}$ with $|\bar{C}|\leq k$, for at least $\alpha^2n$ values of $i\in [r]$, there are vertex-disjoint sets $\bar{V}_1,\ldots,\bar{V}_{\alpha n}\subset V(G)$ such that,
for each $j\in [\alpha n]$, $|\bar{V}_j|=2k+2|\bar{C}|+2$ and $G[\bar{V}_j]$ contains both a matching of $k+|\bar{C}|+1$ colour-$c_0$ edges and a $(\bar{C}\cup C_i)$-rainbow matching with $k+|\bar{C}|$ edges.
\label{prop-pseud-add-new-2}

\end{enumerate}
\end{defn}

\begin{figure}[b]
\begin{center}\begin{tikzpicture}[scale=0.8]
\def\vxrad{0.07cm}
\def\horunit{1.2}
\def\edgelength{0.4}
\def\betweenrows{0.5}


\def\gapratio{1}
\def\biggergap{2}

\foreach \num in {0,1}
{
\coordinate (B\num) at ($(0,0)+\num*4*\gapratio*(1,0)+\gapratio*(-0.5,0.5)$);
\coordinate (C\num) at ($(0,0)+\num*4*\gapratio*(1,0)+\gapratio*(0.5,0.5)$);
\coordinate (E\num) at ($0.5*(B\num)+0.5*(C\num)+0.86602540*\gapratio*(0,1)$);
}

\coordinate (C1) at ($0.5*(B0)+0.5*(C0)-(0,1)$);
\coordinate (C2) at ($(C1)+(0,1)$);
\coordinate (C3) at ($0.5*\gapratio*(1,0)$);
\coordinate (C4) at ($-1*(C3)$);
\coordinate (X1) at (B0);
\coordinate (X2) at (C0);
\coordinate (X3) at ($(B0)-(0,1)$);
\coordinate (X4) at ($(C0)-(0,1)$);

\def\bit{0.375}
\draw ($(C2)+(0,\bit)$) node {$c''$};
\draw ($(C1)+(0,-\bit)$) node {$c'$};
\draw ($(C4)-(\bit,0)$) node {$c$};

\draw [white] ($(C3)+2.5*(\gapratio,0)$) -- ($(C3)+2.5*(\gapratio,0)$);


\def\sm{0.2};

\foreach \num/\numm/\coll/\fillcoll in {0/0/red!50/red!25,0/1/blue!50/blue!25,0/2/green!50/green!25,0/3/orange!50/orange!25}
{
\begin{scope}[transform canvas={rotate=\numm*90}]
{
\draw [thick,\coll] (C\num) -- (B\num);
}
\end{scope}
}
\foreach \num in {0}
{
\foreach \numm in {0,1,2,3}
{
\begin{scope}[transform canvas={rotate=\numm*90}]
{
\draw [fill] (B\num) circle [radius=\vxrad];
}
\end{scope}
}
}

\begin{scope}[transform canvas={shift={(2*\gapratio,0)}}]
\foreach \num/\numm/\coll/\fillcoll in {0/0/red!50/red!25,0/1/blue!50/blue!25,0/2/green!50/green!25,0/3/orange!50/orange!25}
{
\begin{scope}[transform canvas={rotate=\numm*90}]
{
\draw [thick,\coll] (C\num) -- (B\num);
}
\end{scope}
}
\foreach \num in {0}
{
\foreach \numm in {0,1,2,3}
{
\begin{scope}[transform canvas={rotate=\numm*90}]
{
\draw [fill] (B\num) circle [radius=\vxrad];
}
\end{scope}
}
}
\draw ($(C2)+(0,\bit)$) node {$c''$};
\draw ($(C1)+(0,-\bit)$) node {$c'$};
\draw ($(C4)-(\bit,0)$) node {$c$};
\end{scope}

\draw [white] (0,1.25) -- (0,-1.25);

\def\bitt{0.625};
\draw ($(X1)+\bitt*(-0.5,0.5)$) node {$u_1$};
\draw ($(X2)+\bitt*(0.5,0.5)$) node {$u_4$};
\draw ($(X3)+\bitt*(-0.5,-0.5)$) node {$u_2$};
\draw ($(X4)+\bitt*(0.5,-0.5)$) node {$u_3$};

\begin{scope}[transform canvas={shift={(2*\gapratio,0)}}]
\draw ($(X1)+\bitt*(-0.5,0.5)$) node {$v_1$};
\draw ($(X3)+\bitt*(-0.5,-0.5)$) node {$v_2$};
\draw ($(X4)+\bitt*(0.5,-0.5)$) node {$v_3$};
\draw ($(X2)+\bitt*(0.5,0.5)$) node {$v_4$};
\end{scope}

\end{tikzpicture}\hspace{0.3cm}
\begin{tikzpicture}
\draw [white] (-0.5,0) -- (0.5,0);
\draw [dashed] (0,-1.2) -- (0,1);.2
\end{tikzpicture}
\begin{tikzpicture}[scale=0.8]
\def\vxrad{0.07cm}
\def\horunit{1.2}
\def\edgelength{0.4}
\def\betweenrows{0.5}


\def\gapratio{1}
\def\biggergap{2}

\def\hggt{2.5}

\def\hanglow{-1.75}
\def\widdd{9}
\draw [white] (\widdd,-\hanglow) -- (\widdd,-\hanglow-\hggt);

\foreach \num in {1}
{
\coordinate (A\num) at ($(0,0)+\num*4*\gapratio*(1,0)$);
\coordinate (B\num) at ($(0,0)+\num*4*\gapratio*(1,0)+\gapratio*(1,0)$);
\coordinate (C\num) at ($(0,0)+\num*4*\gapratio*(1,0)+\gapratio*(2,0)$);
\coordinate (D\num) at ($(0,0)+\num*4*\gapratio*(1,0)+\gapratio*(3,0)$);
\coordinate (E\num) at ($0.5*(B\num)+0.5*(C\num)+0.86602540*\gapratio*(0,1)$);
\coordinate (F\num) at ($0.5*(B\num)+0.5*(C\num)-0.86602540*\gapratio*(0,1)$);
}
\foreach \num in {2}
{
\coordinate (B\num) at ($(0,0)+\num*4*\gapratio*(1,0)+\gapratio*(0,0)$);
\coordinate (D\num) at ($(0,0)+\num*4*\gapratio*(1,0)+\gapratio*(1,0)$);
\coordinate (A\num) at ($0.5*(D1)+0.5*(B\num)+0.86602540*\gapratio*(0,1)$);
\coordinate (C\num) at ($0.5*(B\num)+0.5*(D\num)+0.86602540*\gapratio*(0,1)$);
\coordinate (E\num) at ($0.5*(B\num)+0.5*(C\num)+0.86602540*\gapratio*(0,1)$);
\coordinate (F\num) at ($0.5*(B\num)+0.5*(C\num)-0.86602540*\gapratio*(0,1)$);
}
\coordinate (G1) at ($0.5*(D1)+0.5*(A2)+0.86602540*\gapratio*(0,1)$);
\coordinate (G2) at ($0.5*(D1)+0.5*(A2)+0.86602540*\gapratio*(0,1)+4*\gapratio*(1,0)$);
\coordinate (A30) at ($(A2)+2*\gapratio*(1,0)$);
\foreach \num in {3,4}
{
\coordinate (A\num) at ($(0,0)+\num*4*\gapratio*(1,0)+\biggergap*(1,0)$);
\coordinate (B\num) at ($(0,0)+\num*4*\gapratio*(1,0)+\gapratio*(1,0)+\biggergap*(1,0)$);
\coordinate (C\num) at ($(0,0)+\num*4*\gapratio*(1,0)+\gapratio*(2,0)+\biggergap*(1,0)$);
\coordinate (D\num) at ($(0,0)+\num*4*\gapratio*(1,0)+\gapratio*(3,0)+\biggergap*(1,0)$);
\coordinate (E\num) at ($0.5*(B\num)+0.5*(C\num)+0.86602540*\gapratio*(0,1)$);
\coordinate (F\num) at ($0.5*(B\num)+0.5*(C\num)-0.86602540*\gapratio*(0,1)$);
}

\coordinate (G3) at ($0.5*(D3)+0.5*(A4)+0.86602540*\gapratio*(0,1)$);

\def\sm{0.2};
\foreach \num in {2}
{
\draw [thick,black!50,fill=black!25,densely dotted] (A\num) -- (B\num);
\draw [thick,black!50,fill=black!25,densely dotted] (C\num) -- (D\num);
}

\foreach \num in {2}
{
\draw [thick,red!50,fill=red!25] (C\num) -- (B\num);
}

\foreach \num/\numplus in {1/2}
{
\draw [thick,blue!50,fill=blue!25] (A\numplus) -- (D\num);
}

\foreach \num/\numplus in {2/3}
{
\draw [thick,orange!50,fill=orange!25] (A30) -- (D\num);
}

\def\bit{0.375}
\draw ($0.5*(D1)+0.5*(A2)-(0.7*\bit,-0.1*\bit)$) node {$c_1$};
\draw ($0.5*(B2)+0.5*(C2)-(0.4*\bit,-0.6*\bit)$) node {$c_2$};
\draw ($0.5*(A30)+0.5*(D2)+(0.7*\bit,-0.1*\bit)$) node {$c_3$};
\draw ($0.5*(A2)+0.5*(B2)-(0.2*\bit,0.8*\bit)$) node {$c_0$};
\draw ($0.5*(C2)+0.5*(D2)-(0.2*\bit,0.8*\bit)$) node {$c_0$};

\draw ($(D1)+(-\bit,0)$) node {$u$};
\draw ($(B2)-(0,\bit)$) node {$w_2$};
\draw ($(A2)+(0,\bit)$) node {$w_1$};
\draw ($(C2)+(0,\bit)$) node {$w_3$};
\draw ($(D2)-(0,\bit)$) node {$w_4$};
\draw ($(A30)+(\bit,0)$) node {$v$};

\foreach \lett in {A,B,C,D}
\foreach \num in {2}
{
\draw [fill] (\lett\num) circle [radius=\vxrad];
}
\foreach \lett in {D}
\foreach \num in {1}
{
\draw [fill] (\lett\num) circle [radius=\vxrad];
}


\draw [fill] (A30) circle [radius=\vxrad];

\end{tikzpicture}
\begin{tikzpicture}
\draw [white] (-0.5,0) -- (0.5,0);
\draw [dashed] (0,-1.2) -- (0,1);.2
\end{tikzpicture}
\begin{tikzpicture}[scale=0.8]
\def\vxrad{0.07cm}
\def\horunit{1.2}
\def\edgelength{0.4}
\def\betweenrows{0.5}


\def\gapratio{1}
\def\biggergap{2}

\def\hggt{2.5}

\def\hanglow{-1.75}
\def\widdd{9}
\draw [white] (\widdd,-\hanglow) -- (\widdd,-\hanglow-\hggt);

\foreach \num in {1}
{
\coordinate (A\num) at ($(0,0)+\num*4*\gapratio*(1,0)$);
\coordinate (B\num) at ($(0,0)+\num*4*\gapratio*(1,0)+\gapratio*(1,0)$);
\coordinate (C\num) at ($(0,0)+\num*4*\gapratio*(1,0)+\gapratio*(2,0)$);
\coordinate (D\num) at ($(0,0)+\num*4*\gapratio*(1,0)+\gapratio*(3,0)$);
\coordinate (E\num) at ($0.5*(B\num)+0.5*(C\num)+0.86602540*\gapratio*(0,1)$);
\coordinate (F\num) at ($0.5*(B\num)+0.5*(C\num)-0.86602540*\gapratio*(0,1)$);
}
\foreach \num in {2}
{
\coordinate (B\num) at ($(0,0)+\num*4*\gapratio*(1,0)+\gapratio*(0,0)$);
\coordinate (D\num) at ($(0,0)+\num*4*\gapratio*(1,0)+\gapratio*(1,0)$);
\coordinate (A\num) at ($0.5*(D1)+0.5*(B\num)+0.86602540*\gapratio*(0,1)$);
\coordinate (A30) at ($(A2)+2*\gapratio*(1,0)$);
\coordinate (C\num) at ($0.5*(B\num)+0.5*(D\num)+0.86602540*\gapratio*(0,1)$);
\coordinate (E\num) at ($(0,0)+\num*4*\gapratio*(1,0)+\gapratio*(2,0)$);
\coordinate (F\num) at ($(A2)+3*\gapratio*(1,0)$);
}
\coordinate (G1) at ($0.5*(D1)+0.5*(A2)+0.86602540*\gapratio*(0,1)$);
\coordinate (G2) at ($0.5*(D1)+0.5*(A2)+0.86602540*\gapratio*(0,1)+4*\gapratio*(1,0)$);

\foreach \num in {3,4}
{
\coordinate (A\num) at ($(0,0)+\num*4*\gapratio*(1,0)+\biggergap*(1,0)$);
\coordinate (B\num) at ($(0,0)+\num*4*\gapratio*(1,0)+\gapratio*(1,0)+\biggergap*(1,0)$);
\coordinate (C\num) at ($(0,0)+\num*4*\gapratio*(1,0)+\gapratio*(2,0)+\biggergap*(1,0)$);
\coordinate (D\num) at ($(0,0)+\num*4*\gapratio*(1,0)+\gapratio*(3,0)+\biggergap*(1,0)$);
\coordinate (E\num) at ($0.5*(B\num)+0.5*(C\num)+0.86602540*\gapratio*(0,1)$);
\coordinate (F\num) at ($0.5*(B\num)+0.5*(C\num)-0.86602540*\gapratio*(0,1)$);
}

\coordinate (G3) at ($0.5*(D3)+0.5*(A4)+0.86602540*\gapratio*(0,1)$);

\def\sm{0.2};
\foreach \num in {2}
{
\draw [thick,black!50,fill=black!25,densely dotted] (A\num) -- (B\num);
\draw [thick,black!50,fill=black!25,densely dotted] (C\num) -- (D\num);
\draw [thick,black!50,fill=black!25,densely dotted] (A30) -- (E\num);
\draw [thick,black!50,fill=black!25,densely dotted] (D1) -- (F\num);
}

\foreach \num in {2}
{
\draw [thick,red!50,fill=red!25] (C\num) -- (B\num);
}

\foreach \num/\numplus in {1/2}
{
\draw [thick,blue!50,fill=blue!25] (A\numplus) -- (D\num);
}

\foreach \num/\numplus in {2/3}
{
\draw [thick,orange!50,fill=orange!25] (A30) -- (D\num);
}

\foreach \num/\numplus in {2/3}
{
\draw [thick,green!50,fill=green!25] (E\num) -- (F\num);
}

\def\bit{0.375}
\draw ($0.5*(D1)+0.5*(A2)-(0.7*\bit,-0.1*\bit)$) node {$d$};
\draw ($0.5*(B2)+0.5*(C2)-(0.4*\bit,-0.6*\bit)$) node {$c_1$};
\draw ($0.5*(A30)+0.5*(D2)+(0.4*\bit,-0.6*\bit)$) node {$c_2$};
\draw ($0.5*(A30)+0.5*(D2)+(0.7*\bit,-0.1*\bit)+(\gapratio,0)$) node {$c_3$};


\foreach \lett in {A,B,C,D,E,F}
\foreach \num in {2}
{
\draw [fill] (\lett\num) circle [radius=\vxrad];
}
\foreach \lett/\num in {D/1}
{
\draw [fill] (\lett\num) circle [radius=\vxrad];
}


\draw [fill] (A30) circle [radius=\vxrad];

\end{tikzpicture}
\end{center}

\vspace{-0.4cm}

\caption{Relevant structures for \ref{prop-pseud-abs-prime}, \ref{prop-pseud-add-2} and \ref{prop-pseud-add-3} respectively in the definition of proper-pseudorandomness (see Definition~\ref{defn:pseud}), with labels relevant for the proof of Proposition~\ref{prop:nearcompletepseudorandom}. In the picture for \ref{prop-pseud-abs-prime}, in \ref{prop-pseud-abs-prime} we have $e=u_1u_2$ and $f=v_1v_2$. In the picture for \ref{prop-pseud-add-2}, in \ref{prop-pseud-add-2} we would take, for some $i$, $V_i=\{w_1,w_2,w_3,w_4\}$ and $C_i=\{c_1,c_2,c_3\}$.
In the picture for \ref{prop-pseud-add-3}, the dotted lines have colour $c_0$, and in \ref{prop-pseud-add-3} we would take, for some $i$, $V_i$ to be the set of pictured vertices and $C_i=\{c_1,c_2,c_3\}$.
}\label{fig:pseudpics}
\end{figure}
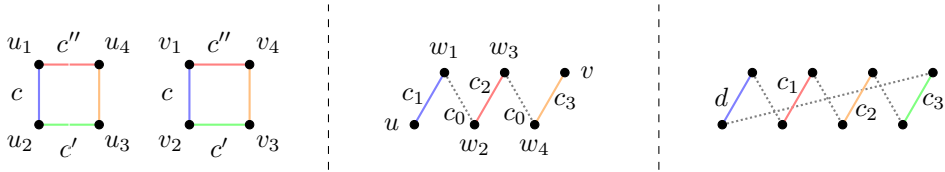


\subsection{Derivation of Theorem~\ref{thm:RBSeven} from the first technical theorem}\label{sec:what}
In order to prove Theorem~\ref{thm:RBSeven} from Theorem~\ref{thm-technical}, we show that, for large $n$, an optimally coloured copy of $K_{n,n}$ is properly pseudorandom, using the following slightly more general result which is also used to deduce Theorems~\ref{thm:generalLS} and~\ref{thm:symbolnum} in Sections~\ref{sec:otherthmsfromgen1} and~\ref{sec:otherthmsfromgen2}. Some of the structures considered in the proof are depicted in Figure~\ref{fig:pseudpics} with appropriate labelling.


\begin{prop}\label{prop:nearcompletepseudorandom}
Let $1/n\llpoly \eta\llpoly \eps\leq 1$. Let $G$ be a properly coloured bipartite graph with vertex classes $A$ and $B$ such that $|A|=|B|=n$, $|C(G)|\geq n$ and $\delta(G)\geq (1-\eta)n$. Suppose that each colour of $G$ appears at least $(1-\eta)n$ times on the edges of $G$.

 Then, $G$ is $(n,1,\eps)$-properly-pseudorandom.
\end{prop}
\begin{proof} Let $p=1$ and $\alpha=p^{12}/10^{100}=1/10^{100}$. Let $C=C(G)$ and note that $n\leq |C|\leq n^2/(1-\eta)n\leq (1+\eps)n$, so that \ref{prop-pseud-basic-1} holds.
Furthermore, $\mathcal{H}(G)$ is a 3-partite simple 3-uniform hypergraph with vertex degrees at least $(1-\eta)n$. Therefore, $\mathcal{H}(G)_{AB}$, $\mathcal{H}(G)_{BC}$ and $\mathcal{H}(G)_{AC}$ are all bipartite graphs with $(1\pm 2\eta)n$ vertices in each class and with minimum degree at least $(1-\eta)n$, so that, then, vertices in the same class have common neighbourhood with size at least $(1-4\eta)n$. Therefore, each of these graphs is $(n,1,4\eta)$-typical, and hence $(n,1,\eps)$-typical, so that \ref{prop-pseud-basic-2new} holds. We will now show that \ref{prop-pseud-abs-prime}--\ref{prop-pseud-add-new-2} hold.

\smallskip

\ref{prop-pseud-abs-prime}: We will prove the slightly stronger result, that, for each $c\in C(G)$ and $e\in E_c(G)$ there are no such edges $f$ as in \ref{prop-pseud-abs-prime}. For this, let $c\in C(G)$ and let $e=u_1u_2,f=v_1v_2\in E_c(G)$ be distinct, so that $\{u_1,u_2\}$ and $\{v_1,v_2\}$ are disjoint as the colouring is proper. We will show that there are at least $n^2/2$ choices for $(c',c'')$ such that there are a pair of vertex-disjoint rainbow 4 cycles containing $e$ and $f$ respectively which each have edges with colour $c'$ and $c''$ next to the edge with colour $c$ (that is, $e$ or $f$).

Note that, for each vertex $v\in V(G)$, $v$ has a neighbouring edge of all but at most $3\eta n$ colours of $G$, as $|C(G)|\leq (1+2\eta)n$ and $\delta(G)\geq (1-\eta)n$, and $v$ has at most $n-\delta(G)\leq \eta n$ non-neighbours. Pick $c'\in C(G)$ which does not appear on $G[V(\{e,f\})]$ so that both $u_2$ and $v_2$ have colour-$c'$ neighbours in $G$, noting that there are at least $n-4-6\eta n\geq n-7\eta n$ choices for $c'$. Let $u_3$ be the colour-$c'$ neighbour of $u_2$ in $G$ and let $v_3$ be the colour-$c'$ neighbour of $v_2$ in $G$. Note that $u_3,v_3\notin \{u_1,u_2,v_1,v_2\}$ by the choice of $c'$, and $u_3\neq v_3$ as both $u_2u_3$ and $v_2v_3$ have colour $c'$ and $G$ is properly coloured.

Pick $c''\in C(G)$ which does not appear on $G[V(\{u_1,u_2,u_3,v_1,v_2,v_3\})]$ for which the following hold.

\vspace{-0.2cm}

\begin{enumerate}[label = \emph{\roman{enumi})}]\setlength\itemsep{-0.3em}
\item  There is no path in $G$ of length 2 with colours $c''$ and $c$ between two vertices in $\{u_1,u_2,u_3,v_1,v_2,v_3\}$.\label{snip1}
\item Both $u_3$ and $v_3$ have a neighbouring edge with colour $c''$ in $G$, and the neighbour along this edge is also a neighbour of $u_2$ and $v_2$, respectively. (Note that vertices in the same class have at least $(1-2\eta)n$ common neighbours, so that all but at most $4\eta n$ colours appear from $u_3$ to a common neighbour of $u_2$ and $u_3$, for example.)\label{snip2}
\end{enumerate}
Note that there are at least $n-\binom{6}{2}(1+2)-2\cdot 4\eta n\geq n-9\eta n$ choices for $c''$.
Let $u_4$ be the colour-$c''$ neighbour of $u_1$ in $G$ and let $v_4$ be the colour $c''$ neighbour of $v_1$ in $G$ (possible by \ref{snip2}). Note that $u_4,v_4\notin \{u_1,u_2,u_3,v_1,v_2,v_3\}$ by \ref{snip1}, that $u_3u_4,v_3v_4\in E(G)$ by \ref{snip2}, and that $u_4\neq v_4$ as both $u_1u_4$ and $v_1v_4$ have colour $c''$ and $G$ is properly coloured.
By \ref{snip1}, we have that neither of the edges $u_3u_4$ or $v_3v_4$ have colour $c$
(or, as the colouring is proper, either $c'$ or $c''$).
Let $S_1$ and $S_2$ be the 4-cycles in $G$ with vertices $u_1u_2u_3u_4$ and $v_1v_2v_3v_4$ in that order, respectively, noting that we have that these are rainbow cycles with $e=u_1u_2\in E(S_1)$ and $f=v_1v_2\in E(S_2)$. Furthermore, the colour sets of the neighbouring edges of $e$ in $S_1$ and the neighbouring edges of $f$ in $S_2$ are both $\{c',c''\}$.

Therefore, as we had at least $(n-7\eta n)(n-9\eta n)\geq n^2/2$ choices for $(c',c'')$, and any pair of such 4-cycles $(S_1,S_2)$ uniquely determines $\{c',c''\}$ from the colours of the edges neighbouring $e$, we have that there are at least $n^2/4\geq \alpha n^2$ pairs $(S_1,S_2)$ such that $S_1$ and $S_2$ are vertex-disjoint rainbow 4-cycles with $e\in E(S_1)$ and $f\in E(S_2)$ and the colour sets of the neighbouring edges of $e$ in $S_1$ and the neighbouring edges of $f$ in $S_2$ are the same.
As this holds for arbitrary $c\in C(G)$ and distinct $e,f\in E_c(G)$, we have that \ref{prop-pseud-abs-prime} holds.

\smallskip

\ref{prop-pseud-add-2}: Let $u\in A$, $v\in B$ and $c_0\in C(G)$. Let $s$ be the largest integer for which there are disjoint sets ${V}_1,\ldots,{V}_{s}\subset V(G)$ with size $4$ and disjoint sets $C_1,\ldots,C_s\subset C(G)\setminus \{c_0\}$ with size 3 such that,
 for each $i\in [s]$, $G[{V}_i]$ contains 2 colour-$c_0$ edges and $G[\{u,v\}\cup {V}_i]$ contains a $C_i$-rainbow matching of order 3. Suppose, for contradiction, that $s<\alpha n$. Let $V=\cup_{i\in [s]}V_i$ and $C=\cup_{i\in [s]}C_i$, so that $|V|\leq 4\alpha n$ and $|C|\leq 3\alpha n$.

Recall that $G$ contains at least $(1-\eta)n$ edges with colour $c_0$, at most $|V|$ of which have an edge in $V$. Thus, there are at least $(1-\eta)n-|V|-3\eta n-|C|\geq n/2$ choices for a colour-$c_0$ edge $w_1w_2$ so that $w_1$ is a neighbour of $u$, $uw_1$ has colour, $c_1$ say, not in $C$, and $w_1,w_2\notin V$. Given such an $w_1w_2$, there are then at least $(1-\eta)n-(|V|+1)-3\eta n-2(|C|+1)\geq n/2$ choices for a colour-$c_0$-edge $w_3w_4$ so that $w_2w_3$ and $w_4v$ are edges of $G$ with colours, $c_2$ and $c_3$ say, which are not in $C\cup \{c_1\}$, and $w_3,w_4\notin V\cup \{w_1,w_2\}$.
Note that, given any $c\in C(G)\setminus \{c_0\}$, there is at most 1 choice for $(w_1,w_2,w_3,w_4,c_1,c_2,c_3)$ for which $c_2=c_3=c$. Therefore, as $(n/2)^2-n>0$, we can choose $(w_1,w_2,w_3,w_4,c_1,c_2,c_3)$ so that $c_2\neq c_3$. Setting $V_{s+1}=\{w_1,w_2,w_3,w_4\}\subset V(G)\setminus V$ and $C_{s+1}=\{c_1,c_2,c_3\}\subset C(G)\setminus C$, observe that $G[{V}_{s+1}]$ contains 2 colour-$c_0$ edges ($w_1w_2$ and $w_3w_4$) and $G[{V}_{s+1}\cup \{u,v\}]$ contains the matching $\{uw_1,w_2w_3,w_4v\}$ which is $C_{s+1}$-rainbow, where $|C_{s+1}|=3$. This contradicts the choice of $s$, and therefore $s\geq \alpha n$. Thus, \ref{prop-pseud-add-2} holds.

\smallskip

For \ref{prop-pseud-add-3}--\ref{prop-pseud-add-new-2}, we first prove the following stronger version of \ref{prop-pseud-add-new-1}:

\begin{itemize}
\item[\ref{prop-pseud-add-new-1}'] For any $c_0\in C(G)$, $0\leq k\leq 120$, and any $\bar{C}\subset C(G)\setminus \{c_0\}$ with $|\bar{C}|=k$, there are vertex-disjoint sets $\bar{V}_1,\ldots,\bar{V}_{\alpha n}\subset V(G)$ such that,
for each $i\in [\alpha n]$, $|\bar{V}_i|=2k+2$ and $G[V_i]$ contains both a matching of $k+1$ colour-$c_0$ edges and a $\bar{C}$-rainbow matching with $k$ edges.
\end{itemize}

Let $c_0\in C(G)$, $0\leq k\leq 120$ and $\bar{C}\subset C(G)\setminus \{c_0\}$ with $|\bar{C}|=k$. Let $s$ be the largest integer for which there are disjoint sets $\bar{V}_1,\ldots,\bar{V}_{s}\subset V(G)$ with size $2k+2$ such that, for each $i\in [\alpha n]$, $G[\bar{V}_i]$ contains both a matching of $k+1$ colour-$c_0$ edges and a $\bar{C}$-rainbow matching with $k$ edges. Suppose, for contradiction, that $s< \alpha n$, and let $V=\cup_{i\in [s]}V_i$, so that $|V_i|\leq (2k+2)\alpha n$.
Let $\hat{V}\subset V(G)\setminus {V}$ be a maximal set for which there is a set $\hat{C}\subset \bar{C}$ with $|\hat{V}|=2|\bar{C'}|$ such that $G[\hat{V}]$ contains an exactly-$\hat{C}$-rainbow matching and $|\hat{C}|$ colour-$c_0$ edges. If $\hat{C}=\bar{C}$, then, using that $|E_{c_0}(G)|\geq (1-\eta)n$, let $e$ be a colour-$c_0$ edge of $G$ with no vertices in $\hat{V}\cup \bar{V}$. Then, setting $\bar{V}_{s+1}=\hat{V}\cup V(e)$, note that $G[\bar{V}_{s+1}]$ contains both a matching of $k+1$ colour-$c_0$ edges and a $\bar{C}$-rainbow matching with $k$ edges, contradicting the maximality of $s$. Therefore, we can assume that $\hat{C}\neq \bar{C}$.

 Let $V'$ be the set of vertices that are reachable from $V\cup \hat{V}$ by a path with length at most $k$ which has colours in $\bar{C}\cup\{c_0\}$, and note that $|V'|\leq \sum_{i=0}^{2k+2}(k+1)^{i}|V\cup \hat{V}|\leq 2(k+1)^{2k+2}\cdot 2\alpha n\leq n/2$. Let $\ell=|\bar{C}\setminus \hat{C}|\geq 1$ and label the colours in $\bar{C}\setminus \hat{C}$ as $c_1,\ldots,c_\ell$.
For each vertex $v\in V(G)\setminus V'$, consider the maximal path $P_v$ starting at $v$ which uses colours $c_0,c_1,c_0,c_2,\ldots,c_0,c_\ell,c_0$ in that order (possibly stopping early), and note that $V(P_v)$ is disjoint from $V\cup \hat{V}$. Note that, for each $i\in [2\ell+2]$, the $i$th vertices of the paths $P_v$, $v\in V(G)\setminus V'$, are distinct where they exist as the colouring is proper, and, as each colour has $(1-\eta)n$ edges, for each $i\in [2\ell+2]$, at most $\eta n$ paths $P_v$, $v\in V(G)\setminus V'$, reach the $i$th vertex with no neighbour of the right colour to extend this to find an $(i+1)$th vertex. Thus, the number of paths $P_v$, $v\in V(G)\setminus V'$, with length $2\ell+1$ is at least $n/2-2\ell\cdot \eta n>0$. That is, using the vertices of some path $P_v$ with length $2\ell+1$, we can find vertices $v_0v_1\ldots v_{2\ell+1}$ in $V(G)\setminus (V\cup \hat{V})$ such that $v_0v_1$ has colour $c_0$ and, for each $1\leq i\leq \ell$, $v_{2i-1}v_{2i}$ has colour $c_i$ and $v_{2i,2i+1}$ has colour $c_0$.

Observe that if $v_i=v_j$ for some $0\leq i<j\leq 2\ell+1$ which minimises $j-i$, then $v_i,\ldots,v_j$ is an even cycle (as $G$ is bipartite) with edges which alternate in colour between $c_0$ and distinct colours in $\bar{C}\setminus \hat{C}=\{c_1,\ldots,c_\ell\}$. Then, $\hat{V}\cup\{v_i,\ldots,v_j\}$ contradicts the maximality of $\hat{V}$. Thus, the vertices $v_0,v_1,\ldots,v_{2\ell}$ are distinct vertices in $V(G)\setminus (V\cup \hat{V})$.
Finally, noting that $V_{s+1}=\hat{V}\cup \{v_0,v_1,\ldots,v_{2\ell+1}\}$ contains $|\hat{C}|+\ell+1=|\bar{C}|+1$ colour-$c_0$ edges, and, as $\bar{C}=\hat{C}\cup \{c_1,\ldots,c_\ell\}$, an exactly-$\bar{C}$-rainbow matching, this contradicts the choice of $s$. Thus, we have $s\geq \alpha n$. Therefore, \ref{prop-pseud-add-new-1}' holds.

We now prove \ref{prop-pseud-add-3}--\ref{prop-pseud-add-new-2}. First, note that \ref{prop-pseud-add-new-1} follows directly from \ref{prop-pseud-add-new-1}'.

\smallskip

\ref{prop-pseud-add-3}: Let $c_0,d\in C(G)$ be distinct. Let $s$ be the largest integer for which there are vertex-disjoint sets $V_1,\ldots,V_{s}\subset V(G)$ with size $8$ and disjoint sets $C_1,\ldots,C_s\subset C(G)\setminus \{c_0,d\}$ of size 3 such that, for each $i\in [s]$, $G[V_i]$ contains both a matching of $4$ colour-$c_0$ edges and an exactly-$(\{d\}\cup C_i)$-rainbow matching. Suppose, for contradiction, that $s< \alpha n/12$, and let $V=\cup_{i\in [s]}V_i$ and $C=\cup_{i\in [s]}C_i$, so that $|V|<2\alpha n/3$ and $|C|\leq \alpha n/4$.
 Using \ref{prop-pseud-add-new-1}' with $k=1$ and $\bar{C}=\{d\}$, find a set $\hat{V}\subset V(G)\setminus V$ of 4 vertices such that $G[\hat{V}]$ contains 2 edges of colour $c_0$ and one edge of colour $d$. Label vertices $u,v\in \hat{V}$ so that $u\in A$, $v\in B$, and the edge of $G[\hat{V}\setminus \{u,v\}]$ has colour $d$.
Using \ref{prop-pseud-add-2}, take a set $\hat{V}'\subset V(G)\setminus (V\cup \hat{V})$ of 4 vertices and a set $\hat{C}'\subset C(G)\setminus (C\cup \{c_0,d\})$ of 3 colours such that $G[\hat{V}']$ contains 2 colour-$c_0$ edges and $G[\hat{V}''\cup \{u,v\}]$  contains an exactly-$\hat{C}'$-rainbow matching. Letting $V_{s+1}=\hat{V}\cup \hat{V}'$ and $C_{s+1}=\hat{C}'$ contradicts the choice of $s$, and therefore $s\geq \alpha n/12$. Thus, \ref{prop-pseud-add-3} holds.

\smallskip

\ref{prop-pseud-add-new-2}:
Set $k=100$ and let $c_0\in C(G)$. Take the largest $r\in \N$ for which there are disjoint sets ${V}_1,\ldots,{V}_{r}$ in $V(G)$ and disjoint sets $C_1,\ldots,C_{r}$ in $C\setminus\{c_0\}$ with $|V_i|=2k$ and $|C_i|=k$ for each $i\in [r]$ such that $G[V_i]$ contains an exactly-$C_i$-rainbow matching and a perfect matching of colour-$c_0$ edges, and fix such sets $V_i$ and $C_i$, $i\in [r]$. Suppose, for contradiction, that $r< 2\alpha^2n$. Let $s\leq 15$ be maximal such that there are disjoint sets $W_1,\ldots,W_s$ in $V(G)\setminus (V_1\cup\ldots\cup V_r)$ and disjoint sets $D_1,\ldots,D_s$ in $C(G)\setminus (C_1\cup \ldots \cup C_r)$ with $|W_i|=8$ and $|D_i|=4$ for each $i\in [s]$, so that $G[W_i]$ contains a matching of 4 colour-$c_0$ edges and a perfectly-$D_i$-rainbow matching. Note that if $s<15$ then we can easily reach a contradiction by \ref{prop-pseud-add-3} as $|\cup_{i\in[r]}V_i|\leq 4k\alpha^2n$. Thus, $s=15$, whence letting $V_{r+1}=\cup_{i\in [s]}W_i$ and $C_{r+1}=\cup_{i\in[s]}D_i$ gives a contradiction to the choice of $r$, so we have $r\geq 2\alpha^2n$.

Now, we show that the sets $V_i$ and $C_i$, $i\in [r]$, have the property for \ref{prop-pseud-add-new-2}. Let $\bar{C}\subset C(G)\setminus \{c_0\}$ with $|\bar{C}|\leq k$. For each $i\in [r]$ with $\bar{C}\cap C_i=\emptyset$ (and thus for at least $\alpha^2n$ values of $i\in [r]$), by \ref{prop-pseud-add-new-1}', there are vertex-disjoint sets $\bar{V}_1,\ldots,\bar{V}_{\alpha n}\subset V(G)$ such that, for each $j\in [\alpha n]$, $|\bar{V}_j|=2k+2|\bar{C}|+2$ and $G[V_j]$ contains both a matching of $k+|\bar{C}|+1$ colour-$c_0$ edges and a $(\bar{C}\cup C_i)$-rainbow matching with $k+|\bar{C}|$ edges. Thus, \ref{prop-pseud-add-new-2} holds.
\end{proof}

Using Proposition~\ref{prop:nearcompletepseudorandom}, we can now deduce Theorem~\ref{thm:RBSeven} from Theorem~\ref{thm-technical}, as follows.

\begin{proof}[Proof of Theorem~\ref{thm:RBSeven} from Theorem~\ref{thm-technical}] Let $n$ be large enough that we can choose $\eta$ and $\eps$ to satisfy $1/n\llpoly \eta \llpoly \eps \llpoly \log^{-1}n$. Let $L$ be a Latin square of order $n$, and let $G$ be the corresponding edge-coloured copy of $L$ as described at the start of Section~\ref{sec:prelim}. Then, as $G$ is properly coloured with $n$ colours, by Proposition~\ref{prop:nearcompletepseudorandom}, $G$ is $(n,1,\eps)$-properly-pseudorandom. Therefore, by Theorem~\ref{thm-technical}, $G$ contains a rainbow matching with at least $n-1$ edges, and thus $L$ contains a transversal with at least $n-1$ cells.
\end{proof}


\subsection{Derivation of Theorem~\ref{thm:generalLS} from the first technical theorem}\label{sec:otherthmsfromgen1}
For Theorem~\ref{thm:generalLS}, we want to show that if $G$ is a properly coloured copy of $K_{n,n}$, and $n$ is large, then $G$ has a rainbow matching with at least $n-1$ edges. If $G$ is far from an optimally coloured copy of $K_{n,n}$, then we use the following result, which is a simplified version of a result by the current author, Pokrovskiy and Sudakov~\cite{montgomery2018decompositions}.

\begin{theorem}\cite[Theorem 1.9]{montgomery2018decompositions}\label{thm-farnoworry}
Let $1/n\llpoly \eps$. Let $G$ be a properly coloured copy of $K_{n,n}$ with at most $(1-\eps)n$ colours having more than $(1-\eps)n$ edges. Then, $G$ contains a rainbow matching with $n$ edges.
\end{theorem}

When $G$ does not satisfy the condition of Theorem~\ref{thm-farnoworry}, then, letting $C_0$ be the set of colours that appear on almost $n$ edges, we wish to find a small rainbow matching with $n-|C_0|$ edges in the rarer colours, where $C_0$ has size close to $n$. Removing this small rainbow matching will then give a $|C_0|$ by $|C_0|$ bipartite graph with $|C_0|$ large colours. We will then remove a further small rainbow matching covering the few vertices which do not have high degree in the subgraph only with edges which have large colours. Finally, then, we can apply Proposition~\ref{prop:nearcompletepseudorandom} to show that it is properly pseudorandom, before applying Theorem~\ref{thm-technical}. To find the initial small rainbow matching, we will use the following result for the rarest colours, finding a rainbow matching using the colours that are far from having $n$ edges. To interpret this result, it may be helpful to note that a rainbow bipartite graph with $d$ vertices in each class connected to all vertices in the other class, and no other edges, has a maximal rainbow matching of at most $2d$ edges but minimum degree $d$, and hence the matching with size $1.8d$ in Proposition~\ref{prop:smallcolours0} is relatively close to optimal. (For Theorem~\ref{thm:generalLS} we will only use a rainbow matching with $d$ edges from Proposition~\ref{prop:smallcolours0}, but use the stronger bound for Theorem~\ref{thm:symbolnum} later.)

\begin{prop}\label{prop:smallcolours0} Let $d,n\in \N$ with $d\leq n/100$. Let $G$ be a properly coloured bipartite graph, with $n$ vertices in each class and $\delta(G)\geq d$, in which each colour appears at most $n/100$ times.

Then, $G$ contains a rainbow matching with at least $1.8d$ edges.
\end{prop}
\begin{proof} Let $A$ and $B$ be the vertex classes of $G$. Let $M$ be a maximal rainbow matching in $G$ and suppose, for contradiction, that $r=|M|<1.8d$. Label vertices so that $M=\{a_ib_i:i\in [r]\}$ with $a_i\in A$ and $b_i\in B$ for each $i\in [r]$.
By the maximality of $M$, the colour of every edge with no vertex in $V(M)$ must be in $C(M)$, so that there are at most $1.8d\cdot n/100$ edges with no vertex in $V(M)$. Therefore, as $\delta(G)\geq d$ and $|V(G)\setminus V(M)|\geq 2(n-1.8d)$, the number of edges of $G$ with a vertex in $M$ is at least
\begin{equation}\label{eqn:noncrossedges}
2(n-1.8d)d-2\cdot 1.8d\cdot \frac{n}{100}\geq 2nd-4d^2-0.04nd\geq 2nd-0.08nd=0.92nd,
\end{equation}
as $d\leq n/100$.

Now, suppose there is some $i\in [r]$ such that $d(a_i),d(b_i)\geq 4d$. Then, as $a_i$ has at least $1.8d+1$ neighbours in $B\setminus V(M)$, we can select some $b_i'\in N(a_i)\setminus V(M)$ such that $a_ib'_i$ does not have colour in $C(M)$. Similarly, as $b_i$ has at least $1.8d+2$ neighbours in $A\setminus V(M)$, we can select some $a_i'\in N(b_i)\setminus V(M)$ such that $a'_ib_i$ does not have colour in $C(M)$ or the same colour as $a_ib_i'$. Then, $M-a_ib_i+a_ib_i'+a_i'b_i$ is a larger rainbow matching in $G$, a contradiction. Thus, we can assume that, for each $i\in [r]$, we have either $d(a_i)<4d$ or $d(b_i)<4d$, so that $d(a_i)+d(b_i)<n+4d$.
Thus, the number of edges with a vertex in $M$ is at most $(n+4d)\cdot 1.8d\leq 1.8nd+8d^2\leq 1.88nd$ as $d\leq n/100$, contradicting \eqref{eqn:noncrossedges}. Thus, we have $r\geq 1.8d$ and the desired matching exists.
\end{proof}

Using Propositions~\ref{prop:nearcompletepseudorandom} and \ref{prop:smallcolours0} and Theorem~\ref{thm-farnoworry}, we can now deduce Theorem~\ref{thm:generalLS} from Theorem~\ref{thm-technical}, as follows. 

\begin{proof}[Proof of Theorem~\ref{thm:generalLS} from Theorem~\ref{thm-technical}]  Let $n$ be large enough that we can choose $\eta$, $\gamma$ and $\eps$ to satisfy $1/n\llpoly \eta \llpoly \gamma \llpoly \eps \llpoly \log^{-1}n$.
Let $L$ be a Latin array of order $n$, and let $G$ be the corresponding edge-coloured copy of $K_{n,n}$ as described at the start of Section~\ref{sec:prelim}. Let $C_0$ be the set of colours that have at least $(1-\eta)n$ edges in $G$. If $|C_0|\leq (1-\eta)n$ then, by Theorem~\ref{thm-farnoworry}, $G$ contains a rainbow matching with $n$ edges, and therefore $L$ contains a full transversal, so we may assume that $|C_0|\geq (1-\eta)n$.

Let $C_1$ be the set of colours not in $C_0$ that have at least $n/100$ edges in $G$, and let $C_2=C(G)\setminus (C_1\cup C_2)$. Let $d=\max\{0,n-|C_0\cup C_1|\}$, and note that $d\leq \eta n$ and, if $G'$ is the subgraph of $G$ of edges with colour in $C_2$, then $\delta(G')\geq d$. Thus, by Proposition~\ref{prop:smallcolours0}, $G'$ contains a $C_2$-rainbow matching, $M_2$ say, with $d$ edges. Let $d'=\max\{0,n-|C_2|-d\}$, so that $d'\leq |C_1|$ and $d'\leq \eta n$. Greedily, using that each colour in $C_1$ has at least $n/100$ edges, let $M_1$ be a $C_1$-rainbow matching with $d'$ edges in $G-V(M_2)$.

Let $V$ be the set of vertices in $G$ adjacent to at most $(1-\gamma)n$ edges with colour in $C_0$. Note that there are at least $(1-\eta)^2n^2$ edges with colour in $C_0$ in $G$, and hence $G$ has at most $2\eta n^2$ edges with colour not in $C_0$, so that $|V|\leq \gamma n$ as $\eta\llpoly\gamma$. Using that $|M_1\cup M_2|\leq 2\eta n$, greedily, then, find a rainbow matching $M_3$ in $G-V(M_1\cup M_2)$ of $|V|$ edges which uses colours not in $V(M_1\cup M_2)$ and such that $V\subset V(M_3)$.

Now, let $G''$ be the subgraph of $G-V(M_1\cup M_2\cup M_3)$ with colours in $C_0\setminus C(M_3)$, and let $m=|G''|/2=n-|M_1\cup M_2\cup M_3|$. Note that $G''$ is a bipartite graph with $m\geq (1-2\eta-\gamma)n\geq (1-2\gamma)n$ vertices in each class, and each colour in $C_0\setminus C(M_3)$ appears at least $(1-\eta)n-2|M_1\cup M_2\cup M_3|\geq (1-3\gamma)n$ times on $G'$. Finally, each $v\in V(G'')\subset V(G)\setminus V$ has $d_{G''}(v)\geq (1-\gamma)n-|M_1\cup M_2\cup M_3|\geq (1-3\gamma )n$.
Thus, as $m=n-|M_1\cup M_2\cup M_3|= n-d-d'-|M_3|\leq |C_0\setminus C(M_3)|$, by Proposition~\ref{prop:nearcompletepseudorandom}, $G''$ is $(m,1,\eps)$-properly-pseudorandom. Therefore, by Theorem~\ref{thm-technical}, $G''$ contains a rainbow matching, $M_0$ say, with at least $m-1$ edges.
Noting that $M_0\cup M_1\cup M_2\cup M_3$ is a rainbow matching with $n-1$ edges, we have that $G$ contains such a matching, so $L$ contains a transversal with at least $n-1$ cells, as required.
\end{proof}


\subsection{Derivation of Theorem~\ref{thm:symbolnum} from the second technical theorem}\label{sec:otherthmsfromgen2}

To prove Theorem~\ref{thm:symbolnum} we work similarly to our proof of Theorem~\ref{thm:generalLS}, but wish to use the extra colours present in the corresponding coloured graph to remove a slightly larger rainbow matching. To aid with this, we first prove the following proposition.

\begin{prop}\label{prop:smallcolours} Let $1/n\llpoly \eta\ll 1$. Suppose $G$ is a copy of $K_{n,n}$ which is properly coloured. Let $C_0\subset C(G)$ be the set of colours which appear on $G$ at least $(1-\eta)n$ times, and suppose that $|C_0|\geq (1-\eta)n$. If $G$ has at least $250n$ colours, then it contains a rainbow matching of at least $n-|C_0|+100$ edges with no colours in $C_0$.
\end{prop}
\begin{proof} Let $C_1$ be the set of colours in $C(G)\setminus C_0$ which appear at least $n/100$ times in $G$, and let $C_2=C(G)\setminus (C_0\cup C_1)$. 
As there are at least $(1-\eta)^2n^2$ edges with colour in $C_0$ in $G$, there are at most $2\eta n^2$ edges with colour in $C_1$ in $G$, so that $|C_1|\leq 200\eta n\leq n/10^3$. As $|C_0|\leq n^2/(1-\eta)n\leq 3n/2$, then, we have $|C_2|\geq 148n$.

Let $d=\max\{n-|C_1|-|C_2|,0\}\leq \eta n$. Let $G'$ be the subgraph of $G$ of edges with colour in $C_2$, and note that $\delta(G')\geq d$. If $d\geq 125$, then, by Proposition~\ref{prop:smallcolours0}, there is a $C_2$-rainbow matching $M\subset G'$ with $|M|\geq 1.8d\geq d+100$. If $d\leq 125$, then let $M$ be a maximal $C_2$-rainbow matching in $G$. Form $H\subset G$ with $M\subset E(H)$ by taking one edge with each colour in $C_2$ in $G$. Note that, for each edge in $M$, at most 1 vertex has a neighbour in $H$ outside of $V(M)$, and that $H$ has no edges with no vertices in $V(M)$. Thus, $|C_2|=e(H)\leq |M|^2+|M|(n-|M|)=n|M|$, and hence $|M|\geq |C_2|/n\geq 225\geq d+100$. Therefore, in each case, we can take a $C_2$-rainbow matching $M$ with $d+100$ edges in $G$.

Using that $|C_1|\leq n/10^3$ and $d+100\leq \eta n+100$, greedily find an exactly $C_1$-rainbow matching $M'$ in $G-V(M)$. Then, we have that $M\cup M'$ is a rainbow matching in $G$ with no colours in $C_0$ and at least $|C_1|+d+100\geq n-|C_0|+100$, as required.
\end{proof}


Using Propositions~\ref{prop:nearcompletepseudorandom} and~\ref{prop:smallcolours} and Theorem~\ref{thm-farnoworry}, we can now deduce  Theorem~\ref{thm:symbolnum} from Theorem~\ref{thm-technical-variant}.

\begin{proof}[Proof of Theorem~\ref{thm:symbolnum} from Theorem~\ref{thm-technical-variant}]
Let $n$ be large enough that we can choose $\eta$, $\gamma$ and $\eps$ to satisfy $1/n\llpoly \eta \llpoly \gamma \llpoly \eps \llpoly \log^{-1}n$.
Let $L$ be a Latin array of order $n$ with at least $150n$ symbols, and let $G$ be the corresponding edge-coloured copy of $K_{n,n}$ as described at the start of Section~\ref{sec:prelim}. Let $C_0$ be the set of colours that have at least $(1-\eta)n$ edges in $G$. If $|C_0|\leq (1-\eta)n$ then, by Theorem~\ref{thm-farnoworry}, $G$ contains a rainbow matching with $n$ edges, and therefore $L$ contains a full transversal, so we may assume that $|C_0|\geq (1-\eta)n$.

Now, by Proposition~\ref{prop:smallcolours}, $G$ contains a rainbow matching, $M$ say, with $\max\{0,n-|C_0|+100\}$ edges and no colours in $C_0$. Note that $|M|\leq \eta n+100\leq 2\eta n$. Let $V$ be the set of vertices in $V(G)\setminus V(M)$ with at least $\gamma n$ neighbouring edges in $G$ with colour not in $C_0$. Note that there are at most $2\eta n^2$ edges in $G$ without colour in $C_0$, so that $|V|\leq \gamma n$. Greedily, then, we can find a matching $M'$ of $|V|$ edges with colours not in $C(M)$ and $V\subset V(M')$.

Let $G'$ be the graph on $V(G)\setminus V(M\cup M')$ with colours in $C_0\setminus C(M')$ and let $m=|G'|/2$. Note that $G'$ is a bipartite graph with $m\geq (1-2\eta-\gamma)n$ vertices in each class. Furthermore, each colour in $C_0$ appears at least $(1-\eta)n-2(2\eta+\gamma)n\geq (1-3\gamma)m$ times on $G'$, and each vertex in $V(G')$ has at least $m-\gamma n-|V|\geq (1-3\gamma)m$ neighbours in $G'$.
Thus, by Proposition~\ref{prop:nearcompletepseudorandom}, $G'$ is $(m,1,\eps)$-properly-pseudorandom. Finally, the number of colours in $G'$ is
\[
|C_0\setminus C(M')|\geq |C_0|-|M'|\geq (n-|M|+100)-|M'|=m+100,
\]
so that, by Theorem~\ref{thm-technical-variant}, $G'$ contains a $C$-rainbow matching with $m$ edges.
In combination with $M\cup M'$, this demonstrates that $G$ contains a rainbow matching with $|M\cup M'|+m=n$ edges, and therefore $L$ contains a full transversal.
\end{proof}


\subsection{Sublinear graph expansion}
\label{subsec:expand}
As discussed in Section~\ref{sec:alg}, when developing our colour classes we will find expander subgraphs in an auxiliary graph and then connect pairs of vertices despite the deletion of vertices or edges. For a general overview of sublinear expansion and its uses, see the recent survey by Letzter~\cite{shohamsurvey}. We use an natural alteration of Koml\'os-Szemer\'edi expansion (see~\cite{K-Sz-1,K-Sz-2}) that requires sets to expand while avoiding edges in an arbitrary subgraph with low degree (see also, for example, a similar expansion in~\cite{haslegrave2020extremal}). The expansion we use is sublinear, but we will find expanders with at most $n$ vertices yet minimum degree at least $\log^Cn$ (for some large constant $C$) so do not need to be as careful with the rate of expansion for small sets as \cite{K-Sz-1,K-Sz-2}, allowing us to use the following simpler definition of expansion (where we will ultimately use this with $\alpha=\Theta(1/\log n)$).

\begin{defn} An $n$-vertex graph $G$ is an \emph{$(\alpha,\Delta)$-expander} if, for every subgraph $K\subset G$ with $\Delta(K)\leq \Delta$, and every $U\subset V(G)$ with $|U|\leq 2n/3$, we have
\[
|N_{G-K}(U)|\geq \alpha|U|.
\]
\end{defn}

Following Koml\'os and Szemer\'edi~\cite{K-Sz-1,K-Sz-2} with only small modifications due to the graph $K$ in this definition, we can show that any graph $G$ contains an expander of this form with comparable average degree. We note that the following theorem and Lemma~\ref{lem-connectinexp} are the only results proved here that we use outside of this section.

\begin{theorem}
Let $n\geq 2$ and $\alpha=1/16\log n$. Then, every $n$-vertex graph $G$ contains a subgraph $H\subset G$ which is an $(\alpha,\alpha\cdot  d(H))$-expander with $d(H)\geq d(G)/2$ and $\delta(H)\geq d(H)/2$.\label{thm-expander}
\end{theorem}

Before proving Theorem~\ref{thm-expander}, we prove the following result for the iterative step in the process.

\begin{prop}\label{prop-forprocess} Let $d,\alpha>0$. Suppose that $G$ is a graph with $d(G)=d$, that $K\subset G$ has $\Delta(K)\leq \alpha d$, and that $U\subset V(G)$ satisfies $|N_{G-K}(U)|<\alpha|U|$.

Then, either $d(G-U)\geq d$ or $d(G[U\cup N_{G-K}(U)])\geq (1-2\alpha)d$.
\end{prop}
\begin{proof}
Suppose otherwise, so that $d(G-U)<d$ and $d(G[U\cup N_{G-K}(U)])<(1-2\alpha)d$. Then, with $n=|G|$, we have
\begin{align*}
2e(G)&\leq 2e(G-U)+e_K(U,V(G)\setminus U)+2e(G[U\cup N_{G-K}(U)])\\
\\
&< d\cdot (n-|U|)+|U|\cdot \Delta(K)+(1-2\alpha)d\cdot |U\cup N_{G-K}(U)|\\
&\leq d\cdot (n-|U|)+|U|\cdot \alpha d+ (1-2\alpha)d\cdot (|U|+\alpha |U|)\\
&=dn-2\alpha^2d|U|\leq dn,
\end{align*}
a contradiction to $d(G)=d$.
\end{proof}

\begin{proof}[Proof of Theorem~\ref{thm-expander}]
Set $G_0=G$ and carry out the following process indexed by $\ell$, starting with $\ell=0$.
\begin{itemize}
\item If $\delta(G_\ell)\geq d(G_\ell)/2$ and $G_\ell$ is an $(\alpha,\alpha\cdot d(G_\ell))$-expander, then stop the process. Otherwise, let $d_\ell=d(G_\ell)$ and $n_\ell=|G_\ell|$, and do the following.
\item If $\delta(G_\ell)<d_\ell/2$, then remove a vertex with minimum degree from $G_\ell$ to get $G_{\ell+1}$, noting that, as is well known, $d(G_{\ell+1})\geq d_\ell$.
\item If $\delta(G_\ell)\geq d_\ell/2$, then, using that $G_\ell$ is not an $(\alpha,\alpha d_\ell)$-expander, find a set $U_\ell\subset V(G_\ell)$ with $|U_\ell|\leq  2n/3$ and a graph $K_\ell\subset G_\ell$ with $\Delta(K_\ell)\leq \alpha d_\ell$ such that $|N_{G_\ell-K_\ell}(U_\ell)|< \alpha |U_\ell|$. By Proposition~\ref{prop-forprocess}, we have either $d(G_\ell-U_\ell)\geq d_\ell$ or $d(G[U_\ell\cup N_{G_\ell-K_\ell}(U_\ell)])\geq (1-2\alpha)d_\ell$. In the former case let $G_{\ell+1}=G_\ell-U_\ell$, and in the latter case let $G_{\ell+1}=G[U_\ell\cup N_{G_\ell-K_\ell}(U_\ell)]$. Note that in the latter case we have
\[
|G_{\ell+1}|\leq (1+\alpha)|U_\ell|\leq (1+\alpha)\cdot 2n_\ell/3\leq 3n_\ell/4<n_\ell.
\]
\end{itemize}

As we always have $0<|G_{\ell+1}|<|G_{\ell}|$, this process eventually terminates (potentially with the graph consisting of a single vertex which trivially satisfies the conditions to stop the process). Let $H$ be the graph at the end of the process. We have, then, that $H$ is an $(\alpha,\alpha\cdot d(H))$-expander with $\delta(H)\geq d(H)/2$, and we need only show that $d(H)\geq d(G)/2$.
At each stage $\ell$ in the process apart from the final stage, either
\begin{enumerate}[label = \roman{enumi})]\setlength\itemsep{-0.3em}
\item $d(G_{\ell+1})\geq d(G_{\ell})$, or
\item $d(G_{\ell+1})\geq (1-2\alpha)d(G_\ell)$ and $|G_{\ell+1}|\leq 3|G_{\ell}|/4$.
\end{enumerate}
As $(3/4)^{4\log n}|G_0|<1$, the latter case occurs for at most $4\log n$ steps. Thus,
\[
d(H)\geq (1-2\alpha)^{4\log n}d(G)\geq (1-2\alpha\cdot 4\log n)d(G)= d(G)/2,
\]
 as required.
\end{proof}

As is typical, we will use expansion to connect vertices or sets by showing that iteratively expanding a set in an expander will result in a large set, as follows.

\begin{lemma}\label{lem-expexpands} Let $n,m\in \N$, $\Delta>0$ and $1/16\log n\leq \alpha \leq 1$. Let $H$ be an $(\alpha,\Delta)$-expander with at most $n$ vertices. Let $K\subset H$ and $V\subset V(H)$ satisfy $\Delta(K)\leq \Delta$ and $|V|\leq \alpha m/2$. Let $U\subset V(H)\setminus V$ satisfy $|U|\geq m$ and let $\ell=64\log^{2}n$.

Then, $|B^{\ell}_{H-K-V}(U)|>|H|/2$.
\end{lemma}
\begin{proof} Suppose, for contradiction, that $|B^{\ell}_{H-K-V}(U)|\leq |H|/2$. For each $i\in [\ell]_0$, let $U_i=B^i_{H-K-V}(U)$, so that $|U_i|\leq |U_\ell|<|H|/2$ and $|U_i|\geq |U_0|=|U|\geq m$. As $H$ is an $(\alpha,\Delta)$-expander, we have
\[
|N_{H-K-V}(U_{i})|= |N_{H-K}(U_i)\setminus V|\geq \alpha |U_i|-|V|\geq \alpha |U_i|/2.
\]
Therefore, $|U_{i+1}|\geq (1+\alpha/2)|U_i|$ for each $i\in [\ell-1]_0$, and thus
\[
|U_\ell|\geq (1+\alpha/2)^{\ell}|U_0|\geq (1+\alpha\cdot 32\log n/2)^{\ell/32\log n}\geq 2^{\ell/32\log n}\geq n\geq |H|,
\]
a contradiction.
\end{proof}
For the proof of our next lemma, it is convenient to record the following very simple result.

\begin{prop}\label{prop:verysimple}
For any $c> 0$, $\lceil \lceil c\rceil /2 \rceil\leq \lceil 2c/3\rceil$.
\end{prop}
\begin{proof}
If $c\geq 3$, then $2c/3\geq (c+1)/2\geq \lceil c\rceil/2$, so taking ceilings gives the result. If $3/2< c\leq 3$, then $\lceil 2c/3\rceil=2$ and $\lceil \lceil c\rceil /2\rceil\leq \lceil 3/2\rceil=2$. If $0< c\leq 3/2$, then $\lceil \lceil c\rceil /2 \rceil=1=\lceil 2c/3\rceil$.
\end{proof}

Our next consequence of expansion is that if the union of some disjoint vertex sets expands despite some deletion of vertices and edges, then one of these sets expands well despite the same deletion of vertices and edges, as follows.

\begin{lemma}\label{lem-expexpands-thingummy} Let $n,m\in \N$ with $n\geq e^{50}$, $\Delta>0$ and $1/16\log n\leq \alpha\leq 1$. Let $H$ be an $(\alpha,\Delta)$-expander with at most $n$ vertices. Let $K\subset H$ and $V\subset V(H)$ satisfy $\Delta(K)\leq \Delta$ and $|V|\leq \alpha m/8$. Let $r\in \N$ and let $U_1,\ldots,U_r$ be disjoint subsets of $V(H)\setminus V$ with $|\cup_{i\in [r]}U_i|\geq m$. Let $\ell=195\log^{3}n$.

Then, there is some $i\in [r]$ for which $|B^{\ell}_{H-K-V}(U_i)|>|H|/2$.
\end{lemma}
\begin{proof}
 Note that, from the conditions, we have that $|H|\geq m$.
Let $\ell_0=64\log^2 n$ and let $0\leq j\leq 1+3\log n$ be the largest such $j$ for which there is a set $I_j\subset [r]$ with $|I_j|\leq \lceil r\cdot (2/3)^j\rceil$ and $|B_{H-K-V}^{\ell_0j}(\cup_{i\in I_j}U_i)|> m/4$. Note that we could take $j=0$ and $I_j=[r]$ for $j=0$, and therefore such a $j$ and an accompanying $I_j$ must exist. If $|I_j|\geq 2$, then note that, as $(2/3)^3< e^{-1}$, we have $j<3\log n$. Partition $I_j=I_{j+1}\cup I'_{j+1}$ as equally as possible.
Note that, by Proposition~\ref{prop:verysimple}, $|I_{j+1}|,|I'_{j+1}|\leq \lceil |I_j|/2\rceil\leq \lceil r\cdot (2/3)^{j+1}\rceil$. Furthermore, by Lemma~\ref{lem-expexpands}, we have
\[
|B_{H-K-V}^{(j+1)\ell_0}(\cup_{i\in I_{j+1}}U_i)|+|B_{H-K-V}^{(j+1)\ell_0}(\cup_{i\in I_{j+1}'}U_i))|\geq |B_{H-K-V}^{\ell_0}\big(B_{H-K-V}^{j\ell_0}(\cup_{i\in I_{j}}U_i)\big)|>|H|/2\geq m/2.
\]
Therefore, either $|B_{H-K-V}^{(j+1)\ell_0}(\cup_{i\in I_{j+1}}U_i))|>m/4$ or $|B_{H-K-V}^{(j+1)\ell_0}(\cup_{i\in I_{j+1}'}U_i))|>m/4$, a contradiction to the choice of $j$.
Thus, we must have that $|I_j|=1$.

Let $i\in [r]$ be such that $I_j=\{i\}$. Finally, note that, as $|B_{H-K-V}^{j\ell_0}(U_i)|>m/4$, by Lemma~\ref{lem-expexpands}, we have $|B^{\ell_0(j+1)}_{H-K-V}(U_i)|>|H|/2$. As $n\geq e^{50}$, we have $\ell_0(j+1)\leq \ell_0(2+3\log n)\leq \ell$, and so this completes the proof.
\end{proof}

We now prove the main result of expansion that we use. It shows that, even if a set of edges (those in $E(K)$) and a set of vertices (those in $V$) are removed, we can find a large set of vertices from which we can connect any pair of vertices with multiple internally vertex-disjoint paths.

\begin{lemma}\label{lem-connectinexp} There exists some $n_0\in \N$ such that the following holds for each $n\geq n_0$. Let $\Delta\in \N$, $1/16\log n\leq \alpha \leq 1$ and $d\geq 4\Delta$.  Let $H$ be an $(\alpha,\Delta)$-expander with at most $n$ vertices which satisfies $\delta(H)\geq d$ and suppose that
\begin{equation}\label{eqn:rbound}
r\leq \frac{\alpha d}{10^6\log^3n}.
\end{equation}

Let $V\subset V(H)$, and let $K\subset H$ satisfy $\Delta(K)\leq \Delta$. Then, there is a set $B\subset V(H)$ with $V\subset B$ and
\begin{equation}\label{eqn:Bbound}
|B|\leq 10^4|V| \cdot \frac{\Delta(H)}{\alpha d}
\end{equation}
such that, given any distinct $x,y$ in $V(H)\setminus B$, there are $r$ $x,y$-paths in $H-K-V$ with length at most $400\log^3n$ which are internally vertex disjoint.
\end{lemma}
\begin{proof}
Let $\ell=200\log^3n$ and pick $n_0\geq e^{50}$. Let $B'$ be the set of vertices $v\in V(H)\setminus V$ for which there is a set $U_v\subset V(H)\setminus \{v\}$ with $|U_v|\leq 2r\ell$ such that $|N^\ell_{H-K-V-U_v}(v)|<|H|/2$, and let $B=B'\cup V$.
First, note that if \eqref{eqn:Bbound} holds, then $B$ has the desired property. Indeed, given any distinct $x,y$ in $V(H)\setminus B$, let $P_1,\ldots,P_{r'}$ be a largest collection of $x,y$-paths in $H-K-V$ with length at most $2\ell=400\log^3n$ which are internally vertex disjoint and let $U_{xy}=V(\cup_{i\in [r']}P_i)\setminus \{x,y\}$. If $r'<r$, then $|U_{xy}|<2r\ell$, so that, as $x,y\notin B$, $|N^\ell_{H-K-V-U_{xy}}(x)|,|N^\ell_{H-K-V-U_{xy}}(y)|>|H|/2$, and therefore $H-K-V-U_{xy}$ contains an $x,y$-path with length at most $2\ell$, contradicting the choice of $r'$. Thus, $r'\geq r$. As $x$ and $y$ were arbitrary distinct vertices in $V(H)\setminus B$, $B$ has the desired property.

Therefore, suppose, for contradiction, that \eqref{eqn:Bbound} does not hold, so that, as $\Delta(H)\geq \delta(H)\geq d$ and $\alpha\leq 1$
\begin{equation}\label{eqn:Bprime}
|B'|=|B|-|V|\geq 10^4|V|\cdot \frac{\Delta(H)}{\alpha d}-|V|\geq (10^4-1)|V|\cdot \frac{\Delta(H)}{\alpha d}.
\end{equation}
 For the largest $s$ possible, take distinct vertices $v_1,\ldots,v_s\in B'$ and disjoint sets $A_1,\ldots,A_s\subset V(H)\setminus V$ such that $A_i\subset N_{H-K}(v_i)$ for each $i\in [s]$ and $|A_i|=\lfloor d/2\rfloor\geq d/4$. Let $A=\cup_{i\in [s]}A_i$ and, suppose, for contradiction, that $|A|< \frac{32|V|}{\alpha}$, and hence $s=|A|/\lfloor d/2\rfloor <\frac{128|V|}{\alpha d}<|B'|/2$ by \eqref{eqn:Bprime}.
Then, as $\delta(H)\geq d$, every vertex $v\in B'\setminus \{v_1,\ldots,v_s\}$ has at least $d-\lfloor d/2\rfloor \geq d/2$ neighbours in $H$ in $V\cup A$, so that we have
\begin{equation}\label{eqn:eBVA}
e_H(B'\setminus \{v_1,\ldots,v_s\},V\cup A)\geq (|B'|-s)\cdot \frac{d}{2}\geq \frac{d|B'|}{4}\overset{\eqref{eqn:Bprime}}{\geq} \frac{(10^4-1)|V|\cdot \Delta(H)}{4\alpha}.
\end{equation}
On the other hand, as $|A|< \frac{32|V|}{\alpha}$,
\begin{equation*}\label{eqn:eBVA2}
e_H(B'\setminus \{v_1,\ldots,v_s\},V\cup A)\leq |V\cup A|\cdot \Delta(H)< 2\cdot \frac{32|V|}{\alpha}\cdot \Delta(H),
\end{equation*}
which contradicts \eqref{eqn:eBVA} as $(10^4-1)/4\geq 2\cdot 32$.

Thus, $|A|\geq \frac{32|V|}{\alpha}$, and hence $|V|\leq \frac{\alpha |A|}{32}$. Let $U=\cup_{i\in [s]}U_{v_i}$, recalling that, for each $i\in [s]$, $|U_{v_i}|\leq 2r\ell$ and $|N^\ell_{H-K-V-U_{v_i}}(v_i)|<|H|/2$. Note that, as $s=|A|/\lfloor d/2\rfloor$,
\[
|U\cup V|\leq |U|+\frac{\alpha |A|}{32}\leq 2r\ell\cdot s+\frac{\alpha |A|}{32}\leq \frac{4|A|}{d}\cdot 2r\ell+\frac{\alpha |A|}{32} \overset{\eqref{eqn:rbound}}{\leq} \frac{\alpha |A|}{32}.
\]
Let $A'=A\setminus U$ and note that $|U\cup V|\leq \alpha |A|/16\leq \alpha |A'|/8$ and $A'=\cup_{i\in [s]}(A_i\setminus U)$. Therefore, by Lemma~\ref{lem-expexpands-thingummy} with $m=|A'|$, we have that there is some $j\in [s]$ with  $|B^{\ell-1}_{H-K-U-V}(A_j\setminus U)|> |H|/2$. As $v_j\notin U_j\cup V$, we thus have $|B^{\ell}_{H-K-U_{v_j}-V}(v_j)|> |H|/2$, a contradiction to the choice of $U_{v_j}$.
\end{proof}


\subsection{Concentration inequalities}\label{sec:conc}

We will use some concentration bounds for our analysis, starting with the following standard version of Chernoff's bound.

\begin{lemma}\label{Lemma_Chernoff}
Let $X$ be a binomial random variable with parameters $(n,p)$.
Then, for each $\eps\in (0,1)$, we have
$$\mathbb P\big(|X-pn|> \eps pn\big)\leq
2e^{-\frac{\eps^2pn}{3}}.$$
\end{lemma}

We will use also the following concentration result due to McDiarmid (see, for example,~\cite{janson2011random}, for an exposition of Lipschitz random variables and, in particular,~\cite[Remark 2.28]{janson2011random}).

\begin{lemma}\label{lem:mcd}
Suppose that a random variable $X:\prod_{i=1}^n\Omega_i\to \R$ is $k$-Lipschitz. Then,
\[
\P(|X-\E X|>t)\leq 2\exp(-t^2/k^2n).
\]
\end{lemma}

We will also use the following standard form of Azuma's inequality for a sub-martingale (see, for example~\cite{janson2011random}, for an exposition of martingales and Azuma's inequality, and, in particular, the note after~\cite[Remark 2.26]{janson2011random}).

\begin{theorem}[Azuma's inequality]\label{azuma} If $X_0,X_1,\ldots,X_n$ is a sub-martingale, and $|X_i-X_{i-1}|\leq c_i$ for each $1\leq i\leq n$, then, for each $t>0$, $\P(X_n-X_0\leq -t)\leq \exp\left(\frac{-t^2}{2\sum_{i=1}^nc_i^2}\right)$.
\end{theorem}

\subsection{Numbers of 4-cycles}\label{sec:count4}
We will use a lower bound on the number of 4-cycles in an $n$-vertex graph that depends on its number of edges. This follows from the well-known proof of Sidorenko's conjecture in the (simple) case for the 4-cycle using the Cauchy-Schwarz inequality. For completeness, we include this short proof, as follows.

\begin{prop}\label{numberof4cycles} Let $1/n\llpoly \xi$. Then, any $n$-vertex graph with at least $\xi n^2$ edges contains at least $10\xi^4n^4$ labelled copies of $C_4$.
\end{prop}
\begin{proof}
Let $r$ be the number of labelled copies of $C_4$ in $G$, and let $s$ be the number of 4-tuples $(w,x,y,z)$ of vertices in $G$ such that $wx,xy,yz,zw\in E(G)$. Note that the number of 4-tuples in which some vertex appears at least twice is certainly at most $\binom{4}{2}\cdot n\cdot n^2=6n^3$, and hence $r\geq s-6n^3$. Now,
\begin{align*}
s&=\sum_{w,y\in V(G)}|N(w)\cap N(y)|^2\geq \frac{1}{n^2}\left(\sum_{w,y\in V(G)}|N(w)\cap N(y)|\right)^2
=\frac{1}{n^2}\left(\sum_{x\in V(G)}|N(x)|^2\right)^2
\\
&\geq \frac{1}{n^4}\left(\sum_{x\in V(G)}|N(x)|\right)^4= \frac{1}{n^4}\left(2e(G)\right)^4\geq 16\xi^4n^4,
\end{align*}
using the Cauchy-Schwarz inequality twice.
Thus, as $1/n\llpoly \xi$, we have $r\geq 16\xi^4n^4-6n^3\geq 10\xi^4n^4$.
\end{proof}



\section{Almost-full transversals}\label{sec:semirandom}
In this section we prove Theorem~\ref{thm:RSBcoveringstep} using a result proved by the author, Pokrovskiy and Sudakov~\cite{montgomery2018decompositions} using the semi-random method, and adapting work by the author, Pokrovskiy and Sudakov~\cite{montgomery2021proof}. We do not use the semi-random method directly, and deduce the result we use  (Theorem~\ref{thm:RSBcoveringstep}) mainly from known results. For discussion of the semi-random method we refer the reader to~\cite{montgomery2018decompositions}, or to the recent survey by Kang, Kelly, K\"uhn, Osthus and Methuku~\cite{kang2021graph}.

In this section, we will work with simple 3-uniform 3-partite hypergraphs, and prove the following result, which is, effectively, Theorem~\ref{thm:RSBcoveringstep} restated using simple 3-uniform 3-partite hypergraphs and using only the condition from proper pseudorandomness required for the matching to be found (see Section~\ref{sec:typical} for the relevant definitions concerning typicality).

\begin{theorem}\label{thm:RSBcoveringstephyper}
Let $1/n\llpoly \eps \llpoly \eta \llpoly  p,q\leq 1$. Let $2q/3\leq q_A,q_B,q_C\leq q$. Let $\mathcal{H}$ be simple 3-partite 3-uniform hypergraph which is $(n,p,\eps)$-typical with vertex classes $A$, $B$ and $C$. Independently, let $A'$ be a $q_A$-random subset of $A$, let $B'$ be a $q_B$-random subset of $B$ and let $C'$ be a $q_C$-random subset of $C$. Then, with high probability, the following holds.

Given any sets $\bar{A}\subset A$, $\bar B\subset B$, $\bar{C}\subset C$ with size $qn$ such that $A'\cup B'\cup C'\subset \bar{A}\cup \bar{B}\cup \bar{C}$, there is a matching in $\mathcal{H}(\bar{A},\bar{B},\bar{C})$ with at least $qn-\eta n$ edges.
\end{theorem}

Note that Theorem~\ref{thm:RSBcoveringstep} follows immediately by considering the hypergraph $\mathcal{H}(G)$ (see Definition~\ref{defn:HG}) and using \ref{prop-pseud-basic-2new}, as it follows from $1/n\ll p,q$ and $\eta\llpoly \log^{-1}n$ that $\eta\llpoly p,q$.
In Section~\ref{sec:quoted}, we will state three results we use from~\cite{KPSY,montgomery2018decompositions}, developing them through corollaries for ease of application, before proving Theorem~\ref{thm:RSBcoveringstephyper} in Section~\ref{sec:newsec}.

\subsection{Results of typicality}\label{sec:quoted}

The first result we need is the following result (stated in a slightly simplified form) of the author with Pokrovskiy and Sudakov~\cite{montgomery2018decompositions} that says the subgraph of a typical bipartite graph chosen by including edges of a random set of colours is itself likely to be typical. (The definition of typicality in~\cite{montgomery2018decompositions} is for balanced bipartite graphs only, but coincides with Definition~\ref{defn:typical} when $|A|=|B|$, though we record the parameters in a different order.)

\begin{lemma}[\cite{montgomery2018decompositions},~Lemma~5.3a)]\label{lem:randcolstypical}
Let $1/n\llpoly \eps,p,q\leq 1$. Let $G$ be a properly coloured bipartite graph with vertex classes $A$ and $B$ which is $(n,p,\eps)$-typical with $|A|=|B|$.
Let $C$ be a random subset of $C(G)$ formed by including each element independently at random with probability $q$. Let $H$ be the subgraph of $G$ whose edges are exactly those in $E(G)$ with colour in $C$.
Then, with high probability, $H$ is $(n,pq,2\eps)$-typical.
\end{lemma}

We will use this through the following corollary, where the bipartite graph does not need to be balanced.

\begin{corollary}\label{cor:randcolstypical}
Let $1/n\llpoly \eps\llpoly \alpha \llpoly p,q\leq 1$. Let $G$ be a properly coloured bipartite graph with vertex classes $A$ and $B$ which is $(n,p,\eps)$-typical.
Let $C$ be a random subset of $C(G)$ formed by including each element independently at random with probability $q$. Let $H$ be the subgraph of $G$ whose edges are exactly those in $E(G)$ with colour in $C$.
Then, with high probability, $H$ is $(n,pq,\alpha)$-typical.
\end{corollary}
\begin{proof} 
Let $G$ be $G'$ with at most $2\eps n$ vertices removed from $A$ or $B$ so that it is a balanced graph, and note that $G'$ is $(n,p,\alpha/4)$-typical as $\eps\llpoly \alpha \llpoly p$. Let $H'=H[V(G')]$, so that by Lemma~\ref{lem:randcolstypical}, with high probability $H'$ is $(n,p,\alpha/2)$-typical.
 As $H'$ is $H$ with at most $2\eps n$ vertices removed,
it then follows that $H$ is $(n,pq,\alpha)$-typical. Thus, with high probability, $H$ is $(n,pq,\alpha)$-typical.
\end{proof}

We will also use the following implication of typicality due to Keevash, Pokrovskiy, Sudakov and Yepremyan~\cite{KPSY}. We note that our definition of a typical graph is recorded differently to~\cite{KPSY}, where their $(\eps,p,n)$-typical bipartite graph corresponds to our $(n,p,n^{-\eps})$-typical bipartite graph when the class sizes are equal.

\begin{lemma}[\cite{KPSY},~Lemma~2.7]\label{lem:typicalpseudo} Let $n\in\N$, $\eps,p,\gamma\in (0,1]$ with $8\eps\leq \gamma$.
Let $G$ be a bipartite graph with vertex classes $A$ and $B$ which is $(n,p,\eps)$-typical with $|A|=|B|=n$.
Then, for every $X\subset A$ and $Y\subset B$ with $|Y|\geq\gamma^{-1}p^{-2}$,
\[
\big|e(X,Y)-p|X||Y|\big|\leq 2|X|^{1/2}|Y|\gamma^{1/2}n^{1/2}p.
\]
\end{lemma}

In particular, we will apply Lemma~\ref{lem:typicalpseudo} through the following corollary.

\begin{corollary}\label{cor:typicalpseudo}
Let $1/n\llpoly\eps \llpoly \eta, p,q\leq 1$. Let $G$ be a bipartite graph with vertex classes $A$ and $B$ which is $(n,p,\eps)$-typical. 
Then, for every $X\subset A$ and $Y\subset B$ with $|B|\geq qn$,
\[
\big|e(X,Y)-p|X||Y|\big|\leq \eta n^2.
\]
\end{corollary}
\begin{proof} Let $\gamma=8\eps$, and note that $qn\geq \gamma^{-1}p^{-2}$ as $1/n\llpoly \eps,p,q$. Then, for every $X\subset A$ and $Y\subset B$ with $|Y|\geq qn$, by Lemma~\ref{lem:typicalpseudo}, we have, noting that $|X|,|Y|\leq 2n$
\[
\big|e(X,Y)-p|X||Y|\big|\leq 2|X|^{1/2}|Y|\gamma^{1/2}n^{1/2}p\leq 16n^2\gamma^{1/2}p\leq \eta n^2,
\]
where we have used that $\eps\llpoly \eta,p$. \end{proof}

Finally, we will need the following result of the author, Pokrovskiy and Sudakov~\cite{montgomery2018decompositions}, proved using the semi-random method.

\begin{lemma}[\cite{montgomery2018decompositions}, Lemma~4.6]\label{Lemma_MPS_nearly_perfect_matching}
Let $1/n\llpoly \eps \llpoly \gamma \llpoly  p\leq 1$.
Let $G$ be a properly coloured balanced bipartite graph with $|G|=(1\pm \eps)2n$ , $d_G(v)=(1\pm \eps)p n$ for all $v\in V(G)$, and
suppose that, for each $c\in C(G)$, $|E_c(G)|\leq (1+\eps)pn$.
Then, $G$ contains a rainbow matching with $(1-\gamma)n$ edges.
\end{lemma}

We will apply Lemma~\ref{Lemma_MPS_nearly_perfect_matching} through the following corollary.

\begin{corollary}\label{cor:nibble} Let $1/n\llpoly \eps \llpoly \gamma \llpoly  p\leq 1$. Let $\mathcal{H}$ be a simple 3-partite 3-uniform hypergraph, whose vertex classes each have size $(1\pm \eps)n$, such that $d(v)=(1\pm \eps)pn$ for each $v\in V(\mathcal{H})$. Then, $\mathcal{H}$ contains a matching with at least $(1-\gamma)n$ edges.
\end{corollary}
\begin{proof} 
Let $\eps'$ satisfy $\eps \llpoly \eps' \llpoly \gamma$.
Let the vertex classes of $\mathcal{H}$ be $A$, $B$ and $C$.
Let $G$ be a coloured bipartite graph with vertex classes $A$ and $B$, with an edge $ab$ with colour $c$ for each $a\in A$, $b\in B$ and $c\in C$ with $abc\in E(\mathcal{H})$. Note that $G$ is properly coloured because $\mathcal{H}$ is simple, and that $d_G(v)=(1\pm\eps)pn$ for each $v\in V(G)$
 and $|E_c(G)|=d_{\mathcal{H}}(c)\leq (1+\eps)pn$ for each $c\in C(G)$.
Let $G'$ be $G$ with at most $2\eps n$ vertices removed from $A$ or $B$ so that $G'$ is a balanced bipartite graph, and note that $d_G(v)=(1\pm\eps')pn$. Then, by Lemma~\ref{Lemma_MPS_nearly_perfect_matching}, $G'$, and hence $G$, contains a rainbow matching with $(1-\gamma)n$ edges. This corresponds to a matching in $\mathcal{H}$, so that $\mathcal{H}$ contains a matching with at least $(1-\gamma)n$ edges, as required.
\end{proof}


\subsection{Proof of Theorem~\ref{thm:RSBcoveringstephyper}}\label{sec:newsec}

Theorem~\ref{thm:RSBcoveringstephyper} concerns matchings of edges between vertex sets $\bar{A}\subset A$, $\bar{B}\subset B$ and $\bar{C}\subset C$, where most of the vertices in each set lie in a random set. We will first prove a lemma where two of the sets $\bar{A}$, $\bar{B}$ and $\bar{C}$ are random and one of them is arbitrary, i.e., Lemma~\ref{lem:matchinhyper}. In the proof of this lemma, and the subsequent proof of Theorem~\ref{thm:RSBcoveringstep}, we follow work by the author, Pokrovskiy and Sudakov~\cite{montgomery2021proof} (in a slightly different setting).

\begin{lemma}\label{lem:matchinhyper}
Let $1/n\llpoly \eps \llpoly \eta \llpoly  p,q\leq 1$. Let $\mathcal{H}$ be a 3-partite linear 3-uniform hypergraph with vertex classes $A$, $B$ and $C$ which is \emph{$(n,p,\eps)$-typical}. Let $A'\subset A$ and $B'\subset B$ be subsets chosen by including each element of $A$ and $B$, respectively, independently at random with probability $q$.

Then, with high probability, given any set $C'\subset C$ with size $qn$, there is a matching in $\mathcal{H}[A',B',C']$ with size at least $(q-\eta)n$.
\end{lemma}
\begin{proof} Let $\alpha$ and $\gamma$ satisfy $\eps\llpoly\alpha\llpoly \gamma\llpoly \eta$.
Let $H_1$ be the bipartite graph with vertex sets $A$ and $C$ and edges $ac$ with colour $b$ with $a\in A$, $b\in B$ and $c\in C$ if $abc\in E(H_1)$. As $\mathcal{H}$ is $(n,p,\eps)$-typical, we have that $H_1$ is $(n,p,\eps)$-typical.
Let $H'_1$ be the subgraph of $H_1$ of edges with colour in $B'$, so that, by Corollary~\ref{cor:randcolstypical}, with high probability, we have that $H_1'$ is $(n,pq,\alpha)$-typical. Let $H_2'$ be the bipartite graph with vertex sets $B$ and $C$ and edges $bc$ with $b\in B$ and $c\in C$ if there is some $a\in A'$ with $abc\in E(\mathcal{H})$. Similarly, with high probability, we have that $H_2'$ is $(n,pq,\alpha)$-typical. Therefore, by Corollary~\ref{cor:typicalpseudo}, with high probability we can assume the following.

\stepcounter{propcounter}
\begin{enumerate}[label = {{\textbf{\Alph{propcounter}\arabic{enumi}}}}]
\item For every $X\subset A$ and $Y\subset C$ with $|Y|\geq qn/2$, we have $|e_{H_1'}(X,Y)-pq|X||Y||\leq \gamma n^2$.\label{eqn:H1edges}
\item For every $X\subset B$ and $Y\subset C$ with $|Y|\geq qn/2$, we have $|e_{H_2'}(X,Y)-pq|X||Y||\leq \gamma n^2$.\label{eqn:H2edges}
\end{enumerate}

Note that, as $\mathcal{H}$ is simple, for each $v\in C$, $d_\mathcal{H}(v,A'\cup B')$ is binomially distributed with mean $q^2\cdot d_\mathcal{H}(v)=(1\pm \eps)q^2pn$, as $\mathcal{H}$ is $(n,p,\eps)$-typical. Thus, using a union bound and Lemma~\ref{Lemma_Chernoff}, we have the following properties with high probability.

\begin{enumerate}[label = {{\textbf{\Alph{propcounter}\arabic{enumi}}}}]\addtocounter{enumi}{2}
\item For each $v\in C$, we have $d_\mathcal{H}(v,A'\cup B')=(1\pm 2\eps)pq^2n$.\label{eqn:Cdegree}
\item $|A'|=(1\pm \eps)qn$ and  $|B'|=(1\pm 2\eps)qn$.\label{eqn:ABprime}
\end{enumerate}

Thus, with high probability, we can assume that \ref{eqn:H1edges}--\ref{eqn:ABprime} hold. We will now show that, when this occurs, $A'$ and $B'$ have the property in the lemma. For this, let $C'\subset C$ be an arbitrary set with size $qn$. To apply Corollary~\ref{cor:nibble}, we wish to identify an almost-regular subgraph of $\mathcal{H}[A'\cup B'\cup C']$. We start by showing there are few vertices in $A'$ with degree in $\mathcal{H}[A'\cup B'\cup C']$ which is much more or less than $pq^n$.
Let
\[
\bar{A}_0=\{v\in A':d_\mathcal{H}(v,B'\cup C')> (1+\sqrt{\gamma})pq^2n\}\;\;\text{and}\;\;
\bar{A}_1=\{v\in A':d_\mathcal{H}(v,B'\cup C')< (1-\sqrt{\gamma})pq^2n\}.
\]
Note that, as $\mathcal{H}$ is a simple hypergraph, for each vertex $v$ and subset $V\subset C$, we have $d_\mathcal{H}(v,B'\cup V)=d_{H_1'}(v,V)$. Therefore, as $|C'|=qn$, there are at least $(1+\sqrt{\gamma})pq|\bar{A}_0||C'|$ edges between $\bar{A}_0$ and $C'$ in $H_1'$. Thus, by \ref{eqn:H1edges}, we have $\sqrt{\gamma}pq|\bar{A}_0||C'|\leq \gamma n^2$, so that $|\bar{A}_0|\leq \gamma^{1/3}pq^2n$ as $\gamma\llpoly p,q$ and $|C'|=qn$.
Similarly, we have $|\bar{A}_1|\leq \gamma^{1/3}pq^2n$, and, furthermore, setting
\[
\bar{B}_0=\{v\in B':d_\mathcal{H}(v,A'\cup C')> (1+\sqrt{\gamma})pq^2n\}\;\;\text{and}\;\;
\bar{B}_1=\{v\in B':d_\mathcal{H}(v,A'\cup C')< (1-\sqrt{\gamma})pq^2n\},
\]
it follows from \ref{eqn:H2edges} that $|\bar{B}_0|,|\bar{B}_1|\leq \gamma^{1/3}pq^2n$.

Let $\bar{A}= A'\setminus(\bar{A}_0\cup \bar{A}_1)$, $\bar{B}= B'\setminus (\bar{B}_0\cup \bar{B}_1)$ and $\bar{C}= C'$. From \ref{eqn:ABprime}, we have $|\bar{A}|=(1\pm 3\gamma^{1/3})qn$ and $|\bar{B}|=(1\pm 3\gamma^{1/3})qn$, and we have $|\bar{C}|=qn$.
Let $\bar{\mathcal{H}}=\mathcal{H}[\bar{A}\cup \bar{B}\cup \bar{C}]$ and note that, for each $v\in V(\bar{\mathcal{H}})$, we have (using \ref{eqn:Cdegree} if $v\in C$ and otherwise that $v\notin \bar{A}_0\cup \bar{A}_1\cup \bar{B}_0\cup \bar{B}_1$)
that
\[
d_{\bar{\mathcal{H}}}(v)=(1\pm \sqrt{\gamma})pq^2n\pm |\bar{A}_0\cup \bar{A}_1\cup \bar{B}_0\cup \bar{B}_1|=(1\pm 5\gamma^{1/3})pq^2n,
\]
where we have used that $|\bar{A}_0\cup \bar{A}_1\cup \bar{B}_0\cup \bar{B}_1|\leq 4\gamma^{1/3}pq^2n$.
Therefore, by Corollary~\ref{cor:nibble}, $\bar{\mathcal{H}}$, and hence $\mathcal{H}[A',B',C']$, contains a matching with at least $qn-\eta n$ edges, as required.
\end{proof}

Using Lemma~\ref{lem:matchinhyper}, it is now straight forward to prove Theorem~\ref{thm:RSBcoveringstephyper}.


\begin{proof}[Proof of Theorem~\ref{thm:RSBcoveringstephyper}] Let $\bar{q}$ satisfy $q-\eta/4\leq 3\bar{q}\leq q-\eta/10$ and $\bar{q}n\in \N$. Note that, as $q_A,q_B,q_C\geq 2\bar{q}$ we can take disjoint sets $A_1\cup A_2\subset A'$, $B_1\cup B_2\subset B'$ and $C_1\cup C_2\subset C'$ so that $A_1,A_2, B_1,B_2,C_1$ and $C_2$ are
each $\bar{q}$-random subsets (of $A$, $B$ or $C$).
By Lemma~\ref{lem:matchinhyper} applied three times, and by Lemma~\ref{Lemma_Chernoff} (and then using $3\bar{q}\leq q-\eta/10$), with high probability, we have the following properties.

\stepcounter{propcounter}
\begin{enumerate}[label = {{\textbf{\Alph{propcounter}\arabic{enumi}}}}]
\item For any $\bar{A}\subset A$ with $|\bar{A}|\geq \bar{q}n$, there is a matching in $\mathcal{H}[\bar{A},B_1,C_1]$ with size at least $(\bar{q}-\eta/4)n$.\label{prop-fin-cover-1}
\item For any $\bar{B}\subset B$ with $|\bar{B}|\geq \bar{q}n$, there is a matching in $\mathcal{H}[A_1,\bar{B},C_2]$ with size at least $(\bar{q}-\eta/4)n$.\label{prop-fin-cover-2}
\item For any $\bar{C}\subset C$ with $|\bar{C}|\geq \bar{q}n$, there is a matching in $\mathcal{H}[A_2,B_2,\bar{C}]$ with size at least $(\bar{q}-\eta/4)n$.\label{prop-fin-cover-3}
\item $|A_1\cup A_2|,|B_1\cup B_2|,|C_1\cup C_2|\leq (2\bar{q}+\eta/20)(1+\eps)n\leq (2\bar{q}+\eta/10)n\leq (q-\bar{q})n$. \label{prop-fin-cover-4}
\end{enumerate}

We now claim that we have the property in the theorem.
To see this, take arbitrary sets $\bar{A}\subset A$, $\bar B\subset B$, $\bar{C}\subset C$ with size $qn$ such that $A'\cup B'\cup C'\subset \bar{A}\cup \bar{B}\cup \bar{C}$. Let $\bar{A}'=\bar{A}\setminus A'$, $\bar{B}'=\bar{B}\setminus B'$ and $\bar{C}'=\bar{C}\setminus C'$, noting that, by \ref{prop-fin-cover-4}, we have $|\bar{A}'|,|\bar{B}'|,|\bar{C}'|\geq \bar{q}n$.
By \ref{prop-fin-cover-1},~\ref{prop-fin-cover-2} and~\ref{prop-fin-cover-3}, there are matchings in $\mathcal{H}[\bar{A}',B_1,C_1]$, $\mathcal{H}[A_1,\bar{B}',C_2]$, and $\mathcal{H}[A_2,B_2,\bar{C}']$, each with size at least $(\bar{q}-\eta/4)n$. Combining these gives a matching with at least $(3\bar{q}-3\eta/4)n\geq (q- \eta)n$ edges.
\end{proof}


\section{Exchangeable colour classes}\label{sec:exchangingcolours}
In this section, we find our colour classes, as discussed in Section~\ref{sec:alg}. We will work more generally in an $n$-vertex graph properly coloured graph $G$ with at most $kn$ colours, where $1/n\ll 1/k$, so that $G$ does not need to be bipartite or a complete bipartite graph.
If two matchings in $G$ have the same vertex set and the same colour set, except one has an edge of colour $c$ while the other has an edge of colour $d$, then this allows us to `switch' between using edges of colour $c$ and $d$. We call such a structure a \emph{colour switcher}, as follows, where we additionally restrict colour switchers to edge disjoint matchings whose union consists of vertex-disjoint 4-cycles to simplify some counting later. A colour switcher is depicted in Figure~\ref{fig:switcher}.

\begin{defn}\label{defn:colourswitcher}
Given colours $c,d\in C(G)$, and an even integer $k$, a \emph{$c,d$-colour-switcher of order $k$} (often shortened to $c,d$-switcher) is a pair $S=(M_1,M_2)$ of rainbow matchings in $G$ of order $k$ with $V(M_1)=V(M_2)$ such that $C(M_1)\setminus\{c\}=C(M_2)\setminus \{d\}$, $c\in C(M_1)$, $d\in C(M_2)$, and $M_1\cup M_2$ is a union of rainbow 4-cycles.

We say $S$ has \emph{vertex set $V(S)=V(M_1)=V(M_2)$} and \emph{colour set $C(S)=C(M_1)\setminus \{c\}=C(M_2)\setminus\{d\}$}.
\end{defn}


As discussed in Section~\ref{sec:alg}, we wish to group colours together into classes so that, for each class of colours $C$, we can find switchers between any pair of colours in $C$. We want to do this \emph{robustly}, so that if we have found some structure in $G$ using colours in $\bar{C}$ and vertices in $\bar{V}$, we can find switchers between any pair of colours in $C$ which do not use colours in $\bar{C}$ or vertices in $\bar{V}$. For large colour classes, though, this is too ambitious, so instead we will have that there is some small set $B\subset C$ (depending on $\bar{C}$ and $\bar{V}$) so that this is true for pairs of colours from $C\setminus B$. A final complication is that, in our applications, after determining $B$ we will need a small further robustness, further avoiding sets $\bar{V}$ and $\bar{C}$ in the following definition -- this plays a crucial, but small role, as the following definition is used for small values of $L$.

\begin{defn}\label{defn:exchange} For each $\eps,\eta>0$ and $L,\ell\in \N$, say a colour set $C\subset C(G)$ is \emph{$(\eps,\eta,L,\ell)$-exchangeable in $G$}, if, for each $\bar{C}\subset C(G)$ and $\bar{V}\subset V(G)$ with $|C\cap \bar{C}|\leq \eps |C|$ and $|\bar{C}|,|\bar{V}|\leq \eps |G|$, there is a set $B\subset C$ with $|B|\leq \eta |C|$ and $C\cap \bar{C}\subset B$ such that the following hold.
\begin{itemize}
\item If $C\cap \bar{C}=\emptyset$, then $B=\emptyset$.
\item For each $\hat{C}\subset C(G)$ and $\hat{V}\subset V(G)$ with $|\hat{C}|,|\hat{V}|\leq L$, and each distinct $c,d\in C\setminus B$, $G$ contains a $c,d$-switcher with order at most $\ell$ and no vertex in $\bar{V}\cup \hat{V}$ or colour in $\bar{C}\cup \hat{C}$.
\end{itemize}
\end{defn}

Our aim in this section is to prove the next result, Theorem~\ref{thm-manyinter}, which shows that given any properly coloured $n$-vertex graph $G$, to find a collection $\mathcal{C}$ of set of colours so that each $C\in \mathcal{C}$ is $(\eps,\eta,L,\ell)$-exchangeable (for certain parameters), as in \eref{prop-C-1}, and no vertex appears in too many of the sets in $\mathcal{C}$, as in \eref{prop-C-2}. Finally, we want that most colour switchers of order 4 consisting of two 4-cycles switch between colours which are $\mathcal{C}$-equivalent using the following definition.
\begin{defn} Given any collection $\mathcal{C}$ of sets of colours, we say $c$ and $d$ are $\mathcal{C}$-equivalent if either $c=d$ or there is some $C\in \mathcal{C}$ with $c,d\in C$.
\end{defn}
As we will see in Section~\ref{subsec:count}, there are $\Theta(n^5)$ colour switchers of order 4 in an optimally coloured $K_{n,n}$ (potentially switching between the same colour) -- therefore, `most colour switchers' here will mean all but at most $\xi n^5$ switchers for some small $\xi$, as in \eref{prop-C-3}. That is, we will prove the following result.


\begin{theorem}\label{thm-manyinter} Let $1/n\ll 1/k$ and $1/n\llpoly \eps\llpoly \eta\llpoly \xi,\log^{-1}n,1/L$. Let $G$ be a properly coloured $n$-vertex graph with $|C(G)|\leq kn$. For each $\{c,d\}\in C(G)^{(2)}$, let $w_{cd}$ be the number of $c,d$-switchers of order 4 in $G$. Then, there is a collection $\mathcal{C}$ of subsets of $C(G)$ satisfying the following properties with $\ell=2000\log^3n$.

\stepcounter{propcounter}
\begin{enumerate}[label = {\emph{\textbf{\Alph{propcounter}\arabic{enumi}}}}]
\item Each $C\in \mathcal{C}$ is $(\eps,\eta,L,\ell)$-exchangeable.\label{prop-C-1}
\item Each colour $c\in C(G)$ appears in at most $\frac{\log^2 n}{\xi}$ of the sets $C\in \mathcal{C}$ with $|C|>2$.\label{prop-C-2}
\item If $\mathcal{I}\subset C(G)^{(2)}$ is the set of non-$\mathcal{C}$-equivalent colour pairs, then\label{prop-C-3}
\begin{equation}\label{eqn-wesum}
\sum_{e\in \mathcal{I}}w_e\leq \xi n^5.
\end{equation}
\end{enumerate}
\end{theorem}

We prove Theorem~\ref{thm-manyinter} in this section. To do so, we will (effectively) consider the complete auxiliary graph $R$ with vertex set $C(G)$ where each edge $cd\in E(R)$ is given a weight $w_{cd}$, the number of $c,d$-switchers in $G$ of order 4. For $w_0,w_1,w_2$ defined more precisely later, we consider edges to be very light if they have weight at most $w_0=n^{3-o(1)}$, light if they have weight above $w_0$ and at most $w_1=n^{3-o(1)}$, heavy if they have weight above $w_2=n^{4-o(1)}$, and otherwise consider them to have moderate weight if their weight is strictly between $w_1$ and $w_2$ (where we later split the moderately-weighted edges further based on their weight, though using the same methods for each group of moderately-weighted eges).
After proving some basic counting results for colour switchers of order 4 in Section~\ref{subsec:count} (gathered together in Proposition~\ref{prop-switchoverlap}), we deal with edges with heavy weight in Section~\ref{subsec:heavy}, edges with light weight in Section~\ref{subsec:light}, and moderately-weighted edges in Section~\ref{subsec:moderate}, before proving Theorem~\ref{thm-manyinter} in Section~\ref{subsec:manyinterproof}. Before proceeding, we first give an overview of the proof of Theorem~\ref{thm-manyinter}.


Note that \eref{prop-C-2} does not consider any sets of 2 colours in the set of colour classes $\mathcal{C}$, so that we may as well add any set of 2 colours to $C$ if it satisfies \eref{prop-C-1}. When $C=\{c,d\}$, if $\eps,\eta<1/2$, then the definition for $C$ to be $(\eps,\eta,L,\ell)$-exchangeable simplifies to be equivalent to the condition, for each $\tilde{C}\subset C(G)$ and $\tilde{V}\subset V(G)$ with $|\tilde{C}|,|\tilde{V}|\leq \eps|G|+L$, $G$ contains a $c,d$-switcher with order at most $\ell$ and no vertex in $\tilde{V}$ or colour in $\tilde{C}$.
In Section~\ref{subsec:heavy}, we show that, for any heavy edge in $R$, the 2 colours it contains have enough switchers between them of order 4 that they will form an exchangeable set together. Thus, we will have the heavy edges contribute no weight to the sum in \eqref{eqn-wesum}.


The very light edges will each have weight at most $\xi n^3/4k^2$. Thus, even if every edge of $R$ is very light and joins two vertices which are not $\mathcal{C}$-equivalent at the end of the construction of $\mathcal{C}$, they will contribute together a manageable amount to the sum in \eqref{eqn-wesum}.

We will show that light edges corresponding to colour pairs not already in $\mathcal{C}$ will also contribute a manageable amount to the sum in \eqref{eqn-wesum}. Working by contradiction, if this is not true, so that these edges form a relatively dense graph (i.e., with at least $\mu |C(G)|^2$ edges for some $\mu$ with $1/n\llpoly \mu$). Using Proposition~\ref{numberof4cycles}, we can find such a light edge, $cd$ say, whose colours are connected in $R$ by many paths of light edges with length 3. We aim to use each path to find 3 switchers of order 4 and combine them to get a $c,d$-switcher. There is a little added complication here in finding 3 switchers that combine well (i.e., that do not share vertices or colours), and this is why we need to work with edges that are not very light, but this is done in Section~\ref{subsec:light}.


Finding exchangeable classes using the edges of $R$ with moderate weight is the trickiest part of this section, and is where we use sublinear expansion as discussed in Section~\ref{subsec:expand} along with the results given there. If $H\subset R$ is an expander subgraph (in the sense of Section~\ref{subsec:expand}) in which the weights of the edges differ by at most a factor of 2 and are all moderate, then we can show $V(H)\subset C(G)$ is an exchangeable set of colours in $G$ (under certain conditions). Roughly speaking, for $c,d\in V(H)$, we can find a $c,d$-switcher by finding a path within $H$ using the expansion conditions, before finding a switcher of order 4 between each of the pairs of colours which appear as an edge of this path. Once we have found such a path we find such switchers greedily and then chain them together to get a $c,d$-switcher. To find switchers that chain together without vertex or colour conflicts we need each edge in the path to have enough weight, which will follow as they are not light, or very light, edges. (This is why we consider light edges separately.) We will also need that $H$ has a large enough minimum degree, and therefore the heavy edges may not contain a suitable expander, even if their total weight would be a significant contributor to \eqref{eqn-wesum}. (This is why we consider heavy edges separately.) We use that the weights on the edges of $H$ differ only by a factor of 2 in order to get a sufficiently robust connection property in $H$. The conditions we require of $H$ are a little technical, but they are given in Section~\ref{subsec:moderate}, where we prove that when these hold $V(H)$ is an exchangeable set of colours (see Lemma~\ref{lem-exchangefromexpansion}).

Having introduced our methods to deal with heavy, light, and moderately-weighted edges in Sections~\ref{subsec:heavy},~\ref{subsec:light} and~\ref{subsec:moderate}, it is a relatively simple task to prove Theorem~\ref{thm-manyinter} in Section~\ref{subsec:manyinterproof}. Adding all pairs of exchangeable colours to $\mathcal{C}$, the above discussion indicates we need only worry about the weight of moderate edges whose colours have not been added to $\mathcal{C}$. We partition these edges into $2\log n$ graphs $R_i$ so that in each graph $R_i$ the weight of any pair of edges differs by at most a factor of 2. Removing a maximal collection of edge-disjoint expanders (with appropriate parameters) from each $R_i$, by Theorem~\ref{thm-expander} we will have few edges in $R_i$ remaining. Adding the vertex set of each expander to $\mathcal{C}$ will mean that only these few remaining edges will contribute to the sum in \eqref{eqn-wesum}, which will result in \eref{prop-C-3} holding, while \eref{prop-C-1} will hold for the colour sets added here from our work in Section~\ref{subsec:moderate}. Each time a colour appears in an expander, a sizeable portion of the weight on its neighbouring edges in $R$ will end up on edges in that expander, and this will result in \eref{prop-C-2} holding.


\subsection{Counting colour switchers of order 4}\label{subsec:count}

Suppose $G$ is any properly coloured graph with $n$ vertices. Considering the number of colour switchers of order 4 in $G$, note that any switcher $S$ of order 4 is the disjoint union of 2 rainbow 4-cycles, and so is reduced to $O(1)$ choices by the choice of the three colours in $C(S)$ and a vertex in each of the 4-cycles (we make these arguments more precisely in the proof of Proposition~\ref{prop-switchoverlap} below). Thus, $G$ will have $O(n^5)$ switchers of order 4.
We will now prove several bounds for the number of switchers in $G$ of order 4 satisfying additional properties, for example containing a fixed vertex. These are gathered into Proposition~\ref{prop-switchoverlap}, and proved in full, but let us note first that these bounds are very simple and confirm only what is to be expected from considering the ``degrees of freedom''.
That is, though we have a bound of $O(n^5)$ for the switchers $S=(M,M')$ of order 4 in $G$, if we fix one colour in $C(S)$ and one colour in $C(M)\triangle C(M')$ then the number of switchers drops to $O(n^3)$ (see Proposition~\ref{prop-switchoverlap}i)).
If we require $C(M)\triangle C(M')$  to contain a fixed colour and require $V(M)$ to contain a fixed vertex, then the number also drops to $O(n^3)$ (see Proposition~\ref{prop-switchoverlap}ii)).
If we specify two edges of the same colour to be in $M\cup M'$ (of $O(n^3)$ choices for these two edges) and specify one further colour for $C(M)$, then the number drops to $O(n)$ (see Proposition~\ref{prop-switchoverlap}iii)).
If we specify two edges of the same colour to be in $M\cup M'$ and specify one vertex for $V(M)$ which is not in either of these two edges, then the number similarly drops to $O(n)$ (see Proposition~\ref{prop-switchoverlap}iv)).
Finally, if we fix only one edge for $M\cup M'$ (of $O(n^2)$ choices), then the number drops to $O(n^3)$ (see Proposition~\ref{prop-switchoverlap}v)).

\begin{prop}\label{prop-switchoverlap}  Let $G$ be a properly coloured $n$-vertex graph and $c\in C(G)$. Then, the following hold.

\begin{enumerate}[label = \emph{\textbf{\roman{enumi})}}]
\item If $c'\in C(G)\setminus \{c\}$, then there are at most $20n^3$ switchers $S$ in $G$ of order 4 for which there is some $d\in C(G)$ such that $S$ is a $c,d$-switcher with $c'\in C(S)$.\label{prop-switchoverlap-1}
\item If $v\in V(G)$, then there are at most $50n^3$ switchers $S$ in $G$ of order 4 for which there is some $d\in C(G)$ such that $S$ is a $c,d$-switcher with $v\in V(S)$.\label{prop-switchoverlap-2}
\item If $e,f\in E_c(G)$ are distinct and $c'\in C(G)\setminus\{c\}$, then there are at most $10^3n$ switchers $S=(M,M')$ in $G$ of order 4 with $e,f\in M\cup M'$ and $c'\in C(M\cup M')$. \label{prop-switchoverlap2-1}
\item If $e,f\in E_c(G)$ are distinct and $v\in V(G)\setminus(V(e)\cup V(f))$, then there are at most $600n$ switchers $S=(M,M')$ in  $G$ of order 4 with $e,f\in M\cup M'$ and $v\in V(M)$.\label{prop-switchoverlap2-2}
\item If $e\in E(G)$, then there are at most $100n^3$ switchers $S=(M,M')$ in  $G$ of order 4 with $e\in M\cup M'$.\label{prop-switchoverlap2-3}
\end{enumerate}
\end{prop}
\begin{proof} First note that, given a set of 3 distinct colours $C\subset C(G)$ and a vertex $w\in V(G)$, a path with length 3 which is $C$-rainbow and has $w$ as an end-vertex is, as $G$ properly coloured, determined uniquely by the order of colours on the path's edges starting from $w$. Thus, for each $C\subset C(G)$ with $|C|=3$ there are at most $3n$ $C$-rainbow paths in $G$ with length 3, and hence at most $3n$ rainbow 4-cycles using each colour in $C$.

Note too, that, for any $d,d'\in C(G)$, if $S=(M,M')$ is a $d,d'$-switcher of order 4 then $M\cup M'$ contains two (vertex-disjoint) rainbow 4-cycles with colour sets $C(S)\cup \{d\}$ and $C(S)\cup\{d'\}$ respectively, and, given two such 4-cycles, choosing $d$ from among the colours of the first cycle (with 4 choices) uniquely determines $d',M$ and $M'$. Thus, for each case \eref{prop-switchoverlap-1}--\eref{prop-switchoverlap2-3}, we will give an upper bound for the number of pairs of rainbow 4-cycles $(S_1,S_2)$ sharing at least 3 colours with the natural restriction for that case (as given in each case below) before multiplying the bound by 4.

\smallskip

\noindent\eref{prop-switchoverlap-1}: Let $c'\in C(G)\setminus \{c\}$. We bound above the number of pairs of rainbow 4-cycles $(S_1,S_2)$ with $c,c'\in C(S_1)$ and $C(S_1)\setminus \{c\}\subset C(S_2)$. Now, there are at most $3n^2/2$ choices for a 4-cycle $S_1$ with $c,c'\in C(S_1)$. Indeed, firstly, if the edges with colour $c$ and colour $c'$ share a vertex, then the 4-cycle is determined by the shared vertex and the vertex appearing in neither edge with colour $c$ or $c'$ ($\leq n^2$ choices). Secondly, if the edges with colour $c$ and $c'$ share no vertices then the 4-cycle is determined by the colour-$c$ edge ($\leq n/2$ choices), the colour-$c'$ edge ($\leq n/2$ choices) and choosing 1 of the 2 possible ways of completing this edge pair into a 4-cycle, for at most $n^2/2$ choices in this second case, and at most $3n^2/2$ choices in total, as claimed. Given a rainbow 4-cycle $S_1$ with $c,c'\in C(S_1)$, from above we have that there are at most $3n$ rainbow 4-cycles $S_2$ with $C(S_1)\setminus\{c\}\subset C(S_2)$. Thus, in total, there are at most $9n^3/2\leq 5n^3$ pairs of rainbow 4-cycles $(S_1,S_2)$ such that $c,c'\in C(S_1)$ and $C(S_1)\setminus\{c\}\subset C(S_2)$, so that, multiplying this bound by 4, we have that \eref{prop-switchoverlap-1} holds.

\medskip

\noindent{\eref{prop-switchoverlap-2}}: Let $v\in V(G)$. We bound above the number of pairs of rainbow 4-cycles $(S_1,S_2)$ with $c\in C(S_1)$, $C(S_1)\setminus \{c\}\subset C(S_2)$ and $v\in V(S_1)\cup V(S_2)$, counting separately the cases when $v\in V(S_1)$ and when $v\in V(S_2)$. There will be at most $6n^3$ pairs in each case, for at most $12n^3$ pairs in total, so that multiplying this bound by 4 gives \eref{prop-switchoverlap-2}.

Note that there are at most $2n^2$ choices for a rainbow 4-cycle $S_1$ with $c\in C(S_1)$ and $v\in V(S_1)$. Indeed, firstly, if $v$ appears in the edge of colour-$c$ in $S_1$ then this determines the colour-$c$ edge in $S_1$, and there are at most $n^2$ ways to extend this edge to a 4-cycle. Secondly, if $v$ is not in the colour-$c$ edge in $S_1$, then there are at most $n/2$ choices for the colour-$c$ edge, at most $n$ choices for the 4th vertex of $S_1$ and at most 2 ways to choose a 4-cycle through these 4 vertices which contains the chosen colour-$c$ edge, for at most $n^2$ choices in total here, and at most $2n^2$ choices then across both cases.
 Now, given a rainbow 4-cycle $S_1$ with $c\in C(S_1)$ and $v\in V(S_1)$, there are at most $3n$ rainbow 4-cycles $S_2$ with $C(S_1)\setminus\{c\}\subset C(S_2)$. Thus, in total, there are at most $6n^3$ pairs of rainbow 4-cycles $(S_1,S_2)$ such that $c\in C(S_1)$, $C(S_1)\setminus\{c\}\subset C(S_2)$ and $v\in V(S_1)$, as required.

Now, by picking first a colour-$c$ edge ($\leq n/2$ choices) and 2 other vertices in order, we have that there are at most $n^3/2$ rainbow 4-cycles $S_1$ with $c\in C(S_1)$. If $S_2$ is a rainbow 4-cycle containing $v$ with $C(S_1)\setminus\{c\}\subset C(S_2)$, then note it contains a path with length 3 and colour set $C(S_1)\setminus\{c\}$ starting at $v$ ($\leq 6$ choices), or 2 edges neighbouring $v$ with colour in $C(S_1)\setminus \{c\}$ (of $\leq 3$ choices) with an edge attached with the remaining colour in $C(S_1)\setminus\{c\}$ attached to one of these neighbours ($\leq 2$ choices), making at most 12 choices for $S_2$. Thus, in total, there are at most $6n^3$ pairs of rainbow 4-cycles $(S_1,S_2)$ such that $c\in C(S_1)$, $C(S_1)\setminus\{c\}\subset C(S_2)$ and $v\in V(S_2)$, as required, completing the proof of \eref{prop-switchoverlap-2}.

\medskip

\noindent{\eref{prop-switchoverlap2-1}}: Let $e,f\in E_c(G)$ be distinct and $c'\in C(G)\setminus\{c\}$. We bound above the number of pairs of rainbow 4-cycles $(S_1,S_2)$ sharing at least 3 colours with $e,f\in E(S_1\cup S_2)$ and $c'\in C(S_1\cup S_2)$. Note that if there are at most $54n$ such pairs $(S_1,S_2)$ with $e\in E(S_1)$, $f\in E(S_2)$ and $c'\in C(S_1)$, then by symmetry we have that there are most $4\cdot 54n\leq 250n$ such pairs with $e,f\in E(S_1\cup S_2)$ and $c'\in C(S_1\cup S_2)$ (as $e$ and $f$ appear in different 4-cycles), and then multiplying this bound by 4 gives \eref{prop-switchoverlap2-1}.

 If $S_1$ is a rainbow 4-cycle with $e\in E(S_1)$ and $c'\in C(S_1)$, then $S_1$ is either determined by the vertex in $V(e)$ appearing in a colour-$c'$ edge in $S_1$ ($\leq 2$ choices) and the vertex in $S_1$ not in $e$ or the colour-$c'$ edge ($\leq n$ choices), or determined by the colour-$c'$ edge disjoint from $e$ ($\leq n/2$ choices) and the choice  of joining this into a 4-cycle with $e$ ($\leq 2$ choices), making at most $2n+2\cdot n/2=3n$ rainbow 4-cycles $S_1$ with $e\in E(S_1)$ and $c'\in C(S_1)$.
 Given such an $S_1$, if $S_2$ is a rainbow 4-cycle containing $f$ which shares at least 3 colours with $S_1$, then $S_2$ contains a rainbow path $P$ of length 3 containing the edge $f$ and having 2 colours in $C(S_1)\setminus\{c\}$. To count such paths we can label $f=xy$ arbitrarily, and choose the length of the path in $P-xy$ attached to $x$ (3 choices), and then the choosing the colours for the non-$f$ edges of $P$ ($\leq 3\cdot 2=6$ choices), for at most $18$ choices for $S_2$. Thus, in total, there are at most $3n\cdot 18=54n$ pairs of rainbow 4-cycles $(S_1,S_2)$ sharing 3 colours with $e\in E(S_1)$, $f\in E(S_2)$ and $c'\in C(S_1)$, as required.

\medskip

\noindent{\eref{prop-switchoverlap2-2}}: Let $e,f\in E_c(G)$ be distinct and $v\in V(G)\setminus (V(e)\cup V(f))$. We bound above the number of pairs of rainbow 4-cycles $(S_1,S_2)$ sharing at least 3 colours with $e,f\in E(S_1\cup S_2)$ and $v\in V(S_1\cup S_2)$. Note that if there are at most $36n$ such pairs $(S_1,S_2)$ with $e\in E(S_1)$, $f\in E(S_2)$ and $v\in V(S_1)$, then by symmetry we have that there are most $4\cdot 36n\leq 150n$ such pairs with $e,f\in E(S_1\cup S_2)$ and $v\in V(S_1\cup S_2)$, and then multiplying this bound by 4 gives \eref{prop-switchoverlap2-2}.

If $S_1$ is a rainbow 4-cycle with $e\in E(S_1)$ and $v\in V(S_1)$, then, as $v\notin V(e)$. $S_1$ is determined by the vertex not appearing in $V(e)\cup \{v\}$ ($\leq n$ choices) and the choice of the 4-cycle containing $e$ on the resulting 4 vertices ($\leq 2$ choices). That is, there are at most $2n$ rainbow 4-cycles $S_1$ with $e\in E(S_1)$ and $v\in V(S_1)$. As for \eref{prop-switchoverlap2-1}, given such an $S_1$, there are at most 18 choices for a rainbow 4-cycle containing $f$ sharing at least 3 colours with $S_1$, and therefore, in total at most $36n$ pairs $(S_1,S_2)$ of rainbow 4-cycles sharing 3 colours with $e\in E(S_1)$, $f\in E(S_2)$ and $v\in V(S_1)$, as required.

\medskip

\noindent{\eref{prop-switchoverlap2-3}}: Let $e\in E_c(G)$. We bound above the number of pairs of rainbow 4-cycles $(S_1,S_2)$ sharing 3 colours with $e\in E(S_1)\cup E(S_2)$. Note that there are at most $n^2$ rainbow 4-cycles $S_1$ with $e\in E(S_1)$, by choosing the two vertices in $V(S_1)\setminus V(e)$ in turn. Given any such $S_1$, and picking one of the 4 sets $C\subset C(S_1)$ with size 3, we have from above that there are at most $3n$ rainbow 4-cycles using each colour in $C$. Thus, there are at most $12n^3$ pairs of rainbow 4-cycles $(S_1,S_2)$ sharing 3 colours with $e\in E(S_1)$. Similarly, there are at most $12n^3$ such pairs with $e\in E(S_2)$ instead of $e\in E(S_1)$. Thus, there are at most $24n^3$ pairs of rainbow 4-cycles $(S_1,S_2)$ sharing 3 colours with $e\in E(S_1)\cup E(S_2)$. Multiplying this bound by 4 then gives \eref{prop-switchoverlap2-3}.
\end{proof}


\subsection{Exchangeable colour pairs: heavy edges}\label{subsec:heavy}
We can now use Proposition~\ref{prop-switchoverlap} to show that if there are many $c,d$-switchers of order 4 (i.e., if $cd$ is a heavy edge in the auxiliary graph described at the start of this section) then $\{c,d\}$ is an exchangeable set using only switchers of order 4, as follows.

\begin{lemma}\label{lem-highwexchange}
Let $1/n\ll 1/k\leq 1$ and $\eps\in (0,1/2)$. Let $G$ be a properly coloured $n$-vertex graph.
For each distinct $c,d\in C(G)$, if there are at least $150\eps n^4$ $c,d$-switchers in $G$ of order 4, then $\{c,d\}$ is $(\eps,0,\eps n,4)$-exchangeable.
\end{lemma}
\begin{proof}
Let $C=\{c,d\}$. As $\eps|C|<1$, to show that $C$ is $(\eps,0,\eps n,4)$-exchangeable, we must have $B=\emptyset$ for any choice of $\bar{C},\bar{V}$ used in Definition~\ref{defn:exchange}. Thus, combining, for example, the sets $\bar{C}$ and $\hat{C}$ in Definition~\ref{defn:exchange} into $\tilde{C}$, it is sufficient to show that
 given any $\tilde{C}\subset C(G)$ and $\tilde{V}\subset V(G)$ with $|\tilde{C}|,|\tilde{V}|\leq 2\eps n$, there is a $c,d$-switcher in $G$ with order 4 and no vertices in $\tilde{V}$ or colours in $\tilde{C}$.

Let then $\tilde{C}\subset C(G)$ and $\tilde{V}\subset V(G)$ with $|\tilde{C}|,|\tilde{V}|\leq 2\eps n$. By Proposition~\ref{prop-switchoverlap}i) and ii) respectively, there are at most $20n^3\cdot |\tilde{C}|$ $c,d$-switchers containing a colour in $\tilde{C}$ and at most $50n^3\cdot |\tilde{V}|$  $c,d$-switchers containing a vertex in $\tilde{V}$.
Therefore, as
\[
20n^3\cdot |\tilde{C}|+50n^3\cdot |\tilde{V}|\leq 140\eps n^4 < 150\eps n^4,
\]
there is a $c,d$-switcher in $G$ with order 4 and no vertices in $\tilde{V}$ and no colours in $\tilde{C}$, as required.
\end{proof}


\subsection{Exchangeable colour pairs: light edges}\label{subsec:light}
As sketched at the start of this section, we now show that the auxiliary graph with vertex set $C(G)$ cannot have very many  edges which are not `very light', as follows, which we use to bound the number of `light' edges which do not form an exchangeable colour pair.

\begin{lemma}\label{lem-manyinter-1} Let $1/n\llpoly \eps \llpoly \lambda,\xi,1/k$. Let $G$ be a properly coloured graph with $|C(G)|\leq kn$ and $|G|=n$.
Then, there are at most $\xi n^2$ pairs of colours $c,d\in C(G)$ such that there are at least $\lambda n^3$ $c,d$-switchers of order 4 in $G$ but $\{c,d\}$ is not $(\eps,0,\eps n,12)$-exchangeable.
\end{lemma}
\begin{proof} Pick $\bar{\eps}$ so that $\eps\llpoly\bar{\eps}\llpoly \lambda,\xi,1/k$.
Let $I$ be the set of pairs of colours $c,d\in C(G)$ such that there are at least $\lambda n^3$ switchers of order 4 in $G$ but $\{c,d\}$ is not $(\eps,0,\eps n,12)$-exchangeable, and suppose, for contradiction, that $|I|>\xi n^2$. Let $R$ be the graph with vertex set $C(G)$ and edge set $I$. Take a maximal subgraph $R'\subset R$ with $V(R')=V(R)$ for which there is a set of switchers $S_e$, $e\in E(R')$, such that the following hold.

\stepcounter{propcounter}
\begin{enumerate}[label = {\textbf{\Alph{propcounter}\arabic{enumi}}}]
\item For each $e\in E(R')$, $S_e$ is an $e$-switcher with order 4.\label{prop-E-2}
\item For each distinct $e,f\in E(R')$ which share a vertex, $S_e$ and $S_f$ are colour-disjoint and vertex-disjoint.\label{prop-E-3}
\end{enumerate}

\begin{claim} We have $e(R')\geq \lambda \xi n^2/10^{4}$.\label{clm-Ebig}
\end{claim}
\begin{proof}[Proof of Claim~\ref{clm-Ebig}]
Let $R''=R\setminus R'$ and let $S_e$, $e\in E(R')$, satisfy \ref{prop-E-2} and \ref{prop-E-3}. For each $c\in C(G)=V(R')$, let $C_c=\cup_{y\in N_{R'}(c)}C(S_{cd})$ and $V_c=\cup_{d\in N_{R'}(c)}V(S_{cd})$, and note that, by \ref{prop-E-2} and \ref{prop-E-3}, $|C_c|=3d_{R'}(c)$ and $|V_c|=8d_{R'}(c)$. By the maximality of $R'\subset R$, for each edge $cd\in E(R'')=E(R)\setminus E(R')$, every $c,d$-switcher in $G$ with order 4 must have a vertex in $V_c\cup V_d$ or a colour in $C_c\cup C_d$. Using this, direct the edges of $E(R'')$ such that, if $\vec{cd}\in E(R'')$, then at least half of the $c,d$-switchers in $G$ with order 4 share a vertex with $V_c$ or a colour with $C_c$. As $cd\in I$, we therefore have that at least $\lambda n^3/2$ $c,d$-switchers in $G$ of order 4 which share a vertex with $V_c$ or a colour with $C_c$. Therefore, for each $c\in C(G)$, counting these over $d\in N^+_{R''}(c)$, we have, using Proposition~\ref{prop-switchoverlap}, that
\begin{equation*}\label{eqn-RvsR}
d^{+}_{R''}(c)\cdot \frac{\lambda n^3}{2}\leq 20n^3\cdot |C_c|+50n^3\cdot |V_c|= 460 n^3 \cdot d_{R'}(c),
\end{equation*}
and, therefore, $d_{R''}^+(c)\leq 920 \cdot d_{R'}(c)/\lambda$.
Thus, we have
\[
e(R'')=\sum_{c\in C(G)}d^+_{R''}(c) \leq \sum_{c\in C(G)}\frac{920 \cdot d_{R'}(c)}{\lambda} =\frac{920 \cdot 2e(R')}{\lambda}.
\]
As $e(R'')+e(R')=e(R)\geq \xi n^2$, we have $e(R')\geq \xi n^2/(1+1840 \cdot \lambda^{-1})\geq \lambda \xi n^2/10^{4}$.
\renewcommand{\qedsymbol}{$\boxdot$}
\end{proof}
\renewcommand{\qedsymbol}{$\square$}

Suppose then that we have switchers $S_e=(V_e,C_e)$, $e\in E(R')$, of order 4 for which \ref{prop-E-2} and \ref{prop-E-3} hold, where $e(R')\geq  \lambda \xi n^2/10^{4}$.
By Proposition~\ref{numberof4cycles}, $R'$ contains at least $\bar{\eps}n^4$ labelled 4-cycles, and therefore, as $|C(G)|\leq kn$, there is some $cd\in E(R')$ such that $cd$ is in at least $\bar{\eps}n^2/k^2$ labelled 4-cycles in $R'$. We will show that $\{c,d\}$ is $(\eps,0,\eps n,12)$-exchangeable, which will be a contradiction as $cd\in E(R')\subset E(R)=I$.

As in the proof of Lemma~\ref{lem-highwexchange}, to show this,
let $\tilde{C}\subset C(G)$ and $\tilde{V}\subset V(G)$ satisfy $|\tilde{C}|,|\tilde{V}|\leq 2\eps n$. By \ref{prop-E-3}, the switchers $S_{cc'}$, $c'\in N_{R'}(c)\setminus \{d\}$, are colour- and vertex-disjoint.
Therefore, at most $4\eps n+1$ colours $c'\in N_{R'}(c)\setminus \{d\}$ are such that $S_{cc'}$ has a colour in $\tilde{C}\cup \{d\}$ or a vertex in $\tilde{V}$. Thus, there are at least $\tilde{\eps}n^2/2k^2$ pairs $c',c''$ of distinct colours in $C(G)\setminus \{c,d\}$ such that $cc',c'c'',c''d\in E(R')$ and $S_{cc'}$ has no colour in $\tilde{C}\cup \{d\}$ and no vertex in $\tilde{V}$.

Pick then $c'\in N_{R'}(c)\setminus \{d\}$ such that $S_{cc'}$ has no colour in $\tilde{C}\cup \{d\}$ and no vertex in $\tilde{V}$, and there are at least $\bar{\eps}n/2k^3$ values of $c''\in C(G)\setminus \{c,d,c'\}$ -- say those in $D$ -- such that $c'c'',c''d\in E(R')$. That is, with $D=(N_{R'}(c')\cap N_{R'}(d))\setminus\{c\}$, we have $|D|\geq \bar{\eps}n/2k^3$.
Now, for each $c''\in D$, label the switchers so that $S_{c'c''}=(M_{c'',1},M_{c'',2})$ and $S_{dc''}=(M_{c'',3},M_{c'',4})$. By \ref{prop-E-3}, $S_{c'c''}$ and $S_{dc''}$ are colour- and vertex-disjoint. Note, therefore, that $S_{c''}:=(M_{c'',1}\cup M_{c'',4},M_{c'',2}\cup M_{c'',3})$ is a $c',d$-switcher of order 8 with $V(S_{c''})=V(S_{c'c''})\cup V(S_{dc''})$ and $C(S_{c''})=\{c''\}\cup C(S_{c'c''})\cup C(S_{dc''})$.

By \ref{prop-E-3}, each vertex in $V(G)$ appears in at most one set $V(S_{c'c''})$, $c''\in D$, and at most one set $V(S_{dc''})$, $c''\in D$, and therefore appears in at most 2 sets $V(S_{c''})$, $c''\in D$. Similarly, each colour in $C(G)$ appears in at most one set $C(S_{c'c''})$, $c''\in D$, and at most one set $C(S_{dc''})$, $c''\in D$, and at most one set $\{c''\}$, $c''\in D$, and therefore appears in at most 3 sets $C(S_{c''})$, $c''\in D$.
Thus, as $|D|\geq \bar{\eps}n/2k^3$, and $|\tilde{C}\cup \{c,c',d\}\cup C(S_{cc'})|\leq 3\eps n$ and $|\tilde{V}\cup V(S_{cc'})|\leq 3\eps n$, there is some $c''\in D\setminus (\tilde{C}\cup \{c,c',d\}\cup C(S_{cc'}))$ such that $C(S_{c''})$ has no colour in $\tilde{C}\cup \{c,c',d\}\cup C(S_{cc'})$ or vertex in $\tilde{V}\cup V(S_{cc'})$.

Finally, note that, labelling matchings so that $S_{cc'}=(M_1,M_2)$ is a $c,c'$-switcher and $S_{c''}=(M_3,M_4)$ is a $c',d$-switcher, then $(M_1\cup M_3,M_2\cup M_4)$ is a $c,d$-switcher of order 12 in $G$ with no vertices in $\tilde{V}$ or colours in $\tilde{C}$ (with vertex set $V(S_{cc'})\cup V(S_{c''})$ and colour set $C(S_{cc'})\cup C(S_{c''})\cup \{c'\}$). Thus, $\{c,d\}$ is $(\eps,0,\eps n,12)$-exchangeable, a contradiction.
\end{proof}


\subsection{Exchangeable colour classes: moderately-weighted edges}\label{subsec:moderate}
We now prepare to deal with the `moderately-weighted' edges from the proof sketch at the start of this section. The following lemma shows that if a collection of these edges form the edge set of an expander, then the colours in the vertex set of this expander form an exchangeable colour class, as follows.
\begin{lemma}\label{lem-exchangefromexpansion} Let $1/n\ll 1/k$ and $\alpha=1/16\log (kn)$. Let $\eps\llpoly \eta\llpoly \log^{-1}n$. Let $G$ be a properly coloured $n$-vertex graph with $|C(G)|\leq kn$. Let $L\in \N$ and $\ell=2000\log^3n$, and suppose $\bar{d},w$ and $\bar{\Delta}$ satisfy
\begin{equation}\label{eqn:bounds}
\bar{d}w\geq \frac{320\eps n^4}{\alpha},\;\;\; w\geq 10^6Ln^3\log^3n, \;\;\; \bar{\Delta}\leq \frac{\bar{d}}{{\eta}}\;\text{ and }\;L\leq \frac{\bar{d}}{10^{10}\log^4n}.
\end{equation}
Suppose $H$ is an $(\alpha,\alpha \bar{d})$-expander with $\delta(H)\geq \bar{d}$, $\Delta(H)\leq \bar{\Delta}$ and $V(H)\subset C(G)$ such that, for each $cd\in E(H)$, the number of $c,d$-switchers of order $4$ in $G$ is at least $w$.

Then, $V(H)$ is $(\eps,\eta,L,\ell)$-exchangeable in $G$.
\end{lemma}
\begin{proof}
 Let $C=V(H)$. To show that $C$ is $(\eps,\eta,L,\ell)$-exchangeable in $G$, first take arbitrary sets $\bar{C}\subset C(G)$ and $\bar{V}\subset V(G)$ with $|C\cap \bar{C}|\leq \eps |C|$ and $|\bar{V}|,|\bar{C}|\leq \eps n$. We need to show that there is a set $B\subset C$ with $|B|\leq \eta |C|$ and $C\cap \bar{C}\subset B$ such that if $C\cap \bar{C}=\emptyset$, then $B=\emptyset$, and, for every distinct $c,d\in C\setminus B$ and every $\hat{V}\subset V(G)$ and $\hat{C}\subset C(G)$ with $|\hat{V}|,|\hat{C}|\leq L$, $G$ contains a $c,d$-switcher with no vertices in $\bar{V}\cup \hat{V}$ or colours in $\bar{C}\cup \hat{C}$.

We will find $B$ by applying Lemma~\ref{lem-connectinexp}, for which we first choose a subgraph $K\subset H$ and show that $\Delta(K)\leq \alpha \bar{d}$.  For this, let $K\subset H$ be the subgraph of $H$ where $cd\in E(H)$ is an edge of $K$ if the number of $c,d$-switchers in $G$ with order 4 and a vertex in $\bar{V}$ or a colour in $\bar{C}$ is at least $w/2$.

\begin{claim}\label{clm-deltaK}
We have $\Delta(K)\leq \alpha \bar{d}$.
\end{claim}
\begin{proof}[Proof of Claim~\ref{clm-deltaK}]
For each $c\in V(H)$, by Proposition~\ref{prop-switchoverlap}i), there are at most $20n^3\cdot |\bar{C}|$ switchers containing a colour in $\bar{C}$ which switch $c$ into another colour. Similarly, by Proposition~\ref{prop-switchoverlap}ii),
there are at most $50n^3\cdot |\bar{V}|$ switchers containing a vertex in $\bar{V}$ which switch $c$ into another colour.
For each $d\in N_K(c)$, there are at least $w/2$ $c,d$-switchers containing a vertex in $\bar{V}$ or a colour in $\bar{C}$.
Therefore,
\[
d_K(c)\cdot w/2 \leq 20n^3\cdot |\bar{C}|+50n^3\cdot |\bar{V}|\leq 70\eps n^4,
\]
so that the claim follows from \eqref{eqn:bounds} as $\bar{d}w\geq 320\eps n^4/\alpha$.
\renewcommand{\qedsymbol}{$\boxdot$}
\end{proof}
\renewcommand{\qedsymbol}{$\square$}

Let $r=2L$ and $\Delta=\alpha\bar{d}$, so that $H$ is an $(\alpha,\Delta)$-expander with at most $kn$ vertices which satisfies $\delta(H)\geq \bar{d}$ and $\Delta(H)\leq \bar{\Delta}$. By Claim~\ref{clm-deltaK}, we have $\Delta(K)\leq \Delta$. Finally, note that, from \eqref{eqn:bounds},
\[
r=2L\leq \frac{\alpha \bar{d}}{10^6\log^3(kn)},
\]
as $1/n\ll 1/k$. By Lemma~\ref{lem-connectinexp} applied with $V=C\cap \bar{C}$, then, there is a set $B\subset C$ with $C\cap \bar{C}\subset B$ and
\begin{equation}\label{eqn:newBbound}
|B|\leq 10^4|C\cap \bar{C}|\cdot \frac{\bar{\Delta}}{\alpha \bar{d}}\overset{\eqref{eqn:bounds}}{\leq} 10^4\cdot \eps|C|\cdot \frac{1}{\alpha \eta}\leq \eta |C|
\end{equation}
such that, for each distinct $c,d\in C\setminus B$, $H-K-(C\cap \bar{C})=H-K-\bar{C}$ contains $2L$ internally vertex-disjoint $c,d$-paths with length at most $400\log^{3}(kn)\leq 500\log^3n$.

We now show that $B$ has the property required. Note that $|B|\leq \eta |C|$ and, by \eqref{eqn:newBbound}, if $C\cap \bar{C}=\emptyset$ then $B=\emptyset$.
Fix, then, distinct $c,d\in C\setminus B$ and let $\hat{V}\subset V(G)$ and $\hat{C}\subset C(G)$ satisfy $|\hat{V}|,|\hat{C}|\leq L$. We have that $H-K-\bar{C}$ contains a $c,d$-path with length at most $500\log^3n$ which has no internal vertex in $\hat{C}$. Let $P$ be such a path and let $\ell_0=e(P)$. Direct the edges of $P$ from $c$ to $d$, additionally labelling them $e_1,\ldots,e_{\ell_0}$ in order along the path.
For each $i\in [\ell_0]$, label vertices so that $e_i=\vec{c_id_i}$. For each $i=1,\ldots,\ell_0$ in turn, let $S_i$ be a $c_i,d_i$-switcher of order 4 in $G$ with no vertices in $\bar{V}\cup \hat{V}\cup (\cup_{j<i}V(S_j))$ or colours in $\bar{C}\cup \hat{C}\cup V(P)\cup (\cup_{j<i}C(S_j))$.

To see this is possible, note that when we seek the switcher $S_i$, we wish to avoid $|\hat{V}\cup (\cup_{j<i}V(S_j))|\leq L+8\ell_0$ vertices in addition to those in $\bar{V}$ and $|\hat{C}\cup (\cup_{j<i}C(S_j))\cup V(P)|\leq L+4\ell_0$ colours in addition to those in $\bar{C}$. By i) and ii) in Proposition~\ref{prop-switchoverlap}, there are at most
\[
50n^3\cdot (L+8\ell_0)+20n^3\cdot (L+4\ell_0)= 70Ln^3+480\ell_0n^3\leq 100Ln^3+2\cdot 10^5n^3\log^3n\overset{\eqref{eqn:bounds}}{<}w/2
\]
$c_i,d_i$-switchers with a vertex in $\hat{V}\cup (\cup_{j<i}V(S_j))$ or a colour in $\hat{C}\cup (\cup_{j<i}C(S_j))\cup V(P)$. Thus, as $c_id_i\in E(H)\setminus E(K)$, we can select the required switcher $S_i$.

Finally, note that the switchers $S_i$, $i\in [\ell_0]$, combine to form a $c,d$-switcher with no vertex in $\bar{V}\cup \hat{V}$ or colour in $\bar{C}\cup \hat{C}$ with order at most $4\ell_0\leq 2000\log^3n$, as required.
\end{proof}


\subsection{Proof of Theorem~\ref{thm-manyinter}}\label{subsec:manyinterproof}

Following the sketch at the start of this section, we now use the methods we have developed to deal with light, moderately-weighted and heavy edges, to prove Theorem~\ref{thm-manyinter}.

\begin{proof}[Proof of Theorem~\ref{thm-manyinter}]
Let $K$ be such that $\eta\llpoly 1/K\llpoly \xi,\log^{-1}n,1/L$ and let
\begin{equation}\label{eqn:wbounds}
w_0=\frac{\xi n^3}{4k^2}, \;\;\;\; w_1=K n^3, \;\;\text{ and }\;\; w_2=150\eps n^4.
\end{equation}
Let $\mathcal{J}_0$, $\mathcal{J}_1$, $\mathcal{J}_2$ and $\mathcal{J}_3$ be the sets of pairs of distinct colours $c,d\in C(G)$ with, respectively, $w_{cd}\leq w_0$, $w_0<w_{cd}\leq w_1$, $w_1< w_{cd}\leq w_2$ and $w_{cd}>w_2$.

Let
\[
\mathcal{C}_0=\{\{c,d\}:c,d\in C(G),c\neq d\text{ and }\{c,d\}\text{ is }(\eps,\eta,L,\ell)\text{-exchangeable}\}.
\]
By Lemma~\ref{lem-highwexchange}, each $C\in \mathcal{J}_3$ is $(\eps,0,\eps n,4)$-exchangeable, and hence as $\ell\geq 4$ and $K\leq \eps n$, each $C\in \mathcal{J}_3$ is $(\eps,\eta,L,\ell)$-exchangeable, and thus in $\mathcal{C}_0$.
By Lemma~\ref{lem-manyinter-1} applied with $\xi'=\xi/4K$ and $\lambda=\xi/4k^2$, at most $\xi n^2/4K$ sets $C\in \mathcal{J}_1$ are not $(\eps,0,\eps n,12)$-exchangeable, and therefore, similarly, $|\mathcal{J}_1\setminus \mathcal{C}_0|\leq \xi n^2/4K$.
Thus, we have
\begin{equation}\label{eqn-J1}
\sum_{cd\in (\mathcal{J}_0\cup \mathcal{J}_1\cup \mathcal{J}_3)\setminus \mathcal{C}_0}w_{cd}\leq |\mathcal{J}_0|\cdot w_0+|\mathcal{J}_1\setminus \mathcal{C}_0|\cdot w_1\leq \binom{|C(G)|}{2}\cdot \frac{\xi n^3}{4k^2}+\frac{\xi n^2}{4K}\cdot K n^3\leq \frac{\xi n^5}{2}.
\end{equation}

Let $R$ be the graph with vertex set $C(G)$ and edge set $\mathcal{J}_2$, so that, for each $cd\in E(R)$, $w_1<w_{cd}\leq w_2$. Let $r=2\log n$, and note that $w_2\leq nw_1$. Take then integers $\bar{w}_1=w_1<\bar{w}_2<\ldots < \bar{w}_{r+1}=w_2$ such that $\bar{w}_{i+1}\leq 2\bar{w}_i$ for each $i\in [r]$. For each $i\in [r]$, let $R_i\subset R$ be the subgraph of edges $e\in E(R)$ with $\bar{w}_i< w_e\leq \bar{w}_{i+1}$.
For each $i\in [r]$, let
\begin{equation}\label{eqn:diDi}
d_i=\left(\frac{\xi}{8kr}\right)\frac{n^4}{\bar{w}_i}\;\;\text{ and }\;\;\Delta_i=\left(\frac{400kr}{\xi}\right)d_i.
\end{equation}

\begin{claim}\label{clm:DeltaR} For each $i\in [r]$, $\Delta(R_i)\leq \Delta_i$.
\end{claim}
\begin{proof}
For each $c\in C(G)$, as $|G|=n$, by Proposition~\ref{prop-switchoverlap}ii), there are at most $50n^4$ switchers of order 4 which switch $c$ with some other colour $d\in C(G)\setminus \{c\}$. Therefore, for each $i\in [r]$, $\Delta(R_i)\cdot \bar{w}_i\leq 50n^4$, and hence \renewcommand{\qedsymbol}{$\boxdot$}
\[
\Delta(R_i)\leq \frac{50n^4}{\bar{w}_i}\overset{\eqref{eqn:diDi}}{=}\frac{400krd_i}{\xi}\overset{\eqref{eqn:diDi}}{=}\Delta_i.\qedhere
\]
\end{proof}
\renewcommand{\qedsymbol}{$\square$}

Let $\alpha=1/16\log (kn)$. Let $H_1,\ldots,H_{s}$ be a maximal collection of edge-disjoint graphs such that
there are $i_1,\ldots,i_s\in [r]$ for which the following hold for each $j\in [s]$,
\begin{itemize}
\item $H_j\subset R_{i_j}$.
\item $H_j$ is an $(\alpha,\alpha\cdot  d_{i_j})$-expander with $\delta(H_j)\geq d_{i_j}$.
\end{itemize}

For each $i\in [r]$, we now check the conditions on $d_i,\bar{w}_i,\Delta_i$ corresponding to \eqref{eqn:bounds} to then apply Lemma~\ref{lem-exchangefromexpansion} to show that any $V(H_j)$ with $i_j=i$ is $(\eps,\eta,L,\ell)$-exchangeable. First note that, for each $i\in [r]$, as $\eps\ll \xi,\log^{-1}n$, we have
\[
d_{i}\bar{w}_i\overset{\eqref{eqn:diDi}}{=}\frac{\xi n^4}{8kr}\geq \frac{320 \eps n^4}{\alpha}.
\]
Next, for each $i\in [r]$, as $1/K\ll \log^{-1}n,1/L$, we have
\[
\bar{w}_i\geq w_1\overset{\eqref{eqn:wbounds}}{=} Kn^3\geq  10^6Ln^3\log^3n.
\]
Then, as $\eta\ll\xi,\log^{-1}n$, for each $i\in [r]$,
\[
\frac{\Delta_i}{d_i}\overset{\eqref{eqn:diDi}}{=}\frac{400kr}{\xi}\leq \frac{1}{\eta}.
\]
Finally, for each $i\in [r]$, as $\eps\llpoly \xi,\log^{-1}n$, we have
\[
d_i\overset{\eqref{eqn:diDi}}{=}\frac{\xi n^4}{8k\bar{w}_ir}\geq \frac{\xi n^4}{8kw_2r}\overset{\eqref{eqn:wbounds}}{=} \frac{\xi}{1200\eps kr}\geq 10^{10}L\log^4n.
\]

Therefore, for each $j\in [s]$, as $\Delta(H_j)\leq \Delta(R_{i_j})\leq \Delta_{i_j}$ by Claim~\ref{clm:DeltaR}, and, for each $cd\in E(H)$, the  number of $c,d$-switchers of order $4$ in $G$ is at least $\bar{w}_{i_j}$, by Lemma~\ref{lem-exchangefromexpansion}, as $H_j$ is an $(\alpha,\alpha\cdot  d_{i_j})$-expander with $\delta(H_j)\geq d_{i_j}$, we have that $V(H_{j})$ is $(\eps,\eta,L,\ell)$-exchangeable.
Furthermore, by the maximality of $H_1,\ldots,H_{s}$, for each $i\in [r]$, we have that $R_i':=R_i\setminus (H_1\cup\ldots\cup H_{s})$ contains no $(\alpha,\alpha\cdot d_{i})$-expander with minimum degree at least $d_i$, and, hence, by Theorem~\ref{thm-expander}, $R_i'$ has average degree at most $4d_i$. This implies that, for each $i\in [r]$, the edges of $R_i'$ have low total weight, as follows.
\begin{claim}\label{clm:simple} For each $i\in [r]$,
\[
\sum_{cd\in E(R_i')}w_{cd}\leq \frac{\xi n^5}{2r}.
\]
\end{claim}
\begin{proof}[Proof of Claim~\ref{clm:simple}] As $|{R}'_i|\leq |C(G)|\leq kn$ and $d(R_i')\leq 4d_i$, we have
\[
\sum_{cd\in E(R_i')}w_{cd}\leq \frac{d(R_i')}{2}\cdot |R_i'|\cdot \bar{w}_{i+1}\leq 2d_i\cdot kn \cdot 2\bar{w_i}\overset{\eqref{eqn:diDi}}{=} \frac{\xi n^5}{2r},
\]
as required.
\renewcommand{\qedsymbol}{$\boxdot$}
\end{proof}
\renewcommand{\qedsymbol}{$\square$}

Let $\mathcal{C}=\mathcal{C}_0\cup \{V(H_j):j\in [s]\}$. We will show that $\mathcal{C}$ satisfies the property in the theorem. We have immediately that \eref{prop-C-1} holds from our work so far. Let $\mathcal{I}\subset C(G)^{(2)}$ be the set of pairs of colours $\{c,d\}$ which are not $\mathcal{C}$-equivalent. Then,
\begin{align*}
\sum_{e\in \mathcal{I}}w_e&\leq \sum_{i\in \{0,1,2,3\}}\sum_{e\in \mathcal{I}\cap \mathcal{J}_i}w_e\overset{\eqref{eqn-J1}}{\leq} \frac{\xi n^5}{2}+\sum_{e\in \mathcal{I}\cap \mathcal{J}_2}w_e \\
&\leq \frac{\xi n^5}{2}+\sum_{i\in [r]}\sum_{e\in E(R_i')}w_e\overset{\text{Claim~\ref{clm:simple}}}{\leq} \xi n^5.
\end{align*}
and thus \eref{prop-C-3} holds.

It is left then only to prove that \eref{prop-C-2} holds. For each $c\in C(G)$, recalling that $\delta(H_j)\geq d_{i_j}$ for each $j\in [s]$, we have
\begin{align*}
\sum_{j\in [s]:c\in V(H_j)}\;\sum_{d\in N_{H_j}(c)}w_{cd}&\geq \sum_{j\in [s]:c\in V(H_j)} d_{H_j}(c)\cdot \bar{w}_{i_j}\geq \sum_{j\in [s]:c\in V(H_j)} d_{i_j}\cdot \bar{w}_{i_j}\overset{\eqref{eqn:diDi}}{=} \sum_{i\in [s]:c\in V(H_j)}\frac{\xi n^4}{8kr}\\
&\geq |\{C\in \mathcal{C}:c\in C,|C|>2\}|\cdot \frac{\xi n^4}{8kr}.
\end{align*}
On the other hand, as the graphs $H_i$ are edge-disjoint, using Proposition~\ref{prop-switchoverlap}ii) and that $|G|=n$, we have
\[
\sum_{j\in [s]:c\in V(H_j)}\sum_{d\in N_{H_j}(c)}w_{cd}\leq \sum_{d\in C(G)\setminus \{c\}}w_{cd}\leq 50n^4.
\]
Therefore, for each $c\in C(G)$, $|\{C\in \mathcal{C}:c\in C,|C|>2\}|\leq 400kr/\xi$, so that, as $1/n\ll 1/k$ and $r=2\log n$, we have that $|\{C\in \mathcal{C}:c\in C,|C|>2\}|\leq \log^2n/\xi$. Thus, \eref{prop-C-2} holds, completing the proof.
\end{proof}


\section{Switching edges of the same colour}\label{sec:exchangingedges}
As discussed in Section~\ref{subsec:discuss}, we will build absorbers capable of absorbing sets which are the vertex sets of monochromatic matchings. Here, we build towards this by considering matchings with a fixed colour set, $C$, but which can switch between using the vertex set $V\cup V(e)$ and $V\cup V(f)$, where $e$ and $f$ are edges, making an \emph{$e,f$-edge-switcher} (also $e,f$-switcher), defined as follows (see also Figure~\ref{fig:edgeswitcher}).

\begin{defn}\label{defn:edgeexchanger}
Given vertex-disjoint edges $e,f\in E(G)$, and an integer $k$, an \emph{$e,f$-edge-switcher of order $k$}  (often shortened to $e,f$-switcher) is a pair $S=(V,C)$ with $V\subset V(G)\setminus (V(e)\cup V(f))$ and $C\subset C(G)$, such that $|V|=2k-2$, $|C|=k$, and $G[V\cup V(e)]$ and $G[V\cup V(f)]$ both contain an exactly-$C$-rainbow matching.

We say $S$ has vertex set $V(S)=V$ and colour set $C(S)=C$.
\end{defn}

We are interested in pairs of edges $e,f$ which have switchers avoiding any arbitrary set of vertices and colours (which can represent part of our construction). Where we can \emph{robustly} find an $e,f$-switcher like this, we say the two edges are switchable, as follows.

\begin{defn}\label{defn:switchable} Given $\beta>0$ and $k\in \N$, say edges $e,f\in E(G)$ are \emph{$(\beta,k)$-switchable in $G$} if, for any $\bar{V}\subset V(G)$ and $\bar{C}\subset C(G)$ with $|\bar{V}|,|\bar{C}|\leq \beta |G|$, there is an $e,f$-switcher in $G$ with order at most $k$ with no vertices in $\bar{V}$ or colours in $\bar{C}$.
\end{defn}

We now state the main result of this section. In any properly coloured graph $G$ with certain properties (\eref{cond-thmee2} and \eref{cond-thmee0} below, where the latter condition will come from \ref{prop-pseud-abs-prime} in the definition of proper pseudorandomness), we find a subgraph $H\subset G$ of most of the edges of $G$ (see \eref{prop-thmee1} and \eref{prop-thmee2}) such that any pair of edges of the same colour in $H$ are switchable (see \eref{prop-thmee3}).

\begin{theorem}\label{thm-exchangeedges}
 Let $1/n\ll p\leq 1$. Let $1/n\llpoly  \beta\llpoly \alpha, \log^{-1}n$. Let $G$ be an $n$-vertex properly coloured graph satisfying the following properties.
 \stepcounter{propcounter}
 \begin{enumerate}[label = {{\emph{\textbf{\Alph{propcounter}\arabic{enumi}}}}}]
 \item For each $c\in C(G)$, we have $|E_c(G)|\geq pn$.\label{cond-thmee2}
 \item For each pair of edges $e$ and $f$ with the same colour, $c_{ef}$ say, there are at least $pn^2$ triples $(c,d,(M,M'))$ where $c,d\in C(S)\setminus \{c_{ef}\}$ and $(M,M')$ is a $c,d$-switcher of order 4 with $e\in M$ and $f\in M'$.
 \label{cond-thmee0}
\end{enumerate}

Then, there is a subgraph $H\subset G$ such that the following hold.
\begin{enumerate}[label = {{\emph{\textbf{\Alph{propcounter}\arabic{enumi}}}}}]\addtocounter{enumi}{2}
\item At most $\alpha n$ colours appear on $G$ but not $H$.\label{prop-thmee1}
\item Each colour appearing in $H$ has at most $\alpha n$ edges in $G-H$.\label{prop-thmee2}
\item Any $e,f\in E(H)$ with the same colour are $(\beta,4\log^4n)$-switchable.\label{prop-thmee3}
\end{enumerate}
\end{theorem}

To see how we prove Theorem~\ref{thm-exchangeedges}, it is helpful to recall the $e_1,e_2$-edge-switcher discussed in Section~\ref{subsec:discuss} and depicted in Figure~\ref{fig:edgeswitcher}. Here, labelling two edges $e,f\in E(G)$ with colour $c$ as $e=u_1v_1$ and $f=u_2v_2$, we pick   two colours $d,d'\in C(G)$ such that we can let $w_i$ and $x_i$ be the $d$- and $d'$-neighbour of $u_i$ and $v_i$ respectively for each $i\in [2]$. If $w_1x_1$ and $w_2x_2$ have the same colour, $c'$ say, which is not $c$, and $u_1,v_1,w_1,x_1,u_2,v_2,w_2,x_2$ are distinct vertices, then $(\{w_1,x_1,w_2,x_2\},\{d,d,c'\})$ is an $e,f$-edge-switcher (see Figure~\ref{fig:edgeswitcher}). Due to \eref{cond-thmee0}, there will be many potential choices of $d,d'$, except we cannot guarantee that $w_1x_1$ and $w_2x_2$ have the same colour. Instead, we apply Theorem~\ref{thm-manyinter} to get a collection of colour classes $\mathcal{C}$, and show that on average over $e,f,d$ and $d'$ as above we can expect $w_1x_1$ and $w_2x_2$ to have colours, $c_1$ and $c_2$ say, which are $\mathcal{C}$-equivalent.
Using the definition of exchangeability we can then find a $c_1,c_2$-colour-switcher using new colours and vertices, and combine this with $(\{w_1,x_1,w_2,x_2\},\{d,d,c'\})$ in the natural way to get an $e,f$-edge-switcher. The necessity that the $c_1,c_2$-colour-switcher is found without vertices in $\{w_1,x_1,w_2,x_2\}$ or colours in $\{d,d,c'\})$ is the reason for the sets $\hat{C},\hat{V}$ and the parameter $L$ in Definition~\ref{defn:exchange}.

We will prove Theorem~\ref{thm-exchangeedges} throughout this section, keeping our notation structure as we deduce properties for various lemmas before completing the proof in Section~\ref{sec-EE-final}. In Section~\ref{sec-EE-setup}, we will set up constants we will use throughout this section and apply Theorem~\ref{thm-manyinter} to find our colour classes, $\mathcal{C}$. In Section~\ref{sec-EE-collrob}, we prove a collective exchangeability property the classes in $\mathcal{C}$ share (as opposed to the individual exchangeability guaranteed by \emph{\ref{prop-C-1}}).
We define switchable edges in Section~\ref{sec-EE-switch} and show that any pair of switchable edges of the same colour are switchable. In Section~\ref{sec-EE-manyswitch}, we show that most of the edges of $G$ are switchable. Finally, in Section~\ref{sec-EE-final}, we define $H$ and conclude it has the properties required in Theorem~\ref{thm-exchangeedges}, completing the proof of this theorem.


\subsection{Proof of Theorem~\ref{thm-exchangeedges}: set up and colour classes}\label{sec-EE-setup}
Take additional constants $\eps,\eta,\xi$ and $\alpha_0$ satisfying
\[
\beta \llpoly \eps\llpoly \eta\llpoly \xi\llpoly \alpha_0\llpoly \alpha,\log^{-1}n
\]
and let $\ell=\log^4n$ and $L=8$.

For each $\{c,d\}\in C(G)^{(2)}$, let $w_{cd}$ be the number of $c,d$-switchers of order 4 in $G$. Note that, by \emph{\ref{cond-thmee2}}, $|C(G)|\leq n/p$. Using Theorem~\ref{thm-manyinter},
 let $\mathcal{C}$ be a collection of subsets of $C(G)$ satisfying the following properties.

\stepcounter{propcounter}
\begin{enumerate}[label = {{\textbf{\Alph{propcounter}\arabic{enumi}}}}]
\item Each $C\in \mathcal{C}$ is $(\eps,\eta,L,\ell)$-exchangeable.\label{prop-calCC-1}
\item Each colour $c\in C(G)$ appears in at most $\frac{\log^2 n}{\xi}$ of the sets $C\in \mathcal{C}$.\label{prop-calCC-2}
\item If $\mathcal{I}\subset C(G)^{(2)}$ is the set of pairs of colours which are not $\mathcal{C}$-equivalent, then\label{prop-calCC-3}
\begin{equation*}
\sum_{e\in \mathcal{I}}w_e\leq \xi n^5.
\end{equation*}
\end{enumerate}


\subsection{Proof of Theorem~\ref{thm-exchangeedges}: colour classes are collectively robust}\label{sec-EE-collrob}
The exchangeability property of the classes in $\mathcal{C}$ given by \ref{prop-calCC-1} depends only on each set. Here, we turn this into a property the sets share.
\begin{lemma}\label{lem-newlemma}
Given any $\bar{C}\subset C(G)$ and $\bar{V}\subset V(G)$ with $|\bar{C}|,|\bar{V}|\leq 10\beta  n$, there is a set $B\subset C(G)$ with $|B|\leq \xi n$ and $\bar{C}\subset B$ such that the following holds.

For each $C\in \mathcal{C}$ and distinct $c,d\in C\setminus B$, and each $\hat{C}\subset C(G)$ and $\hat{V}\subset V(G)$ with $|\hat{C}|,|\hat{V}|\leq L$, there is a $c,d$-switcher with no vertices in $\bar{V}\cup \hat{V}$ or colours in $\bar{C}\cup \hat{C}$.
\end{lemma}
\begin{proof}
Let $\bar{C}\subset C(G)$ and $\bar{V}\subset V(G)$ with $|\bar{C}|,|\bar{V}|\leq 10\beta  n$ be arbitrary. Let $r=|\mathcal{C}|$ and $\mathcal{C}=\{C_1,\ldots,C_r\}$.
Let $I=\{i\in[r]:|\bar{C}\cap C_i|\geq \eps |C_i|\text{ and }|C_i|>2\}$ and $B_0=\cup_{i\in I}C_i$. For each $i\in I$, we have $|C_i|\leq |\bar{C}\cap C_i|/\eps$, and therefore
\begin{equation}\label{eqn:CB12}
|B_0|=|\cup_{i\in I}C_i|\leq \frac{1}{\eps}\sum_{i\in I}|\bar{C}\cap C_i|\overset{\ref{prop-calCC-2}}{\leq} \frac{1}{\eps}|\bar{C}|\cdot \frac{\log^2n}{\xi} \leq \frac{10\beta n \cdot \log^2 n}{\eps\xi}\leq \frac{\xi n}{2},
\end{equation}
as $\beta\llpoly \eps,\xi,\log^{-1}n$.

For each $i\in [r]\setminus I$ with $|C_i|=2$, let $B_i=\bar{C}\cap C_i$.
For each $i\in [r]\setminus I$ with $|C_i|>2$, using that $|\bar{C}\cap C_i|\leq \eps |C_i|$ and that, by \ref{prop-calCC-1}, $C_i\in \mathcal{C}$ is $(\eps,\eta,L,\ell)$-exchangeable, let $B_i\subset C_i$ satisfy $|B_i|\leq \eta |C_i|$, $\bar{C}\cap C_i\subset B_i$ and the following property.
\stepcounter{propcounter}
\begin{enumerate}[label = {{{\textbf{\Alph{propcounter}}}}}]
\item For  each distinct $c,d\in C_i\setminus B_i$ and each $\hat{C}\subset C(G)$ and $\hat{V}\subset V(G)$ with $|\hat{V}|,|\hat{C}|\leq L$, $G$ contains a $c,d$-switcher with order at most $\ell$ and no vertices in $\bar{V}\cup \hat{V}$ or colours in $\bar{C}\cup \hat{C}$.\label{prop:blah}
\end{enumerate}
Note that, for each $i\in [r]\setminus I$ with $|C_i|=2$ we have that \ref{prop:blah} also holds (trivially, if $B_i\cap C_i\neq \emptyset$ as there are no such distinct $c,d$, and, otherwise, as $C_i$ is $(\eps,\eta,L,\ell)$-exchangeable).
Furthermore, we have
\begin{align}
|\cup_{i\in [r]\setminus I}B_i|&\leq |\bar{C}|+\sum_{i\in [r]\setminus I:|C_i|>2}\eta |C_i|\overset{\ref{prop-calCC-2}}{\leq} 10\beta n+\frac{\eta\cdot \log^2 n}{\xi}\cdot |C(G)|\nonumber
\\
&\overset{\emph{\ref{cond-thmee2}}}{\leq} 10\beta n+\frac{\eta\cdot\log^2 n\cdot n}{\xi\cdot p}\leq \frac{\xi n}{2}.\label{eqn:CB22}
\end{align}
Let $B=B_0\cup(\cup_{i\in [r]\setminus I}B_i)$, so that, by \eqref{eqn:CB12} and \eqref{eqn:CB22}, $|B|\leq \xi n$.

Now, for any $i\in I$, there are no distinct $c,d\in C_i\setminus B$ as $C_i\subset B_0$. For any $i\in [r]\setminus I$, we have by \ref{prop:blah} that, for any distinct $c,d\in C_i\setminus B$, and any  $\hat{C}\subset C(G)$ and $\hat{V}\subset V(G)$ with
 $|\hat{V}|,|\hat{C}|\leq L$, $G$ contains a $c,d$-switcher with order at most $\ell$ and no vertices in $\bar{V}\cup \hat{V}$ or colours in $\bar{C}\cup \hat{C}$. Thus, $B$ has the property we required.
\end{proof}


\subsection{Proof of Theorem~\ref{thm-exchangeedges}: switchable edges}\label{sec-EE-switch}

We now define switchable edges, and show that any two switchable edges of the same colour can be robustly switched. While the definitions can naturally be more general, we define them formally only for the constants $\beta,\ell$ which we have already chosen, as follows.

\begin{defn}\label{defn:switchablesingle}
Say an edge $e$ with colour $c$ is \emph{$(2\beta ,2\ell)$-switchable in $G$} if, for at least $2|E_c(G)|/3$ edges $f\neq e$ with colour $c$, $e$ and $f$ are together $(2\beta,2\ell)$-switchable in $G$.
\end{defn}

In combination with Definition~\ref{defn:switchable}, we have defined both a single edge and a pair of edges to be switchable, and sometimes refer to a pair of edges as \emph{together switchable} to emphasis we have the latter definition, not that each of these edges is switchable on its own.

\begin{lemma}\label{lem-switchgoodpairs} Let $c\in C(G)$ and suppose $e_1$ and $e_2$ are distinct edges in $G$ with colour $c$ which are both $(2\beta ,2\ell)$-switchable. Then, $e_1$ and $e_2$ are together $(\beta,4\ell)$-switchable.
\end{lemma}

\begin{proof}
 Let $\bar{C}\subset C(G)$ and $\bar{V}\subset V(G)$ satisfy $|\bar{V}|,|\bar{C}|\leq \beta n$.
By Definition~\ref{defn:switchable} and \itref{cond-thmee2}, as $e_1,e_2$ are $(2\beta ,2\ell)$-switchable edges with colour $c$, there is some edge $f\notin\{e_1,e_2\}$ with colour $c$ and no vertices in $\bar{V}$ such that, for each $i\in [2]$, given any sets $V'\subset V(G)$ and $C'\subset C(G)$ with $|V'|,|C'|\leq 2\beta  n$, there is an $e_i,f$-switcher in $G$ with order at most $2\ell$ with no vertices in $V'$ or colours in $C'$.

Thus, there is an $e_1,f$-switcher $S_1=(V_1,C_1)$ with vertices not in $V\cup V(e_2)$ and colours not in $C$, and order at most $2\ell$. Similarly, there is an $e_2,f$-switcher $S_2=(V_2,C_2)$ with vertices not in $V\cup V(e_1)\cup V_1$, colours not in $C\cup C_1$, and order at most $2\ell$. Then, note that $(V_1\cup V(f)\cup V_2,C_1\cup C_2)$ is  an $e_1,e_2$-switcher with vertices not in $\bar{V}$ and colours not in $\bar{C}$. Thus, as $\bar{C}\subset C(G)$ and $\bar{V}\subset V(G)$ were chosen arbitrarily, $e_1$ and $e_2$ are together $(\beta,4\ell)$-switchable.
\end{proof}


\subsection{Proof of Theorem~\ref{thm-exchangeedges}: switchable edges are plentiful}\label{sec-EE-manyswitch}
We now show that most of the edges of $G$ are switchable. To do this, we first show for Lemma~\ref{lem:notswitchable} below that if two edges $e,f$ of the same colour are not switchable then they lie in many switchers of order 4 which do not switch between colours that are $\mathcal{C}$-equivalent. We prove this by contradiction, and note that the construction of an $e,f$-switcher at the end of the proof lies in two cases (when $d=d'$ and when $d\neq d'$, for some colours $d,d'$ found in the construction) which directly correlate to the switcher sketched at the start of this section in two cases (there, when $c_1=c_2$ and when $c_1\neq c_2$).

\begin{lemma}\label{lem:notswitchable}
Let $c\in C(G)$ and let $e,f\in E(G)$ be distinct edges with colour $c$ which are not together $(2\beta,2\ell)$-switchable.

Then, there are at least $pn^2/2$ triples $(d,d',(M,M'))$ such that $d,d'\in C(G)\setminus \{c\}$, $d$ and $d'$ are not $\mathcal{C}$-equivalent, and $(M,M')$ is a $d,d'$-switcher of order 4 with $e\in M$ and $f\in M'$.
\end{lemma}
\begin{proof}
Let $c\in C(G)$ and suppose that $e,f\in E(G)$ are distinct edges with colour $c$ which are not together $(2\beta,2\ell)$-switchable. Then, by Definition~\ref{defn:switchable}, there are sets $\bar{V}\subset V(G)$ and $\bar{C}\subset C(G)$ with $|\bar{V}|,|\bar{C}|\leq 2\beta  n$ such that there is no $e,f$-switcher in $G$ with order at most $2\ell$ and no vertices in $\bar{V}$ or colours in $\bar{C}$.

Let $\mathcal{S}$ be the set of triples $(d,d,(M,M'))$ where $d,d'\in C(G)\setminus \{c\}$ and $(M,M')$ is a $d,d'$-switcher of order 4 with $e\in M$ and $f\in M'$.
Note that, by \itref{cond-thmee0}, $|\mathcal{S}|\geq pn^2$. Let $\mathcal{S}'$ be the set of triples $(d,d',S)\in \mathcal{S}$ for which $d$ and $d'$ are $\mathcal{C}$-equivalent. Suppose, for contradiction, that $|\mathcal{S}'|> pn^2/2$, as, otherwise, the triples in $\mathcal{S}\setminus \mathcal{S}'$ satisfy the requirements in the lemma.

Now, using Lemma~\ref{lem-newlemma}, let $B\subset C(G)$ satisfy $|B|\leq \xi n$, $\bar{C}\subset B$ and the following property.
\stepcounter{propcounter}
\begin{enumerate}[label = {{{\textbf{\Alph{propcounter}}}}}]
\item For each $C\in \mathcal{C}$ and distinct $d,d'\in C\setminus B$, and each $\hat{C}\subset C(G)$ and $\hat{V}\subset V(G)$ with $|\hat{V}|,|\hat{C}|\leq L$, there is a $d,d'$-switcher with no vertices in $\bar{V}\cup \hat{V}$
or colours in $\bar{C}\cup \hat{C}$.\label{prop:creak}
\end{enumerate}
\noindent Note that, by Proposition~\ref{prop-switchoverlap}iii) and iv), we have the following.

\begin{itemize}
\item The number of triples $(d,d',S)\in \mathcal{S}'$ where $C(S)\cup\{d,d'\}$ has a colour in $B$ is at most $10^3n|B|$.
\item The number of triples $(d,d',S)\in \mathcal{S}'$ where $S$ has a vertex in $\bar{V}$ is at most $600n\cdot |\bar{V}|$.
\end{itemize}
Thus, recalling that $1/n\ll p$ and $\beta,\xi\ll \log^{-1}n$, as
\[
600n\cdot |\bar{V}|+10^3n\cdot |B|\leq 10^3n\cdot (2\beta  n+ \xi n)<pn^2/2< |\mathcal{S}'|,
\]
there is some $(d,d',S)\in \mathcal{S}'$ such that $V(S)$ has no vertex in $\bar{V}$ and $C(S)\cup\{d,d'\}$ has no colour in $B$.

If $d=d'$, then, letting $S=(M,M')$, $M$ and $M'$ are two matchings of order 4 with the same vertex set and the same colour set, and with $e\in M$ and $f\in M'$. Thus, $(V(S)\setminus (V(e)\cup V(f)),C(S)\cup\{d\})$ is an $e,f$-switcher with no vertices in $\bar{V}$ and no colours in $\bar{C}$ and order 4, a contradiction to the choice of $\bar{V}$ and $\bar{C}$.

Therefore, we must have $d\neq d'$. As $(d,d',S)\in \mathcal{S}'$, $d$ and $d'$ are $\mathcal{C}$-equivalent and thus there is some $C\in \mathcal{C}$ such that $d,d'\in C$. As $d,d'\notin B$ and $|V(S)|,|C(S)|\leq 8= L$, by \ref{prop:creak} there is a $d,d'$-switcher, $S'$ say, of order at most $\ell$ with no vertices in $\bar{V}\cup V(S)$ or colours in $\bar{C}\cup C(S)$.
Note that $((V(S)\cup V(S'))\setminus (V(e)\cup V(f)),(C(S)\cup C(S')\setminus \{c\})\cup \{d,d'\})$ is an $e,f$-switcher of order at most $2\ell$ with no vertices in $\bar{V}$ or colours in $\bar{C}$, again contradicting the choice of $\bar{V}$ and $\bar{C}$.
\end{proof}

We now show that Lemma~\ref{lem:notswitchable} and \ref{prop-calCC-3} imply that most of the edges of $G$ are switchable, as follows.

\begin{lemma}\label{lem-mostedgesgood} All but at most $\alpha_0 n^2$ edges of $G$ are $(2\beta ,2\ell)$-switchable.
\end{lemma}
\begin{proof}
Let $F$ be the set of edges of $G$ which are not $(2\beta ,2\ell)$-switchable. For each $e\in F$, let $F_e$ be the set of edges $f\neq e$ in $G$ with the same colour as $e$ for which there is a set, $\mathcal{S}_{e,f}$ say, of at least $pn^2/2$ triples $(d,d',(M,M'))$ such that $d$ and $d'$ are not $\mathcal{C}$-equivalent and $(M,M')$ is a $d,d'$-switcher of order 4 with $e\in M$ and $f\in M'$.
By Definition~\ref{defn:switchable} and Lemma~\ref{lem:notswitchable}, we have, for each $e\in F$, that $|F_e|\geq |E_c(G)|/3-1$ so that, by \itref{cond-thmee2}, $|F_e|\geq pn/4$. Thus,
\begin{equation}\label{lowbound}
\sum_{e\in F}\sum_{f\in F_e}|\mathcal{S}_{e,f}|\geq |F|\cdot \frac{pn}{4}\cdot \frac{pn^2}{2}.
\end{equation}

Note that, given any triple $(d,d',(M,M'))$ where $(M,M')$ is a $d,d'$-switcher of order 4, there are at most 4 pairs of edges $e\in M$ and $f\in M'$ such that $e$ and $f$ have the same colour. Thus, if $\mathcal{I}\subset C(G)^{(2)}$ is the set of pairs $\{d,d'\}$ which are not $\mathcal{C}$-equivalent, then,
\[
\sum_{e\in F}\sum_{f\in F_e}|\mathcal{S}_{e,f}|\leq 8\sum_{dd'\in \mathcal{I}}w_{dd'}\overset{\ref{prop-calCC-3}}{\leq} 8\xi n^5.
\]
Therefore, in combination  with \eqref{lowbound}, and as $\xi\llpoly \alpha_0$ and $1/n\ll p$, we have
\[
|F|\leq \frac{8\cdot 8\xi n^5}{p^2n^3}\leq \alpha_0 n^2,
\]
as required.
\end{proof}


\subsection{Proof of Theorem~\ref{thm-exchangeedges}: choosing $H$}\label{sec-EE-final}
Finally in this section, we can now define the graph $H\subset G$ required in Theorem~\ref{thm-exchangeedges}. To do this, simply let $H$ be the graph with $V(H)=V(G)$ whose edges are the $(2\beta,2\ell)$-switchable edges of colours which have at most $\alpha n$ edges which are not $(2\beta,2\ell)$-switchable in $G$. Then, \eref{prop-thmee2} holds directly from this definition, while \eref{prop-thmee3} holds by this definition and Lemma~\ref{lem-switchgoodpairs}, as $\ell=\log^4n$.

As every colour that does not appear on $H$ has at least $\alpha n$ edges which are not $(2\beta,2\ell)$-switchable in $G$, by Lemma~\ref{lem-mostedgesgood}, the number of such colours is at most
 \[
\frac{\alpha_0 n^2}{\alpha n}\leq \alpha n,
 \]
as $\alpha_0\llpoly \alpha$. Thus, \itref{prop-thmee1} holds, as required, completing the proof of Theorem~\ref{thm-exchangeedges}.


\section{Switching edges with colours from the same class}\label{sec:exchangingedgesinclass}
In our last section, given any graph $G$ satisfying some mild conditions (\eref{cond-thmee2} and \eref{cond-thmee0}), we found a subgraph $H\subset G$ containing most of edges of most the colours of $G$ in which any pair of edges with the same colour were switchable (\eref{prop-thmee1}--\eref{prop-thmee3}). In this relatively short section, using the same conditions (\eref{cond-thmeec-2} and \eref{cond-thmeec-0}), we use Theorem~\ref{thm-exchangeedges} to find this subgraph $H$, and then apply Theorem~\ref{thm-manyinter} to $H$. This allows us to find a collection $\mathcal{C}$ of colour classes such that any pair of edges in $H$ with $\mathcal{C}$-equivalent colours are exchangeable (see \emph{\ref{prop-thmeec-3}} below),
 while most of the colour switchers in $H$ switch colours which are $\mathcal{C}$-equivalent (see \emph{\ref{prop-thmeec-4}} below). This results in Theorem~\ref{thm-exchangeedges-inclasses}, which is the main result of this and the last two sections.

In the original graph $G$, when we found a colour class $C$ which was exchangeable and chose distinct colours $c,d$ in $C$, we could robustly find a $c,d$-colour-switcher $(M,M')$ in $G$, but doing so while avoiding colours in $C$ was a delicate business (requiring the set $B$ in Definition~\ref{defn:exchange}). Key to the proof of Theorem~\ref{thm-exchangeedges-inclasses} is that if $S=(M,M')$ is a $c,d$-colour-switcher in $H$, then we can use this to create a $c,d$-colour-switcher avoiding colours in any relatively small set $\bar{C}$. Indeed, for each colour $c'\in C(S)$ if $e\in M$ and $f\in M'$ are distinct edges with colour $c'$, we can find an $e,f$-edge-switcher in $G$ with no colours in $\bar{C}$ or vertices in $V(M\cup M')$ (as $e,f\in E(H)$), and use this $e,f$-edge-switcher to cover $V(e)$ or $V(f)$ instead of the choice of the colour-$c'$ edge from $\{e,f\}$. This gives a $c,d$-colour-switcher in which $c'$ is not used, carrying this out for  each pair of edges in $M\cup M'$ with the same colour allows us to construct a $c,d$-colour-switcher avoiding all the colours in $C(S)$ while not using any additional colours in $\bar{C}$.

\begin{theorem}\label{thm-exchangeedges-inclasses}
 Let $1/n\ll p\leq 1$. Let $1/n\llpoly  \beta\llpoly \alpha, \log^{-1}n$. Let $G$ be an $n$-vertex properly coloured graph satisfying the following properties.
 \stepcounter{propcounter}
 \begin{enumerate}[label = {{\emph{\textbf{\Alph{propcounter}\arabic{enumi}}}}}]
 \item For each $c\in C(G)$, we have $|E_c(G)|\geq pn$.\label{cond-thmeec-2}
 \item For each pair of edges $e$ and $f$ with the same colour, $c_{ef}$ say, there are at least $pn^2$ triples $(c,d,(M,M'))$  where $c,d\in C(S)\setminus \{c_{ef}\}$ and $(M,M')$  is a $c,d$-switcher of order 4 with $e\in M$ and $f\in M'$.
 \label{cond-thmeec-0}
\end{enumerate}

Then, there is a subgraph $H\subset G$ and a collection $\mathcal{C}$ of subsets of $C(H)$ satisfying the following properties.
\begin{enumerate}[label = {{\emph{\textbf{\Alph{propcounter}\arabic{enumi}}}}}]\addtocounter{enumi}{2}
\item At most $\alpha n$ colours appear on $G$ but not $H$.\label{prop-thmeec-1}
\item Each colour appearing in $H$ has at most $\alpha n$ edges in $G-H$.\label{prop-thmeec-2}
\item Any distinct $e,e'\in E(H)$ with $\mathcal{C}$-equivalent colours are $(\beta,32\log^8n)$-switchable in $G$.\label{prop-thmeec-3}
\item If $\mathcal{I}\subset C(H)^{(2)}$ is the set of non-$\mathcal{C}$-equivalent colour pairs, and, for each $\{c,d\}\in C(H)^{(2)}$, $w_{cd}$ is the number of $c,d$-switchers of order 4 in $H$, then\label{prop-thmeec-4}
\begin{equation*}
\sum_{cd\in \mathcal{I}}w_{cd}\leq \alpha n^5.
\end{equation*}
\end{enumerate}
\end{theorem}

\begin{proof}
Take $\eta$ satisfying $\beta\llpoly \eta \llpoly \alpha,\log^{-1}n$, and let $\ell=4\log^4n$.
By \itref{cond-thmeec-2}, \itref{cond-thmeec-0} and Theorem~\ref{thm-exchangeedges} with $\beta'=2\beta$ and $\alpha'=\alpha$, take a subgraph $H\subset G$ such that \itref{prop-thmeec-1} and \itref{prop-thmeec-2} are satisfied and the following holds.
\stepcounter{propcounter}
\begin{enumerate}[label = {{{\textbf{\Alph{propcounter}\arabic{enumi}}}}}]
\item Any edges $e,f\in E(H)$ with the same colour are $(2\beta,\ell)$-switchable in $G$. \label{prop-thmee3forH0new}
\end{enumerate}

For each $\{c,d\}\in C(G)^{(2)}$, let ${w}_{cd}$ be the number of $c,d$-switchers of order 4 in $H$. Using Theorem~\ref{thm-manyinter} with $\xi=\alpha$,
 let $\mathcal{C}$ be a collection of subsets of $C(H)$ such that \itref{prop-thmeec-4} is satisfied and the following holds.

\begin{enumerate}[label = {{\textbf{\Alph{propcounter}\arabic{enumi}}}}]\addtocounter{enumi}{1}
\item Each $C\in \mathcal{C}$ is $(2\beta,\eta,0,\ell)$-exchangeable.\label{prop-calCC-1-forH0new}
\end{enumerate}

It is left then only to show that \itref{prop-thmeec-3} holds. For this, let $e$ and $e'$ be edges of $H$ with colour $c$ and $c'$, respectively, such that $c$ and $c'$ are $\mathcal{C}$-equivalent. We will show that $e$ and $e'$ are $(\beta,2\ell^2)$-switchable in $G$.
Note that, if $c=c'$, then $e$ and $e'$ are $(2\beta,\ell)$-switchable in $G$ by \ref{prop-thmee3forH0new}, and hence $(\beta,2\ell^2)$-switchable in $G$. Suppose then that $c\neq c'$, so that there is some set $C\in \mathcal{C}$ with $c,c'\in C$.

Let $\bar{V}\subset V(G)$ and $\bar{C}\subset C(G)$ be arbitrary sets satisfying $|\bar{V}|,|\bar{C}|\leq \beta n$.
By \ref{prop-calCC-1-forH0new}, we have that $C$ is $(2\beta,\eta,0,\ell)$-exchangeable in $H$. Therefore, as $|V(e)\cup V(e')\cup \bar{V}|\leq 4+\beta n\leq 2\beta n$, by Definition~\ref{defn:exchange} (applied with $\bar{C}'=\emptyset$ so that the property given holds for $B=\emptyset$),
$H$ contains a $c,c'$-switcher, $(M,M')$ say, of order at most $\ell$ with no vertices in $V(e)\cup V(e')\cup \bar{V}$.
Let $r=|M|-1<\ell$, and let $d_1,\ldots,d_r$ be the colours in $C(M)\setminus \{c\}=C(M')\setminus \{c'\}$. For each $i\in [r]$, let $f_i$ be the edge in $M$ with colour $d_i$ and let $f_i'$ be the edge in $M'$ with colour $f_i$ (noting that $f_i'\neq f_i$ as by the definition of a $c,d$-switcher $M\cup M'$ is a union of 4-cycles). Let $f$ be the edge in $M$ with colour $c$ and let $f'$ be the edge in $M'$ with colour $c'$, and note that $M=\{f,f_1,\ldots,f_r\}$ and $M'=\{f',f_1',\ldots,f_r'\}$.

Note that, for each $i\in [r]$, $f_i$ and $f_i'$ are edges with the same colour in $H$, and therefore are $(2\beta,\ell)$-switchable by \ref{prop-thmee3forH0new}. Greedily, then, for each $i\in [r]$, find a $f_i,f_i'$-switcher, $(V_i,C_i)$, in $H$ with order at most $\ell$ and with no vertices in $V(e)\cup V(e')\cup \bar{V}\cup V(M)\cup(\cup_{j<i}V_j)$ or colours in $\bar{C}\cup (\cup_{j<i}C_j)$. Note that this is possible as, for each $i\in [r]$,
\begin{equation}\label{eqn:example}
|V(e)\cup V(e')\cup \bar{V}\cup V(M)\cup(\cup_{j<i}V_j)|\leq 4+\beta n+2\ell+\ell\cdot 2\ell\leq 2\beta n,
\end{equation}
and $|\bar{C}\cup (\cup_{j<i}C_j)|\leq \beta n+\ell\cdot\ell\leq 2\beta n$.

Now, $e$ and $f$ both have colour $c$ and are in $E(H)$, so therefore, by \ref{prop-thmee3forH0new} and similar calculations to \eqref{eqn:example}, we can find an $e,f$-switcher $(\hat{V},\hat{C})$ with no vertices in $V(e)\cup V(e')\cup \bar{V}\cup V(M)\cup(\cup_{i\in [r]}V_i)$ or colours in $\bar{C}\cup (\cup_{i\in [r]}C_i)$.
Similarly, as $e'$ and $f'$ both have colour $c'$ and are in $E(H)$, we can find an $e',f'$-switcher $(\hat{V}',\hat{C}')$ in $G$ with no vertices in $V(e)\cup V(e')\cup \bar{V}\cup V(M)\cup(\cup_{i\in [r]}V_j)\cup \hat{V}$ or colours in $\bar{C}\cup (\cup_{i\in [r]}C_j)\cup \hat{C}$.

Letting $\tilde{V}=V(M)\cup(\cup_{i\in [r]}V_i)\cup \hat{V}\cup \hat{V}'$ and $\tilde{C}=\cup_{i\in [r]}C_i)\cup \hat{C}\cup \hat{C}'$, we can then observe that $(\tilde{V},\tilde{C})$ is
 an $(e,e')$-switcher in $G$ with no vertex in $\bar{V}$ or colour in $\bar{C}$ and with order at most $2\ell^2$. Indeed, for example, there is an exactly-$C_i$-rainbow matching with vertex $V(f_i)\cup V_i$ for each $i\in [r]$, a perfectly $\hat{C}$-rainbow matching with vertex set $\hat{V}\cup V(f)$, and an exactly-$\hat{C}'$-rainbow matching with
 vertex set $\hat{V}'\cup V(e')$. As $M=\{f,f_1,\ldots,f_r\}$, these matchings combine to give an exactly-$\tilde{C}$-rainbow matching with vertex set $\tilde{V}\cup V(e')$. If, instead, these matchings cover sets $V(f_i')\cup V_i$ for each $i\in [r]$, $\hat{V}\cup V(e)$ and $\hat{V}'\cup V(f')$, as $M'=\{f',f_1',\ldots,f_r'\}$ and $V(M)=V(M')$, we can combine them to give an exactly-$\tilde{C}$-rainbow matching with vertex set $\tilde{V}\cup V(e)$. Thus, $(\tilde{V},\tilde{C})$ is
the required $(e,e')$-switcher, completing the proof.
\end{proof}


\section{Absorption structure}\label{sec:absorbingedges}
In any properly coloured graph $G$ satisfying certain conditions (see \eref{cond-thmae2} and \eref{cond-thmae0} below) we now find a subgraph $H\subset G$ using most of the edges and most of the colours of $G$ (see \eref{prop-thmae1} and \eref{prop-thmae2}) such that, for any set $E$ of a certain number of edges of the same colour in $H$, we can robustly find an absorber capable of absorbing any \emph{vertex set} of the right number of edges from $E$ (see \eref{prop-thmae3}). This will  give us the following, which is the main result of this section.

\begin{theorem}\label{thm-absorbedges}
 Let $1/n\ll p\leq 1$ and $1/n\llpoly \gamma\llpoly \beta\llpoly \alpha, \log^{-1}n$. Let $G$ be a properly coloured $n$-vertex graph satisfying the following properties.
 \stepcounter{propcounter}
 \begin{enumerate}[label = {{\emph{\textbf{\Alph{propcounter}\arabic{enumi}}}}}]
 \item For each $c\in C(G)$, we have $|E_c(G)|\geq pn$.\label{cond-thmae2}
 \item For each edge $e$, with colour $c$ say, for all but at most ${n}^{2/3}$ edges $f\neq e$ with colour $c$ the following holds. There are at least $pn^2$ triples $(d,d',(M,M'))$ where $d,d'\in C(S)\setminus \{c\}$, $(M,M')$ is a $d,d'$-switcher of order 4 with $e\in M$ and $f\in M'$, such that $e$, $f$, and the edges of $M\cup M'$ with colour in $\{d,d'\}$ form a matching.
 \label{cond-thmae0}
\end{enumerate}

Then, there is a subgraph $H\subset G$ such that the following hold.
\begin{enumerate}[label = {{\emph{\textbf{\Alph{propcounter}\arabic{enumi}}}}}]\addtocounter{enumi}{2}
\item At most $\alpha n$ colours appear on $G$ but not $H$.\label{prop-thmae1}
\item Each colour appearing in $H$ has at most $\alpha n$ edges in $G-H$.\label{prop-thmae2}
\item Given any $m_0,m_1\in\N$ with $m_0\leq m_1\leq \gamma n$, and any monochromatic set $E\subset E(H)$ with $|E|=m_1$, and any sets $\bar{V}\subset V(G)$ and $\bar{C}\subset C(G)$ with $|\bar{V}|,|\bar{C}|\leq \beta n$, there are sets $\tilde{V}\subset V(G)\setminus (\bar{V}\cup V(E))$ and $\tilde{C}\subset C(G)\setminus \bar{C}$, such that $|\tilde{V}|=2|\tilde{C}|-2m_0\leq \beta n$ and, given any $E'\subset E$ with $|E'|=m_0$, there is an exactly-$\tilde{C}$-rainbow matching in $G[\tilde{V}\cup V(E')]$.\label{prop-thmae3}
\end{enumerate}
\end{theorem}

To prove Theorem~\ref{thm-absorbedges}, we start by applying Theorem~\ref{thm-exchangeedges-inclasses} to get a large subgraph $H_0\subset G$ and a collection of colour classes $\mathcal{C}$ such that any pair of edges of $H_0$ with $\mathcal{C}$-equivalent colours are switchable. The graph $H\subset H_0$ is then chosen as the edges of $H_0$ whose colour is `good' for constructing absorbers for an arbitrary set $E$ of edges of that colour for property \eref{prop-thmae3}.
As discussed in Section~\ref{subsec:discuss}, using distributive absorption, we can robustly build an absorber of this nature for $E$, a set of edges with colour $c$ say, if, roughly speaking (i.e., ignoring the role of $E'$ in Section~\ref{sec:abs}), for any set $\hat{E}\subset E$ of 100 edges, we can robustly find a small `absorber' capable of absorbing the vertex set of any 1 edge from $\hat{E}$. However, if we can do this not for an arbitrary set $\hat{E}$ but for some specific set $\hat{E}'$ of 100 edges with colour $c$ which are not in $E$, then, using that we can robustly construct $e,f$-switchers for any $e\in \hat{E}$ and $f\in \hat{E}'$, we can use an absorber for $\hat{E}'$ to construct one for $\hat{E}$. Therefore, we seek to find sets $\hat{E}'$ of 100 edges with colour $c$ and matching absorbers, so that these sets of edges and absorbers are all vertex disjoint.

For this, we use essentially the same construction as sketched out in Section~\ref{sec:abs} (and depicted in Figure~\ref{fig:abs}). We look for vertex-disjoint 4-cycles $u_iv_iw_ix_iu_i$, $i\in [100]$, in $H_0$ with edge colours $c,d,c'_i,d'$ in that order, for some $d,d'$ and $c_i'$, $i\in [100]$, and an edge $wx$ such that the colour of $wx$ is $\mathcal{C}$-equivalent to each colour $c_i'$, $i\in [100]$. This allows us to find, for each $i\in [100]$, a $w_ix_i,wx$-switcher $(\hat{V}_i,\hat{C}_i)$ using new vertices and colours, using the property of $H_0$ that pairs of edges with $\mathcal{C}$-equivalent colours are switchable.
Then, as discussed in Section~\ref{sec:abs}, $(\{w,x\}\cup (\cup_{i\in [100]}\hat{V}_i),\{d,d'\}\cup(\cup_{i\in [100]}\hat{C}_i))$ can `absorb' $\{u_i,v_i\}$ for any one edge $u_iv_i$, $i\in [100]$. The existence of these cycles $u_iv_iw_ix_iu_i$, $i\in [100]$, and the associated edge $wx$ are closely linked to colour switchers in $H_0$ of order 4, so we use the property of $H_0$ that most of the colour switchers of order 4 switch between $\mathcal{C}$-equivalent colours to show that most colours have many vertex-disjoint sets of these cycles. More specifically, we do this via the notion of a `good' colour (see Definition~\ref{defn:goodcolour}).

We will prove Theorem~\ref{thm-absorbedges} from Section~\ref{sec-AE-setup} to Section~\ref{sec-AE-absorber3}, before deducing Theorem~\ref{thm:RSBabsorption} (our main absorption structure theorem) in Section~\ref{sec-AE-final}.
In Section~\ref{sec-AE-setup}, we will set up the proof of Theorem~\ref{thm-absorbedges} and apply Theorem~\ref{thm-exchangeedges-inclasses} to get the subgraph $H_0\subset G$. In Section~\ref{sec-AE-good} we will define which colours are \emph{good}, show that most colours are good, and choose $H\subset H_0$ as the subgraph of edges with a good colour. In Section~\ref{sec-AE-absorber1}, we show that if a colour is good then we can robustly find absorbers for sets of 100 edges of that colour, by finding 4-cycles as sketched above. In Section~\ref{sec-AE-absorber2}, we use this to show that if a colour is good then we can robustly find absorbers for \emph{specific} sets of 100 edges of that colour. In Section~\ref{sec-AE-absorber3}, we use distributive absorption to build this into a global absorption property for sets of edges with the same, good, colour, completing the proof of Theorem~\ref{thm-absorbedges}. Finally, we then deduce Theorem~\ref{thm:RSBabsorption} in Section~\ref{sec-AE-final}.


\subsection{Proof of Theorem~\ref{thm-absorbedges}: set up and application of Theorem~\ref{thm-exchangeedges-inclasses}}\label{sec-AE-setup}
Let $1/n\ll p\leq 1$, and, taking additional variables $\eps,\eta,\xi$ in addition to those in the statement of Theorem~\ref{thm-absorbedges}, let
\[
1/n\llpoly \gamma\llpoly \beta \llpoly\eps\llpoly \eta\llpoly \xi\llpoly \alpha,\log^{-1}n.
\]
Let $G$ be a properly coloured $n$-vertex graph satisfying \eref{cond-thmae2} and \eref{cond-thmae0}. Thus, we wish to find a subgraph $H\subset G$ for which \eref{prop-thmae1}--\eref{prop-thmae3} hold. For this, let $\ell=32\log^8n$, and, using Theorem~\ref{thm-exchangeedges-inclasses}, take a subgraph $H_0\subset G$ and a collection $\mathcal{C}$ of subsets of $C(H_0)$ such that the following hold.
\stepcounter{propcounter}
\begin{enumerate}[label = {{{\textbf{\Alph{propcounter}\arabic{enumi}}}}}]
\item At most $\xi n$ colours appear on $G$ but not $H_0$.\label{prop-thmee1forH0}
\item Each colour appearing in $H_0$ has at most $\xi n$ edges in $G-H_0$.\label{prop-thmee2forH0}
\item Any distinct edges $e,f\in E(H_0)$ whose colours are $\mathcal{C}$-equivalent are $(2\eps,\ell)$-switchable. \label{prop-thmee3forH0}
\item If $\mathcal{I}\subset C(H_0)^{(2)}$ is the set of non-$\mathcal{C}$-equivalent pairs of colours in $C(H_0)$, and, for each $\{c,d\}\in C(G)^{(2)}$, $w_{cd}$ is the number of $c,d$-switchers of order 4 in $H_0$ then\label{prop-calCC-3-forH0}
\begin{equation*}\label{eqn-wesum-sec6}
\sum_{cd\in \mathcal{I}}w_{cd}\leq \xi n^5.
\end{equation*}
\end{enumerate}


\subsection{Proof of Theorem~\ref{thm-absorbedges}: good colours and choosing $H$}\label{sec-AE-good}
We will define certain colours of $H_0$ to be \emph{$p$-good}, as follows.
\begin{defn}\label{defn:goodcolour}
A colour $c\in C(H_0)$ is \emph{$p$-good} if there are at most $p^3 n^4/100$ distinct triples $(d,d',(M,M'))$ such that $d,d'\in C(G)\setminus\{c\}$, $(M,M')$ is a $d,d'$-switcher of order 4  in $G$ with $c\in C(S)$ and, either

i) $M\cup M'$ contains an edge in $G- H_0$, or

ii) $d$ and $d'$ are not $\mathcal{C}$-equivalent.
\end{defn}

We now show that few colours are not $p$-good.

\begin{lemma}\label{lem-mostcoloursgood} There are at most $\alpha n/2$ colours $c\in C(H_0)$ which are not $p$-good. \end{lemma}
\begin{proof}
Note that, by \eref{cond-thmae2}, \ref{prop-thmee1forH0} and \ref{prop-thmee2forH0}, we have $e(G-H_0)\leq (n/p)\cdot \xi n+\xi n^2\leq 2\xi n^2/p$. Thus,
by Proposition~\ref{prop-switchoverlap}v), there are at most $e(G-H_0)\cdot 100n^3\leq 200\xi n^5/p$ triples $(d,d',(M,M'))$ such that $(M,M')$ is a $d,d'$-switcher of order 4  in $G$ with $(M\cup M')\cap E(G-H_0)\neq \emptyset$.
By \ref{prop-calCC-3-forH0}, there are at most $2\xi n^5$ triples $(d,d',(M,M'))$ such that $(M,M')$ is a $d,d'$-switcher of order 4  in $H_0$ and $d$ and $d'$ are not $\mathcal{C}$-equivalent.

Thus, there are at most $300\xi n^5/p$ triples $(d,d',(M,M'))$ such that $(M,M')$ is a $d,d'$-switcher of order 4  in $G$ and either i) $M\cup M'$ contains an edge in $G-H_0$ or ii) $d$ and $d'$ are not $\mathcal{C}$-equivalent. For each such triple $(d,d',(M,M'))$, there are 3 colours in $C(M)\setminus \{d\}$. Therefore, the number of colours which are not $p$-good is at most
\[
\frac{3\cdot 300 \xi n^5/p}{p^3n^4/100}\leq \frac{\alpha n}{2},
\]
as $1/n\ll p$ and $\xi\ll\alpha,\log^{-1}n$.
\end{proof}


\subsection{Proof of Theorem~\ref{thm-absorbedges}: choosing $H$ and finding local 1-in-100 absorbers}\label{sec-AE-absorber1}
We can now define the graph $H\subset G$ required in Theorem~\ref{thm-absorbedges}. To do this, simply let $H$ be the graph with $V(H)=V(G)$ whose edges are the edges of $H_0$ with a $p$-good colour.
Note that \itref{prop-thmae2} holds directly from \ref{prop-thmee2forH0}, and \itref{prop-thmae1} holds from \ref{prop-thmee1forH0} and Lemma~\ref{lem-mostcoloursgood}. Thus, it is left only to show that \itref{prop-thmae3} holds.

We start by showing that $p$-good colours robustly have sets of 100 edges and matching absorbers that can absorb any one of the vertex sets of these edges. Such an absorber is defined formally as follows, after which we show the robust existence of such a set of edges along with an absorber. The absorber constructed is illustrated, with different notation, in Figure~\ref{fig:abs}.

\begin{defn}\label{defn:Eabs} Given a set of edges $E\subset E(H_0)$, and sets $C\subset C(G)$ and $V\subset V(G)\setminus V(E)$, we say $(V,C)$ is an \emph{$E$-absorber of order $k$ in $G$} if $|V|=2k-2$, $|C|=k$, and, for each $e\in E$, there is an exactly-$C$-rainbow matching in $G[V\cup V(e)]$.
\end{defn}

\begin{lemma}\label{lem:arblocab} For each $p$-good colour $c\in C(H_0)$ and each set $\bar{C}\subset C(G)$ and $\bar{V}\subset V(G)$ with $|\bar{C}|,|\bar{V}|\leq \eps n$, there is a set $E\subset E_c(H_0)$ of 100 edges such that $V(E)\cap \bar{V}=\emptyset$ and sets $V\subset V(G)\setminus (\bar{V}\cup V(E))$ and $C\subset C(G)\setminus \bar{C}$ such that $({V},{C})$ is a $E$-absorber of order at most $200\ell$.
\end{lemma}
\begin{proof}
 Fix a $p$-good colour $c\in C(G)$ and sets $\bar{C}\subset C(G)$ and $\bar{V}\subset V(G)$ with $|\bar{C}|,|\bar{V}|\leq \eps n$.
Let $\mathcal{S}_0$ be the set of triples $(d,d',(M,M'))$ where $d,d'\in C(M\cup M')\setminus \{c\}$, $(M,M')$ is a $d,d'$-switcher of order 4, and $M$ and $M'$ both contain colour-$c$ edges whose neighbouring edges in $M\cup M'$ do not have colour in $\{d,d'\}$. By \itref{cond-thmae0}, for each $e\in E_c(H_0)$, for all but at most $n^{2/3}$ $f\in E_c(H_0)\setminus \{e\}$, there are at least $pn^2$ triples
$(d,d',(M,M'))\in \mathcal{S}_0$ with $e\in M$ and $f\in M'$, such that $e$, $f$, and the edges of $M\cup M'$ with colour in $\{d,d'\}$ form a matching.
By Proposition~\ref{prop-switchoverlap}iii) and iv), for each $e,f$, the number of such triples $(d,d',(M,M'))\in \mathcal{S}$ with $V(M)\cap \bar{V}\neq \emptyset$ or $C(M\cup M')\cap (\bar{C}\setminus \{c\})\neq \emptyset$ is at most $600n\cdot |\bar{V}|+10^3n\cdot |\bar{C}|\leq pn^2/2$, where we have used that $1/n\ll p$ and $\eps\llpoly \log^{-1}n$.
Finally, note that, by \ref{prop-thmee2forH0} and \itref{cond-thmae2}, and as $1/n\ll p$ and $\xi\llpoly \log^{-1}n$, there are at least $pn/2$ edges with colour $c$ in $H_0$ as $c\in C(H_0)$.

Therefore, as $c$ is $p$-good, if we let $\mathcal{S}_1\subset \mathcal{S}_0$ be the set of triples $(d,d',(M,M'))\in \mathcal{S}_0$ such that
\begin{itemize}
\item $V(M)\cap \bar{V}\neq \emptyset$ and $C(M\cup M')\cap (\bar{C}\cup \{c\})\neq \emptyset$,
\item $d$ and $d'$ are $\mathcal{C}$-equivalent, and
\item $M\cup M'\subset E(H_0)$,
\end{itemize}
then we have that, as $|E_c(H_0)|\geq pn/2$,
\[
|\mathcal{S}_1|\geq \frac{|E_c(H_0)|\cdot (|E_c(H_0)|-n^{2/3})}{2}\cdot \left(pn^2-\frac{pn^2}{2}\right)-\frac{p^3n^4}{100}\geq \frac{p^2n^2}{16}\cdot \frac{pn^2}{2}-\frac{p^3n^4}{100}\geq \frac{p^3n^4}{50}.
\]

Now, for each $(d,d',(M,M'))\in \mathcal{S}_1$, we have that $c,d\in C(M)$, and $|C(M)\setminus \{c,d\}|=2$. Therefore, as $|C(G)|\leq n/p$ by \itref{cond-thmae2}, there is some $c_1,c_2\in C(G)\setminus \{c\}$ and $e\in E_c(G)$ for which there are at least $p^5n/50$ triples $(d,d',(M,M'))\in \mathcal{S}_1$ with $C(M)=\{c,d,c_1,c_2\}$ and $e\in M$. Fix such $c_1,c_2\in C(G)\setminus \{c\}$ and $e\in E_c(G)$, and say $\mathcal{S}_2$ is the set of such triples, so that $|\mathcal{S}_2|\geq p^5n/50$.

Now, let $r=|\mathcal{S}_2|$ and enumerate the triples in $\mathcal{S}_2$ as $(d_i,d_i',(M_i,M'_i))$, $i\in [r]$. For each $i\in [r]$, as in Figure~\ref{fig:absactual}, label edges so that $M_i=\{e_{i,1},e_{i,2},e_{i,3},e_{i,4}\}$ and $M'_i=\{f_{i,1},f_{i,2},f_{i,3},f_{i,4}\}$, and
\begin{itemize}
\item $e_{i,1}$ and $f_{i,1}$ have colour $c$,
\item $e_{i,2}$ and $f_{i,2}$ have colour $d_i$ and $d'_i$ respectively,
\item $e_{i,3}$ and $f_{i,3}$ have colour $c_1$, and
\item $e_{i,4}$ and $f_{i,4}$ have colour $c_2$.
\end{itemize}
For each $i\in [r]$, we have $e_{i,1}=e$ as $e\in M$ has colour $c$. Furthermore, $e_{i,1}f_{i,3}e_{i,2}f_{i,4}$ is a 4-cycle with colours $c,c_1,d_i,c_2$ in that order, so there are at most 2 possibilities for the edges $f_{i,3}$ and $f_{i,4}$, and hence for $e_{i,2}$ and the colour $d_i$. Thus, we can take some colour $d$ and edges $f_3,f_4,e_2$ and a set $\mathcal{S}_3\subset \mathcal{S}_2$ with $|\mathcal{S}_3|\geq p^5n/100$ such that, for each $i\in [r]$ with $(d_i,d'_i,(M_i,M'_i))\in \mathcal{S}_3$, we have $e_{i,2}=e_2$, $f_{i,3}=f_3$, $f_{i,4}=f_4$ and $d_i=d$.

\begin{figure}
\begin{center}
\picfirsttrue\picsecondfalse\picthirdfalse
\begin{tikzpicture}[scale=1]
\def\vxrad{0.07cm}
\def\horunit{1.2}
\def\edgelength{0.4}
\def\betweenrows{0.5}


\def\gapratio{1}

\foreach \num/\parity/\parityy in {1/-1/-1,3/-1/4}
{
\coordinate (X\num) at ($(0,0)+\parityy*\gapratio*(\horunit,0)+(0.5*\parity*\horunit,0.5*\horunit)$);
\coordinate (W\num) at ($(0,0)+\parityy*\gapratio*(\horunit,0)+(-0.5*\parity*\horunit,0.5*\horunit)$);
\coordinate (U\num) at ($(0,0)+\parityy*\gapratio*(\horunit,0)+(0.5*\parity*\horunit,-0.5*\horunit)$);
\coordinate (V\num) at ($(0,0)+\parityy*\gapratio*(\horunit,0)+(-0.5*\parity*\horunit,-0.5*\horunit)$);
}

\def\uppp{2.5*\horunit}
\coordinate (X2) at ($(X1)+(-\uppp,0)$);
\coordinate (W2) at ($(W1)+(-\uppp,0)$);
\coordinate (U2) at ($(U1)+(-\uppp,0)$);
\coordinate (V2) at ($(V1)+(-\uppp,0)$);

\coordinate (X) at ($0.5*(X1)+0.5*(W3)+(0,1.5)-0.5*(\horunit,0)$);
\coordinate (W) at ($0.5*(X1)+0.5*(W3)+(0,1.5)+0.5*(\horunit,0)$);

\foreach \x/\lab/\offs in {U/u/-1.2,V/v/-1.2}
\foreach \num/\labb in {1/1,2/2,3/k}
{
}
\foreach \x/\lab/\offs in {W/x/1,X/w/-1}
\foreach \num/\labb in {1/1,2/2,3/k}
{
}
\foreach \num/\labb in {1/1}
{
\draw  ($0.5*(U\num)+0.5*(V\num)-(0,0.3)$) node {$f_{i,1}$};
}
\foreach \num/\labb in {2/2}
{
\draw  ($0.5*(U\num)+0.5*(V\num)-(0,0.3)$) node {$f_{i,3}$};
}

\ifpicfirst
\foreach \x/\y/\col in {U/V/red}
\foreach \n in {1}
{
\draw [thick,\col,densely dashed] (\x\n) -- (\y\n);
}
\foreach \x/\y/\col in {U/V/orange}
\foreach \n in {2}
{
\draw [thick,\col] (\x\n) -- (\y\n);
}
\foreach \x/\y/\col in {U/X/orange}
\foreach \n in {1}
{
\draw [thick,\col] (\x\n) -- (\y\n);
}
\foreach \x/\y/\col in {U/X/red}
\foreach \n in {2}
{
\draw [thick,\col, densely dashed] (\x\n) -- (\y\n);
}
\foreach \x/\y/\col in {W/V/magenta}
\foreach \n in {1}
{
\draw [thick,\col] (\x\n) -- (\y\n);
}
\foreach \x/\y/\col in {W/V/blue}
\foreach \n in {2}
{
\draw [thick,\col] (\x\n) -- (\y\n);
}
\foreach \x/\y/\col in {X/W/blue}
\foreach \n in {1}
{
\draw [thick,\col] (\x\n) -- (\y\n);
}
\foreach \x/\y/\col in {X/W/magenta}
\foreach \n in {2}
{
\draw [thick,\col] (\x\n) -- (\y\n);
}



\foreach\num in {2}
{
\draw [red] ($0.5*(U\num)+0.5*(X\num)+(0.25,0)$) node {$c$};
\draw [blue] ($0.5*(V\num)+0.5*(W\num)+(-0.25,0)$) node {$d_i$};
\draw  ($0.5*(U\num)+0.5*(X\num)-(0.35,0)$) node {$e_{i,1}$};
\draw  ($0.5*(V\num)+0.5*(W\num)+(0.35,0)$) node {$e_{i,2}$};
}

\foreach\num in {2}
{
\draw [orange] ($0.5*(U\num)+0.5*(V\num)+(0,0.2)$) node {$c_1$};
\draw [magenta] ($0.5*(W\num)+0.5*(X\num)+(0,-0.2)$) node {$c_2$};
}
\foreach\num in {1}
{
\draw [red] ($0.5*(U\num)+0.5*(V\num)+(0,0.2)$) node {$c$};
\draw [orange] ($0.5*(U\num)+0.5*(X\num)+(0.25,0)$) node {$c_1$};
\draw ($0.5*(U\num)+0.5*(X\num)+(-0.35,0)$) node {$e_{i,3}$};
\draw [magenta] ($0.5*(V\num)+0.5*(W\num)+(-0.25,0)$) node {$c_2$};
\draw ($0.5*(V\num)+0.5*(W\num)+(0.35,0)$) node {$e_{i,4}$};
\draw  ($0.5*(W\num)+0.5*(X\num)+(0,0.25)$) node {$f_{i,2}$};
\draw [blue] ($0.5*(W\num)+0.5*(X\num)+(0,-0.25)$) node {$d_i'$};
}
\foreach\num in {2}
\draw  ($0.5*(W\num)+0.5*(X\num)+(0,0.25)$) node {$f_{i,4}$};

\coordinate (R1) at ($0.5*(X1)+0.5*(W1)$);
\coordinate (R2) at ($0.5*(X2)+0.5*(W2)$);
\coordinate (R3) at ($0.5*(X3)+0.5*(W3)$);

\coordinate (T) at ($0.5*(X)+0.5*(W)$);

\foreach \num in {2}
{
\coordinate (S\num1) at ($0.1*(R\num)+0.9*(T)$);
\coordinate (S\num2) at ($0.9*(R\num)+0.1*(T)$);
\coordinate (M\num) at ($0.5*(S\num1)+0.5*(S\num2)$);
}
\foreach \num in {1,3}
{
\coordinate (S\num1) at ($0.1*(R\num)+0.9*(T)$);
\coordinate (S\num2) at ($0.9*(R\num)+0.1*(T)+\num*(0.2,0)-2*(0.2,0)$);
\coordinate (M\num) at ($0.5*(S\num1)+0.5*(S\num2)$);
}

\ifpicfirst
\foreach \num in {1,2}
{
}
\fi

\foreach \num in {2}
{
}
\foreach \num/\lab in {1/1}
{
}
\foreach \num/\lab in {3/100}
{
}

\foreach \x in {X,W,U,V}
\foreach \n in {1,2}
{
\draw [fill] (\x\n) circle [radius=\vxrad];
}


\end{tikzpicture}
\begin{tikzpicture}
\draw [white] (-0.5,0) -- (0.5,0);
\draw [dashed] (0,-1.2) -- (0,1);.2
\end{tikzpicture}
\begin{tikzpicture}[scale=1]
\def\vxrad{0.07cm}
\def\horunit{1.2}
\def\edgelength{0.4}
\def\betweenrows{0.5}


\def\gapratio{1}

\foreach \num/\parity/\parityy in {1/-1/-1,2/-1/1,3/-1/4}
{
\coordinate (X\num) at ($(0,0)+\parityy*\gapratio*(\horunit,0)+(0.5*\parity*\horunit,0.5*\horunit)$);
\coordinate (W\num) at ($(0,0)+\parityy*\gapratio*(\horunit,0)+(-0.5*\parity*\horunit,0.5*\horunit)$);
\coordinate (U\num) at ($(0,0)+\parityy*\gapratio*(\horunit,0)+(0.5*\parity*\horunit,-0.5*\horunit)$);
\coordinate (V\num) at ($(0,0)+\parityy*\gapratio*(\horunit,0)+(-0.5*\parity*\horunit,-0.5*\horunit)$);
}

\coordinate (X) at ($0.5*(X1)+0.5*(W3)+(0,1.5)-0.5*(\horunit,0)$);
\coordinate (W) at ($0.5*(X1)+0.5*(W3)+(0,1.5)+0.5*(\horunit,0)$);

\foreach \x/\lab/\offs in {U/u/-1.2,V/v/-1.2}
\foreach \num/\labb in {1/1,2/2,3/k}
{
}
\foreach \x/\lab/\offs in {W/x/1,X/w/-1}
\foreach \num/\labb in {1/1,2/2,3/k}
{
}
\foreach \num/\labb in {1/1,2/2,3/100}
{
\draw  ($0.5*(U\num)+0.5*(V\num)-(0,0.3)$) node {$f_{\labb,1}$};
}

\draw ($(X)-(0.5,0)$) node {$e_2$:};

\ifpicfirst
\foreach \x/\y/\col in {U/V/red}
\foreach \n in {1,2,3}
{
\draw [thick,\col,densely dashed] (\x\n) -- (\y\n);
}
\foreach \x/\y/\col in {U/X/orange}
\foreach \n in {1,2,3}
{
\draw [thick,\col] (\x\n) -- (\y\n);
}
\foreach \x/\y/\col in {W/V/magenta}
\foreach \n in {1,2,3}
{
\draw [thick,\col] (\x\n) -- (\y\n);
}
\foreach \x/\y/\col in {X/W/blue}
\foreach \n in {1,2,3}
{
\draw [thick,\col] (\x\n) -- (\y\n);
}

\draw [thick,blue] (W) -- (X);


\foreach\num in {2}
{
\draw [orange] ($0.5*(U\num)+0.5*(X\num)+(0.25,0)$) node {$c_1$};
\draw [magenta] ($0.5*(V\num)+0.5*(W\num)+(-0.25,0)$) node {$c_2$};
}

\foreach\num in {2}
{
\draw [red] ($0.5*(U\num)+0.5*(V\num)+(0,0.2)$) node {$c$};
\draw [blue] ($0.5*(W\num)+0.5*(X\num)+(0,-0.25)$) node {$d_2'$};
}
\draw [blue] ($0.5*(W)+0.5*(X)+(0,0.2)$) node {$d$};
\foreach\num in {1}
{
\draw [red] ($0.5*(U\num)+0.5*(V\num)+(0,0.2)$) node {$c$};
\draw [orange] ($0.5*(U\num)+0.5*(X\num)+(0.25,0)$) node {$c_1$};
\draw [magenta] ($0.5*(V\num)+0.5*(W\num)+(-0.25,0)$) node {$c_2$};
\draw [blue] ($0.5*(W\num)+0.5*(X\num)+(0,-0.25)$) node {$d_1'$};
}
\foreach\num in {3}
{
\draw [red] ($0.5*(U\num)+0.5*(V\num)+(0,0.2)$) node {$c$};
\draw [orange] ($0.5*(U\num)+0.5*(X\num)+(0.25,0)$) node {$c_1$};
\draw [magenta] ($0.5*(V\num)+0.5*(W\num)+(-0.25,0)$) node {$c_2$};
\draw [blue] ($0.5*(W\num)+0.5*(X\num)+(0,-0.25)$) node {$d_{100}'$};
}

\coordinate (R1) at ($0.5*(X1)+0.5*(W1)$);
\coordinate (R2) at ($0.5*(X2)+0.5*(W2)$);
\coordinate (R3) at ($0.5*(X3)+0.5*(W3)$);

\coordinate (T) at ($0.5*(X)+0.5*(W)$);

\foreach \num in {2}
{
\coordinate (S\num1) at ($0.1*(R\num)+0.9*(T)$);
\coordinate (S\num2) at ($0.9*(R\num)+0.1*(T)$);
\coordinate (M\num) at ($0.5*(S\num1)+0.5*(S\num2)$);
}
\foreach \num in {1,3}
{
\coordinate (S\num1) at ($0.1*(R\num)+0.9*(T)$);
\coordinate (S\num2) at ($0.9*(R\num)+0.1*(T)+\num*(0.2,0)-2*(0.2,0)$);
\coordinate (M\num) at ($0.5*(S\num1)+0.5*(S\num2)$);
}

\ifpicfirst
\foreach \num in {1,2,3}
{
\draw [<->] (S\num1) -- (S\num2);
}
\fi

\foreach \num in {2}
{
\draw [fill=white,white] (M\num) circle [radius=0.25];
\draw (M\num) node {$({V}_\num,{C}_\num)$};
}
\foreach \num/\lab in {1/1}
{
\draw [fill=white,white] (M\num) circle [radius=0.5];
\draw ($(M\num)+(0,0)$) node {$({V}_{\lab},{C}_{\lab})$};
}
\foreach \num/\lab in {3/100}
{
\draw [fill=white,white] (M\num) circle [radius=0.5];
\draw ($(M\num)+(0.25,0)$) node {$({V}_{\lab},{C}_{\lab})$};
}

\foreach \x in {X,W,U,V}
\foreach \n in {1,2,3}
{
\draw [fill] (\x\n) circle [radius=\vxrad];
}

\draw [fill] (W) circle [radius=\vxrad];
\draw [fill] (X) circle [radius=\vxrad];

\foreach \y in {-1,0,1}
{
\draw  [fill] ($0.25*(X2)+0.25*(V2)+0.25*(W3)+0.25*(U3)+\y*(0.2,0)$) circle [radius=0.25*\vxrad];
\draw  [fill] ($0.5*(M2)+0.5*(M3)+\y*(0.15,0)$) circle [radius=0.25*\vxrad];
}

\end{tikzpicture}
\end{center}

\vspace{-0.4cm}

\caption{On the left, the matchings $M_i=\{e_{i,1},e_{i,2},e_{i,3},e_{i,4}\}$ and $M'_i=\{f_{i,1},f_{i,2},f_{i,3},f_{i,4}\}$. On the right, the absorber construction for the absorption of any $V(f_{i,1})$, $i\in [100]$.}
\label{fig:absactual}
\end{figure}
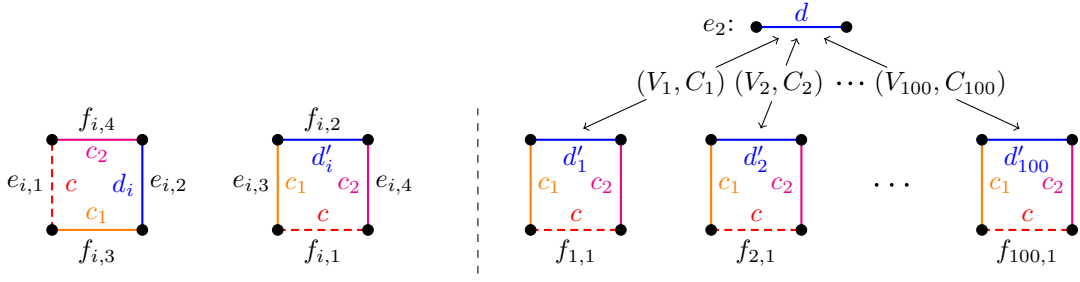

Now, note that, for each $i\in [r]$, $f_{i,1}e_{i,3}f_{i,2}e_{i,4}$ is a 4-cycle with colours $c,c_1,d_i',c_2$ in order. As each vertex is the endvertex of at most 6 paths of length 3 in $G$ which are $\{c,c_1,c_2\}$-rainbow, and a middle vertex of at most 6 such paths, each vertex is in at most 12 rainbow 4-cycles in $G$ with colour set including $c,c_1$ and $c_2$. As each 4-cycle can appear as at most 2 cycles  $f_{i,1}e_{i,3}f_{i,2}e_{i,4}$, $i\in [r]$, we therefore have that each cycle $f_{i,1}e_{i,3}f_{i,2}e_{i,4}$, $i\in [r]$, shares a vertex with at most $12\cdot 2\cdot 4<100$ of the cycles $f_{j,1}e_{j,3}f_{j,2}e_{j,4}$, $j\in [r]$. Thus, as $|\mathcal{S}_3|/100\geq p^5n/10^4\geq 100$, we can assume, by relabelling, that $(d_i,d'_i,(M_i,M'_i))\in \mathcal{S}_3$ for each $i\in [100]$ and that the cycles $f_{i,1}e_{i,3}f_{i,2}e_{i,4}$, $i\in [100]$, are pairwise vertex disjoint. For each $i\in [100]$, we have, as $(d_i,d_i',(M_i,M_i'))\in \mathcal{S}_3$, that $e_2=e_{i,2}$ shares no vertices with $f_{i,1}e_{i,3}f_{i,2}e_{i,4}$, and has colour $d=d_i$, which is $\mathcal{C}$-equivalent to $d_i'$, the colour of $f_{i,2}$, as $(d_i,d_i',(M_i,M_i'))\in \mathcal{S}_1$.

Let $E=\{f_{i,1}:i\in [100]\}$, so that $E\subset E_c(H_0)$ and $V(E)\cap \bar{V}=\emptyset$. We now construct an $E$-absorber, as depicted in Figure~\ref{fig:absactual} (see also Figure~\ref{fig:abs}).
 Greedily, for each $i=1,\ldots,100$ in turn, using that $e_2$ and $f_{i,2}$ are edges in $H_0$ whose colours are $\mathcal{C}$-equivalent, and \ref{prop-thmee3forH0}, let $(V_i,C_i)$ be an $e_2,f_{i,2}$-switcher of order at most $\ell$ with no vertices in $V(E)\cup V(e_2)\cup(\cup_{j\in [100]}V(f_{j,2}))\cup \bar{V}\cup (\cup_{j<i}V_j)$ or colours in $\{c_1,c_2\}\cup \bar{C}\cup (\cup_{j<i}C_j)$. Note that this is possible as
\[
|V(E)\cup V(e_2)\cup(\cup_{j\in [100]}V(f_{j,2}))\cup \bar{V}\cup (\cup_{j<i}V_j)|\leq 200+2+200+\eps n+100\cdot 2\ell\leq 2\eps n,
\]
and, similarly, $|\{c_1,c_2\}\cup \bar{C}\cup (\cup_{j<i}C_j)|\leq 2\eps n$.

Finally, letting $\tilde{V}=(\cup_{j\in [100]}V(f_{j,2}))\cup V(e_2)\cup (\cup_{i\in [100]}V_i)$ and $\tilde{C}=\{c_1,c_2\}\cup (\cup_{i\in [100]}C_i)$, note that $(\tilde{V},\tilde{C})$ is an $E$-absorber with order at most $200\ell$ and no vertices in $\bar{V}$ or colours in $\bar{C}$.
Indeed, $|\{c_1,c_2\}\cup (\cup_{i\in [100]}C_i)|\leq 2+100\ell\leq 200\ell$, and, for each $i\in [100]$,
 taking a $C_i$-rainbow matching with vertex set $V_i\cup V(e_2)$, a $C_j$-rainbow matching with vertex set $V_j\cup V(f_{j,2})$ for each $j\in [100]\setminus \{i\}$, and the $\{c_1,c_2\}$-rainbow matching $\{e_{i,3},e_{i,4}\}$ which has vertex set $V(\{f_{i,1},f_{i,2}\})$, we get an exactly-$\tilde{C}$-rainbow matching with vertex set $\tilde{V}\cup V(f_{i,1})$, as required.
\end{proof}


\subsection{Proof of Theorem~\ref{thm-absorbedges}: finding specific 1-in-100 local absorbers}\label{sec-AE-absorber2}
Given any $p$-good colour, we now use Lemma~\ref{lem:arblocab} and \ref{prop-thmee3forH0} to find absorbers for any set of 100 edges in $H_0$ of that colour, as follows.
\begin{lemma}\label{lem-exchangeedges}
Given any $p$-good colour $c\in C(H_0)$, and a set $E\subset E_c(H_0)$ with $|E|\leq 100$, and any sets $\bar{V}\subset V(G)$ and $\bar{C}\subset C(G)$ with $|\bar{V}|,|\bar{C}|\leq \eps n/2$, there are sets $V\subset V(G)\setminus (\bar{V}\cup V(E))$ and $C\subset C(G)\setminus \bar{C}$, such that $(V,C)$ is an $E$-absorber of order at most $300\ell$.
\end{lemma}
\begin{proof}
Let $c\in C(H_0)$ be $p$-good, let $E\subset E_c(H_0)$ with $|E|\leq 100$, and let sets $\bar{V}\subset V(G)$ and $\bar{C}\subset C(G)$ satisfy $|\bar{V}|,|\bar{C}|\leq \eps n/2$. By \ref{prop-thmee2forH0} and \eref{cond-thmae2}, as $1/n\ll p$ and $\xi\llpoly \log^{-1}n$, $H_0$ has at least $pn/2$ many colour-$c$ edges, and so, adding edges to $E$ from $E_c(H_0)$ if necessary, we can assume that $|E|=100$.

Note that $|\bar{V}\cup V(E)|\leq \eps n$, so therefore, by  Lemma~\ref{lem:arblocab},  there is a set $\bar{E}\subset E_c(H_0)$ of $100$ edges such that $V(\bar{E})\cap (\bar{V}\cap V(E))=\emptyset$ and sets $\hat{V}\subset V(G)\setminus (\bar{V}\cup V(E\cup\bar{E}))$ and $\hat{C}\subset C(G)\setminus \bar{C}$ with $|\hat{V}|=2|\hat{C}|-2$ and $|\hat{C}|\leq 200\ell$ and such that the following holds.
\stepcounter{propcounter}
\begin{enumerate}[label = {{{\textbf{\Alph{propcounter}}}}}]
\item For each $e\in \bar{E}$, there is a $\hat{C}$-rainbow perfect matching in $G[\hat{V}\cup V(e)]$.\label{prop-OGab}
\end{enumerate}

Next, take an arbitrary bijection $\phi:E\to \bar{E}$. Let $E'\subset E$ be a maximal set for which there are  vertex- and colour-disjoint $e,\phi(e)$-switchers, $S_e$, $e\in E'$, in $G$, with vertices in $V(G)\setminus (\hat{V}\cup\bar{V}\cup V(E\cup \bar{E}))$ and colours in $C(G)\setminus (\hat{C}\cup \bar{C})$, each with order at most $\ell$, and fix such switchers $S_e$, $e\in E'$.

Suppose, for contradiction, that $E'\neq E$, and pick $e\in E\setminus E'$.
Note that
\[
|(\cup_{e\in E'}V(S_e))\cup \hat{V}\cup \bar{V}\cup V(E\cup \bar{E})|\leq 100\cdot 2\cdot \ell+2\cdot 200\ell+\frac{\eps n}{2}+400\leq \eps n,
\]
and, similarly, $|(\cup_{e\in E'}C(S_e))\cup \hat{C}\cup \bar{C}|\leq \eps n$. Therefore, by \ref{prop-thmee3forH0}, there is an $e,\phi(e)$-switcher with order at most $\ell$ and no vertex in
$(\cup_{e\in E'}V(S_e))\cup \hat{V}\cup \bar{V}\cup V(E\cup \bar{E})$ or colour in $(\cup_{e\in E'}C(S_e))\cup \hat{C}\cup \bar{C}$, a contradiction. Thus, we have $E'=E$.

Let $\tilde{V}=\hat{V}\cup (\cup_{e\in E}V(S_e))\cup V(\bar{E})$ and $\tilde{C}=\hat{C}\cup (\cup_{e\in E}C(S_e))$, so that $\tilde{V}\cap \bar{V}=\emptyset$ and $\tilde{C}\cap \bar{C}=\emptyset$. To complete the proof, we show that $(\tilde{V},\tilde{C})$ is an $E$-absorber of order at most $300\ell$. Note first that $|\tilde{C}|\leq 100\cdot \ell+200\ell=300\ell$.
Fix then an arbitrary $e\in E$. Using that $S_e$ is an $e,\phi(e)$-switcher, take a $C(S_e)$-rainbow matching $M$ with vertex set $V(S_e)\cup V(e)$.
For each $f\in E\setminus\{e\}$, using that $S_f$ is an $e,\phi(f)$-switcher, take a $C(S_f)$-rainbow matching $M_f$ with vertex set $V(S_f)\cup V(\phi(f))$.
Using \ref{prop-OGab}, take a $\bar{C}$-rainbow matching $M'$ with vertex set $\bar{V}\cup V(\phi(e))$. Note that $M\cup M'\cup(\cup_{f\in E\setminus \{e\}}M_e)$ is a $\tilde{C}$-rainbow matching with vertex set $\tilde{V}\cup V(e)$. Thus, as $e\in E$ was arbitrary, $(\tilde{V},\tilde{C})$ is an $E$-absorber with order $|\tilde{C}|\leq 300\ell$, as required.
\end{proof}


\subsection{Proof of Theorem~\ref{thm-absorbedges}: global absorption}\label{sec-AE-absorber3}

We can now finish the proof of Theorem~\ref{thm-absorbedges} by showing that \itref{prop-thmae3} holds. That is, we prove the following lemma.

\begin{lemma}
Given any $m_0,m_1\in\N$ with $m_0\leq m_1\leq \gamma n$, and any monochromatic set $E\subset E(H)$ with $|E|=m_1$, and any sets $\bar{V}\subset V(G)$ and $\bar{C}\subset C(G)$ with $|\bar{V}|,|\bar{C}|\leq \beta n$, there are sets $\tilde{V}\subset V(G)\setminus (\bar{V}\cup V(E))$ and $\tilde{C}\subset C(G)\setminus \bar{C}$, such that $|\tilde{V}|=2|\tilde{C}|-2m_0\leq \beta n$ and, given any $E'\subset E$ with $|E'|=m_0$, there is an exactly-$\tilde{C}$-rainbow matching in $G[\tilde{V}\cup V(E')]$.
\end{lemma}
\begin{proof}
Let $m_0\leq m_1\leq\gamma n$ and $c\in C(H)$. Let $E\subset E(H)$ satisfy $|E|=m_1$, and $\bar{V}\subset V(G)$ and $\bar{C}\subset C(G)$ satisfy $|\bar{V}|,|\bar{C}|\leq \beta n$. Note that, by the definition of $H$, $c$ is $p$-good.  Furthermore, as we have shown that
 \itref{prop-thmae2} holds, and as $1/n\ll p$ and $\alpha\ll \log^{-1}n$,
by \itref{cond-thmae2} there are at least $pn/2$ edges in $H$ with the same colour as the edges in $E$.
As $1/n\llpoly \gamma$, we can assume that the conclusion of Lemma~\ref{Lemma_H_graph} holds for $h=3\gamma n$. Let $\hat{E}\subset E(H)$ be a set of $h+m_1-m_0\geq 2\gamma n+m_1$ edges with colour $c$, such that $E\subset \hat{E}$. Let $E_Y\cup E_Z$ be a partition of $\hat{E}$
so that $|E_Y|=2\gamma n$ and $E\subset E_Z$, and thus, as well, $|E_Z|=\gamma n+m_1-m_0$. Let $X$ be a set of $3\gamma n$ new vertices. Using the property of $h$ from Lemma~\ref{Lemma_H_graph} (and discarding $|Z|-|E_Z|$ vertices from $Z$ in the graph mentioned there), take an auxiliary bipartite graph $K$ with maximum degree at most $100$ and vertex classes $X$ and $E_Y\cup E_Z$, so that the following is true.

\stepcounter{propcounter}
\begin{enumerate}[label = {{\textbf{\Alph{propcounter}}}}]
\item If $E_Z' \subseteq E_Z$ and $|E_Z' | = h/3$, then
there is a matching between $X$ and $E_Y \cup E_Z'$ in $K$.\label{prop-thmaepf-2}
\end{enumerate}

Now, for each $x\in X$, let $E_x=N_K(x)$, so that $|E_x|\leq 100$. From \ref{prop-thmaepf-2}, we will have that, for each $x\in X$, $|E_x|\geq 1$. Let $X'\subset X$ be a maximal set for which there is an $E_x$-absorber $(V_x,C_x)$ for each $x\in X'$, with $|C_x|\leq 300\ell$, so that these absorbers are all pairwise vertex- and colour-disjoint with no vertices in $\bar{V}\cup V(\hat{E})$ or colours in $\bar{C}$.
As $|(\cup_{x\in X'}{V}_x)\cup \bar{V}\cup V(\hat{E})|\leq (h+m_1-m_0)\cdot 600\ell+\beta n+2(h+m_1-m_0)<\eps n/2$ and $|(\cup_{x\in X'}{C}_x)\cup \bar{C}|\leq (h+m_1-m_0)\ell+\beta n<\eps n/2$, by Lemma~\ref{lem-exchangeedges}, we must have that there is no $z\in X\setminus X'$, for otherwise we could find an $E_z$-absorber to contradict the choice of $X'$.

Thus, we have an $E_x$-absorber $(V_x,C_x)$ for each $x\in X$. Let $\tilde{V}=V(\hat{E}\setminus E)\cup (\cup_{x\in X}V_x)$ and $\tilde{C}=\cup_{x\in X}C_x$, so that $\tilde{V}\subset V(G)\setminus (\bar{V}\cup V(E))$ and $\tilde{C}\subset C(G)\setminus \bar{C}$. We will show that $\tilde{V}$ and $\tilde{C}$ have the property we require. For this, first note that, using Definition~\ref{defn:Eabs},
\[
|\tilde{V}|=2|\hat{E}\setminus E|+\sum_{x\in X}|{V}_x|{=}2(h-m_0)+\sum_{x\in X}(2|{C}_x|-2)=2(h-m_0)+2|\tilde{C}|-2|X|=2|\tilde{C}|-2m_0,
\]
and $|\tilde{C}|\leq h\cdot 300\ell=3\gamma n\cdot \ell\leq \beta n/2$, so that $|\tilde{V}|<2|\tilde{C}|\leq \beta n$.

Finally, given any $E'\subset E$ with $|E'|=m_0$, we have $|(E_Z\setminus E)\cup E'|=h/3$. Thus, by \ref{prop-thmaepf-2}, there is a matching, $M$ say, between $E_Y\cup (E_Z\setminus E)\cup E'=(\hat{E}\setminus E)\cup E'$ and $X$ in $K$. Labelling the edges in $(\hat{E}\setminus E)\cup E'$ as $e_x$, $x\in X$, such that $xe_x\in M$ for each $x\in X$, as $(V_x,C_x)$ is an $E_x$-absorber for each $x\in X$, and $e_x\in E_x$ as $xe_x\in M\subset E(K)$, there is a ${C}_x$-rainbow matching, $M_x$ say, in $G$ with vertex set ${V}_x\cup V(e_x)$. Then, $\cup_{x\in X}M_x$ is a $\tilde{C}$-rainbow matching in $G$ with vertex set $\cup_{x\in X}({V}_x\cup V(e_x))=(\cup_{x\in X}V_x)\cup V(\hat{E}\setminus E)\cup E')=\tilde{V}\cup V({E}')$.
Thus, $(\tilde{V},\tilde{C})$ has the property we want, completing the proof of the lemma, and therefore, as we have shown that \itref{prop-thmae3} holds, the proof of Theorem~\ref{thm-absorbedges}.
\end{proof}


\subsection{Proof of Theorem~\ref{thm:RSBabsorption} from Theorem~\ref{thm-absorbedges}}\label{sec-AE-final}

Having proved Theorem~\ref{thm-absorbedges}, we are finally in a position to prove Theorem~\ref{thm:RSBabsorption}.

\begin{proof}[Proof of Theorem~\ref{thm:RSBabsorption}]
We start by recalling the initial starting situation of Theorem~\ref{thm:RSBabsorption}, as follows. Let $1/n\ll p,q_V,q_C\leq 1$. Let $1/n\llpoly\eps \llpoly \gamma \llpoly \beta\llpoly \alpha \llpoly\log^{-1}n$. Let $G$ be a $(n,p,\eps)$-properly-pseudorandom
bipartite graph with vertex classes $A$ and $B$.
Independently, let $V$ be a $q_V$-random subset of $V(G)$ and let $C$ be a $q_C$-random subset of $C(G)$.
Our aim is to show that, with high probability, for all but at most $\alpha n$ colours $c\in C$, there is a set $E_c\subset E_{c}(G)$ with $|E_{c}(G)\setminus E_c|\leq \alpha n$ with the absorption property given in \itref{prop:thmabs}.

Let $G'$ be the subgraph of $G[V]$ with colours in $C$. We will show that \eref{cond-thmae2} and \eref{cond-thmae0} hold for an application of Theorem~\ref{thm-absorbedges} to $G'$, which we note has at most $2n$ vertices. Take $q$ with $1/n\ll q\ll p,q_V,q_C$. For each $c\in C(G)$, we have $|E_c(G)|\geq pn/2$ by \ref{prop-pseud-basic-2new} as $G$ is $(n,p,\eps)$-properly-pseudorandom. Thus, for each $c\in C(G)$, by Lemma~\ref{Lemma_Chernoff}, with probability $1-o(n^{-1})$, we have that $|E_c(G')|=|E_c(G[V])|\geq 2qn$. Thus, with high probability, we have the following property, where we also use that $|C(G)|\geq n/2$ by \ref{prop-pseud-basic-2new} and pply Lemma~\ref{Lemma_Chernoff} again.
\stepcounter{propcounter}
\begin{enumerate}[label = {{{\textbf{\Alph{propcounter}\arabic{enumi}}}}}]
\item We have $|C|\geq q_Cn/2$ and, for each $c\in C$, we have $|E_c(G')|\geq 2qn$.\label{cond-thmae2-nice}
\end{enumerate}

We will now show that, with high probability, for each $c\in C(G)$ and $e\in E_c(G)$, for all but $\sqrt{n}$ edges $f\in E_c(G)\setminus \{e\}$ we have the following property.

\begin{enumerate}[label = {{{\textbf{\Alph{propcounter}\arabic{enumi}}}}}]\addtocounter{enumi}{1}
\item If $c\in C$ and $e,f\in E_c(G')$, then there are at least $4qn^2$ triples $(d,d',(M,M'))$ in $G'$ where $d,d'\in C(M\cup M')\setminus \{c\}$, $(M,M')$ is a $d,d'$-switcher of order 4 with $e\in M$ and $f\in M'$, such that $e,f$, and the edges of $M\cup M'$ with colour in $\{d,d'\}$ form a matching.
\label{cond-thmae0-nice}
\end{enumerate}

As $G$ is $(n,p,\eps)$-properly-pseudorandom, we have by \ref{prop-pseud-abs-prime}, that for each $c\in C(G)$ and $e\in E_c(G)$ there is some set $F_{c,e}\subset E_c(G)$ with $|F_{c,e}|\leq \sqrt{n}$ such that, for all $f\in E_c(G)\setminus (\{e\}\cup F_{c,e})$, the following hold with $\alpha'=p^{12}/10^{100}$.
\begin{enumerate}[label = {{{\textbf{\Alph{propcounter}\arabic{enumi}}}}}]\addtocounter{enumi}{2}
\item There are at least $\alpha' n^2$ pairs $(S_1,S_2)$ such that $S_1$ and $S_2$ are vertex-disjoint rainbow 4-cycles with $e\in E(S_1)$ and $f\in E(S_2)$,
the colour sets of the neighbouring edges of $e$ in $S_1$ and the neighbouring edges of $f$ in $S_2$ are the same.\label{lastprop}
\end{enumerate}

We will show that, with high probability, for each $c\in C(G)$, $e\in E_c(G)$ and $f\in E_c(G)\setminus(\{c\}\cup F_{c,e})$,  \ref{cond-thmae0-nice} holds for $c,e$ and $f$.
This follows simply from the following claim and a union bound over all $c\in C(G)$, $e\in E_c(G)$ and $f\in E_c(G)\setminus(\{c\}\cup F_{c,e})$.

\begin{claim}
For each $c\in C(G)$, $e\in E_c(G)$ and $f\in E_c(G)\setminus(\{c\}\cup F_{c,e})$, \ref{cond-thmae0-nice} holds with probability $1-o(n^{-3})$.
\end{claim}
\begin{proof}
Let $c\in C(G)$ $e\in E_c(G)$ and $f\in E_c(G)\setminus(\{c\}\cup F_{c,e})$. Let $\mathcal{S}_{c,e,f}$ be the set of triples $(d,d',(M,M'))$ where $d,d'\in C(M\cup M')\setminus \{c\}$, $(M,M')$ is a $d,d'$-switcher of order 4 in $G$ with $e\in M$ and $f\in M'$, and the edges $e$ and $f$ form a matching with the edges in $M\cup M'$ with colour in $\{d,d'\}$.
Let $X_{c,e,f}$ be the number of $(d,d',(M,M'))\in \mathcal{S}_{c,e,f}$ with $(E(M)\cup E(M'))\subset E(G')$. Note that if $X_{c,e,f}\geq 4qn^2$, then \ref{cond-thmae0-nice} holds. Therefore, it is sufficient to show that $X_{c,e,f}\geq 4qn^2$ with probability $1-o(n^{-3})$.

Now, as $f\in E_c(G)\setminus(\{c\}\cup F_{c,e})$, we have
$|\mathcal{S}_{c,e,f}|\geq \alpha'n^2$ by \ref{lastprop}. Indeed, suppose we have $(S_1,S_2)$ such that $S_1$ and $S_2$ are vertex-disjoint rainbow 4-cycles with $e\in E(S_1)$ and $f\in E(S_2)$ and the colour sets of the neighbouring edges of $e$ in $S_1$ and the neighbouring edges of $f$ in $S_2$ are the same. Let $M$ and, respectively, $M'$, be the rainbow matchingswith vertex set $V(S_1\cup S_2)$ contained in
 $E(S_1\cup S_2)$ which contains $e$ but not $f$, and, respectively, $f$ but not $e$. Then, if $d$ is the colour of the edge in $(E(S_1)\cap E(M))\setminus \{e\}$ and $d'$ is the colour of the edge in $(E(S_2)\cap E(M'))\setminus \{f\}$, we can see that $(M,M')$ is a $d,d'$-switcher of order 4 in $G$ with $e\in M$ and $f\in M'$ in which $e$ and $f$ form a matching with the edges in $M\cup M'$ with colour in $\{d,d'\}$. Thus, as there are at least $\alpha'n^2$ such pairs $(S_1,S_2)$ by the definition of $F_{c,e}$, we have that  $|\mathcal{S}_{c,e,f}|\geq \alpha' n^2$.

Note that, given $(d,d',(M,M'))\in \mathcal{S}_{c,e,f}$, we have that $(C(M\cup M'))\setminus \{c\}\subset C$ and $(V(M\cup M'\setminus \{e,f\})\subset V$ with probability at least $q_C^4q_V^4$. Thus, $\E(X_{c,e,f})\geq q_C^4q_V^4\alpha' n^2\geq 8qn^2$.

Now, by Proposition~\ref{prop-switchoverlap}iii) and iv), for each $c'\in C(G)$, there are at most $10^3n$ triples $(d,d',(M,M'))\in \mathcal{S}_{c,e,f}$ with $c'\in C(M\cup M')\setminus \{c\}$, and, for each $v'\in V(G)$, there are at most $600n$ triples $(d,d',(M,M'))\in \mathcal{S}_{c,e,f}$ with $v'\in V((M\cup M')\setminus \{e,f\})$. Therefore, $X_{c,e,f}$ is $10^3n$-Lipschitz. Thus, by Lemma~\ref{lem:mcd} with $t=qn^2$ (and applied with $n'=|C(G)|+|V(G)|\leq 2n/p+2n\leq 4n/p$, using \ref{prop-pseud-basic-2new}), we have that $\P(X_{c,e,f}\leq 4qn^2)=o(n^{-3})$, and thus the claim holds.
\renewcommand{\qedsymbol}{$\boxdot$}
\end{proof}
\renewcommand{\qedsymbol}{$\square$}

Thus, with high probability, we can assume that \ref{cond-thmae2-nice} holds and, for each $c\in C(G)$, $e\in E_c(G)$ and $f\in E_c(G)\setminus(\{e\}\cup F_{c,e})$ (and hence for all but $\sqrt{n}$ edges $f\in E_c(G)\setminus \{e\}$), \ref{cond-thmae0-nice} holds. Under this assumption, we now show that the property in Theorem~\ref{thm:RSBabsorption} holds.
Applying Theorem~\ref{thm-absorbedges} (formally to $G'$ with $2n-|G'|$ new vertices added and no additional edges, and using $\alpha'=\alpha/2$, $\beta'=\beta/2$, $\gamma'=\gamma/2$, $p'=q$ and $n'=2n$), we have that there is a subgraph $H\subset G'$ such that the following hold.
\begin{enumerate}[label = {{{\textbf{\Alph{propcounter}\arabic{enumi}}}}}]\addtocounter{enumi}{3}
\item At most $\alpha n$ colours in $C(G')$ do not appear in $H$.\label{prop-thmae1-nice}
\item Each colour appearing in $H$ has at most $\alpha n$ edges in $G'-H$.\label{prop-thmae2-nice}
\item Given any $m_0,m_1\in\N$ with $m_0\leq m_1\leq \gamma n$, and any monochromatic set $E\subset E(H)$ with $|E|=m_1$, and any sets $\bar{V}\subset V(G)$ and $\bar{C}\subset C(G)$ with $|\bar{V}|,|\bar{C}|\leq \beta n$, there are sets $\tilde{V}\subset V\setminus (\bar{V}\cup V(E))$ and $\tilde{C}\subset \bar{C}\setminus C'$, such that $|\tilde{V}|=2|\tilde{C}|-2m_0\leq \beta n$ and, given any $E'\subset E$ with $|E'|=m_0$, there is an exactly-$\tilde{C}$-rainbow matching in $G[\tilde{V}\cup V(E')]$.\label{prop-thmae3-nice}
\end{enumerate}

Thus, for all but at most $\alpha n$ colours $c\in C$ (i.e., all but those in $C\setminus C(H)$, so that \ref{prop-thmae1-nice} implies the bound), setting $E_c=E_c(H)$, we have that $|E_c(G[V])\setminus E_c|\leq \alpha n$ by \ref{prop-thmae2-nice}, and that \itref{prop:thmabs} holds by \ref{prop-thmae3-nice} and \ref{cond-thmae2-nice} (where the last condition is used to add the vertex set and edge set of rainbow much to any sets $\tilde{V}$ and $\tilde{C}$ from \ref{prop-thmae3-nice} to get $V^\mathrm{abs}$ and $C^\mathrm{abs}$ respectively with from with $|V^\mathrm{abs}|=2\beta n-\ell_0$ and $|C^\mathrm{abs}|=\beta n$). Thus, we have the property in Theorem~\ref{thm:RSBabsorption}.
\end{proof}

\section{Addition structure}\label{sec:addition}

In this section, we construct our main addition structure and supplementary addition structure and thus prove Theorems~\ref{thm:RSBaddition} and~\ref{thm:RSBaddition-variant}. We restate the main conditions we use for this from the definition of proper-pseudorandomness (see Definition~\ref{defn:pseud}), in a slightly modified form for convenience. We will define what it means for sets $V\subset V(G)$  and $C\subset C(G)$ to $\alpha$-support an addition structure in a graph $G$, where in Definition~\ref{defn:pseud} the corresponding conditions are for $(V(G),C(G))$ to $\alpha$-support an addition structure.

\begin{defn}\label{defn:addsupport} Given $\alpha\in (0,1)$ and a properly coloured bipartite graph $G$ with vertex classes $A$ and $B$ with $|A|=|B|=n$, and sets $V\subset V(G)$ and $C\subset C(G)$, we say $(V,C)$ \emph{$\alpha$-supports an addition structure in $G$} if the following hold.

\begin{itemize}
\item[\ref{prop-pseud-add-2}]
For each $u\in A$, $v\in B$ and $c_0\in C(G)$, there are disjoint sets ${V}_1,\ldots,{V}_{\alpha n}\subset V\setminus \{u,v\}$ and disjoint sets $C_1,\ldots,C_{\alpha n}$ in $C\setminus\{c_0\}$ such that,
for each $i\in [\alpha n]$, $|{V}_i|=4$, $|C_i|=3$, $G[V_i]$ contains 2 colour-$c_0$ edges and $G[\{u,v\}\cup V_i]$ contains an exactly-$C_i$-rainbow matching in $E(G)\setminus \{uv\}$.
\item[\ref{prop-pseud-add-3}] For each distinct $c_0,d\in C(G)$, there are disjoint sets $V_1,\ldots,V_{\alpha n/12}$ in $V$ and disjoint sets $C_1,\ldots,C_{\alpha n/12}$ in $C\setminus\{c_0,d\}$,
so that, for each $i\in [\alpha n/12]$, $|V_i|=8$ and $|C_i|=3$, and $G[V_i]$ contains a matching of $4$ colour-$c_0$ edges
and an exactly-$(C_i\cup \{d\})$-rainbow matching.
\item[\ref{prop-pseud-add-new-1}]  For any $c_0\in C(G)$, $0\leq k\leq 20$, and any $\bar{C}\subset C(G)\setminus \{c_0\}$ with $|\bar{C}|\geq 5k$, there are vertex-disjoint sets $\bar{V}_1,\ldots,\bar{V}_{\alpha n}\subset V$
such that,
for each $i\in [\alpha n]$, $|\bar{V}_i|=2k+2$ and $G[\bar{V}_i]$ contains both a matching of $k+1$ colour-$c_0$ edges and a $\bar{C}$-rainbow matching with $k$ edges.
\item[\ref{prop-pseud-add-new-2}] Setting $k=100$, for each $c_0\in C(G)$, there is some $r\in \N$ and disjoint sets ${V}_1,\ldots,{V}_{r}$ in $V$ and disjoint sets $C_1,\ldots,C_{r}$ in $C\setminus\{c_0\}$
 with $|V_i|=2k$ and $|C_i|=k$ for each $i\in [r]$ such that $G[V_i]$ contains an exactly-$C_i$-rainbow matching and a perfect matching of colour-$c_0$ edges, and the following holds. For every $\bar{C}\subset C(G)\setminus \{c_0\}$ with $|\bar{C}|\leq k$, for at least $\alpha^2n$ values of $i\in [r]$, there are vertex-disjoint sets $\bar{V}_1,\ldots,\bar{V}_{\alpha n}\subset V$ such that,
for each $j\in [\alpha n]$, $|\bar{V}_j|=2k+2|\bar{C}|+2$ and $G[\bar{V}_j]$ contains both a matching of $k+|\bar{C}|+1$ colour-$c_0$ edges and a $(\bar{C}\cup C_i)$-rainbow matching with $k+|\bar{C}|$ edges.
\end{itemize}
\end{defn}

In the rest of this section, we will first show that if $(V,C)$ supports an addition structure in $G$ then taking random subsets of $V$ and $C$ and removing a few edges from $G$ will likely still give us support for an addition structure in $G$ (Lemma~\ref{lem-inherit-addition-support} and Lemma~\ref{lem-remove-edges-addition-support} in Sections~\ref{sec-add-1} and~\ref{sec-add-2}, respectively). Using the addition support property, we then construct our main addition structure in Section~\ref{sec-add-3} which can incorporate all the missing vertices and all but 100 missing colours (giving Lemma~\ref{lem-addstruct}), before constructing our supplementary addition structure in Section~\ref{sec-add-43} which can incorporate all but 1 of 100 missing colours (giving Lemma~\ref{lem-addstruct-part2}). Finally, we then prove Theorem~\ref{thm:RSBaddition} in Section~\ref{sec-add-4} before giving the slight alterations needed to prove Theorem~\ref{thm:RSBaddition-variant} in Section~\ref{sec-add-5}.


\subsection{Random subsets likely still support addition}\label{sec-add-1}
Using Chernoff's bound and a union bound, it is simple to show that if $(V,C)$ supports an addition structure in $G$, then random subsets $V'\subset V$ and $C'\subset C$ are likely to together also support an addition structure in $G$, as follows.

\begin{lemma}\label{lem-inherit-addition-support}
Let $1/n\llpoly \beta \llpoly \alpha,q_V,q_C$ and let $G$ be a properly coloured bipartite graph with vertex classes $A$ and $B$, where $|A|=|B|=n$. Let $V\subset V(G)$ and $C\subset C(G)$  be such that $(V,C)$ $\alpha$-supports an addition structure in $G$. Let $V'\subset V$ and $C'\subset C$ be independent $q_V$-random and $q_C$-random subsets, respectively. Then, with high probability, $(V',C')$ $\beta$-supports an addition structure in $G$.
\end{lemma}
\begin{proof}
We will show in turn that each of \ref{prop-pseud-add-2}--\ref{prop-pseud-add-new-2} holds with high probability for $(V',C')$ to $\beta$-support an addition structure in $G$.

\smallskip

\ref{prop-pseud-add-2}:
Let $u\in A$, $v\in B$ and $c_0\in C(G)$.
By \ref{prop-pseud-add-2} for $(V,C)$ to $\alpha$-support an addition structure in $G$ there are disjoint sets ${V}_1,\ldots,{V}_{\alpha n}\subset V\setminus \{u,v\}$ and disjoint sets $C_1,\ldots,C_{\alpha n}$ in $C\setminus\{c_0\}$ such that,
for each $i\in [\alpha n]$, $|{V}_i|=4$, $|C_i|=3$, $G[V_i]$ contains 2 colour-$c_0$ edges and $G[\{u,v\}\cup V_i]$ contains an exactly-$C_i$-rainbow matching in $E(G)\setminus \{uv\}$.
Let $X_{u,v,c_0}$ be the number of $i\in [\alpha n]$ with $V_{i}\subset V'$ and $C_i\subset C'$. Then, $X_{u,v,c_0}$ is binomially distributed with mean $\alpha n\cdot q_V^4\cdot q_C^3\geq 2\beta n$. Therefore, the probability that $X_{u,v,c_0}\geq \beta n$ is, by Lemma~\ref{Lemma_Chernoff},
at least $1-\exp(-\beta n/12)$. Thus, by a union bound, with high probability the corresponding version of \ref{prop-pseud-add-2} holds for $(V',C')$ to $\beta$-support an addition structure in $G$.

\smallskip

\ref{prop-pseud-add-3}: Let $c_0,d\in C(G)$ be distinct. By \ref{prop-pseud-add-3} for $(V,C)$ to $\alpha$-support an addition structure in $G$, there are disjoint sets $V_1,\ldots,V_{\alpha n/10}$ in $V$ and disjoint sets $C_1,\ldots,C_{\alpha n/10}$ in $C\setminus\{c_0,d\}$,
so that, for each $i\in [\alpha n/12]$, $|V_i|=8$ and $|C_i|=3$, and $G[V_i]$ contains a matching of $4$ colour-$c_0$ edges
and an exactly-$(C_i\cup \{d\})$-rainbow matching. Let $X_{c_0,d}$ be the number of $i\in [\alpha n]$ with $V_{i}\subset V'$ and $C_i\subset C'$. Then, $X_{c_0,d}$ is binomially distributed with mean $\alpha n\cdot q_V^8\cdot q_C^3\geq 2\beta n$. Therefore, the probability that $X_{c_0,d}\geq \beta n/12$ is, by Lemma~\ref{Lemma_Chernoff},
at least $1-\exp(-\beta n/12)$. Thus, by a union bound, with high probability the corresponding version of \ref{prop-pseud-add-3} holds for $(V',C')$ to $\beta$-support an addition structure in $G$.

\smallskip

\ref{prop-pseud-add-new-1}: Let $c_0\in C(G)$ and $0\leq k\leq 20$, and let $\bar{C}\subset C(G)\setminus \{c_0\}$ satisfy $|\bar{C}|\geq 5k$.
By \ref{prop-pseud-add-new-1} for $(V,C)$ to $\alpha$-support an addition structure in $G$, there are vertex-disjoint sets $\bar{V}_1,\ldots,\bar{V}_{\alpha n}\subset V$
such that,
for each $i\in [\alpha n]$, $|\bar{V}_i|=2k+2$ and $G[V_i]$ contains both a matching of $k+1$ colour-$c_0$ edges and a $\bar{C}$-rainbow matching with $k$ edges.
Let $X_{c_0,k,\bar{C}}$ be the number of $i\in [\alpha n]$ with $V_{i}\subset V'$. Then, $X_{c_0,k,\bar{C}}$ is binomially distributed with mean $\alpha n\cdot q_V^{2k+2}\geq \alpha n\cdot q_V^{42}\geq 2\beta n$. Therefore, the probability that $X_{c_0,k,\bar{C}}\geq \beta n$ is, by Lemma~\ref{Lemma_Chernoff},
at least $1-\exp(-\beta n/12)$. Thus, as it is sufficient to show it for sets $\bar{C}$ with $|\bar{C}|=5k$, by a union bound, with high probability the corresponding version of \ref{prop-pseud-add-new-1} holds for $(V',C')$ to $\beta$-support an addition structure in $G$.

\smallskip

\ref{prop-pseud-add-new-2}: Set $k=100$ and let $c_0\in C(G)$.
By \ref{prop-pseud-add-new-2} for $(V,C)$ to $\alpha$-support an addition structure in $G$, there is some $r\in \N$ and disjoint sets ${V}_1,\ldots,{V}_{r}$ in $V$ and disjoint sets $C_1,\ldots,C_{r}$ in $C\setminus\{c_0\}$ with $|V_i|=2k$ and $|C_i|=k$ for each $i\in [r]$ such that $G[V_i]$ contains an exactly-$C_i$-rainbow matching and a perfect matching of colour-$c_0$ edges, and for every $\bar{C}\subset C(G)\setminus \{c_0\}$ with $|\bar{C}|\leq k$, for at least $\alpha^2n$ values of $i\in [r]$
\begin{enumerate}[label = $(\ddagger)$]
\item there are vertex-disjoint sets $\bar{V}_1,\ldots,\bar{V}_{\alpha n}\subset V$ such that,
for each $j\in [\alpha n]$, $|\bar{V}_j|=2k+2|\bar{C}|+2$ and $G[V_j]$ contains both a matching of $k+|\bar{C}|+1$ colour-$c_0$ edges and a $(\hat{C}\cup C_i)$-rainbow matching with $k+|\bar{C}|$ edges.\label{doubdag}
\end{enumerate}

Let $I\subset [r]$ be the set of $i\in [r]$ for which $V_i\subset V'$ and $C_i\subset C'$. 
For each $\bar{C}\subset C(G)\setminus \{c_0\}$ with $|\bar{C}|\leq k$, let $I_{\bar{C}}$ be the set of $i\in [\alpha n]$ for which \ref{doubdag} holds.
Similarly to the above arguments for \ref{prop-pseud-add-new-1} (but with sets of size $2k+2|\bar{C}|+2$), as $q_V^{2k+2|\bar{C}|+2}\alpha\geq q_V^{4k+2}\alpha \ggpoly \beta$, with high probability, for every $\bar{C}\subset C(G)\setminus \{c_0\}$ with $|\bar{C}|\leq k$
\begin{itemize}
\item there are vertex-disjoint sets $\bar{V}_1,\ldots,\bar{V}_{\beta n}\subset V$ such that,
for each $j\in [\beta n]$, $|\bar{V}_j|=2k+2|\bar{C}|+2$ and $G[V_j]$ contains both a matching of $k+|\bar{C}|+1$ colour-$c_0$ edges and a $(\hat{C}\cup C_i)$-rainbow matching with $k+|\bar{C}|$ edges.
\end{itemize}

Thus, to complete the proof, we need to show that, with high probability, for each $\bar{C}\subset C(G)\setminus \{c_0\}$ with $|\bar{C}|\leq k$, $|I_{\bar{C}}\cap I|\geq \beta^2 n$. 
Fix, then,  $\bar{C}\subset C(G)\setminus \{c_0\}$ with $|\bar{C}|\leq k$. We have that $|I_{\bar{C}}\cap I|$ is a binomial variable with parameters $|I_{\bar{C}}|\geq \alpha^2n$ and $q_V^{2k}q_C^{k}$, so that, by Lemma~\ref{Lemma_Chernoff},
with probability $1-o(n^{-k})$, $|I_{\bar{C}}\cap J|\geq \beta^2n$.
Taking a union bound over all $\bar{C}\subset C(G)\setminus \{c_0\}$ with $|\bar{C}|\leq k$, we have, then with high probability, that \ref{prop-pseud-add-new-2} holds.
\end{proof}


\subsection{Removing edges maintains support for addition}\label{sec-add-2}
We now show that removing a few edges of one colour from a graph has little impact on any support for an addition structure, as follows.
\begin{lemma}\label{lem-remove-edges-addition-support}
Let $1/n\llpoly \alpha\leq 1/12$ and let $G$ be a properly coloured bipartite graph with vertex classes $A$ and $B$, where $|A|=|B|=n$. Let $V\subset V(G)$ and $C\subset C(G)$ be such that $(V,C)$ $\alpha$-supports an addition structure in $G$. Let $E\subset E(G)$ satisfy $|E|\leq \alpha^2 n/2$. Then, $(V,C)$ $(\alpha/2)$-supports an addition structure in $G-E$.
\end{lemma}
\begin{proof}
Observe that each of the properties \ref{prop-pseud-add-2}--\ref{prop-pseud-add-new-1} require in each case, for some $r\in \{\alpha n/12,\alpha n\}$, the existence of disjoint sets $V_i$, $i\in [r]$, such that $G[V_i]$ has some property. As at most $\alpha^2/2\leq \alpha n/24$ of the graphs $G[V_i]$, $i\in [r]$, can then contain some edge in $E$, the properties \ref{prop-pseud-add-2}--\ref{prop-pseud-add-new-1} can easily be seen to hold for $(V,C)$ to $(\alpha/2)$-support an addition structure in $G-E$.

Consider, then, \ref{prop-pseud-add-new-2}, and set $k=100$. As $(V,C)$  $\alpha$-supports an addition structure in $G$ (and in particular \ref{prop-pseud-add-new-2}), we can take $r\in \N$ and disjoint sets ${V}_1,\ldots,{V}_{r}$ in $V$ and disjoint sets $C_1,\ldots,C_{r}$ in $C\setminus\{c_0\}$
 with $|V_i|=2k$ and $|C_i|=k$ for each $i\in [r]$ such that $G[V_i]$ contains an exactly-$C_i$-rainbow matching and a perfect matching of colour-$c_0$ edges, which demonstrates that \ref{prop-pseud-add-new-2} holds. Let $I\subset [r]$ be the set of $i\in [r]$ for which $G[V_i]$ has no edge in $E$, so that $|I|\geq \alpha^2n/2$. For each $\bar{C}\subset C(G)\setminus \{c_0\}$ with $|\bar{C}|\leq k$, for each $i\in I$, there are vertex-disjoint sets $\bar{V}_1,\ldots,\bar{V}_{\alpha n}\subset V$ such that,
for each $j\in [\alpha n]$, $|\bar{V}_j|=2k+2|\bar{C}|+2$ and $G[\bar{V}_j]$ contains both a matching of $k+|\bar{C}|+1$ colour-$c_0$ edges and a $(\hat{C}\cup C_i)$-rainbow matching with $k+|\bar{C}|$ edges -- and, at most $\alpha^2n/4$ values of $j$ are such that $G[\bar{V}_j]$ has no edges in $E$. As $\alpha^2n-(r-|I|)-|E|\geq \alpha^2n/4$, we thus have that \ref{prop-pseud-add-new-2} holds for $(V,C)$ to $(\alpha/2)$-support an addition structure in $G-E$.
\end{proof}


\subsection{Main addition structure construction}\label{sec-add-3}
We now prove the main part of this section, showing that addition support in the sense of Definition~\ref{defn:addsupport} allows us to construct the main part of our addition structure in the following lemma, which will allow us to incorporate all the missing vertices and at most 100 missing colours in our rainbow matching.


\begin{lemma}\label{lem-addstruct} Let $1/n\llpoly \eta \llpoly \gamma \llpoly \alpha$, and  let $G$ be a properly coloured bipartite graph with vertex classes $A$ and $B$, where $|A|=|B|=n$. Let $V\subset V(G)$ and $C\subset C(G)$ be such that $(V,C)$ $\alpha$-supports an addition structure in $G$, and let $c_0\in C(G)$.

Then, there are sets $\bar{V}\subset V$ and $\bar{C}\subset C\setminus \{c_0\}$ such that the following hold for some $\ell_0,\ell_1\leq \gamma n$.
\stepcounter{propcounter}
\begin{enumerate}[label = {\emph{\textbf{\Alph{propcounter}\arabic{enumi}}}}]
\item $|A\cap \bar{V}|=|B\cap \bar{V}|=\ell_0+\ell_1$ and $|\bar{C}|=\ell_1+100$.\label{prop:addition1}
\item For any $A'\subset V(G)\setminus \bar{V}$, $B'\subset V(G)\setminus \bar{V}$ and $C'\subset C(G)\setminus (\bar{C}\setminus \{c_0\})$ with $|A'|=|B'|\leq \eta n$ and $|C'|\leq \eta n$,
$G[\bar{V}\cup A'\cup B']$ contains vertex-disjoint matchings $M'_1$ and $M'_2$ such that $M'_1$
has $\ell_0+|A'|-|C'|$ edges
in $E_{c_0}(G)$
with both vertices in $\bar{V}$
and $M'_2$ is a $(\bar{C}\cup C')$-rainbow matching with $\ell_1+|C'|$ edges.\label{prop:addition2}
\end{enumerate}
\end{lemma}

Before proving Lemma~\ref{lem-addstruct}, we comment on the proof. The construction of the main additional structure follows the sketch in Section~\ref{sec:add} and Figure~\ref{fig:add}, using the matchings $M^\mathrm{id}$ and $M^\mathrm{rb}$. The only significant difference is that $\bar{C}$ is a slightly larger set than $C(M_2)$ (containing 100 more colours). How we use the structure is slightly different to the simplified form in the sketch in Section~\ref{sec:add}. Instead of maintaining remainder vertices, we effectively first expand the matchings $M_1=M^{\mathrm{id}}$ and $M_2=M^{\mathrm{rb}}$ to cover the vertices in $A'\cup B'$, without worrying how many colours we drop out, doing only stage i) and ii) as outlined in Section~\ref{sec:add}. Then, we effectively incorporate all but 100 of the colours in $\bar{C}\cup C'$ by first performing stage iii), but in which $F_3$ is a rainbow matching of 9 edges whose colours can be any of the colours not yet incorporated (of which there are at least 100, so that we can use \ref{prop-pseud-add-new-1}), before using stage i) and ii) to reincorporate the two remainder vertices which are dropped out in this stage. We say `effectively' here, as, instead of iterating, the proof choses two matchings $M'_1$ and $M'_2$ to maximise certain properties before showing they have the required properties. Indeed, these matchings will together cover all the vertices in $A'\cup B'$ (for otherwise, we could use stage i) and stage ii) to gain a contradiction, as in the proof of Claim~\ref{clm:bycontra1}) and use all but 100 of the colours (for otherwise, we could use stage iii), stage i) and stage ii) to gain a contradiction, as in the proof of Claim~\ref{clm:bycontra2}).

\begin{proof}[Proof of Lemma~\ref{lem-addstruct}]
Let $\bar{\alpha},\hat{\alpha}$ be such that $\eta \llpoly \bar{\alpha}\llpoly \gamma\llpoly \hat{\alpha}\llpoly \alpha$, and note that we may assume $\alpha\leq 1/12$.
Let $V=V_0\cup V_1$ be a partition of $V$ with the location of each $v\in V$ chosen independently at random, and such that $V_0$ is a $(\gamma/8)$-random subset of $V$. Let $C=C_0\cup C_1$ be a partition of $C$ with the location of each $c\in C$ chosen independently at random so that $C_0$ is a $(\gamma/8)$-random subset of $C$.

By Lemma~\ref{lem-inherit-addition-support} (as $(V,C)$ $\alpha$-supports an addition structure in $G$), and by Lemma~\ref{Lemma_Chernoff},  with high probability we have that the following properties hold.
\begin{enumerate}[label = {{\textbf{\Alph{propcounter}\arabic{enumi}}}}]\addtocounter{enumi}{2}
\item $(V_0,C_0)$ $\bar{\alpha}$-supports an addition structure in $G$.\label{prop-minisup-1}
\item $(V_1,C_1)$ $\hat{\alpha}$-supports an addition structure in $G$.\label{prop-minisup-2}
\item $|V_0|,|C_0|\leq \gamma n/4$.\label{prop-minisup-3}
\end{enumerate}
Additionally, we have that, with positive probability (in particular, with probability at least $1-\gamma/8$), $c_0\notin C_0$.
Thus, we can assume we have partitions $V=V_0\cup V_1$ and $C=C_0\cup C_1$ for which \ref{prop-minisup-1}--\ref{prop-minisup-3} hold and such that $c_0\notin C_0$.

Let $M_1=E_{c_0}(G[V_0])$ and $\ell_0=|M_1|$, so that, by \ref{prop-minisup-3}, $\ell_0\leq \gamma n$. We later will use \ref{prop-minisup-1} to show that the necessary properties hold for us to robustly carry out stages i) and iii) as sketched out in Section~\ref{sec:abs}, but first pick the matching $M_2$.

Let $C_0'\subset C_0$ be a maximal set for which we can find pairs of matchings $E_{2,c},F_{2,c}$, $c\in C_0'$, such that
\begin{itemize}
\item $E_{2,c}$ and $F_{2,c}$ are matchings of size 4 with the same vertex set,
\item for each $c\in C_0'$, $E_{2,c}$ is a colour-$c_0$ matching and $F_{2,c}$ is a rainbow matching with $c\in C(F_{2,c})$, and
\item the sets $C(F_{2,c})\setminus \{c\}$, $c\in C_0'$, are disjoint sets in $C_1$, and the sets $V(E_{2,c})$, $c\in C_0'$, are disjoint sets in $V_1$.
\end{itemize}
Suppose, for contradiction, that there is some $c\in C_0\setminus C_0'$. Then, any pair of matchings $E_{2,c}$ and $F_{2,c}$ of order 4 with the same vertex set which are a colour-$c_0$ matching and a $(\{c\}\cup C_1)$-rainbow matching respectively in $G[V_1]$, have a vertex in $\cup_{c\in C_0'}V(E_{2,c})$ or a colour in $\cup_{c\in C_0'}C(F_{2,c})\setminus\{c\}$, where $|\cup_{c\in C_0'}V(E_{2,c})|=8|C_0'|\leq 2\gamma n$ and $|\cup_{c\in C_0'}C(F_{2,c})\setminus\{c\}|\leq 3|C_0'|\leq \gamma n$ by \ref{prop-minisup-3}.
However, by \ref{prop-minisup-2} (and in particular~\ref{prop-pseud-add-3} in Definition~\ref{defn:addsupport}), there are at least $\hat{\alpha}n/12$ ways to choose $E_{2,c},F_{2,c}$ so that the choices have disjoint vertex sets and use disjoint sets of colours for $C(F_{2,c})\setminus \{c\}$. As $\gamma\ll \hat{\alpha}$, this is a contradiction, and thus $C_0'=C_0$, so we can assume we have matchings $E_{2,c},F_{2,c}$ with the properties stated above.

 Let $M_2=\cup_{c\in C_0}F_{2,c}$ and $\ell_1=|M_2|=4|C_0|\leq \gamma n$. Let $\bar{C}\subset C$ have size $\ell_1+100$ with $C(M_2)\subset \bar{C}$ (using, for example, that $|C|\geq 3\alpha n$ by \ref{prop-pseud-add-3} in Definition~\ref{defn:addsupport}). Setting
 $\bar{V}=V(M_1\cup M_2)$, and using that the matchings $M_1$ and $M_2$ are between $A$ and $B$, we have $|\bar{V}\cap A|=|\bar{V}\cap B|=\ell_0+\ell_1$. Thus, \eref{prop:addition1} holds.

\medskip


We now show that \itref{prop:addition2} holds. Let then $A'\subset V(G)\setminus \bar{V}$, $B'\subset V(G)\setminus \bar{V}$ and
$C'\subset C(G)\setminus \bar{C}$ with $|A'|=|B'|\leq \eta n$ and $|C'|\leq \eta n$.
Let $A''\subset A'$, $B''\subset A'$ and $r$  maximise $10|A''|+r$ subject to  $-10|A''|\leq r\leq |C'|$ and that
\begin{itemize}
\item $G[\bar{V}\cup A''\cup B'']$ contains matchings $M'_{1}$ and $M'_{2}$ whose vertex sets partition $\bar{V}\cup A''\cup B''$, $M'_1$ consists of $\ell_0+|A''|-r$ edges in $E_{c_0}(G)$ with both vertices in $\bar{V}$ and $M'_2$ is a $(C'\cup \bar{C})$-rainbow matching with $\ell_1+r$ edges, and $|M_1\setminus M'_1|,|M_2\setminus M'_2|\leq 100|A''|+10r$.
\end{itemize}

Note that this is possible as $A''=B''=C''=\emptyset$ satisfies these conditions with $M_1'=M_1$ and $M_2'=M_2$, and note that we have $|A''|=|B''|$.
Note further that we are done if $A''=A'$ and $r=|C'|$, so that \itref{prop:addition2} follows from the following two claims.

\begin{claim}$A''=A$.\label{clm:bycontra1}
\end{claim}

\begin{claim}$r=|C'|$.\label{clm:bycontra2}
\end{claim}

\begin{proof}[Proof of Claim~\ref{clm:bycontra1}] Assume, for contradiction, that $|A''|<|A'|$, and hence $|B''|<|B'|$, and that $M'_1$ and $M'_2$ have the properties described for $A'$ and $B'$.
Let $x\in A'\setminus A''$ and $y\in B'\setminus B''$.
Let $\hat{C}\subset C_0$ be the set of $c\in C_0$ with $F_{2,c}\subset {M}'_2$. As $M_2=\cup_{c\in C_0}F_{2,c}$ is the union of the vertex-disjoint matchings $F_{2,c}$, $c\in C_0$, and $|M_2\setminus M_2'|\leq 100|{A''}|+10r\leq 10^3\eta n$, we have $|C_0\setminus \hat{C}|\leq 10^3\eta n$.
Let $\hat{V}_0\subset V_0$ be the set of the $v\in V_0$ which occur in no edge in $M_1\setminus M_1'$, so that, as $|M_1\setminus M_1'|\leq 100|A''|+10r\leq 10^3\eta n$, $|V_0\setminus \hat{V}_0|\leq 2\cdot 10^3\eta n$.
As $\eta\llpoly \bar{\alpha}$, and as, by \ref{prop-minisup-1}, $(V_0,C_0)$ $\bar{\alpha}$-supports an addition structure in $G$, and in particular the corresponding version of \ref{prop-pseud-add-2} in Definition~\ref{defn:addsupport}, there is a set $W\subset \hat{V}_1$ with size 4 and a set
$D\subset \hat{C}$ with size 3 such that $G[W]$ contains a matching, $E_1$ say, of 2 colour-$c_0$ edges and $G[W\cup \{x,y\}]$ contains an exactly-$D$-rainbow matching, $F_1$ say.
Note that the edges of $E_1$ have both vertices in $W\subset \hat{V}_1$, so that $E_1\subset E_{c_0}(G[V_1])=M_1$. By the choice of $\hat{V}_1$, then, we have $E_1\subset M_1\cap M_1'$ and, by the choice of $\hat{C}$, we have $F_{c,2}\subset M_2\cap M_2'$ for each $c\in D$.

Let $F_2=\cup_{c\in D}F_{2,c}$, so that $|F_2|=12$, and, as $D\subset \hat{C}$, $F_2\subset M_2$. Let  $D'=\cup_{c\in D}(C(F_{2,c})\setminus\{c\})$ so that $|D'|=9$,
and note that $F_{2}$ is an exactly-$(D\cup D')$-rainbow matching. Recall that, for each $c\in C_0$, $E_{2,c}$ is a matching of 4 colour-$c_0$ edges in $G[V_1]$ with the same vertex set as $F_{2,c}$. Let, then, $E_2=\cup_{c\in D}E_{2,c}$, so that this is a matching of $4|D|=12$ colour-$c_0$-edges with $V(E_2)=V(F_2)\subset V_1$ and $E_2\subset E_{c_0}(G[V])$.

To recap, we have the following.
\begin{enumerate}[label = \roman{enumi})]
\item A set of 3 colours $D\subset C_0$, an exactly-$D$-rainbow matching $F_1$ in $G$ and a matching $E_1\subset M_1\cap M_1'$ of 2 edges such that $V(F_1)=V(E_1)\cup\{x,y\}$.\label{steptofinish1}
\item A set $D'\subset C\setminus \{c_0\}$ of 9 colours, an exactly-$(D\cup D')$-rainbow matching $F_2\subset M_2\cap M_2'$ and a matching $E_2$ of 12 edges in $E_{c_0}(G[V])$ such that $V(E_2)=V(F_2)$.\label{steptofinish2}
\end{enumerate}

Setting $M_1''=M_1'-E_1+E_2$ and $M_2''=M_2'+F_1-F_2$ thus gives two matchings whose vertex sets partition $\bar{V}\cup A''\cup B''\cup \{x,y\}$. Observe that $M_1''$ has $|M_1'|-2+12=\ell_0+(|A''|+1)-(r-9)$ edges. Observe that $M_2''$ has $|M_2'|+3-12=\ell_1+(r-9)$ edges. As, then,
\[
|M_1''\setminus M_1|\leq |M'_1\setminus M_1|+|E_1|\leq 100|A''|+10r+|E_1|\leq 100(|A''|+1)+10(r-9),
\]
and
\[
|M_2''\setminus M_2|\leq |M_2'\setminus M_2|+|F_2|\leq 100|A''|+10r+|F_2|\leq  100(|A''|+1)+10(r-9),
\]
the matchings $M_1''$ and $M_2''$ demonstrate that $A''\cup \{x\}$ and $B''\cup \{y\}$ contradict the maximality of $A''$, $B''$ and $r$ above, completing the proof of Claim~\ref{clm:bycontra1}.
\renewcommand{\qedsymbol}{$\boxdot$}
\end{proof}
\renewcommand{\qedsymbol}{$\square$}

\begin{proof}[Proof of Claim~\ref{clm:bycontra2}] Assume, for contradiction, that $r<|C'|$, and that $M'_1$ and $M'_2$ have the properties described for $A'$, $B'$ and $r$. Let $\tilde{C}$ be the set of colours in $\bar{C}\cup C'$ which do not appear on $M_2'$, so that, as $|\bar{C}|=\ell_1+100$, we have $|\tilde{C}|=100+|C'|-r>100$.
 Let $\hat{V}_0\subset V_0$ be the set of $v\in V_0$ which occur in no edge in $M_1\setminus M_1'$, so that, as $|M_1\setminus M_1'|\leq 100|A''|+10r\leq 10^3\eta n$, $|V_0\setminus \hat{V}_0|\leq 2\cdot 10^3\eta n$.
Using \ref{prop-minisup-1}, and in particular the corresponding version of \ref{prop-pseud-add-new-1} in Definition~\ref{defn:addsupport} applied with colour set $\bar{C}$ and $k=10$,
 find a set $W\subset V(M_1')$ with $|W|=22$ such that $G[W]$ contains both a matching of $11$ colour-$c_0$ edges, $E_3$ say, and a $\tilde{C}$-rainbow matching, $F_3$ say, with 10 edges. Let $w$ and $z$ be the two vertices in $W$ which are not in $F_3$.

 Let $\hat{C}\subset C_0$ be the set of $c\in C_0$ with $F_{2,c}\subset {M}'_2$. As $M_2=\cup_{c\in C_0}F_{2,c}$ is the union of the vertex-disjoint matchings $F_{2,c}$, $c\in C_0$, and $|M_2\setminus M_2'|\leq 100|{A''}|+r\leq 10^3\eta n$, we have $|C_0\setminus \hat{C}|\leq 10^3\eta n$.
 As $\eta\ll \bar{\alpha}$, and as, by \ref{prop-minisup-1}, $(V_1,C_0)$ $\bar{\alpha}$-supports an addition structure in $G$, and in particular the corresponding version of \ref{prop-pseud-add-2}, there is a set $W'\subset \hat{V}_1\setminus V(E_3)$ with size 4 and a set
 $D\subset \hat{C}$ with size 3 such that $G[W']$ contains a matching of 2 colour-$c_0$ edges, $E_1$ say, and $G[W'\cup \{w,z\}]$ contains an exactly-$D$-rainbow matching, $F_1$ say.
Note that the edges of $E_1$ have both vertices in $W'\subset \hat{V}_1$, so that $E_1\subset E_{c_0}(G[V_1])=M_1$. By the choice of $\hat{V}_1$, then, as $W'\subset \hat{V}_1\setminus V(E_3)$, we have $E_1\subset M_1\setminus (M_1'\cup E_3)$ and by the choice of $\hat{C}$, we have $F_{c,2}\subset M_2$ for each $c\in D_1\cup D_2$.

 Let $F_2=\cup_{c\in D}F_{2,c}$, so that $|D|=3$, $|F_2|=12$, and, as $D\subset \hat{C}$, $F_2\subset M_2$. Let  $D'=\cup_{c\in D}(C(F_{2,c})\setminus\{c\})$ so that $|D'|=9$, and note that $F_{2}$ is an exactly-$(D\cup D')$-rainbow matching. Recall that, for each $c\in C_0$, $E_{2,c}$ is a matching of 4 colour-$c_0$ edges in $G[V_1]$ with the same vertex set as $F_{2,c}$. Let, then, $E_2=\cup_{c\in D}E_{2,c}$, so that this is a matching of $4|D|=12$ colour-$c_0$-edges with $V(E_2)=V(F_2)\subset V_1$ and $E_2\subset E_{c_0}(G[V])$.

 To recap, we have the following.
 \begin{enumerate}[label = \roman{enumi})]
\item A $\tilde{C}$-rainbow matching $F_3$ with 10 edges and vertices in $V(M)\setminus (V(M')\cup V(F_1))$, and a matching $E_3\subset (M_1\cap M_1')\setminus E_1$ of $11$ edges, such that $V(F_3)\cup \{w,x\}=V(E_3)$.
 \item A set of 3 colours $D\subset C_0$, an exactly-$D$-rainbow matching $F_1$ in $G$ and a matching $E_1\subset M_1\cap M_1'$ of 2 edges such that $V(F_1)=V(E_1)\cup\{w,x\}$.\label{steptofinish3}
 \item A set $D'\subset C\setminus \{c_0\}$ of 9 colours, an exactly-$(D\cup D')$-rainbow matching $F_2\subset M_2\cap M_2'$ and a matching $E_2$ of 12 edges of $E$ such that $V(E_2)=V(F_2)$.\label{steptofinish4}
 \end{enumerate}

 Setting $M_1''=M_1'-E_3-E_1+E_2$ and $M_2''=M_2'+F_3+F_1-F_2$ thus gives two matchings whose vertex sets partition $\bar{V}\cup A''\cup B''$. Observe that $M_1''$ has $|M_1'|-11-2+12=\ell_0+|A''|-(r+1)$ edges. Observe that $M_2''$ has $|M_2'|+10+3-12=\ell_1+r+1$ edges. Finally,
 \[
 |M_1''\setminus M_1|\leq |M'_1\setminus M_1|+|E_1|\leq 100|A''|+10r+|E_1|\leq 100|A''|+10(r+1),
 \]
 and
 \[
 |M_2''\setminus M_2|\leq |M_2'\setminus M_2|+|F_2|\leq 100|A''|+10r+|F_2|\leq  100|A''|+10(r+1).
 \]
 Therefore, the matchings $M_1''$ and $M_2''$ demonstrate that $A''$, $B''$ and $r+1$ contradict the maximality of $A''$, $B''$, $r$ above, completing the proof of Claim~\ref{clm:bycontra2}.
 \renewcommand{\qedsymbol}{$\boxdot$}
 \end{proof}
 \renewcommand{\qedsymbol}{$\square$}
This completes the proof of the two claims, and hence the lemma.
\end{proof}


\subsection{Supplementary addition structure construction}\label{sec-add-43}
We now construct the supplementary addition structure, showing that addition support in the sense of Definition~\ref{defn:addsupport} allows us to construct an addition structure which can incorporate all but 1 of 100 missing colours, at the expense of dropping out two remainder vertices.

\begin{lemma}\label{lem-addstruct-part2} Let $1/n\llpoly \alpha$, and let $G$ be a properly coloured bipartite graph with vertex classes $A$ and $B$, where $|A|=|B|=n$. Let $V\subset V(G)$ and $C\subset C(G)$ be such that $(V,C)$ $\alpha$-supports an addition structure in $G$, and let $c_0\in C(G)$.

Then, there are sets $\bar{V}\subset V$ and $\bar{C}\subset C\setminus \{c_0\}$ such that the following hold for some $\ell_0,\ell_1\in \N$.
\stepcounter{propcounter}
\begin{enumerate}[label = {\emph{\textbf{\Alph{propcounter}\arabic{enumi}}}}]
\item $|A\cap \bar{V}|=|B\cap \bar{V}|=\ell_0+\ell_1+100$ and $|\bar{C}|=\ell_1$.\label{prop:supadd1}
\item For any $\hat{C}\subset C(G)\setminus \bar{C}$ with $|\hat{C}|=100$,  $G[\bar{V}]$ contains vertex-disjoint matchings $M'_1$ and $M'_2$ such that $M'_1$ has $\ell_0$ edges in $E_{c_0}(G)$ with both vertices in $\bar{V}$ and $M'_2$ is a
$(\hat{C}\cup \bar{C})$-rainbow matching with $\ell_1+99$ edges.\label{prop:supadd2}
\end{enumerate}
\end{lemma}
\begin{proof} Let $k=100$. By \ref{prop-pseud-add-new-2}, there are disjoint sets ${V}_1,\ldots,{V}_{\alpha n}$ in $V$ and disjoint sets $C_1,\ldots,C_{\alpha n}$ in $C\setminus\{c_0\}$ with $|V_i|=2k$ and $|C_i|=k$ for each $i\in [\alpha n]$ such that $G[V_i]$ contains an exactly-$C_i$-rainbow matching ($M_{2,i}$, say) and a perfect matching of colour-$c_0$ edges ($M_{1,i}$, say), and the following holds.
\begin{enumerate}[label = {{\textbf{\Alph{propcounter}\arabic{enumi}}}}]\addtocounter{enumi}{2}
\item For every $\bar{C}\subset C(G)\setminus \{c_0\}$ with $|\bar{C}|=k$, for at least $\alpha^2n$ values of $i\in [\alpha n]$, there are vertex-disjoint sets $\bar{V}_1,\ldots,\bar{V}_{\alpha n}\subset V$ such that,
for each $j\in [\alpha n]$, $|\bar{V}_j|=2k+2$ and $G[\bar{V}_j]$ contains both a matching of $2k+1$ colour-$c_0$ edges and a $(\bar{C}\cup C_i)$-rainbow matching with $2k$ edges.\label{quickcallback}
\end{enumerate}
Let $I\subset [\alpha n]$ be a random set formed by choosing each element independently at random with probability $2/\alpha^2\sqrt{n}$. Using Lemma~\ref{Lemma_Chernoff} and a union bound,
we can assume that $|I|\leq 4\sqrt{n}/\alpha$ and that \ref{quickcallback} holds with `at least $\alpha^2n$ values of $i\in [\alpha n]$' replaced by `at least $\sqrt{n}$ values of $i\in I$'.
Let $M_2=\cup_{i\in I}M_{2,i}$. Let $M_1$ be the set of colour-$c_0$ edges of $G$ with both vertices in $V\setminus V(M_2)$. Let $\bar{V}=V(M_1\cup M_2)$ and $\bar{C}=C(M_2)$. Let $\ell_0=|M_1|-k$ and $\ell_1=C(M_2)$.
Note that \eref{prop:supadd1} holds.

We now show that \eref{prop:supadd2} holds. For this, let $\hat{C}\subset C(G)\setminus \bar{C}$ satisfy $|\hat{C}|=k$. As $|V(M_2)|\leq 2k|I|\leq 8k\sqrt{n}/\alpha<\alpha n$, by the altered form of~\ref{quickcallback}, there is some $\bar{V}'\subset V(G)\setminus V(M_2)$ such that,
$|\bar{V}'|=2k+2$ and $G[\bar{V}']$ contains both a matching of $2k$ colour-$c_0$ edges, $\bar{M}_1$ say, and a $(\hat{C}\cup C_j)$-rainbow matching with $2k-1$ edges. As $V(\bar{M}_1)\cap V(M_2)=\emptyset$, we have that $\bar{M}_1\subset M_1$.
Let $M_1'=(M_1\setminus \bar{M}_1)\cup M_{1,j}$ so that $|M_1'|=|M_1|-2k+k=\ell_0$. Let $M_2'=(M_2\setminus M_{2,j})\cup \bar{M_2}$, so that $|M_2'|=\ell_1+2k-1-k=\ell_1+k-1$. Note that $M_1'$ and $M_2'$ have the properties required.
\end{proof}


\subsection{Proof of Theorem~\ref{thm:RSBaddition}}\label{sec-add-4}
In this section, we can finally prove Theorem~\ref{thm:RSBaddition}. To do this, we will apply Lemma~\ref{lem-addstruct} and~\ref{lem-addstruct-part2}, and take additional colour-$c_0$ edges (to get a matching $T_0$) and an additional rainbow matching (called $T_1$) so that, combining the vertices and colours used, the size of $V^{\mathrm{add}}$ and $C^{\mathrm{add}}$ is exactly what we want.

\begin{proof}[Proof of Theorem~\ref{thm:RSBaddition}]
As in the statement of Theorem~\ref{thm:RSBaddition}, let $1/n\ll p,q\leq 1$, let $1/n\llpoly \eps \llpoly \eta  \llpoly \gamma \llpoly \alpha \llpoly\log^{-1}n$ and let $G$ be a coloured $(n,p,\eps)$-properly-pseudorandom bipartite graph with vertex classes $A$ and $B$. Furthermore, let $V$ be a $q_V$-random subset of
$V(G)$ and let $C$ be a $q_C$-random subset of $C(G)$. Take an additional variable $\hat{\alpha}$ satisfying $\gamma \llpoly \hat{\alpha}\llpoly \alpha$. Note that, as $1/n\ll p$ and $\alpha \llpoly \log^{-1}n$, by \ref{prop-pseud-add-2}--\ref{prop-pseud-add-new-2} in Definition~\ref{defn:pseud}, $(V(G),C(G))$ $(p^{12}/10^{100})$-supports an addition structure in $G$.
By Lemma~\ref{lem-inherit-addition-support}, and by \ref{prop-pseud-basic-2new} and Lemma~\ref{Lemma_Chernoff}, then, with high probability, we have the following properties.

\stepcounter{propcounter}
\begin{enumerate}[label = {{\textbf{\Alph{propcounter}\arabic{enumi}}}}]
\item $(V,C)$ ($2\alpha$)-supports an addition structure in $G$.\label{MI1}
\item For each $c\in C(G)$, $|E_c(G[V])|\geq pq_V^2n/4$.\label{MI1-1}
\end{enumerate}

 We now show that $G$ has the property required in the lemma. To show this, let $c_0\in C(G)$ and $E_{c_0}\subset E(G)$ with $|E_{c_0}(G)\setminus E_{c_0}|\leq \alpha n$. We will find sets $V^{\mathrm{add}}\subset V$ and $C^{\mathrm{add}}\subset C$ such that \eref{prop:addthm1} and \eref{prop:addthm2} hold, which we restate for convenience here.

 \begin{enumerate}[label = {\textbf{\emph{\Alph{propcounter}\arabic{enumi}}}}]
 \item[\eref{prop:addthm1}] $|A\cap V^{\mathrm{add}}|=|B\cap V^{\mathrm{add}}|=2\gamma n+1$ and $|C^{\mathrm{add}}|=\gamma n+1$.
 \item[\eref{prop:addthm2}] For any $\hat{A}\subset V(G)\setminus V^{\mathrm{add}}$ and $\hat{B}\subset V(G)\setminus V^{\mathrm{add}}$ with $|\hat{A}|=|\hat{B}|\leq \eta n$, and any $\hat{C}\subset C(G)\setminus C^{\mathrm{add}}$ with $|C'|\leq \eta n$, $G[V^{\mathrm{add}}\cup V']$ contains vertex-disjoint matchings $M_1$ and $M_2$ such that $M_1$ has $\gamma n+|\hat{A}|-|\hat{C}|$ edges in $E_c$ with both vertices in $V^{\mathrm{add}}$
 and $M_2$ is a $(C^{\mathrm{add}}\cup \hat{C})$-rainbow matching with size $|C^{\mathrm{add}}\cup \hat{C}|-1$.
 \end{enumerate}

Let $G'=G-(E_c(G)\setminus E_{c_0})$, so that, by Lemma~\ref{lem-remove-edges-addition-support} and \ref{MI1}, $(V,C)$ $\alpha$-supports an addition structure in $G'$. Let $V=V_0\cup V_1$ be a random partition so that $V_1$ is a $\gamma/4$-random subset of $V$. Let $C=C_0\cup C_1$ be a random partition into $(1/2)$-random subsets of $C$.
Then, by Lemma~\ref{lem-inherit-addition-support} and Lemma~\ref{Lemma_Chernoff}, with high probability we have the following properties.

\begin{enumerate}[label = {{\textbf{\Alph{propcounter}\arabic{enumi}}}}]\addtocounter{enumi}{2}
\item $(V_0,C_0)$ $\hat{\alpha}$-supports an addition structure in $G'$.\label{MI2}
\item $(V_1,C_1)$ $(100\eta)$-supports an addition structure in $G'$.\label{MI2-1}
\item $|A\cap V_1|\leq \gamma n/2$.\label{MI3}
\end{enumerate}


As $\eta\llpoly \gamma \llpoly \bar{\alpha}$, by Lemma~\ref{lem-addstruct} and \ref{MI2}, we can find sets $\bar{V}_0\subset V_0$ and $\bar{C}_0\subset C_0\setminus \{c_0\}$ such that the following hold for some $\ell_0,\ell_1\leq \gamma n/4$.

\begin{enumerate}[label = {{\textbf{\Alph{propcounter}\arabic{enumi}}}}]\addtocounter{enumi}{5}
\item $|A\cap \bar{V}_0|=|B\cap \bar{V}_0|=\ell_0+\ell_1$ and $|\bar{C}_0|=\ell_1+100$.\label{prop:add11}
\item For any $A'\subset V(G)\setminus \bar{V}_1$, $B'\subset V(G)\setminus \bar{V}_1$ and $C'\subset C(G)\setminus \bar{C}_1$ with $|A'|=|B'|\leq \eta n$ and $|C'|\leq \eta n$, $G[\bar{V}_1\cup A'\cup B']$ contains vertex-disjoint matchings $M'_1$ and $M'_2$ such that $M'_1$ has $\ell_0+|A'|-|C'|$ edges in $E_{c_0}(G')$ with both vertices in
$\bar{V}_0$ and $M_2'$ is a $(\bar{C}_0\cup C')$-rainbow matching with $\ell_1+|C'|$ edges.\label{prop:add12}
\end{enumerate}

Furthermore, by Lemma~\ref{lem-addstruct-part2} and \ref{MI2}, we can find sets $\bar{V}_1\subset V_1$ and $\bar{C}_1\subset C_1\setminus \{c_0\}$ such that the following hold for some $\ell_0',\ell_1'\leq n$.
\begin{enumerate}[label = {{\textbf{\Alph{propcounter}\arabic{enumi}}}}]\addtocounter{enumi}{7}
\item $|A\cap \bar{V}_1|=|B\cap \bar{V}_1|=\ell_0'+\ell_1'+100$ and $|\bar{C}|=\ell_1'$.\label{prop:supadd1-1}
\item For any $\hat{C}\subset C(G)\setminus \bar{C}_1$ with $|\hat{C}|=100$,  $G[\bar{V}_1]$ contains vertex-disjoint matchings $M''_1$ and $M''_2$ such that $M''_1$ has $\ell'_0$ edges in $E_{c_0}(G')$ with both vertices in $\bar{V}$ and $M''_2$ is a
$(\hat{C}\cup \bar{C}_1)$-rainbow matching with $\ell'_1+99$ edges.\label{prop:supadd2-1}
\end{enumerate}
In fact, then, from \ref{prop:supadd1-1} and \ref{MI3}, we have $\ell_0'+\ell_1'+100\leq \gamma n/2$.

Finally, as $G'[V]$ contains at least $\alpha n$ edges of each colour by \ref{MI1}, and as $|\bar{V}_1|+|\bar{V}_2|= 2(\ell_0+\ell_1+\ell_0'+\ell_1'+100)\leq 10\gamma n$, we can find in $V\setminus (\bar{V}_1\cup \bar{V}_2)$ a matching $T_1$ of $\gamma n-\ell_0-\ell_0'$ colour-$c_0$ edges as $\ell_0+\ell_0'\leq \gamma n/2+\gamma n/4\leq \gamma n/ 4$.
Similarly, we can find, disjointly from $V\setminus (\bar{V}_1\cup \bar{V}_2\cup V(T_1))$, in $G'$, a $C$-rainbow matching $T_2$ with size $\gamma n+1-|\bar{C}_1\cup \bar{C}_2|=\gamma n-\ell_1-\ell_1'-99\geq 0$ and no colours in $\bar{C}_1\cup \bar{C}_2$ (as $G'$ certainly has at least $\alpha n$ colours by \ref{MI1-1}).
Let ${V}^{\mathrm{add}}=\bar{V}_1\cup \bar{V}_2\cup V(T_1)\cup V(T_2)$ and ${C}^{\mathrm{add}}=\bar{C}_1\cup \bar{C}_2\cup C(T_2)$. We will show that \eref{prop:addthm1} and \eref{prop:addthm2} hold.

First, note that $|{C}^{\mathrm{add}}|=\gamma n+1$ by the choice of $T_2$ and, furthermore,
\begin{align*}
|{V}^{\mathrm{add}}|&=2(\ell_0+\ell_1)+2(\ell_0'+\ell_1'+100)+2|T_1|+2|T_2|=4\gamma n+2.\label{eqn:alignforbarC}
\end{align*}
Therefore, as ${V}^{\mathrm{add}}$ is the vertex set of the union of two disjoint matchings $T_0$ and $T_1$ along with, disjointly, balanced sets $\bar{V}_1$ and $\bar{V}_2$ (using \ref{prop:add11} and \ref{prop:supadd1-1}), we have $|A\cap {V}^{\mathrm{add}}|=|B\cap {V}^{\mathrm{add}}|=2\gamma n+1$, and thus \eref{prop:addthm1} holds.


For \itref{prop:addthm2}, let $\hat{A}\subset V(G)\setminus {V}^{\mathrm{add}}$ and $\hat{B}\subset V(G)\setminus {V}^{\mathrm{add}}$ with $|\hat{A}|=|\hat{B}|\leq \eta n$, and $\hat{C}\subset C(G)\setminus {C}^{\mathrm{add}}$ with $|\hat{C}|\leq \eta n$. By \ref{prop:add12}, there are vertex-disjoint matchings $M_1'$ and $M'_2$ in $G[\bar{V}_1\cup \hat{A}\cup \hat{B}]$
such that $M_1'$ has $\ell_0+|\hat{A}|-|\hat{C}|$ edges in $E_c(G)\setminus E_{c_0}$ with both vertices in $\bar{V}_1\subset V$ and $M_2'$ is a $(\hat{C}\cup \bar{C}_0)$-rainbow matching with $\ell_1+|\hat{C}|$ edges. Let $\hat{C}'=(\hat{C}\cup \bar{C}_0)\setminus C(M_2')$, so that $|\hat{C}'|=100$.
Then, by \ref{prop:supadd2-1}, $G[\bar{V}_1]$ contains vertex-disjoint matchings $M''_1$ and $M''_2$ such that $M''_1$ has $\ell_0'$ edges in $E_{c_0}(G')$ with both vertices in $\bar{V}$ and $M''_2$ is a $(\hat{C}\cup \bar{C}_1)$-rainbow matching with $\ell_1'+99$ edges

Letting $M_1=M'_1\cup M''_1\cup T_1$ and $M_2=M_2'\cup M_2''\cup T_2$, we will show these matchings have the property in \itref{prop:addthm2}. Indeed, firstly, we have by construction that $M_1'$, $M_1''$ and $T_1$ are vertex-disjoint matchings of colour-$c_0$ edges in $G'[V]$, so that $M_1$ is a matching of edges in $E=E_{c_0}(G)$ with both vertices in $V$, and is such that
\begin{align*}
|M_1|&=|M'_1|+|M_1''|+|T_1|=\ell_0+|\hat{A}|-|\hat{C}|+\ell_0'+(\gamma n-\ell_0-\ell_0')=\gamma n+|\hat{A}|-|\hat{C}|.
\end{align*}
As $V(M_1')\subset \bar{V}_1\cup \hat{A}\cup \hat{B}$ and $V(M_1'')\subset \bar{V}_2$, we have $V(M_1)\subset V^{\mathrm{add}}$.
Furthermore, $M_2$ is by construction the vertex disjoint union of colour-disjoint rainbow matchings $M_2'$, $M_2''$ and $T_2$, all with colours in $C^\mathrm{add}\cup \hat{C}|$ and vertex-disjoint from $M_1$, and such that
\[
|M_2|=|M'_2|+|M_2''|+|T_2|=\ell_1+|\hat{C}|+\ell_1'+99+(\gamma n-\ell_1-\ell_1'-99)=\gamma n+|\hat{C}|=|C^{\mathrm{add}}\cup \hat{C}|-1.
\]
Thus, as $V(M_1')\subset \bar{V}_1\cup \hat{A}\cup \hat{B}$ and $V(M_2)\subset \bar{V}_2$, we have $V(M_2)\subset V^{\mathrm{add}}$. Thus, \itref{prop:addthm2} holds with $M_1=M^{\mathrm{id}}$ and $M_2=M^{\mathrm{rb}}$, completing the proof.
\end{proof}


\subsection{Proof of Theorem~\ref{thm:RSBaddition-variant}}\label{sec-add-5}

We finish this section by proving Theorem~\ref{thm:RSBaddition-variant}, which is simpler than the proof of Theorem~\ref{thm:RSBaddition} as it only requires the use of the main addition structure.

\begin{proof}[Proof of Theorem~\ref{thm:RSBaddition-variant}] Let $c_0=c$.
As in the proof of Theorem~\ref{thm:RSBaddition}, we have, with high probability, that \ref{MI1} and \ref{MI1-1} hold.
As $\eta\llpoly \gamma \llpoly {\alpha}$, by Lemma~\ref{lem-addstruct} and \ref{MI1}, we can find sets $\bar{V}_1\subset V_1$ and $\bar{C}\subset C\setminus \{c_0\}$ such that the following hold for some $\ell_0,\ell_1\leq \gamma n/2$.

\stepcounter{propcounter}
\begin{enumerate}[label = {{\textbf{\Alph{propcounter}\arabic{enumi}}}}]
\item $|A\cap \bar{V}_1|=|B\cap \bar{V}_1|=\ell_0+\ell_1$ and $|\bar{C}_1|=\ell_1+100$.\label{prop:add11-1}
\item For any $A'\subset V(G)\setminus \bar{V}_1$, $B'\subset V(G)\setminus \bar{V}_1$ and $C'\subset C(G)\setminus \bar{C}$ with $|A'|=|B'|\leq 4\eta n$ and $|C'|\leq \eta n$, $G[\bar{V}_1\cup A'\cup B']$ contains vertex-disjoint matchings $M'_1$ and $M'_2$ such that $M'_1$ has $\ell_0+|A'|-|C'|$ edges in $E_{c_0}(G')$ with both vertices in
$\bar{V}_1$ and $M_2'$ is a $(C'\cup \bar{C}_0)$-rainbow matching with $\ell_1+|C'|$ edges.\label{prop:add12-1}
\end{enumerate}

Finally, as $G'[V]$ contains at least $\alpha n$ edges of each colour by \ref{MI1}, and as $|\bar{V}_1|\leq 2(\ell_0+\ell_1)\leq \gamma n$, we can find in $V\setminus \bar{V}_1$ a matching $T_1$ of $\gamma n-\ell_0$ colour-$c_0$ edges as $\ell_0\leq \gamma n/2$.
Similarly, we can find, disjointly from $V\setminus (\bar{V}_1\cup V(T_1))$, in $G'$, a $C$-rainbow matching $T_2$ with size $\gamma n-|\bar{C}_1|+100\geq 0$ and no colours in $\bar{C}_1\cup \bar{C}_2$ (as $G'$ certainly has at least $\alpha n$ colours by \ref{MI1}).
Let ${V}^{\mathrm{add}}=\bar{V}_1\cup V(T_1)\cup V(T_2)$ and ${C}^{\mathrm{add}}=\bar{C}_1\cup  C(T_2)$. We will show that  \eref{prop:addthmvar1} and \eref{prop:addthmvar2} hold, which we restate for convenience.

\begin{enumerate}[label = {\textbf{\emph{\Alph{propcounter}\arabic{enumi}}}}]
\item[\eref{prop:addthmvar1}] $|A\cap V^{\mathrm{add}}|=|B\cap V^{\mathrm{add}}|=2\gamma n$ and $|{C}^{\mathrm{add}}|=\gamma n+100$.
\item[\eref{prop:addthmvar2}] For any $\hat{A}\subset V(G)\setminus V^{\mathrm{add}}$ and $\hat{B}\subset V(G)\setminus V^{\mathrm{add}}$ with $|\hat{A}|=|\hat{B}|\leq \eta n$,
and any $\hat{C}\subset C(G)\setminus C^{\mathrm{add}}$ with $|\hat{C}|\leq \eta n$, $G[V^{\mathrm{add}}\cup \hat{A}\cup \hat{B}]$ contains vertex-disjoint matchings $M^{\textrm{id}}$ and $M^{\textrm{rb}}$ such that $M^{\textrm{id}}$ has $\gamma n+|\hat{A}|-|\hat{C}|$ edges in $E_c$ with both vertices in $V^{\mathrm{add}}$
and $M^{\textrm{rb}}$ is a $(C^{\mathrm{add}}\cup \hat{C})$-rainbow matching with size $|C^{\mathrm{add}}\cup \hat{C}|-100$.
\end{enumerate}

First, note that $|{C}^{\mathrm{add}}|=\gamma n+100$ by the choice of $T_2$, and, furthermore,
\begin{align}
|{V}^{\mathrm{add}}|&=2(\ell_0+\ell_1)+2|T_1|+2|T_2|=2(\ell_0+\ell_1)+2(\gamma n-\ell_0)+2(\gamma n-|\bar{C}_1|+100)=4\gamma n.\nonumber
\end{align}
Therefore, as $\bar{V}$ is the vertex set of the union of two disjoint matchings $T_0$ and $T_1$ along with, disjointly, a balanced set $\bar{V}_1$ (using \ref{prop:add11-1}), we have $|A\cap \bar{V}|=|B\cap \bar{V}|=2\gamma n$, and thus \eref{prop:addthmvar1} holds.


For \itref{prop:addthmvar2}, let $\hat{A}\subset V(G)\setminus \bar{V}$ and $\hat{B}\subset V(G)\setminus \bar{V}$ with $|\hat{A}|=|\hat{B}|\leq \eta n$, and $\hat{C}\subset C(G)\setminus \bar{C}$ with $|\hat{C}|\leq \eta n$.
By \ref{prop:add12-1}, there are vertex-disjoint matchings $M_1'$ and $M'_2$ in $G[\bar{V}_1\cup \hat{A}\cup \hat{B}]$
such that $M_1'$ has $\ell_0+|\hat{A}|-|\hat{C}|$ edges in $E$ with both vertices in $\bar{V}_1$ and $M_2'$ is a $(\hat{C}\cup \bar{C}_1)$-rainbow matching with $\ell_1+|\hat{C}|$ edges.

Leting $M^{\mathrm{id}}=M'_1\cup T_1$ and $M^{\mathrm{rb}}=M_2'\cup T_2$, we will show these matchings have the property in \itref{prop:addthmvar2}. Indeed, by construction, $M^{\mathrm{id}}$ is a matching of edges in $E_{c_0}(G)$ with both vertices in $V$ and $V(M^{\mathrm{id}})\subset V^{\mathrm{add}}\cup \hat{A}\cup \hat{B}$, such that
\begin{align}
|M^{\mathrm{id}}|&=|M'_1|+|T_1|=\ell_0+|\hat{A}|-|\hat{C}|+(\gamma n-\ell_0)=\gamma n+|\hat{A}|-|\hat{C}|.
\end{align}
Furthermore, $M^{\mathrm{rb}}$ is a $(C^\mathrm{add}\cup \hat{C})$-rainbow matching in $G[V^{\mathrm{add}}\cup \hat{A}\cup \hat{B}]- V(M_1)$ such that
\[
|M^{\mathrm{rb}}|=|M'_2|+|T_2|=\ell_1+|\hat{C}|+\gamma n-|\bar{C}_1|+100=\gamma n +|\hat{C}|=|C^{\mathrm{add}}\cup \hat{C}|-100.
\]
Thus, \itref{prop:addthmvar2} holds, completing the proof.
\end{proof}

\section{Derivation of Theorem~\ref{thm:brouwer} from the technical theorems}\label{sec:brouwer}
We now prove Theorem~\ref{thm:brouwer} by applying Theorem~\ref{thm-technical} to an appropriate properly coloured pseudorandom bipartite graph. To get a  properly coloured pseudorandom bipartite graph from  a Steiner triple system we follow a method of Keevash, Pokrovskiy, Sudakov and Yepremyan~\cite{KPSY}.

\begin{proof}[Proof of Theorem~\ref{thm:brouwer}]
Let $S$ be a Steiner triple system (STS) with vertex set $[n]$. As $S$ is a STS, we know that $n\equiv 1,3\mod 6$. Furthermore, let us assume that $n\equiv 3\mod 6$, where the case where $n\equiv 1\mod 6$ follows similarly after removing an arbitrary vertex.

Let $m=n/3$. Let $[3m]=A\cup B\cup C$ be a partition created by, for each $v\in [3m]$, choosing the set for $v$ independently at random such that $\P(v\in A)=\P(v\in B)=\P(v\in C)=1/3$. Let $G$ be the bipartite graph with vertex classes $A$ and $B$ where $ab$ with $a\in A$ and $b\in B$ is an edge with colour $c$ exactly when $abc\in S$. We will show that, with positive probability, $G$ is $(m,1/3,\eps)$-properly-pseudorandom for some $\eps$ with $1/n\llpoly \eps\llpoly \log^{-1}n$. Therefore, Theorem~\ref{thm-technical} applies to show that a rainbow matching, $M$ say, exists in $G$ with $m-1$ edges. Noting that $\{abc\in S:ab\in E(M)\}$ is a matching in $S$ would then show that $S$ contains a matching with at least $n/3-1$ edges, as required.



Following Keevash, Pokrovskiy, Sudakov and Yepremyan~\cite[Section~6]{KPSY}, it is straightforward to see that, with probability at least $n^{-3}$, we have $|A|=|B|=|C|=m$ and, with probability $1-o(n^{-3})$, $\mathcal{H}(G)$ (as defined in Section~\ref{sec:typical}) is $(m,1/3,\eps)$-typical if $\eps\geq n^{-1/8}$. Thus, with positive probability, \ref{prop-pseud-basic-1} and \ref{prop-pseud-basic-2new} hold for $G$ to be $(m,1/3,\eps)$-properly-pseudorandom.
Therefore, it is sufficient to show that \ref{prop-pseud-abs-prime}--\ref{prop-pseud-add-new-2} hold for $G$ to be $(m,1/3,\eps)$-properly-pseudorandom for some $\eps\llpoly \log^{-1}n$ with probability $1-o(n^{-3})$.

\medskip

\ref{prop-pseud-abs-prime}: Let $c,x_1,x_2$ be distinct with $x_1x_2c\in S$. Let $F$ be the set of pairs of vertices which form an edge with $c$ in $S$, except for $x_1x_2$. For each $f=y_1y_2\in F$, we would ideally like to show there are at least $n^2/8$ choices for $d_1$ and $d_2$ in the picture in Figure~\ref{fig:fourcycle} so that, labelling vertices as pictured, this gives distinct vertices except we may have $d_3=d_4$. This we can only do for all but $O(1)$ edges $f=y_1y_2\in F$, and for convenience use the bound $\sqrt{n}$ as in \ref{prop-pseud-abs-prime}.
For this, first note that, for each  $f=y_1y_2\in F$, using arguments similar to those in the proof of Proposition~\ref{prop:nearcompletepseudorandom}, and labelling vertices as in Figure~\ref{fig:fourcycle} when they are determined, there are at least $n/2$ choices for $d_1$ so that $x_3$ and $y_3$ are not in $\{c,x_1,x_2,y_1,y_2\}$. Furthermore, if $x_3=y_3$, the STS property would imply $x_2=y_2$, so we have that $x_3$ and $y_3$ are distinct. Then, there are at least $n/2$ choices for $d_2$ so that $x_4,y_4,d_3,d_4$ are not in $\{c,x_1,x_2,y_1,y_2,d_1,x_3,y_3\}$, where, by the STS properties $d_2\neq d_3$, $d_2\neq d_4$, $x_4\neq y_4$, but we may have $d_3=d_4$, $x_4=d_4$ or $y_4=d_3$.

\begin{figure}
\begin{center}\begin{tikzpicture}[scale=0.8]
\def\vxrad{0.07cm}
\def\horunit{1.2}
\def\edgelength{0.4}
\def\betweenrows{0.5}


\def\gapratio{1}
\def\biggergap{2}

\foreach \num in {0,1}
{
\coordinate (B\num) at ($(0,0)+\num*4*\gapratio*(1,0)+\gapratio*(-0.5,0.5)$);
\coordinate (C\num) at ($(0,0)+\num*4*\gapratio*(1,0)+\gapratio*(0.5,0.5)$);
\coordinate (E\num) at ($0.5*(B\num)+0.5*(C\num)+0.86602540*\gapratio*(0,1)$);
}

\coordinate (C1) at ($0.5*(B0)+0.5*(C0)-0.86602540*\gapratio*(0,1)-(0,1)$);
\coordinate (C2) at ($0.5*(B0)+0.5*(C0)-0.86602540*\gapratio*(0,1)-(0,1)$);
\coordinate (C3) at ($0.5*\gapratio*(1,0)+0.86602540*\gapratio*(1,0)$);
\coordinate (C4) at ($-1*(C3)$);
\coordinate (X1) at (B0);
\coordinate (X2) at (C0);
\coordinate (X3) at ($(B0)-(0,1)$);
\coordinate (X4) at ($(C0)-(0,1)$);

\def\bit{0.375}
\draw ($(E0)+(0,\bit)$) node {$d_2$};

\draw ($(C1)+(0,-\bit)$) node {$d_1$};
\draw ($(C3)+(\bit,0)$) node {$d_3$};
\draw ($(C4)-(\bit,0)$) node {$c$};


\def\sm{0.2};

\foreach \num/\numm/\coll/\fillcoll in {0/0/red!50/red!25,0/1/blue!50/blue!25,0/2/green!50/green!25,0/3/orange!50/orange!25}
{
\begin{scope}[transform canvas={rotate=\numm*90}]
{
\draw [rounded corners,\coll,fill=\fillcoll] ($(C\num)+(0,-\sm)-0.5*(\gapratio,0)$) -- ($(C\num)+(0,-\sm)$) -- ($(C\num)+(\sm,0)$) -- ($(E\num)+(\sm,0)$) -- ($(E\num)+(0,\sm)$) -- ($(E\num)-(\sm,0)$) -- ($(B\num)-(\sm,0)$) -- ($(B\num)-(0,\sm)$) --
($(C\num)-(0,\sm)-0.5*(\gapratio,0)$);
}
\end{scope}
}
\foreach \num in {0}
{
\foreach \numm in {0,1,2,3}
{
\begin{scope}[transform canvas={rotate=\numm*90}]
{
\draw [fill] (B\num) circle [radius=\vxrad];
\draw [fill] (E\num) circle [radius=\vxrad];
}
\end{scope}
}
}

\begin{scope}[transform canvas={shift={(4*\gapratio,0)}}]
\foreach \num/\numm/\coll/\fillcoll in {0/0/red!50/red!25,0/1/blue!50/blue!25,0/2/green!50/green!25,0/3/orange!50/orange!25}
{
\begin{scope}[transform canvas={rotate=\numm*90}]
{
\draw [rounded corners,\coll,fill=\fillcoll] ($(C\num)+(0,-\sm)-0.5*(\gapratio,0)$) -- ($(C\num)+(0,-\sm)$) -- ($(C\num)+(\sm,0)$) -- ($(E\num)+(\sm,0)$) -- ($(E\num)+(0,\sm)$) -- ($(E\num)-(\sm,0)$) -- ($(B\num)-(\sm,0)$) -- ($(B\num)-(0,\sm)$) --
($(C\num)-(0,\sm)-0.5*(\gapratio,0)$);
}
\end{scope}
}
\foreach \num in {0}
{
\foreach \numm in {0,1,2,3}
{
\begin{scope}[transform canvas={rotate=\numm*90}]
{
\draw [fill] (B\num) circle [radius=\vxrad];
\draw [fill] (E\num) circle [radius=\vxrad];
}
\end{scope}
}
}
\draw ($(E0)+(0,\bit)$) node {$d_2$};

\draw ($(C1)+(0,-\bit)$) node {$d_1$};

\draw ($(C4)-(\bit,0)$) node {$c$};
\end{scope}
\draw ($(C3)+(\bit,0)+(4*\gapratio,0)+0.86602540*\gapratio*(1,0)+(0.175,0)$) node {$d_4$};

\def\bitt{0.625};
\draw ($(X1)+\bitt*(-0.5,0.5)$) node {$x_1$};
\draw ($(X2)+\bitt*(0.5,0.5)$) node {$x_4$};
\draw ($(X3)+\bitt*(-0.5,-0.5)$) node {$x_2$};
\draw ($(X4)+\bitt*(0.5,-0.5)$) node {$x_3$};

\begin{scope}[transform canvas={shift={(4*\gapratio,0)}}]
\draw ($(X1)+\bitt*(-0.5,0.5)$) node {$y_1$};
\draw ($(X3)+\bitt*(-0.5,-0.5)$) node {$y_2$};
\draw ($(X4)+\bitt*(0.5,-0.5)$) node {$y_3$};
\draw ($(X2)+\bitt*(0.5,0.5)$) node {$y_4$};
\end{scope}

\end{tikzpicture}
\end{center}

\vspace{-0.4cm}

\caption{The structure counted in $S$ for \ref{prop-pseud-abs-prime} in the proof of Theorem~\ref{thm:brouwer}, where the vertices $c$, $d_1$ and $d_2$ are repeated for clarity and we may have $d_3=d_4$. When $c,d_1,d_2,d_3,d_4\in C$, $x_1,x_3,y_1,y_3\in A$ and $x_2,x_4,y_2,y_4\in B$, this corresponds to two disjoint rainbow 4-cycles in $G$ containing $x_1x_2$ and $y_1y_2$, respectively.
}\label{fig:fourcycle}
\end{figure}
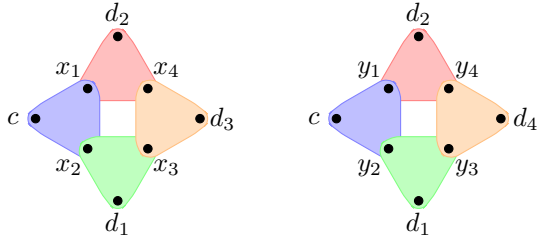

For each $f=y_1y_2\in F$, then, letting $\mathcal{R}_f$ be the set of sequences $(x_3,x_4,y_3,y_4,d_1,d_2,d_3,d_4)$ such that $c,x_1,x_2,y_1,y_2,x_3,x_4,y_3,y_4,d_1,d_2,d_3,d_4$ are distinct, except we may have $d_3=d_4$, $x_4=d_4$ or $y_4=d_3$, and such that
\[
x_2x_3d_1,x_4x_1d_2,x_3x_4d_3,y_2y_3d_1,y_4y_1d_2,y_3y_4d_4\in S,
\]
we have $|\mathcal{R}_f|\geq n^2/2$.



Note that, over all $f\in F$, the sets $\mathcal{R}_f$ are disjoint. Let $\mathcal{R}=\cup_{f\in F}\mathcal{R}_f$.
Noting that we have fixed $c$, $x_1$ and $x_2$, observe the following.
The sequence $(x_3,x_4,y_3,y_4,d_1,d_2,d_3,d_4)\in \mathcal{R}$ is determined by $x_4$ (which fixes $d_2$), $y_4$ (which fixes $y_1$ and $y_2$) and $d_4$ (which fixes $y_3$, $d_1$ and $x_3$), and therefore there are at most $n^2$ such sequences with $x_4=d_4$. Similarly, such a sequence is determined by $y_4$, $y_4$ and $d_3$, so that there are at most $n^2$ such sequences with $y_4=d_3$.

Therefore, for all but at most $\sqrt{n}$ $f\in F$, we must have that there are at least $n^2/8$ sequences $(x_3,x_4,y_3,y_4,d_1,d_2,d_3,d_4)\in \mathcal{R}_f$ with $x_3,x_4,y_3,y_4,d_1,d_2,d_3,d_4$ distinct except for, possibly, $d_3=d_4$. Let $F'\subset F$ be the set of such $f$ and, for each $f\in F'$, let $\mathcal{R}'_f\subset R_f$ be the corresponding set of these sequences. For each $f\in F'$, let $X_{e,f}$ be the number of $(x_3,x_4,y_3,y_4,d_1,d_2,d_3,d_4)\in \mathcal{R}'_f$ with $x_3,y_3\in A$, $x_4,y_4\in B$ and $d_1,d_2,d_3,d_4\in C$. The sequence  $(x_3,x_4,y_3,y_4,d_1,d_2,d_3,d_4)\in \mathcal{R}_f$ is determined uniquely by the specification of one vertex from $\{d_1,x_3,y_3\}$ and one vertex from $\{x_4,y_4,d_2,d_3,d_4\}$, and therefore each vertex appears in at most $8n$ sequences in $\mathcal{R}_f$. Thus, $X_{e,f}$ is $(8n)$-Lipschitz. As $\E X_{e,f}\geq (1/3)^{8}n^2/8$, we have by Lemma~\ref{lem:mcd}, setting $p=1/3$ and $\alpha=p^{12}/10^{100}$, that $|X_{e,f}|\geq \alpha n^2$ with probability $1-o(n^{-6})$. Then, using a union bound over all $c\in S$ and distinct $e,f$ such that $V(e)\cup \{c\},V(f)\cup \{c\}\in S$, we have that \ref{prop-pseud-abs-prime} holds with probability $1-o(n^{-3})$.

\medskip

\ref{prop-pseud-add-2}: Let $u,v,c_0\in V(S)$ be distinct, and let $W\subset V(S)$ with $|W|\leq n/100$ be arbitrary. We will count the number of sequences $(x_1,x_2,x_3,x_4,d_1,d_2,d_3)$ of distinct vertices in $V(S)\setminus \{u,v,c_0\}$ which have edges in $S$ as depicted in Figure~\ref{fig:fivepath}. To start with, let $\mathcal{R}$ be the set of such sequences where we additionally allow $d_2=d_3$. Using arguments similar to those in the proof of Proposition~\ref{prop:nearcompletepseudorandom}, and labelling vertices as in Figure~\ref{fig:fivepath} once they are determined by the edges in $S$, given $u,v,c_0\in V(S)$, we have at least $n/2$ choices for $x_1$ so that $x_1$, $d_1$ and $x_2$ are not in $\{u,v,c_0\}$ (and $d_1\neq x_2$ by the STS property), and then at least $n/2$ choices for $x_4$ so that $x_4$, $d_2$, $x_3$ and $d_3$ are not in $\{u,v,c_0,x_1,x_2,d_1\}$. By the STS property we have $d_3\neq x_3$ and and $d_2\neq x_4$, and thus $|\mathcal{R}|\geq n^2/4$.

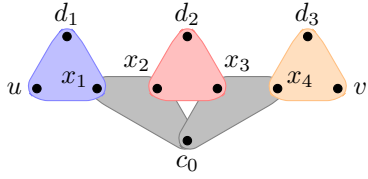
\begin{figure}
\begin{center}\begin{tikzpicture}[scale=0.8]
\def\vxrad{0.07cm}
\def\horunit{1.2}
\def\edgelength{0.4}
\def\betweenrows{0.5}


\def\gapratio{1}
\def\biggergap{2}

\foreach \num in {1,2}
{
\coordinate (A\num) at ($(0,0)+\num*4*\gapratio*(1,0)$);
\coordinate (B\num) at ($(0,0)+\num*4*\gapratio*(1,0)+\gapratio*(1,0)$);
\coordinate (C\num) at ($(0,0)+\num*4*\gapratio*(1,0)+\gapratio*(2,0)$);
\coordinate (D\num) at ($(0,0)+\num*4*\gapratio*(1,0)+\gapratio*(3,0)$);
\coordinate (E\num) at ($0.5*(B\num)+0.5*(C\num)+0.86602540*\gapratio*(0,1)$);
\coordinate (F\num) at ($0.5*(B\num)+0.5*(C\num)-0.86602540*\gapratio*(0,1)$);
}
\coordinate (G1) at ($0.5*(D1)+0.5*(A2)+0.86602540*\gapratio*(0,1)$);
\coordinate (G2) at ($0.5*(D1)+0.5*(A2)+0.86602540*\gapratio*(0,1)+4*\gapratio*(1,0)$);
\coordinate (A30) at ($(A2)+4*\gapratio*(1,0)$);
\foreach \num in {3,4}
{
\coordinate (A\num) at ($(0,0)+\num*4*\gapratio*(1,0)+\biggergap*(1,0)$);
\coordinate (B\num) at ($(0,0)+\num*4*\gapratio*(1,0)+\gapratio*(1,0)+\biggergap*(1,0)$);
\coordinate (C\num) at ($(0,0)+\num*4*\gapratio*(1,0)+\gapratio*(2,0)+\biggergap*(1,0)$);
\coordinate (D\num) at ($(0,0)+\num*4*\gapratio*(1,0)+\gapratio*(3,0)+\biggergap*(1,0)$);
\coordinate (E\num) at ($0.5*(B\num)+0.5*(C\num)+0.86602540*\gapratio*(0,1)$);
\coordinate (F\num) at ($0.5*(B\num)+0.5*(C\num)-0.86602540*\gapratio*(0,1)$);
}

\coordinate (G3) at ($0.5*(D3)+0.5*(A4)+0.86602540*\gapratio*(0,1)$);

\def\sm{0.2};
\foreach \num in {2}
{
\draw [rounded corners,black!50,fill=black!25] ($(A\num)-(0,-\sm)+0.5*(\gapratio,0)$) -- ($(A\num)-(0,-\sm)$) -- ($(A\num)-(\sm,0)$) -- ($(A\num)-(0,\sm)$) -- ($(F\num)+(0,-\sm)$) -- ($(F\num)+(\sm,0)$) -- ($(B\num)+(\sm,0)$) -- ($(B\num)+(0,\sm)$) --
($(A\num)-(0,-\sm)+0.5*(\gapratio,0)$);
}

\foreach \num in {2}
{
\draw [rounded corners,black!50,fill=black!25] ($(D\num)-(0,-\sm)-0.5*(\gapratio,0)$) -- ($(D\num)-(0,-\sm)$) -- ($(D\num)+(\sm,0)$) -- ($(D\num)-(0,\sm)$) -- ($(F\num)+(0,-\sm)$) -- ($(F\num)-(\sm,0)$) -- ($(C\num)-(\sm,0)$) -- ($(C\num)+(0,\sm)$) --
($(D\num)-(0,-\sm)-0.5*(\gapratio,0)$);
}

\foreach \num in {2}
{
\draw [rounded corners,red!50,fill=red!25] ($(C\num)+(0,-\sm)-0.5*(\gapratio,0)$) -- ($(C\num)+(0,-\sm)$) -- ($(C\num)+(\sm,0)$) -- ($(E\num)+(\sm,0)$) -- ($(E\num)+(0,\sm)$) -- ($(E\num)-(\sm,0)$) -- ($(B\num)-(\sm,0)$) -- ($(B\num)-(0,\sm)$) --
($(C\num)-(0,\sm)-0.5*(\gapratio,0)$);
}

\foreach \num/\numplus in {1/2}
{
\draw [rounded corners,blue!50,fill=blue!25] ($(A\numplus)+(0,-\sm)-0.5*(\gapratio,0)$) -- ($(A\numplus)+(0,-\sm)$) -- ($(A\numplus)+(\sm,0)$) -- ($(G\num)+(\sm,0)$) -- ($(G\num)+(0,\sm)$) -- ($(G\num)-(\sm,0)$) -- ($(D\num)-(\sm,0)$) -- ($(D\num)-(0,\sm)$) --
($(A\numplus)-(0,\sm)-0.5*(\gapratio,0)$);
}

\foreach \num/\numplus in {2/3}
{
\draw [rounded corners,orange!50,fill=orange!25] ($(A30)+(0,-\sm)-0.5*(\gapratio,0)$) -- ($(A30)+(0,-\sm)$) -- ($(A30)+(\sm,0)$) -- ($(G\num)+(\sm,0)$) -- ($(G\num)+(0,\sm)$) -- ($(G\num)-(\sm,0)$) -- ($(D\num)-(\sm,0)$) -- ($(D\num)-(0,\sm)$) --
($(A30)-(0,\sm)-0.5*(\gapratio,0)$);
}

\def\bit{0.375}
\draw ($(E2)+(0,\bit)$) node {$d_2$};
\draw ($(G1)+(0,\bit)$) node {$d_1$};
\draw ($(G2)+(0,\bit)$) node {$d_3$};
\foreach \num in {2}
\draw ($(F\num)-(0,\bit)$) node {$c_0$};

\draw ($(D1)+(-\bit,0)$) node {$u$};
\draw ($(B2)+(-0.9*\bit,1.1*\bit)$) node {$x_2$};
\draw ($(A2)+(-\bit,0.4*\bit)$) node {$x_1$};
\draw ($(C2)+(0.9*\bit,1.1*\bit)$) node {$x_3$};
\draw ($(D2)+(\bit,0.4*\bit)$) node {$x_4$};
\draw ($(A30)+(\bit,0)$) node {$v$};

\foreach \lett in {A,B,C,D,E,F}
\foreach \num in {2}
{
\draw [fill] (\lett\num) circle [radius=\vxrad];
}
\foreach \lett in {D}
\foreach \num in {1}
{
\draw [fill] (\lett\num) circle [radius=\vxrad];
}


\draw [fill] (G1) circle [radius=\vxrad];
\draw [fill] (G2) circle [radius=\vxrad];
\draw [fill] (A30) circle [radius=\vxrad];

\end{tikzpicture}
\end{center}

\vspace{-0.4cm}

\caption{The structure counted in $S$ for \ref{prop-pseud-add-2} in the proof of Theorem~\ref{thm:brouwer}, where the vertices $c$, $d_1$ and $d_2$ are repeated for clarity and we may have $d_3=d_4$. When $c_0,d_1,d_2,d_3\in C$, $u,x_2,x_4\in A$ and $x_1,x_3,v\in B$, this corresponds to a $u,v$-path in $G$ with colours $d_1,c_0,d_2,c_0,d_3$ in order.
}\label{fig:fivepath}
\end{figure}

Note that, $(x_1,x_2,x_3,x_4,d_1,d_2,d_3)\in \mathcal{R}$ is determined by $d_2$, $d_3$, so there are at most $n$ such sequences with $d_2=d_3$. Therefore, if $\mathcal{R}'$ is the set of sequences $(x_1,x_2,x_3,x_4,d_1,d_2,d_3)$ of distinct vertices in $V(S)\setminus \{u,v,c_0\}$ with
\begin{equation}\label{eqn:STS}
ux_1d_1,x_1x_2c_0,x_2x_3d_2,x_3x_4c_0,x_4vd_3\in S,
\end{equation}
then we have $|\mathcal{R}'|\geq n^2/8$.

Note further that any sequence $(x_1,x_2,x_3,x_4,d_1,d_2,d_3)\in \mathcal{R}'$ is determined by any two vertices from different sets in $\{x_1,x_2,d_1\}$, $\{d_2\}$ and $\{x_3,x_4,d_3\}$. Therefore, there are at most $7n$ sequences $(x_1,x_2,x_3,x_4,d_1,d_2,d_3)\in \mathcal{R}$ containing any one fixed vertex, and thus at most $7|W|n$ sequences with a vertex in $W$. Therefore, as $|\mathcal{R}'|\geq n^2/8>7|W|n$, we have that there are some distinct vertices  $x_1,x_2,x_3,x_4,d_1,d_2,d_3\in V(S)\setminus (W\cup \{u,v,c_0\})$ for which \eqref{eqn:STS} holds.

As $W\subset V(S)$ was arbitrary with $|W|\leq n/100$, we can find distinct $x_{i,1},x_{i,2},x_{i,3},x_{i,4},d_{i,1},d_{i,2},d_{i,3}\in V(S)\setminus \{u,v,c_0\}$, $i\in [n/800]$, such that, for each $i\in [n/800]$,
\begin{equation}\label{eqn:STS2}
ux_{i,1}d_{i,1},x_{i,1}x_{i,2}c_0,x_{i,2}x_{i,3}d_{i,2},x_{i,3}x_{i,4}c_0,x_{i,4}vd_{i,3}\in S.
\end{equation}

Now, let $X_{u,v,c_0}$ be the number of $i\in [n/800]$ with $x_{i,2},x_{i,4}\in A$, $x_{i,1},x_{i,3}\in B$ and $d_{i,1},d_{i,2},d_{i,3}\in C$, and note that if $X_{u,v,c_0}\geq \alpha n$, $u\in A$ and $v\in B$, then \ref{prop-pseud-add-2} holds with $c_0$, $u$ and $v$. As $X_{u,v,c_0}$ is a binomial random variable with parameters $n/800$ and $(1/3)^7$, we thus have, by Lemma~\ref{Lemma_Chernoff} and a union bound, that, with probability $1-o(n^{-3})$, \ref{prop-pseud-add-2} holds.

\medskip

\ref{prop-pseud-add-3}: Let $c_0,d\in V(S)$ be distinct, and let $W\subset V(S)$ with $|W|\leq n/10^3$ be arbitrary.  We will count the number of sequences $(x_1,\ldots,x_8,d_1,d_2,d_3)$ of distinct vertices in $V(S)\setminus \{c_0,d\}$ which have edges in $S$ as depicted in Figure~\ref{fig:eightcycle}. To start with, let $\mathcal{R}$ be the set of such sequences where we additionally allow $d_2=d_3$. Using arguments similar to those in the proof of Proposition~\ref{prop:nearcompletepseudorandom}, and labelling vertices as in Figure~\ref{fig:eightcycle} once they are determined by the edges in $S$, given $c_0,d\in V(S)$, we have at least $n/2$ choices for $x_1$ (determining $x_8$, $x_7$ and $x_2$), then at least $n/2$ choices for $x_3$ (determining $d_1$ and $x_4$) and then at least $n/2$ choices for $x_5$ (determining $d_2,x_6$ and $d_3$), such that $x_1,\ldots,x_8,d_1,d_2,d_3$ are all distinct except for, possibly, $d_2=d_3$. Thus, we have $|\mathcal{R}|\geq n^3/8$.

\begin{figure}
\begin{center}\begin{tikzpicture}[scale=0.8]
\def\vxrad{0.07cm}
\def\horunit{1.2}
\def\edgelength{0.4}
\def\betweenrows{0.5}


\def\gapratio{1}
\def\biggergap{2}

\foreach \num in {1,2,3,4}
{
\coordinate (A\num) at ($(0,0)+\num*4*\gapratio*(1,0)$);
\coordinate (B\num) at ($(0,0)+\num*4*\gapratio*(1,0)+\gapratio*(1,0)$);
\coordinate (C\num) at ($(0,0)+\num*4*\gapratio*(1,0)+\gapratio*(2,0)$);
\coordinate (D\num) at ($(0,0)+\num*4*\gapratio*(1,0)+\gapratio*(3,0)$);
\coordinate (E\num) at ($0.5*(B\num)+0.5*(C\num)+0.86602540*\gapratio*(0,1)$);
\coordinate (F\num) at ($0.5*(B\num)+0.5*(C\num)-0.86602540*\gapratio*(0,1)$);
}
\coordinate (G1) at ($0.5*(D1)+0.5*(A2)+0.86602540*\gapratio*(0,1)$);
\coordinate (G2) at ($0.5*(D1)+0.5*(A2)+0.86602540*\gapratio*(0,1)+4*\gapratio*(1,0)$);
\coordinate (A30) at ($(A2)+4*\gapratio*(1,0)$);

\coordinate (G3) at ($0.5*(D3)+0.5*(A4)+0.86602540*\gapratio*(0,1)$);

\def\sm{0.2};
\foreach \num in {2,3}
{
\draw [rounded corners,black!50,fill=black!25] ($(A\num)-(0,-\sm)+0.5*(\gapratio,0)$) -- ($(A\num)-(0,-\sm)$) -- ($(A\num)-(\sm,0)$) -- ($(A\num)-(0,\sm)$) -- ($(F\num)+(0,-\sm)$) -- ($(F\num)+(\sm,0)$) -- ($(B\num)+(\sm,0)$) -- ($(B\num)+(0,\sm)$) --
($(A\num)-(0,-\sm)+0.5*(\gapratio,0)$);
}

\foreach \num in {2,3}
{
\draw [rounded corners,black!50,fill=black!25] ($(D\num)-(0,-\sm)-0.5*(\gapratio,0)$) -- ($(D\num)-(0,-\sm)$) -- ($(D\num)+(\sm,0)$) -- ($(D\num)-(0,\sm)$) -- ($(F\num)+(0,-\sm)$) -- ($(F\num)-(\sm,0)$) -- ($(C\num)-(\sm,0)$) -- ($(C\num)+(0,\sm)$) --
($(D\num)-(0,-\sm)-0.5*(\gapratio,0)$);
}

\foreach \num/\coll in {2/red,3/orange}
{
\draw [rounded corners,\coll!50,fill=\coll!25] ($(C\num)+(0,-\sm)-0.5*(\gapratio,0)$) -- ($(C\num)+(0,-\sm)$) -- ($(C\num)+(\sm,0)$) -- ($(E\num)+(\sm,0)$) -- ($(E\num)+(0,\sm)$) -- ($(E\num)-(\sm,0)$) -- ($(B\num)-(\sm,0)$) -- ($(B\num)-(0,\sm)$) --
($(C\num)-(0,\sm)-0.5*(\gapratio,0)$);
}

\foreach \num/\numplus in {1/2,3/4}
{
\draw [rounded corners,blue!50,fill=blue!25] ($(A\numplus)+(0,-\sm)-0.5*(\gapratio,0)$) -- ($(A\numplus)+(0,-\sm)$) -- ($(A\numplus)+(\sm,0)$) -- ($(G\num)+(\sm,0)$) -- ($(G\num)+(0,\sm)$) -- ($(G\num)-(\sm,0)$) -- ($(D\num)-(\sm,0)$) -- ($(D\num)-(0,\sm)$) --
($(A\numplus)-(0,\sm)-0.5*(\gapratio,0)$);
}

\foreach \num/\numplus in {2/3}
{
\draw [rounded corners,green!50,fill=green!25] ($(A30)+(0,-\sm)-0.5*(\gapratio,0)$) -- ($(A30)+(0,-\sm)$) -- ($(A30)+(\sm,0)$) -- ($(G\num)+(\sm,0)$) -- ($(G\num)+(0,\sm)$) -- ($(G\num)-(\sm,0)$) -- ($(D\num)-(\sm,0)$) -- ($(D\num)-(0,\sm)$) --
($(A30)-(0,\sm)-0.5*(\gapratio,0)$);
}

\def\bit{0.375}
\draw ($(E2)+(0,\bit)$) node {$d_1$};
\draw ($(E3)+(0,\bit)$) node {$d_3$};
\draw ($(G1)+(0,\bit)$) node {$d$};
\draw ($(G2)+(0,\bit)$) node {$d_2$};
\draw ($(G3)+(0,\bit)$) node {$d$};
\foreach \num in {2,3}
\draw ($(F\num)-(0,\bit)$) node {$c_0$};

\draw ($(B2)+(-0.9*\bit,-0.4*\bit)$) node {$x_2$};
\draw ($(A2)+(-\bit,0.4*\bit)$) node {$x_1$};
\draw ($(D1)+(-1.1*\bit,-0.4*\bit)$) node {$x_8$};
\draw ($(C2)+(-\bit,0.4*\bit)$) node {$x_3$};
\draw ($(D2)+(-0.9*\bit,-0.4*\bit)$) node {$x_4$};
\draw ($(B3)+(-0.9*\bit,-0.4*\bit)$) node {$x_6$};
\draw ($(A3)+(-\bit,0.4*\bit)$) node {$x_5$};
\draw ($(C3)+(-\bit,0.4*\bit)$) node {$x_7$};
\draw ($(D3)+(-0.9*\bit,-0.4*\bit)$) node {$x_8$};
\draw ($(A4)+(-\bit,0.4*\bit)$) node {$x_1$};

\foreach \lett in {A,B,C,D,E,F}
\foreach \num in {2,3}
{
\draw [fill] (\lett\num) circle [radius=\vxrad];
}

\draw [fill] (D1) circle [radius=\vxrad];
\draw [fill] (A4) circle [radius=\vxrad];

\draw [fill] (G1) circle [radius=\vxrad];
\draw [fill] (G2) circle [radius=\vxrad];
\draw [fill] (G3) circle [radius=\vxrad];

\end{tikzpicture}
\end{center}

\vspace{-0.4cm}

\caption{The structure used in $S$ for \ref{prop-pseud-add-3} in the proof of Theorem~\ref{thm:brouwer}, where the vertex $c_0$ and the edge $x_8x_1d$ are repeated for clarity. When $c_0,d,d_1,d_2,d_3\in C$, $x_1,x_3,x_5,x_7\in A$ and $x_2,x_4,x_6,x_8\in B$, this corresponds to an $8$-cycle in $G$ with colours $d,c_0,d_1,c_0,d_2,c_0,d_3,c_0$ in order.
}\label{fig:eightcycle}
\end{figure}
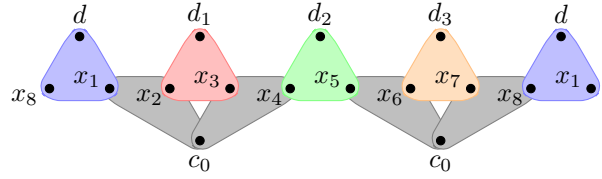

Note that $(x_1,\ldots,x_8,d_1,d_2,d_3)\in\mathcal{R}$ is determined by $d_2$, $d_3$ and $x_5$, so that there are at most $n^2$ such sequences with $d_2=d_3$.
Therefore, if $\mathcal{R}'$ is the set of sequences $(x_1,\ldots,x_8,d_1,d_2,d_3)$ of distinct vertices in $V(S)\setminus \{c_0,d\}$ such that, setting $d_0=d$,
\begin{equation}\label{eqn:STS3}
x_{2i+1}x_{2i+2}d_i\in S\text{ for each }0\leq i\leq 3,\text{ and }x_{2i}x_{2i+1}c_0\in S\text{ for each }0\leq i\leq 2,\text{ and }x_8x_1c_0\in S,
\end{equation}
then we have $|\mathcal{R}'|\geq n^3/16$.

Note that $(x_1,\ldots,x_8,d_1,d_2,d_3)\in \mathcal{R}$ is uniquely determined by any 1 vertex in $\{x_7,x_8,x_1,x_2\}$ in conjunction with any 2 vertices from different sets $\{d_1,x_3,x_4\}$, $\{d_3,x_6,x_5\}$ and $\{d_2\}$.
Therefore, there are at most $11|W|n^2$ sequences
$(x_1,\ldots,x_8,d_1,d_2,d_3)\in \mathcal{R}'$ with a vertex in $W$. Thus, as $|\mathcal{R}|\geq n^3/16>11|W|n^2$, we have that there are some distinct $x_1,x_2,x_3,x_4,d_1,d_2,d_3\in V(S)\setminus (W\cup \{u,v,c_0\})$ for which \eqref{eqn:STS} holds.

As $W\subset V(S)$ was arbitrary with $|W|\leq n/10^3$, we can find distinct vertices $x_{i,j}$, $d_{i,j'}$, $i\in [n/10^5]$, $j\in [8]$, $j'\in [3]$, in $V(S)\setminus \{c_0,d\}$ such that, for each $i\in [n/10^5]$,
\begin{equation}\label{eqn:STS4}
x_{i,2j+1}x_{i,2j+2}d_{i,j}\in S\text{ for each }0\leq j\leq 3,\text{ and }x_{i,2j}x_{i,2j+1}c_0\in S\text{ for each }0\leq j\leq 2,\text{ and }x_{i,8}x_{i,1}c_0\in S.
\end{equation}

Now, let $X_{c_0,d}$ be the number of $i\in [n/10^5]$ with $x_{i,j}\in A$ for each odd $j\in [8]$, $x_{i,j}\in B$ for each even $j\in [8]$ and $d_j\in C$ for each $j\in [3]$. Note that if $c_0,d\in C$ and $X_{c_0,d}\geq \alpha n$, then \ref{prop-pseud-add-3} holds with $c_0$ and $d$. As $X_{c_0,d}$ is a binomial random variable with parameters $n/10^5$ and $(1/3)^{11}$, we thus have, by Lemma~\ref{Lemma_Chernoff} and a union bound, that, with probability $1-o(n^{-3})$, \ref{prop-pseud-add-3} holds.

\medskip

\ref{prop-pseud-add-new-1}: Let $c_0\in V(S)$, $0\leq k\leq 20$ and $\bar{C}\subset V(S)\setminus \{c_0\}$ with $|\bar{C}|=5k$. Let $W\subset V(S)$ with $|W|\leq n/10^{50}$ be arbitrary.
Let $\mathcal{R}$ be the set of sequences $(x_0,x_1,\ldots,x_k,y_0,y_1,\ldots,y_k,d_1,\ldots,d_k)$ such that $x_0,x_1,\ldots,x_k,y_1,\ldots,y_k$ are distinct vertices in $V(S)$ and $d_1,\ldots,d_k$ are distinct vertices in $\bar{C}$, and (see Figure~\ref{fig:covereasy}) such that
\begin{equation}\label{eqn:fin}
x_iy_{i}c_0\in S\text{ for each }0\leq i\leq k,\;\text{ and }\;y_{i-1}x_{i}d_{i}\text{ for each }i\in [k].
\end{equation}
We first show that $|\mathcal{R}|\geq n$. Pick $x_0\in V(S)\setminus \{c_0\}$ and let $y_0$ be such that $x_0y_0c_0\in S$. Then, for each $i\in [k]$ in turn, do the following.
\begin{itemize}
\item Pick $d_i\in \bar{C}\setminus \{d_1,\ldots,d_{i-1}\}$, so that, letting $x_i$ and $y_i$ be such that $y_{i-1}x_{i}d_{i}\in S$ and $x_iy_ic_0\in S$, then $x_i,y_i\notin \{x_0,y_0,\ldots,x_{i-1},y_{i-1}\}$.
\end{itemize}
Note that, when choosing some $d_i$, we are avoiding at most $2i\leq 2k$ vertices for $x_i,y_i$ and each choice of $d_i$ gives a different $x_i$ and hence a different $y_i$ (and there is at most 1 choice of $d_i$ for each possible $y_i$ and each possible $x_i$), so that, as $|\bar{C}|=5k$ it is possible to choose $d_i$ for each $i\in [k]$. Thus, we have $|\mathcal{R}|\geq n$.

Note that $(x_0,x_1,\ldots,x_k,y_0,y_1,\ldots,y_k,d_1,\ldots,d_k)\in \mathcal{R}$ is uniquely determined by $d_1,\ldots,d_k$ and any other one vertex in that sequence. Therefore, there are at most $|\bar{C}|^k\cdot (2k+2)\cdot |W\cup \bar{C}|<n$ such sequences with $x_i$ or $y_i$ in $W\cup \bar{C}$ for some $0\leq i\leq k$. Thus, there is some sequence of distinct vertices $x_0,x_1,\ldots,x_k,y_0,y_1,\ldots,y_k,d_1,\ldots,d_k\in V(S)\setminus (W\cup \{c_0\})$ for which we have that \eqref{eqn:fin} holds with $x_0,x_1,\ldots,x_k,y_0,y_1,\ldots,y_k\notin \bar{C}$.

As $W$ was arbitrary with $|W|\leq 10^{50}$, we can then find sequences $x_{i,0},x_{i,1},\ldots,x_{i,k},y_{i,0},y_{i,1},\ldots,y_{i,k}$, $i\in [n/10^{52}]$, of distinct vertices in $V(S)\setminus (W\cup \{c_0\}\cup \bar{C})$ such that, for each $i\in [n/10^{52}]$,
there are some distinct $d_{i,1},\ldots,d_{i,k}\in \bar{C}$ such that \eqref{eqn:fin} holds.
Now, let $X_{c_0,k,\bar{C}}$ be the number of $i\in [n/10^{52}]$ with $x_{i,j}\in A$ and $y_{i,j}\in B$ for each $0\leq j\leq k$. Note that if $\bar{C}\subset C$, $c_0\in C$ and $X_{c_0,k,\bar{C}}\geq \alpha n$, then \ref{prop-pseud-add-new-1} holds with $c_0$, $k$ and $\bar{C}$. As $X_{c_0,k,\bar{C}}$ is a binomial random variable with parameters $n/10^{52}$ and $(1/3)^{2k+2}$, we thus have, by Lemma~\ref{Lemma_Chernoff} and a union bound, that, with probability $1-o(n^{-3})$, \ref{prop-pseud-add-new-1} holds.

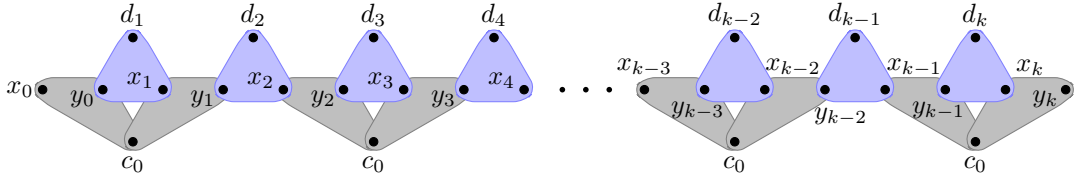
\begin{figure}
\begin{center}\begin{tikzpicture}[scale=0.8]
\def\vxrad{0.07cm}
\def\horunit{1.2}
\def\edgelength{0.4}
\def\betweenrows{0.5}


\def\gapratio{1}
\def\biggergap{2}

\foreach \num in {1,2}
{
\coordinate (A\num) at ($(0,0)+\num*4*\gapratio*(1,0)$);
\coordinate (B\num) at ($(0,0)+\num*4*\gapratio*(1,0)+\gapratio*(1,0)$);
\coordinate (C\num) at ($(0,0)+\num*4*\gapratio*(1,0)+\gapratio*(2,0)$);
\coordinate (D\num) at ($(0,0)+\num*4*\gapratio*(1,0)+\gapratio*(3,0)$);
\coordinate (E\num) at ($0.5*(B\num)+0.5*(C\num)+0.86602540*\gapratio*(0,1)$);
\coordinate (F\num) at ($0.5*(B\num)+0.5*(C\num)-0.86602540*\gapratio*(0,1)$);
}

\def\bit{0.375}

\coordinate (G1) at ($0.5*(D1)+0.5*(A2)+0.86602540*\gapratio*(0,1)$);
\coordinate (G2) at ($0.5*(D1)+0.5*(A2)+0.86602540*\gapratio*(0,1)+4*\gapratio*(1,0)$);
\coordinate (A30) at ($(A2)+4*\gapratio*(1,0)$);
\foreach \num in {3,4}
{
\coordinate (A\num) at ($(0,0)+\num*4*\gapratio*(1,0)+\biggergap*(1,0)$);
\coordinate (B\num) at ($(0,0)+\num*4*\gapratio*(1,0)+\gapratio*(1,0)+\biggergap*(1,0)$);
\coordinate (C\num) at ($(0,0)+\num*4*\gapratio*(1,0)+\gapratio*(2,0)+\biggergap*(1,0)$);
\coordinate (D\num) at ($(0,0)+\num*4*\gapratio*(1,0)+\gapratio*(3,0)+\biggergap*(1,0)$);
\coordinate (E\num) at ($0.5*(B\num)+0.5*(C\num)+0.86602540*\gapratio*(0,1)$);
\coordinate (F\num) at ($0.5*(B\num)+0.5*(C\num)-0.86602540*\gapratio*(0,1)$);
}

\coordinate (G3) at ($0.5*(D3)+0.5*(A4)+0.86602540*\gapratio*(0,1)$);

\def\sm{0.2};
\foreach \num in {1,2,3,4}
{
\draw [rounded corners,black!50,fill=black!25] ($(A\num)-(0,-\sm)+0.5*(\gapratio,0)$) -- ($(A\num)-(0,-\sm)$) -- ($(A\num)-(\sm,0)$) -- ($(A\num)-(0,\sm)$) -- ($(F\num)+(0,-\sm)$) -- ($(F\num)+(\sm,0)$) -- ($(B\num)+(\sm,0)$) -- ($(B\num)+(0,\sm)$) --
($(A\num)-(0,-\sm)+0.5*(\gapratio,0)$);
}

\foreach \num in {1,2,3,4}
{
\draw [rounded corners,black!50,fill=black!25] ($(D\num)-(0,-\sm)-0.5*(\gapratio,0)$) -- ($(D\num)-(0,-\sm)$) -- ($(D\num)+(\sm,0)$) -- ($(D\num)-(0,\sm)$) -- ($(F\num)+(0,-\sm)$) -- ($(F\num)-(\sm,0)$) -- ($(C\num)-(\sm,0)$) -- ($(C\num)+(0,\sm)$) --
($(D\num)-(0,-\sm)-0.5*(\gapratio,0)$);
}

\foreach \num in {1,2,3,4}
{
\draw [rounded corners,blue!50,fill=blue!25] ($(C\num)+(0,-\sm)-0.5*(\gapratio,0)$) -- ($(C\num)+(0,-\sm)$) -- ($(C\num)+(\sm,0)$) -- ($(E\num)+(\sm,0)$) -- ($(E\num)+(0,\sm)$) -- ($(E\num)-(\sm,0)$) -- ($(B\num)-(\sm,0)$) -- ($(B\num)-(0,\sm)$) --
($(C\num)-(0,\sm)-0.5*(\gapratio,0)$);
}

\foreach \num/\numplus in {1/2,3/4}
{
\draw [rounded corners,blue!50,fill=blue!25] ($(A\numplus)+(0,-\sm)-0.5*(\gapratio,0)$) -- ($(A\numplus)+(0,-\sm)$) -- ($(A\numplus)+(\sm,0)$) -- ($(G\num)+(\sm,0)$) -- ($(G\num)+(0,\sm)$) -- ($(G\num)-(\sm,0)$) -- ($(D\num)-(\sm,0)$) -- ($(D\num)-(0,\sm)$) --
($(A\numplus)-(0,\sm)-0.5*(\gapratio,0)$);
}

\foreach \num/\numplus in {2/3}
{
\draw [rounded corners,blue!50,fill=blue!25] ($(A30)+(0,-\sm)-0.5*(\gapratio,0)$) -- ($(A30)+(0,-\sm)$) -- ($(A30)+(\sm,0)$) -- ($(G\num)+(\sm,0)$) -- ($(G\num)+(0,\sm)$) -- ($(G\num)-(\sm,0)$) -- ($(D\num)-(\sm,0)$) -- ($(D\num)-(0,\sm)$) --
($(A30)-(0,\sm)-0.5*(\gapratio,0)$);
}

\draw ($(A1)-1*(\bit,0)$) node {$x_0$};
\draw ($(B1)+(-0.9*\bit,-0.4*\bit)$) node {$y_0$};
\draw ($(D1)+(-0.9*\bit,-0.4*\bit)$) node {$y_1$};
\draw ($(B2)+(-0.9*\bit,-0.4*\bit)$) node {$y_2$};
\draw ($(D2)+(-0.9*\bit,-0.4*\bit)$) node {$y_3$};

\draw ($(B3)+(-0.3*\bit,-0.9*\bit)$) node {$y_{k-3}$};
\draw ($(D3)+(0.7*\bit,-1.2*\bit)$) node {$y_{k-2}$};
\draw ($(B4)+(-0.3*\bit,-0.9*\bit)$) node {$y_{k-1}$};
\draw ($(D4)+(-0.9*\bit,-0.4*\bit)$) node {$y_k$};

\draw ($(A3)+(-0*\bit,1*\bit)$) node {$x_{k-3}$};
\draw ($(C3)+(1.25*\bit,1*\bit)$) node {$x_{k-2}$};
\draw ($(A4)+(1.3*\bit,1*\bit)$) node {$x_{k-1}$};
\draw ($(C4)+(1.1*\bit,1*\bit)$) node {$x_k$};

\draw ($(C1)+(-\bit,0.4*\bit)$) node {$x_1$};
\draw ($(A2)+(-\bit,0.4*\bit)$) node {$x_2$};
\draw ($(C2)+(-\bit,0.4*\bit)$) node {$x_3$};
\draw ($(A30)+(-\bit,0.4*\bit)$) node {$x_4$};

\def\bit{0.375}
\draw ($(E1)+(0,\bit)$) node {$d_1$};
\draw ($(E2)+(0,\bit)$) node {$d_3$};
\draw ($(E3)+(0,\bit)$) node {$d_{k-2}$};
\draw ($(E4)+(0,\bit)$) node {$d_k$};
\draw ($(G1)+(0,\bit)$) node {$d_2$};
\draw ($(G2)+(0,\bit)$) node {$d_4$};
\draw ($(G3)+(0,\bit)$) node {$d_{k-1}$};
\foreach \num in {1,2,3,4}
\draw ($(F\num)-(0,\bit)$) node {$c_0$};

\foreach \lett in {A,B,C,D,E,F}
\foreach \num in {1,2,3,4}
{
\draw [fill] (\lett\num) circle [radius=\vxrad];
}

\draw ($0.5*(A30)+0.5*(A3)+(0.1,0)$) node {\scalebox{2}{\ldots}};

\draw [fill] (G1) circle [radius=\vxrad];
\draw [fill] (G2) circle [radius=\vxrad];
\draw [fill] (G3) circle [radius=\vxrad];
\draw [fill] (A30) circle [radius=\vxrad];

\end{tikzpicture}
\end{center}

\vspace{-0.4cm}

\caption{The structure counted for \ref{prop-pseud-add-new-2} in the proof of Theorem~\ref{thm:brouwer}, where the vertex $c_0$ is repeated for clarity, and the central line of vertices are, in order, $x_{0},y_0,x_1,y_1,\ldots,x_k,y_k$, so that $x_iy_ic_0\in S$ for each $i\in [k]_0$ and $x_{i-1}y_id_i\in S$ for each $i\in [k]$.
}\label{fig:covereasy}
\end{figure}

\medskip

\ref{prop-pseud-add-new-2}: Let $k=100$ and $r=n/10^8$. Let $c_0\in C(G)$. We will randomly choose disjoint sets $V_i$, $W_i$ and $C_i$, $i\in [r]$, in $V(S)\setminus \{c_0\}$, each with size $k$, such that, for most $i\in [r]$, if $V_i\subset A$, $W_i\subset B$ and $C_i\cup \{c_0\}\subset C$, then $G[V_i\cup W_i]$ has an exactly-$C_i$-rainbow matching and a matching of $k$ colour-$c_0$ edges (and then, for \ref{prop-pseud-add-new-2}, consider some of the sets $V_i\cup W_i$ and $C_i$, $i\in [r]$).

 For each $i=1,\ldots,r$ in turn, choose $V_i$, $W_i$ and $C_i$ as follows.
\begin{itemize}
\item Let $Z_i=\{c_0\}\cup (\cup_{i'<i}(V_i\cup W_i\cup C_i))$. Pick an arbitrary $v_{i,1}\in V(S)\setminus Z_i$ which is not in an edge with $c_0$ and 1 vertex from $Z_i$, and let $w_{i,1}$ be such that $v_{i,1}w_{i,1}c_0\in S$. For each $j=2,\ldots,k$, pick $v_{i,j}\in V(S)\setminus Z_i\cup\{v_{i,1},w_{i,1},\ldots,v_{i,j-1},w_{i,j-1},d_{i,1},\ldots,d_{i,j-2}\}$
 uniformly at random from such $v_{i,j}$ for which, letting $d_{i,j-1}$ and $w_{i,j}$ be such that $w_{i,j-1}v_{i,j}d_{i,j-1}\in S$ and $v_{i,j}w_{i,j-1}c_0\in S$, we have that $d_{i,j-1},w_{i,j}$ are not in $Z_i\cup\{v_{i,1},w_{i,1},\ldots,v_{j-1},w_{j-1},d_{i,1},\ldots,d_{i,j-2}\}$.
 Finally, let $d_{i,k}$ be such that $w_{i,k}v_{i,1}d_{i,k}\in S$. Let $V_i=\{v_{i,1},\ldots,v_{i,k}\}$, $W_i=\{w_{i,1},\ldots,w_{i,k}\}$ and $C_i=\{d_{i,1},\ldots,d_{i,k}\}$, noting that these are disjoint sets with size $k$.
\end{itemize}
Note that, for any sequence $(\bar{d}_1,\ldots,\bar{d}_{k-1})$ of vertices (possibly with repetition), at step $i$, the probability that $d_{i,j}=\bar{d}_{j}$ for each $j\in [k-1]$ is at most $1/(n-10(|W|+3k))^{k-1}\geq 1/(n-100kr)^{k-1}$.
We will show that, with high probability, the following holds.
\stepcounter{propcounter}
\begin{enumerate}[label = \textbf{\Alph{propcounter}}]
\item For every $\bar{C}\subset V(S)\setminus \{c_0\}$ with $|\bar{C}|\leq k$, for at least $r/2$ values of $i\in [r]$ (those in $I_{\bar{C}}$, say)joint sets $\bar{V}_1,\ldots,\bar{V}_{\alpha n},\bar{W}_1,\ldots,\bar{W}_{\alpha n}\subset V(S)\setminus (\bar{C}\cup\{c_0\}\cup V_i\cup W_i\cup C_i)$ of size $k+|\bar{C}|+1$ such that, letting $\ell=|\bar{C}|$, for each $i\in [\alpha n]$, we can label the vertices of $\bar{V}_i$, $\bar{W}_i$, and $\bar{C}\cup C_i$ respectively as $\{a_1,\ldots,a_{k+\ell+1}\}$ and $\{b_1,\ldots,b_{k+\ell+1}\}$, and  $\{c_1,\ldots,c_{k+\ell}\}$ so that, for each $j\in [k+\ell+1]$, $c_0a_jb_j\in S$ and, for each $j\in [k+\ell+1]$, $b_ja_{j+1}c_j\in S$.\label{thisthing}
\end{enumerate}

Thus, $V_i$, $W_i$ and $C_i$, $i\in [r]$, can be chosen in this fashion so that \ref{thisthing} holds. Then, selecting the partition $[n]=A\cup B\cup C$ as above, consider the set $I\subset[r]$ of $i\in [r]$ for which $V_i\subset A$, $W_i\subset B$ and $C_i\subset C$. By Lemma~\ref{Lemma_Chernoff} and a union bound, with probability $1-o(n^{-3})$, for each $\bar{C}\subset C(G)\setminus \{c_0\}$ with $|\bar{C}|\leq k$, we have $|I\cap I_{\bar{C}}|\geq \alpha n$ and, for each $i\in I_{\bar{C}}$,
taking the sets $\bar{V}_1,\ldots,\bar{V}_{\alpha n},\bar{W}_1,\ldots,\bar{W}_{\alpha n}$ from \ref{thisthing}, for at least $\alpha^2n$ of $j\in [\alpha n]$ we will have $\bar{V}_j\subset A$ and $\bar{W}_j\subset B$. Thus, with probability $1-o(n^{-3})$, \ref{prop-pseud-add-new-2} holds for $G$ by using $V_i'=V_i\cup W_i$ and $C_i$, $i\in I$.

Therefore it is sufficient to show that, with high probability, \ref{thisthing} holds. Fix then $\bar{C}\subset V(S)\setminus \{c_0\}$ with $|\bar{C}|\leq k$. For ease of notation assume that $|\bar{C}|=k$, where the other cases with $|\bar{C}|<k$ follow similarly (indeed, as looking at Figure~\ref{fig:cover}, we use only that there is at least 1 blue edge between red edges in the sequence). We will show that, for each $i\in [r]$, when $V_i$, $W_i$ and $C_i$ are chosen, with probability $1-o(1)$ we have that \ref{thisthing} holds for $i$. Then, using Lemma~\ref{lem:mcd}, we can show that, with probability $1-o(n^{-k})$, \ref{thisthing} holds for $\bar{C}$.

Let $\lambda=10^{-3}$. Label the vertices in $\bar{C}$ as $c_1,\ldots,c_k$. Let $\mathcal{L}_{\bar{C}}$ be the set of sequences $(d_1,\ldots,d_{k-1})$ of distinct vertices in $V(S)\setminus(\{c_0\}\cup \bar{C})$ for which there are at least $\lambda n$ choices for $x_{1,1}$ for which, as in Figure~\ref{fig:cover}, there are distinct vertices
\begin{equation}\label{lasteq}
x_{1,2},x_{1,3},x_{1,4},x_{2,1},x_{2,2},x_{2,3},x_{2,4},\ldots ,x_{k-1,1},x_{k-1,2},x_{k-1,3},x_{k-1,4},x_{k,1},x_{k,2},x_{k,3},x_{k,4}
\end{equation}
in $V(S)\setminus (\bar{C}\cup \{d_1,\ldots,d_{k-1}\})$ such that, for each $i\in [k]$, $x_{i,1}x_{i,2}c_0,x_{i,2}x_{i,3}c_i,x_{i,3}x_{i,4}c_0\in S$, and, for each $i\in [k-1]$, $x_{i,4}x_{i+1,1}d_i\in S$.

\begin{claim}$|\mathcal{L}_{\bar{C}}|\geq 99n^{k-1}/100$.\label{finalclaim}
\end{claim}

\begin{figure}
\begin{center}\begin{tikzpicture}[scale=0.8]
\def\vxrad{0.07cm}
\def\horunit{1.2}
\def\edgelength{0.4}
\def\betweenrows{0.5}


\def\gapratio{1}
\def\biggergap{2}

\foreach \x in {1,2}
{
\foreach \y in {1,2,3,4}
{
\foreach \num/\labb/\labbb in {\y/\x/\y}
{
\draw ($(0,-0.6)+(0,0)+1*3*\gapratio*(1,0)+\gapratio*(\num,0)+(0,2*\gapratio)+(\x*4*\gapratio-4*\gapratio,0)$) node {$x_{\labb,\labbb}$};
\draw [->] ($(0,-0.6)+(0,0)+1*3*\gapratio*(1,0)+\gapratio*(\num,0)+(0,2*\gapratio)+(\x*4*\gapratio-4*\gapratio,0)-(0,0.2)$) -- ++(0,-0.6);
}
}
}
\foreach \x in {3}
{
\foreach \y in {1}
{
\foreach \num/\labb/\labbb in {\y/\x/\y}
{
\draw ($(0,-0.6)+(0,0)+1*3*\gapratio*(1,0)+\gapratio*(\num,0)+(0,2*\gapratio)+(\x*4*\gapratio-4*\gapratio,0)$) node {$x_{\labb,\labbb}$};
\draw [->] ($(0,-0.6)+(0,0)+1*3*\gapratio*(1,0)+\gapratio*(\num,0)+(0,2*\gapratio)+(\x*4*\gapratio-4*\gapratio,0)-(0,0.2)$) -- ++(0,-0.6);
}
}
}
\foreach \x/\labbbb in {4.5/k}
{
\foreach \y in {1,2,3,4}
{
\foreach \num/\labb/\labbb in {\y/\labbbb/\y}
{
\draw ($(0,-0.6)+(0,0)+1*3*\gapratio*(1,0)+\gapratio*(\num,0)+(0,2*\gapratio)+(\x*4*\gapratio-4*\gapratio,0)$) node {$x_{\labb,\labbb}$};
\draw [->] ($(0,-0.6)+(0,0)+1*3*\gapratio*(1,0)+\gapratio*(\num,0)+(0,2*\gapratio)+(\x*4*\gapratio-4*\gapratio,0)-(0,0.2)$) -- ++(0,-0.7);
}
}
}

\foreach \x/\labbbb in {3.5/k-1}
{
\foreach \y in {1,3}
{
\foreach \num/\labb/\labbb in {\y/\labbbb/\y}
{
\draw ($(0,-0.6)+(0,0)+1*3*\gapratio*(1,0)+\gapratio*(\num,0)+(0,2*\gapratio)+(\x*4*\gapratio-4*\gapratio,0)$) node {$x_{\labb,\labbb}$};
\draw [->] ($(0,-0.6)+(0,0)+1*3*\gapratio*(1,0)+\gapratio*(\num,0)+(0,2*\gapratio)+(\x*4*\gapratio-4*\gapratio,0)-(0,0.2)$) -- ++(0,-0.7);
}
}
}

\foreach \x/\labbbb in {3.5/k-1}
{
\foreach \y in {2,4}
{
\foreach \num/\labb/\labbb in {\y/\labbbb/\y}
{
\draw ($(0,-0.6)+(0,0.4)+(0,0)+1*3*\gapratio*(1,0)+\gapratio*(\num,0)+(0,2*\gapratio)+(\x*4*\gapratio-4*\gapratio,0)$) node {$x_{\labb,\labbb}$};
\draw [->] ($(0,-0.6)+(0,0.4)+(0,0)+1*3*\gapratio*(1,0)+\gapratio*(\num,0)+(0,2*\gapratio)+(\x*4*\gapratio-4*\gapratio,0)-(0,0.2)$) -- ++(0,-1.1);
}
}
}

\foreach \num in {1,2}
{
\coordinate (A\num) at ($(0,0)+\num*4*\gapratio*(1,0)$);
\coordinate (B\num) at ($(0,0)+\num*4*\gapratio*(1,0)+\gapratio*(1,0)$);
\coordinate (C\num) at ($(0,0)+\num*4*\gapratio*(1,0)+\gapratio*(2,0)$);
\coordinate (D\num) at ($(0,0)+\num*4*\gapratio*(1,0)+\gapratio*(3,0)$);
\coordinate (E\num) at ($0.5*(B\num)+0.5*(C\num)+0.86602540*\gapratio*(0,1)$);
\coordinate (F\num) at ($0.5*(B\num)+0.5*(C\num)-0.86602540*\gapratio*(0,1)$);
}
\coordinate (G1) at ($0.5*(D1)+0.5*(A2)+0.86602540*\gapratio*(0,1)$);
\coordinate (G2) at ($0.5*(D1)+0.5*(A2)+0.86602540*\gapratio*(0,1)+4*\gapratio*(1,0)$);
\coordinate (A30) at ($(A2)+4*\gapratio*(1,0)$);
\foreach \num in {3,4}
{
\coordinate (A\num) at ($(0,0)+\num*4*\gapratio*(1,0)+\biggergap*(1,0)$);
\coordinate (B\num) at ($(0,0)+\num*4*\gapratio*(1,0)+\gapratio*(1,0)+\biggergap*(1,0)$);
\coordinate (C\num) at ($(0,0)+\num*4*\gapratio*(1,0)+\gapratio*(2,0)+\biggergap*(1,0)$);
\coordinate (D\num) at ($(0,0)+\num*4*\gapratio*(1,0)+\gapratio*(3,0)+\biggergap*(1,0)$);
\coordinate (E\num) at ($0.5*(B\num)+0.5*(C\num)+0.86602540*\gapratio*(0,1)$);
\coordinate (F\num) at ($0.5*(B\num)+0.5*(C\num)-0.86602540*\gapratio*(0,1)$);
}

\coordinate (G3) at ($0.5*(D3)+0.5*(A4)+0.86602540*\gapratio*(0,1)$);

\def\sm{0.2};
\foreach \num in {1,2,3,4}
{
\draw [rounded corners,black!50,fill=black!25] ($(A\num)-(0,-\sm)+0.5*(\gapratio,0)$) -- ($(A\num)-(0,-\sm)$) -- ($(A\num)-(\sm,0)$) -- ($(A\num)-(0,\sm)$) -- ($(F\num)+(0,-\sm)$) -- ($(F\num)+(\sm,0)$) -- ($(B\num)+(\sm,0)$) -- ($(B\num)+(0,\sm)$) --
($(A\num)-(0,-\sm)+0.5*(\gapratio,0)$);
}

\foreach \num in {1,2,3,4}
{
\draw [rounded corners,black!50,fill=black!25] ($(D\num)-(0,-\sm)-0.5*(\gapratio,0)$) -- ($(D\num)-(0,-\sm)$) -- ($(D\num)+(\sm,0)$) -- ($(D\num)-(0,\sm)$) -- ($(F\num)+(0,-\sm)$) -- ($(F\num)-(\sm,0)$) -- ($(C\num)-(\sm,0)$) -- ($(C\num)+(0,\sm)$) --
($(D\num)-(0,-\sm)-0.5*(\gapratio,0)$);
}

\foreach \num in {1,2,3,4}
{
\draw [rounded corners,red!50,fill=red!25] ($(C\num)+(0,-\sm)-0.5*(\gapratio,0)$) -- ($(C\num)+(0,-\sm)$) -- ($(C\num)+(\sm,0)$) -- ($(E\num)+(\sm,0)$) -- ($(E\num)+(0,\sm)$) -- ($(E\num)-(\sm,0)$) -- ($(B\num)-(\sm,0)$) -- ($(B\num)-(0,\sm)$) --
($(C\num)-(0,\sm)-0.5*(\gapratio,0)$);
}

\foreach \num/\numplus in {1/2,3/4}
{
\draw [rounded corners,blue!50,fill=blue!25] ($(A\numplus)+(0,-\sm)-0.5*(\gapratio,0)$) -- ($(A\numplus)+(0,-\sm)$) -- ($(A\numplus)+(\sm,0)$) -- ($(G\num)+(\sm,0)$) -- ($(G\num)+(0,\sm)$) -- ($(G\num)-(\sm,0)$) -- ($(D\num)-(\sm,0)$) -- ($(D\num)-(0,\sm)$) --
($(A\numplus)-(0,\sm)-0.5*(\gapratio,0)$);
}

\foreach \num/\numplus in {2/3}
{
\draw [rounded corners,blue!50,fill=blue!25] ($(A30)+(0,-\sm)-0.5*(\gapratio,0)$) -- ($(A30)+(0,-\sm)$) -- ($(A30)+(\sm,0)$) -- ($(G\num)+(\sm,0)$) -- ($(G\num)+(0,\sm)$) -- ($(G\num)-(\sm,0)$) -- ($(D\num)-(\sm,0)$) -- ($(D\num)-(0,\sm)$) --
($(A30)-(0,\sm)-0.5*(\gapratio,0)$);
}

\def\bit{0.375}
\draw ($(E1)+(0,-\bit)$) node {$c_1$};
\draw ($(E2)+(0,-\bit)$) node {$c_2$};
\draw ($(E3)+(0.1*\bit,-\bit)$) node {\footnotesize $\footnotesize c_{k-1}$};
\draw ($(E4)+(0,-\bit)$) node {$c_k$};
\draw ($(G1)+(0,-\bit)$) node {$d_1$};
\draw ($(G2)+(0,-\bit)$) node {$d_2$};
\draw ($(G3)+(0.1*\bit,-\bit)$) node {\footnotesize $d_{k-1}$};
\foreach \num in {1,2,3,4}
\draw ($(F\num)-(0,\bit)$) node {$c_0$};

\foreach \lett in {A,B,C,D,E,F}
\foreach \num in {1,2,3,4}
{
\draw [fill] (\lett\num) circle [radius=\vxrad];
}

\draw ($0.5*(A30)+0.5*(A3)+(0.1,0)$) node {\scalebox{2}{\ldots}};

\draw [fill] (G1) circle [radius=\vxrad];
\draw [fill] (G2) circle [radius=\vxrad];
\draw [fill] (G3) circle [radius=\vxrad];
\draw [fill] (A30) circle [radius=\vxrad];

\end{tikzpicture}
\end{center}

\vspace{-0.4cm}

\caption{The structure considered for \ref{prop-pseud-add-new-2} in the proof of Theorem~\ref{thm:brouwer}, where the vertex $c_0$ is repeated for clarity, and the central line of vertices, as labelled from above, are, in order, $x_{1,1},x_{1,2},x_{1,3},x_{1,4},x_{2,1},x_{2,2},x_{2,3},x_{2,4},x_{3,1},\ldots ,x_{k-1,1},x_{k-1,2},x_{k-1,3},x_{k-1,4},x_{k,1},x_{k,2},x_{k,3},x_{k,4}$.
}\label{fig:cover}
\end{figure}
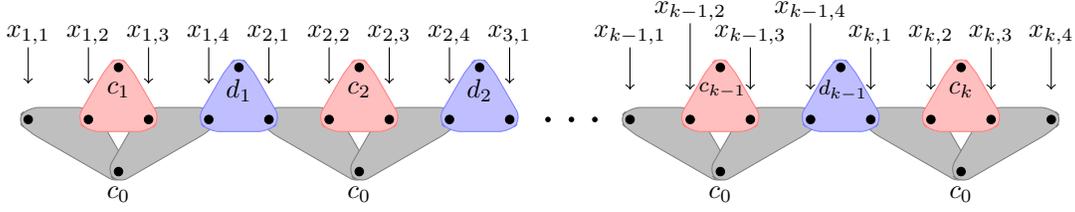

This claim will allow us to complete the proof. Indeed, assuming Claim~\ref{finalclaim} holds,  for each $i\in [k]$, at step $i$, the  probability that $(d_{i,1},\ldots,d_{i,k-1})\in \mathcal{L}_{\bar{C}}$ is chosen is at least
\[
1-(n^{k-1}-|\mathcal{L}_{\bar{C}}|)\cdot \frac{1}{(n-100kr)^{k-1}}\geq \frac{3}{4}.
\]
Therefore, by Azuma's inequality (Theorem~\ref{azuma}), we have that, with high probability, for each $\bar{C}$, that $(d_{i,1},\ldots,d_{i,k-1})\in \mathcal{L}_{\bar{C}}$, for at least $r/2$ values of $i\in [r]$, say those in $I_{\bar{C}}$. Now, for each $i\in I$, there is a set $U_i$ with $|U_i|\geq \lambda n$ such that, setting $x_{1,1}=u$, the vertices at \eqref{lasteq} exist in $V(S)\setminus (\bar{C}\cup \{d_{i,1},\ldots,d_{i,k-1}\})$  such that, for each $j\in [k]$, $x_{j,1}x_{j,2}c_0,x_{j,2}x_{j,3}c_j,x_{j,3}x_{j,4}c_0\in S$, and, for each $i\in [k-1]$, $x_{j,4}x_{j+1,1}d_{i,j}\in S$. As any such vertex in the sequence at \eqref{lasteq} determines all the others and $x_{1,1}$ (for these values of $d_{i,1},\ldots,d_{i,k-1}$), we can take a subset $U'_i\subset U_i$ with $|U'|\geq \alpha n$ such that the corresponding sequences are all disjoint, and therefore taking $\bar{V}_j$ and $\bar{W}_j$ to be the alternating vertices along the sequence corresponding to $j\in U$, we have that \ref{thisthing} holds. Thus, with high probability, \ref{thisthing} holds, as required.
Therefore it is left only to prove Claim~\ref{finalclaim}.

\begin{proof}[Proof of Claim~\ref{finalclaim}]
Let $\mathcal{R}$ be the set of sequences
\begin{equation*}\label{lasteqq}
x_{1,1}x_{1,2},x_{1,3},x_{1,4},x_{2,1},x_{2,2},x_{2,3},x_{2,4},\ldots ,x_{k-1,1},x_{k-1,2},x_{k-1,3},x_{k-1,4},x_{k,1},x_{k,2},x_{k,3},x_{k,4}
\end{equation*}
of disjoint vertices such that, for each $i\in [k]$, $x_{i,1}x_{i,2}c_0,x_{i,2}x_{i,3}c_i,x_{i,3}x_{i,4}c_0\in S$. Choosing one vertex from each block $\{x_{i,1},x_{i,2},x_{i,3},x_{i,4}\}$, $i\in [k]$, determines the sequence in $\mathcal{R}$, and, for each $i\in [k]$, as long we avoid the vertex in an edge with $c_i$ and $c_0$ in $S$, the vertices $x_{i,1},x_{i,2},x_{i,3},x_{i,4}$ are distinct by the STS property. Therefore, for each $i\in [k]$, there are at least $n-6$ choices for $x_{i,1}$ such that $x_{i,1},x_{i,2},x_{i,3},x_{i,4}$ are distinct if $x_{i,1}x_{i,2}c_0,x_{i,2}x_{i,3}c_i,x_{i,3}x_{i,4}c_0\in S$. Furthermore, for each $i,j\in [k]$ with $i\neq j$, if $x_{i,1},x_{i,2},x_{i,3},x_{i,4}$ are chosen, then there are at most 16 choices for $x_{j,1}$ such that if  $x_{j,1},x_{j,2},x_{j,3},x_{j,4}$ can be defined so that $x_{i,1}x_{i,2}c_0,x_{i,2}x_{i,3}c_i,x_{i,3}x_{i,4}c_0\in S$, then $\{x_{j,1},x_{j,2},x_{j,3},x_{j,4}\}$ and $\{x_{i,1},x_{i,2},x_{i,3},x_{i,4}\}$ overlap. Thus,
\[
|\mathcal{R}|\geq \prod_{i=1}^k(n-6-16(i-1))\geq (n-16k)^k.
\]
On the other hand, by the definition of $\mathcal{L}_{\mathcal{C}}$, and as, for each sequence in $\mathcal{R}$ at \eqref{lasteqq}, there is exactly one sequence of vertices $(d_1,\ldots,d_{k-1})$ such that, for each $i\in [k-1]$, $x_{i,4}x_{i+1,1}d_i\in S$, we have
\[
|\mathcal{R}|\leq |\mathcal{L}_{\mathcal{C}}|\cdot n+(n^{k-1}-|\mathcal{L}_{\mathcal{C}}|)\cdot \lambda n,
\]
so that \renewcommand{\qedsymbol}{$\boxdot$}
\[
|\mathcal{L}_{\mathcal{C}}|\geq \frac{|\mathcal{R}|-\lambda\cdot n^k}{(1-\lambda)n}\geq \frac{(n-16k)^k-\lambda\cdot n^k}{(1-\lambda)n}\geq \frac{99 n^{k-1}}{100}.\qedhere
\]
\end{proof}
\renewcommand{\qedsymbol}{$\square$}

This completes the proof that, with positive probability for sufficiently large $n$, $G$ is properly $(m,p,\eps)$-properly-pseudorandom for $p=1/3$ and some $\eps$ with $1/n\llpoly \eps \llpoly \log^{-1}n$ (and thus $1/m\llpoly \eps\llpoly \log^{-1}m$ as $m=n/3$). Therefore, by Theorem~\ref{thm-technical}, there is a choice of $A,B,C$ for which $G$ has a rainbow matching with $m-1$ edges, and, therefore $S$ has a matching with at least $n/3-1$ edges, as required.
\end{proof}

\section{Final remarks}\label{sec:final}

We finish by discussing the limits of our techniques and related problems. For further context and related problems see the recent surveys by Pokrovskiy~\cite{Alexeysurvey} and the author~\cite{Mysurvey}.

\medskip

\noindent\textbf{Slightly strengthening Theorem~\ref{thm:generalLS}.} Best, Pula and Wanless~\cite{best2021small} have conjectured that, for all $n$, any proper colouring of $K_{n,n-1}$ should contain a rainbow matching with at least $n-1$ edges. In other words, they conjecture a version of Theorem~\ref{thm:generalLS} (but for all $n$) where one of the vertices to be omitted has already been deleted. As observed by Georgakopoulos~\cite{georgakopoulos2013delay}, this is equivalent to a special case of a conjecture of Haxell, Wilfong and Winkler (see also~\cite{alon2007edge}). The methods introduced here prove the conjecture of Best, Pula and Wanless for large $n$, though we only sketch this very briefly. Given a graph $G$ which is a properly coloured copy of $K_{n,n-1}$, Theorem~\ref{thm-farnoworry} would give an $(n-1)$-edge rainbow matching, unless the colouring is close enough to an optimal colouring that \ref{isthisadagger} holds, so that the original proof method in Section~\ref{subsec:discuss} can be attempted. Doing so, we could then run the addition structure as far as possible, until there was 1 vertex, $y$ say, not incorporated, but stop at the last step where we want to find the path described using colours $D\cup D'$. Instead of running step iii) to get our final remainder vertices, we would then drop out an additional set $D''$ of four colours whose edges in $M^{\mathrm{rb}}$ are the vertex set of a set of 4 colour-$c_0$ edges, and find a path like that in step iii), but starting from $y$ and using the colours in $D\cup D'$ and all but one colour in $D'''$ while alternating between colour-$c_0$ edges. Then, when we exchange the edges of this path into $M^{\mathrm{rb}}$ and out of $M^{\mathrm{id}}$, we bring $y$ into $M^{\mathrm{rb}}$ and drop out only one remainder vertex (the other end of this path).  Effectively, instead of using two remainder vertices at the very end, this uses one remainder vertex and one remainder colour, to serve the same purpose.

\medskip

\noindent\textbf{Counting near-transversals.} As mentioned in Section~\ref{sec:intro}, Eberhard, Manners and Mrazovi\'c~\cite{greenalites} have determined precisely the asymptotics of the number of full transversals in a Latin Squares of order $n$ if it is the multiplication table of a group $G$ with trivial or non-cyclic 2-Sylow subgroups. In particular, the number of full transversals is $(e^{-1/2}+o(1))|G^{ab}|(n!)^2/n^n$, where $G^{ab}$ is the abelianisation of $G$. Our methods cannot get such tight asymptotics, though the semi-random method can show that many different almost-transversals exists. Using this in combination with the methods introduced here should show that any Latin square of order $n$ has $\exp(\Theta(n\log n))$ transversals with $n-1$ elements.

\medskip

\noindent\textbf{Full transversals when $n$ is odd.} Having proved Theorem~\ref{thm:RBSeven}, it is natural to consider the obstacles of using our methods to prove the Ryser-Brualdi-Stein conjecture for large odd $n$.
It seems likely that any similar methods to those used here that can prove the Ryser-Brualdi-Stein conjecture for large odd $n$ must determine any possible structure behind the colouring much more closely than has been achieved here. That more can be done in this direction is clear: the addition structure greatly simplified the initial approach used in this work (some of these unused ideas will appear in~\cite{BBM}). However, given the `approximate' nature of any such structure, proving the Ryser-Brualdi-Stein conjecture for large odd $n$ with this approach appears to be extremely challenging.

\medskip

\noindent\textbf{Full transversals in special cases.}  The Hall-Paige conjecture (as discussed in Section~\ref{sec:intro}) goes further than simply the group multiplication table special case of the Ryser-Brualdi-Stein conjecture, giving a condition which determines exactly when the corresponding Latin square will have a full transversal, or not. All the Latin squares without a full transversal described here have some precise underlying algebraic structure inherited from a group whose multiplication table has no full transversal. Though no doubt extremely difficult, it seems likely that any Latin square with no full transversal has such a precise underlying algebraic structure.
In this direction, due to Kwan~\cite{kwan2020almost} it is known that almost all Latin squares of order $n$ have a full transversal, who showed moreover that a typical Latin square of order $n$ contains $\left((1-o(1))\frac{n}{e^2}\right)^n$ full transversals using combinatorial tools. (Similarly, due to Kwan~\cite{kwan2020almost}, we also know that a random Steiner triple system is likely to contain many matchings missing at most 1 vertex.) Eberhard, Manners and Mrazovi\'c~\cite{eberhard2023transversals} have improved this with an independent analytic approach, to show that a typical random Latin square has $\left(e^{-1/2}+o(1)\right)(n!)^2/n^n$ transversals.

\medskip

\noindent\textbf{Rainbow paths in properly coloured complete graphs.} The natural analogous problem for complete graphs is the following. Given a properly coloured $n$-vertex complete graph, how long a rainbow cycle (or path) must it contain? In 1989, Andersen~\cite{andersen1989hamilton} conjectured that every properly coloured $n$-vertex complete graph should have a rainbow path with length $n-2$. Using methods of Alon, Pokrovskiy and Sudakov~\cite{alon2017random}, Balogh and Molla~\cite{balogh2019long} showed that there will always be such a path with length $n-O(\sqrt{n}\log n)$. The methods introduced here do not directly apply to this problem, in particular there is no direct analogue for the addition structure used here, but it appears significant progress can be made and this is the subject of forthcoming work by the current author with Benford and Bowtell~\cite{BBM}.





\bibliographystyle{abbrv}
\bibliography{rbs}
\end{document}